\documentclass[smfbook]{book}

\usepackage[a4paper]{geometry}
\usepackage[latin1]{inputenc}
\usepackage[T1]{fontenc}

\usepackage{lipsum} 

\usepackage[english,francais]{babel}

\usepackage{lmodern}
\rmfamily
\DeclareFontShape{T1}{lmr}{b}{sc}{<->ssub*cmr/bx/sc}{}
\DeclareFontShape{T1}{lmr}{bx}{sc}{<->ssub*cmr/bx/sc}{}

\usepackage{amssymb,amsmath,amsthm,amscd}
\usepackage{mathrsfs}

\usepackage{subfig}

\usepackage{tabularx} 
\usepackage{calc} 
\usepackage{graphicx} 

\usepackage{hyperref}

\usepackage[nottoc]{tocbibind}

\usepackage{makeidx}
\makeindex

\usepackage{fancyhdr}
\makeatletter
\let\ps@plain=\ps@empty
\makeatother

\let\origdoublepage\cleardoublepage
\newcommand{\clearemptydoublepage}{%
  \clearpage
  {\pagestyle{empty}\origdoublepage}%
}
\let\cleardoublepage\clearemptydoublepage

\FrenchFootnotes 
\AddThinSpaceBeforeFootnotes 

\theoremstyle{plain} 

\theoremstyle{definition} 

\renewcommand{\proofname}{Démonstration}

\theoremstyle{plain} 
\newtheorem{theorem}{Theorem}[chapter]

\newtheorem{corollary}[theorem]{Corollary}
\newtheorem{proposition}[theorem]{Proposition}
\newtheorem{lemma}[theorem]{Lemma}

\theoremstyle{definition} 
\newtheorem{definition}[theorem]{Definition}
\newtheorem{remark}[theorem]{Remark}

\newcommand{\N}{\mathbb{N}}
\newcommand{\Z}{\mathbb{Z}}
\newcommand{\Q}{\mathbb{Q}}
\newcommand{\R}{\mathbb{R}}
\newcommand{\C}{\mathbb{C}}
\newcommand{\Sp}{\mathbb{S}}
\newcommand{\T}{\mathbb{T}}
\newcommand{\TT}{\mathcal{T}}
\newcommand{\E}{\mathcal{E}}
\newcommand{\ft}{\mathfrak{t}}
\newcommand{\ba}{\mathbf{a}}
\newcommand{\fZ}{\mathfrak{Z}}
\newcommand{\bk}{\mathbf{k}}
\def\a{\alpha}
\def\p{\pi}
\def\r{\rho}
\def\s{\sigma}
\def\t{\tau}
\def\D{\Delta}
\def\x{\chi}
\def\e{\varepsilon}
\def\f{\varphi}
\def\.{\cdot}
\def\ra{\rightarrow}
\def\hra{\hookrightarrow}
\def\g+{g_{+}}
\def\g0{g_{0}}

\providecommand{\norm}[1]{\lVert#1\rVert} 

\title{Global aspects of the reducibility of quasiperiodic cocycles in compact Lie groups}
\author{Nikolaos Karaliolios \footnote{Affiliation: CNRS, IMJ-PRG,
UMR 7586, Univ. Paris Diderot, Sorbonne Paris Cit\'{e}, Sorbonne Universit\'{e}s, UPMC Univ Paris 06, F-75013, Paris, France.
} \footnote{email: nkaraliolios@gmail.com} }

\begin{document}





\pdfbookmark[0]{Résumé}{resume}

\chapter*{Résumé}
 \begin{center} \Large \bf Aspects globaux de la
 r\'{e}ductibilit\'{e} des cocycles quasi-p\'{e}riodiques \`{a}
 valeurs dans des groupes de Lie compacts semi-simples\end{center}
 \begin{center}
Author : Nikolaos Karaliolios \footnote{Affiliation: CNRS, IMJ-PRG,
UMR 7586, Univ. Paris Diderot, Sorbonne Paris Cit\'{e}, Sorbonne Universit\'{e}s, UPMC Univ Paris 06, F-75013, Paris, France.
} \footnote{email: nkaraliolios@gmail.com}
\end{center}
\section*{Résumé}
Ce mémoire porte sur l'étude des cocycles quasi-périodiques à valeurs dans des groupes de Lie compacts semi-simples.
Nous nous restreindrons au cas des cocycles à une fréquence. Nous démontrons que, pour un ensemble de fréquences de mesure de Lebesgue
pleine, l'ensemble des cocycles $C^{\infty }$ qui sont $C^{\infty }$-réductibles sont
$C^{\infty }$-denses. De plus, sous la même condition arithmétique, nous démontrerons que tout cocycle (quitte à
l'itérer afin de simplifier suffisamment l'homotopie du lacet dans le groupe), est presque tore-réductible
(c'est-à-dire qu'il peut être conjugué arbitrairement proche à des cocycles prenant valeurs dans un sous-groupe
abelien donné de $G$).

Le premier pas de la démonstration sera l'obtention de deux invariants de la dynamique, qu'on appellera
\textit{énergie} et \textit{degré}, qui distinguent en particulier les cocycles (presque-)réductibles des cocycles
non-réductibles. On entamera ensuite la démonstration du théorème principal. Nous démontrons dans un second temps
qu'un algorithme dit de \textit{renormalisation} permet de ramener l'étude de tout cocycle à celle des perturbations
de modèles simples indexés par le degré. Nous analysons ensuite ces perturbations par des méthodes inspirées de la
\textit{théorie K.A.M.}.

\subsubsection*{Mots-clefs}

Théorie K.A.M. , Renormalisation , Cocycles Quasi-périodiques , Groupes de Lie Compacts Semi-simples

\begin{otherlanguage}{english}

  \vspace{1cm}

  \begin{center} \rule{\textwidth/3}{1pt} \end{center}

  \vspace{1cm}

  \begin{center} \Large \bf Global aspects of the reducibility of quasiperiodic cocycles in semisimple compact Lie groups \end{center}
  \section*{Abstract}

In this mémoire we study quasiperiodic cocycles in semi-simple compact Lie groups. For the greatest part of our study,
we will focus ourselves to one-frequency cocyles. We will prove that $C^{\infty }$-reducible cocycles are dense in the
$C^{\infty }$ topology, for a full measure set of frequencies. Moreover, we will show that every cocycle (or an
appropriate iterate of it, if homotopy appears as an obstruction) is almost torus-reducible (i.e. can be conjugated
arbitrarily close to cocycles taking values in an abelian subgroup of G). In the course of the proof we will firstly
define two invariants of the dynamics, which we will call \textit{energy} and \textit{degree} and which give a
preliminary distinction between (almost-)reducible and non-reducible cocycles. We will then take up the proof of the
density theorem. We will show that an algorithm of \textit{renormalization} converges to perturbations of simple
models, indexed by the degree. Finally, we will analyze these perturbations using methods inspired
by \textit{K.A.M.} theory.

  \subsection*{Keywords} K.A.M. theory , Renormalization , Quasiperiodic Cocycles , Compact Semisimple Lie Groups

\end{otherlanguage}

\setcounter{tocdepth}{1} 
\pdfbookmark[0]{Table des matières}{tablematieres} 

\tableofcontents  


i\selectlanguage{english}
\chapter{Introduction}

This mémoire is concerned with the study of quasiperiodic cocycles in semisimple compact Lie groups. Cocycles
are discrete dynamical systems whose phase space is a fibered space $X \times E \ra X$. Fibered dynamics is
given by the iteration of a mapping of the type
\begin{equation*}
\begin{array}{cccc}
(T,f): & X \times E & \ra & X \times E \\ 
& (x ,e) & \mapsto  & (Tx, f(x,e))
\end{array}
\end{equation*}
where $T $ is a mapping of $X$ into itself, and $f : X \times E \ra E$. Consequently, the fiber $\{ x \} \times E$
is mapped into the fiber $\{ Tx \} \times E$ following $e \mapsto f(x,e) $. The notation $SW(X,E)$ for the set
of such dynamical systems is classical. If $E$ is a group or a space on which a group acts, this kind of fibered
dynamics is called a cocycle. We can then note a cocycle by
\begin{equation*}
\begin{array}{cccc}
(T,f): & X \times E & \ra & X \times E \\ 
& (x ,e) & \mapsto  & (Tx, f(x).e)
\end{array}%
\end{equation*}
with $f : X \ra E$ and the dot $.$ stands for the group multiplication or action.

The $n$-th iterate of the cocycle $(T,f)$, $n \geq 1$, is of the form
\begin{equation*}
(T,f)^{n}.(x,e) = (T^{n}x,f(T^{n-1}x ) \circ \cdots \circ f(x).e)
\end{equation*}
We say that two cocycles over the same transformation $\psi _{i} = (T,f_{i}) \in SW(X,E)$, $i=1,2$, are
(semi-)\textit{conjugate} iff there exists $g : X \ra E$ such that
\begin{equation*}
\psi _{1} \circ (id, g) = (id, g) \circ \psi _{2}
\end{equation*}
and we remark that it is a notion stronger than that of dynamical (semi-)conjugation by a mapping
$h : X\times E \circlearrowleft $ satisfying $\psi _{1} \circ h = h \circ \psi _{2}$, since conjugation of cocycles
preserves the fibered-space structure of $X\times E$.

In general contexts we suppose that $(X, \mu )$, the \textit{basis} of the dynamics, is a measured space and that
$T$ is ergodic with respect to $\mu$. A particular case, which brings us to the subject of our study, occurs when
$X = \T ^{d} = \R ^{d} / \Z ^{d}$, a $d$-dimensionan torus, and $T = R_{\a } : x \mapsto x+ \a $, is a minimal
translation, and therefore uniquely ergodic with respect to the Haar measure on the torus. These cocycles are
called \textit{quasiperiodic} and $\a $ is called the \textit{frequency} of the cocycle.

We remark here that, depending on the structure of $E$, the
\textit{fibers}, we can define measurable cocycles, or $C^{k}$-differentiable
cocycles, with $k \in \N \cup \{ \infty , \omega \}$ ($\omega $ stands for real analytic), according to the
regularity of the mapping $f$.

Cocycles in linear groups come up naturally in dynamical systems. For example, if $\f$ is a diffeomorphism of the
torus $\T ^{d}$, its differential defines a cocycle on $T \T^{d} \approx \T^{d} \times \R ^{d} $ in a natural way by
\begin{equation*}
(x,y) \mapsto (\f (x) , D \f (x).y)
\end{equation*}

An other class of examples, closer to our subject, is that of fibered linear flows. Such a flow is defined as the
fundamental solution of the system of ODEs
\begin{eqnarray*}
X' &= &F( \theta ). X \\
\theta ' &= & \omega = ( \a ,1) \in \T ^{d+1}
\end{eqnarray*}
where $F : \T ^{d+1} \ra g$, and $g$ is a matrix algebra in $M_{N} (\R )$. The map of first return in the
vertical circle $\T \hra \T ^{d} \times \T $ is a cocycle on $\{ 0 \ \times \T ^{d} \times \R^{N} \} $.
It is quasiperiodic if $\{ \R \omega \mod \Z^{d+1} \} $ is dense in $\T ^{d+1}$.

When $\a $ is rational, such flows are perfectly understood thanks to the Floquet representation of solutions:
The solutions of the system of ODEs
\begin{eqnarray} \label{flot lineaire}
X' = U( \theta ).X \\
\theta ' = k \in \Q ^{d}
\end{eqnarray}
are of the form $X(t)=B(kt + \theta _{0}).e^{t U_{0} (\theta _{0})}.X_{0} $ where $B( \. )$ is a
$2 \Z ^{d}$-periodic map in the matrix group of algebra $g$. For cocycles, the Floquet representaiton of a
solution corresponds to conjugation of a cocycle to a constant one: a cocycle $(R_{\a },f)$ is called constant if
$f:\T ^{d} \ra G$ is a constant mapping. Such cocycles are called \textit{reducible}.

This normal form theorem breaks down when $\a $ is irrational, and the goal of the theory is to examine
the density properties of Floquet-type solutions, and the possibility of approximation
of any given vector field with a field admitting Floquet-type solutions. We now restrict ourselves to the
case of cocycles in $\T ^{d} \times G $, where $G$ is a semisimple compact Lie group, such as $SU(N)$ or $SO(N)$.

The first step is the study of vector fields close to constants, or cocycles close to constants. In this setting
the breakdown of Floquet theory is attributed to \textit{small divisor} phenomena. The use of K.A.M. machinery
allows, however, to give an affirmative answer to both questions, for a full measure set of frequencies
(called Diophantine) (see \cite{KrikAst} and the references therein). Moreover, Floquet representations occur in full
measure sets for generic one-parameter families of cocycles. On the other hand, we know that the Floquet
representation does break down for generic cocycles (cf \cite{El2002a}).

The second step was carried out in a first time by R. Krikorian in his article \cite{Krik2001} on
one-frequency cocycles in $SU(2)$. There, under a full measure condition on the frequency,
the density of reducible cocycles is shown to be true when $G=SU(2)$.
However, not all cocycles can be conjugated arbitrarily close to a reducible cocycle (which amounts to the
approximation by Floquet-type solutions in the formulation above). It
is however shown that a weaker form of approximation can be obtained
: every cocycle can be conjugated arbitrarily close to a cocycle taking values in the subgroup
of diagonal matrices. Cocycles conjugated to diagonal ones are called \textit{torus-reducible} and those
who can be conjugated arbitrarily close to diagonal cocycles are caled \textit{almost torus-reducible}.

The goal of this mémoire is to generalize these results to more general contexts. Here, some complications
due to homotopy have to be faced.
\begin{theorem}
For $a.e. $ frequency in $\T $, and for any given semisimple compact Lie group $G$, reducible cocycles are
dense in $SW^{\infty ,1} (\T ,G)$, the subspace of $SW^{\infty} (\T ,G)$ formed by cocycles homotopic to
$Id : \T \times G \circlearrowleft$.
\end{theorem}
For cocycles non-homotopic to the identity, we obtain
\begin{theorem}
Under the same hypotheses, if a cocycle $(\a ,A(\. )) \in SW^{\infty ,m} (\T ,G)$, the subspace of cocycles whose
$m$-th iterate is homotopic to the $Id$, $(\a ,A(\. ))^{m} \in SW^{\infty ,1} (\T ,G)$ and it is accumulated
by reducible cocycles in $SW^{\infty ,1} (m \T ,G)$.
\end{theorem}
where we remark that the approximation holds for cocycles of longer periods. The second part of R. Krikorian's
theorem also admits a generalization.
\begin{theorem}
Under the same hypotheses on $\a $ and $G $, every cocycle in $SW^{\infty ,1} (\T ,G)$ is almost torus-reducible.
\end{theorem}
again, for non-homotopic cocycles we have to consider an iterate:
\begin{theorem}
Under the same hypotheses on $\a $ and $G $, for every cocycle $(\a ,A(\. )) \in SW^{\infty ,m} (\T ,G)$,
there exists $\x$ dividing $m$ such that $(\a ,A(\. ))^{\x } \in SW^{\infty ,m / \x} (\T ,G)$
is almost torus-reducible.
\end{theorem}
We remark that almost torus-reducibility holds without passing to longer periods, but only by iterating.

The proof follows the line of argument of R. Krikorian's article cited above, but we also use some techniques
appearing in \cite{Fra2004}, and adapt them to the more general algebraic setting. A brief presentation of the
argument is as follows.

A theorem by R. Krikorian stating that a cocycle is $C^{\infty}$-reducible iff the sequence of its iterates is
precompact in $C^{\infty}$. This fact allows us to construct examples of cocycles that are not reducible,
and in fact not even almost reducible. These are the periodic geodesics of the group over any irrational rotation.
Such examples demostrate the fact that almost reducibility is not a global phenomenon, and explains the use
of almost torus-reducibility in its place.

Subsequently, the renormalization scheme allows us to identify such cocycles as the only models of cocycles for
which we can rule out almost reducibility. Disregarding homotopy issues, if there is no way of conjugating a
given cocycle close to a constant one, then it can be conjugated close to a periodic geodesic.

Finally, we take up the local study of the different types of models of dynamics, where a diophantine condition
on the rotation is needed so that K.A.M. techniques can be applicable. The local study of periodic geodesics
over such rotations shows that almost reducibility to the model holds in a small enough neighbourhood, under
the condition that the scheme is initiated on a cocycle obtained as a limit of renormalization. Moreover,
we show that cocycles who are driven by renormalization towards shorter geodesics (or constant cocycles)
are dense in a neighbourhood of these models. Finally, we prove the almost reducibility theorem for cocycles
close to constants, and reprove the theorem of local density of reducible cocycles.

The global almost torus-reducibility theorem has now been proved case-by-case, and the global density theorem
follows from a finite induction argument. If the cocycle can be conjugated to a neighbourhood of a constant cocycle,
it is accumulated by reducible ones. If not, it is almost torus-reducible, and the corresponding periodic geodesic
is accumulated by cocycles conjugate close to shorter periodic geodesics. Since the periodic geodesics form a
discrete set, after a finite number of similar steps and arbitrarily small perturbations we obtain a cocycle
that can be conjugated close to constant ones, and the theorem has been proved.

\textbf{Acknowledgment}: This is a slightly improved version of my
PhD thesis, supervised by prof. Rapha\"{e}l Krikorian. The main part of the work was carried out in the LPMA of
Universit\'{e} Paris VI.
This work was partially supported by the ERC Starting
Grant "Quasiperiodic". We would like to thank the referee for reading
this lengthy manuscript and for pointing out numerous corrections.


\selectlanguage{english}
\chapter{Basic definitions and tools from Algebra and Analysis}
\nocite{MilMor}
\nocite{RudR&C}
\nocite{ShafaAlg}
\nocite{AK2006}
\nocite{AK2011}
\nocite{Gev18}
\nocite{KhanSin87}
\nocite{DolKrik2007}
\nocite{HalmErg}
\nocite{SiegMosCelMec}
\nocite{AnKat1970}

\section{Algebra}

\subsection{Compact Lie groups}

For this section, we refer mainly to \cite{DieudElV}, whose notations we have
adapted, and \cite{Helg}, as well as \cite{BumpLieGroups}, \cite{DuisterKolkLieGroups} and \cite{KrikAst}.
By $G$ we will denote a real connected compact semi-simple Lie group, and
sometimes refer to it simply as a compact group, even though some of the results
presented below are true in a more general context.

\subsubsection{Generalities}

Let $G$ be a real connected compact Lie group, furnished with a
Haar measure, both left- and right-invariant, inducing an invariant
Riemannian metric $d(\.  ,\.  )$. The following theorem by Gleason-Montgomery-Zipplin supplies
a sufficient condition of algebraic and topological nature under which
a topological group is a Lie group.

\begin{theorem}
Let $G^{\prime }$ be a locally compact topological group satisfying the additional
hypothesis of non-existence of small subgroups. Then $G^{\prime }$ is a Lie
group.
\end{theorem}

The hypothesis of non-existence of small subgroups is the following: there
exists a neighborhood of the $Id$ which contains only the trivial
topological subgroup ${Id}$. The study of real compact groups
is capital in the theory of Lie groups, as explained in \cite%
{DieudElV} and their study is to a large extent reduced to that of semisimple ones
\footnote{Semisimplicity will be defined later on by means of a geometric criterion.}.
Under these hypotheses, $G$ admits a representation in a unitary
group. For this reason, with the exception of this first chapter,
we will abuse the notation $A^{\ast }$ for the inverse of $A\in G$.

The tangent space to the identity of $G$, furnished with the linear
antisymmetric commutator $[\.  ,\.  ]:g\times g\ra g$, is
by definition its Lie algebra $g$. The endomorphisms of $g$ preserving its structure as a Lie
algebra are denoted by $Aut(g)$, while the linear transformations of the
vector space $g$ are denoted by $GL(g)$.

The exponential mapping
\begin{equation*}
\begin{array}{cccc}
\exp : & g & \ra & G \\ 
& h & \mapsto & \exp (h)=e^{h}
\end{array}
\end{equation*}
is defined as the time one of the flow generated by the right-invariant
vector field associated to $h$. As a result of the compactness of the group,
the exponential mapping is onto.

We can define the adjoint representation of the group into its algebra as
the differential at the $Id$ of the mapping $B\mapsto A.B.A^{-1}$, for $A\in
G$ fixed. We thus obtain a linear automorphism%
\begin{equation*}
\begin{array}{cccc}
Ad_{A}: & h & \mapsto & \frac{d}{dt}A.B(\.  ).A^{-1}|_{t=0}%
\end{array}%
\end{equation*}%
where $B(\.  )$ is defined locally around $0$ and such that $\frac{d}{dt}%
B(\.  )|_{t=0}=h_{Id}$ and $B(0)= Id$. The image of $G$ in $Aut(g)$ under its adjoint
representation is noted by $Inn(g)$ and called the \textit{inner automorphisms} of $g$%
. It is a subgroup of $Aut(g)$ since $[Ad_{A}.h,Ad_{A}.h^{\prime
}]=Ad_{A}.[h,h^{\prime }]$ holds for all $h$ and $h^{\prime }$ in $g$. We
also have that $A\in Z_{G}$, if, and only if, $%
Ad(A)=Id_{g}$. Therefore, $Inn(g)$ is naturally isomorphic to $G/Z_{G}$,
which is the Lie group that has the same Lie algebra as $G$, but whose center
is reduced to the $Id$.

We can also define the adjoint representation of $g$ into itself via the
commutator: for $\,h$ fixed, let $ad_{h}:h^{\prime }\longmapsto \lbrack
h,h^{\prime }]$ and the two representations are connected by $%
Ad(e^{h}).h^{\prime }=e^{ad(h)}.h^{\prime }$, for all $h,h^{\prime }\in g$.
This representation induces a bilinear symmetric form on $g$, called the
\textit{Cartan-Killing form}, by
\begin{equation*}
\left\langle h,h^{\prime }\right\rangle =-tr(ad_{h}\circ ad_{h^{\prime }})
\end{equation*}
The group and the algebra are \textit{semi-simple} if the Cartan-Killing form is
non-degenerate, and it can be proved that a semi-simple group is compact if
it is \textit{positive definite}, in which case the algebra is said to be compact.
The Riemannian structure induced by a non-degenerate Cartan-Killing form is
compatible with the one induced by the Haar measure.

The center of $G$, denoted by $Z_{G}$ is finite and its cardinal will be
denoted by $c_{G}$.

The adjoint action of $G$ is an isometry of $g$ with respect to the
Cartan-Killing form and the latter satisfies identically
\begin{equation*}
\left\langle \lbrack h,h^{\prime }],h^{\prime \prime }\right\rangle
=\left\langle h,[h^{\prime },h^{\prime \prime }]\right\rangle
\end{equation*}%
We will use the following theorem, from \cite{DieudElV} (\textit{21.6.9}).

\begin{theorem}
\label{Dieud Theorem Comp Gr}Let $G$ be a connected real Lie group (not necessarily compact or semisimple), $Z_{G}$
its center and $g$ its Lie algebra. Then, the following conditions are equivalent:

\begin{enumerate}
\item The group $G/Z_{G}$ is compact

\item The group $G$ is isomorphic to a product $\R^{m}\times G_{1}$, where $G_{1}$ is compact

\item The group $\tilde{G}$, the universal covering of $G$, is isomorphic to
a product $\R^{n}\times K$, where $K$ is a compact semi-simple simply connected Lie group.

\item The Lie algebra $g$ is decomposed in $\mathfrak{c} \oplus \mathfrak{D}(g)$, where $\mathfrak{c}$ is its center and $\mathfrak{D}(g)$ is the
derived algebra $[g,g]$. The restriction of the Cartan-Killing form to $\mathfrak{D}(g)$ is positive non-degenerate.

Under these assumptions, $\mathfrak{D}(g)$ is isomorphic to the Lie algebra
of $K$; its center $Z_{K}$ is finite; $G$ is isomorphic to $\tilde{G}/D$,
where $D$ is a discrete subroup of $\R^{n}\times Z_{K}$, $Z_{G}$ is isomorphic to $(\R^{n}\times Z_{K})/D$ and
if we call $  \check{G} = G/Z_{G}$, we have $Z_{\check{G}}=\{ Id \}$. The groups $Ad(\tilde{G})$, $Ad(G) $
and $Ad(K)$, viewed as subroups of $Aut(g)$ are all equal and isomorphic to $G/Z_{G} = Inn(G)$.
\end{enumerate}
\end{theorem}

\subsubsection{Tori and maximal tori}

For this section, see \cite{DieudElV}, \textit{21.7}. From this section on, the group is supposed compact semisimple,
even though some of the results hold if we drop the semisimpleness assumption.

Let us consider a connected abelian subgroup $\mathcal{H}\subset G$. Since $G$ is
compact, $\mathcal{H}$ is in fact homomorphic to a torus $\R^{d} / \Z^{d}$ and for this reason we will call it a torus of $G$. 
Using standard arguments, we can find $\mathcal{T}$, a maximal abelian subgroup of $G$ (called a \textit{maximal torus}
\footnote{Since the maximal tori of matrix groups such as $SU(w+1)$ are formed by diagonal matrices, we will sometimes
abuse the word diagonal and its derivatives.}), containing $\mathcal{H}$. Maximality is understood in the sense of
set-theoretical inclusion. We have then
\begin{proposition}
\begin{enumerate}
\item \label{Dieud Theorem Tori}Let $\mathcal{T}$ be a maximal torus of $G$.
Then, for any $A\in G$, there exists $S\in G$ such that $S.A.S^{-1 }\in 
\mathcal{T}$. In particular, $G$ agrees set-theoretically with the union of
its maximal tori.

\item If $\mathcal{T}$ and $\mathcal{T}^{\prime }$ are maximal tori of $G$,
there exists $S\in G$ such that $S.\mathcal{T}^{\prime }.S^{-1 }=\mathcal{T}$.

\item The centralizer $Z_{G}(\mathcal{H})$ of a torus $\mathcal{H}\subset G$
is connected and, if $S\in Z_{G}(\mathcal{H})$, then $\mathcal{H\cup }\{S\}$
is contained in a maximal torus.

\item The neutral component $Z_{G}^{0}(S)$ of $Z_{G}(S)$ (the centralizer of $S\in G$) is
equal to the union of the maximal tori containing $S$.
\end{enumerate}
\end{proposition}

By \textit{2}, the dimension of maximal tori of $G$ is a constant
depending only on the group. It is called the \textit{rank} of $G$ and we will
denote it by $w$.

We now give the following definition, concerning the number of maximal tori
containing an element of $G$.

\begin{definition}
An element in $S\in G$ is called \textit{regular} if it is contained in a unique
maximal torus, and \textit{generic} if it generates it topologically, i.e. if the
closure of $\{ S^{k} \} _{k \in \Z} $ is the torus itself. It will be called
\textit{singular} if it is not regular.
\end{definition}

Genericity is a generic condition, just as minimality is generic in the space of translations in $\T ^{w}$.
Therefore, $S$ is regular iff the neutral component of $Z_{G}(S)$ is a
maximal torus and a generic element is a fortiori regular. An example of
a regular but not generic element can be constructed by using any root of $1$ in $\Sp ^{1}$ other than $-1$.
We refer the reader to section \textit{1.1.2} for the canonical embedding
$\Sp ^{1} \hookrightarrow SU(2)$ which gives an example of such an element in
a Lie group. More concretely, we consider elements of $SU(2)$ of the type
\begin{equation*}
\begin{bmatrix}
\exp (2i \pi p/q) & 0 \\ 
0 & \exp (-2i \pi p/q)
\end{bmatrix}
\end{equation*}
with $p,q \in \Z$, $q \neq 0$ and $\dfrac{p}{q} \neq 1 , \frac{1}{2} $. It can be verified that if $S$ is of this form,
it commutes only with diagonal matrices (which are the maximal tori in $SU(2)$, as we will see in subsection \textit{1.1.2}),
but $S^{q} = Id$, so that $S$ is not generic.

\bigskip

The normalizer of a maximal torus $N_{G}(\mathcal{T})=\{S\in G,~S.\mathcal{T}.S^{-1 }\subset \mathcal{T}\}$
in general strictly contains the torus. The
factor group $W(\mathcal{T})=N_{G}(\mathcal{T})/\mathcal{T}$ is a finite
group, called the \textit{Weyl group} with respect to $\mathcal{T}$. Since any two
tori are obtained by conjugation, all the groups $W(\mathcal{T})$ are all
isomorphic one to another and therefore isomorphic to the Weyl group of $G$,
noted $W$. For any torus $\mathcal{H}$, the group $W_{\mathcal{H}}=N_{G}(%
\mathcal{H})/Z_{G}(\mathcal{H})$ is a subgroup of $W$ and we have equality
if $\mathcal{H}$ is maximal. In the particular case where $S\in
G $ is given, the intersection of all maximal tori containing $S$ is a torus 
$\mathcal{H}$ of $G$, and the group $W_{\mathcal{H}}$ is isomorphic to $W$
if, and only if, $\mathcal{H}$ is a maximal torus, i.e. if $s $ is regular.

We also reprove the following lemma, a proof of which using techniques from Algebraic Geometry
can be found in \cite{KrikAst} (lemma 2.3.14).

\begin{lemma}
\label{parties abeliennes}Let $G$ be a real compact semi-simple Lie group.
Then, there exists a minimal positive integer $\chi _{G}$, depending only on the
group $G$ and smaller than $\# W !$ with the following property. If $K$, a subset of $G$, is abelian,
then%
\begin{equation*}
K^{\chi _{G}}=\{ k^{\chi _{G}},k\in K \}
\end{equation*}%
is a subset of a maximal torus of $G$. The integer $\chi _{G}$ is equal to
$1 $ if $G=SU(w+1)$.
\end{lemma}

\begin{proof}
Let $A,B \in G$ such that $A B A^{-1} B^{-1} = Id$ and suppose that there
does not exist a maximal torus passing by both $A$ and $B$.

We can fix, therefore, a maximal torus $\mathcal{T} _{A}$ containing $A$.
Since $B \notin \mathcal{T}_{A}$, there exists $ \tilde{A} \in \mathcal{T}_{A}$ such that
\begin{equation*}
 B \tilde{A} B^{-1}\tilde{A}^{-1} =C' \neq Id
\end{equation*}
where $C' \in Z_{G}^{0}(A)$. For this, one only need consider a path $\Phi (\. ) : [0,1] \ra \TT _{A}$ connecting
$\tilde{A}$ to the $Id$ and to verify that then $\Phi (\. ) B \Phi (\. )^{-1} B^{-1}$ commutes with $A$ at all times,
while it connects $ C' $ with the $Id$. In particular, each generic element $\tilde{A}$ is such.

We also remark that, by the commutation relation $A B A^{-1} B^{-1} = Id$,
we find directly that such a $B$ defines a whole class $B . Z_{G}^{0}(A) $ of elements commuting with $A$
but not on the same torus as $A$.

Since $C = C'.\tilde{A} \in Z_{G}^{0}(A)$ and generic (it is conjugate to $\tilde{A}$ by $B$), it defines a unique
maximal torus $\TT _{C}$. By point \textit{2} of Proposition \ref{Dieud Theorem Comp Gr} applied to the group
$Z_{G}^{0}(A)$ gives the existence of an $S \in Z_{G}^{0}(A)$ such that $S^{-1}.\TT _{C}.S = \TT _{A} $,
which implies that
\begin{equation*}
S.B. \TT _{A}. B^{-1}.S^{-1} = \TT _{A}
\end{equation*}
Therefore, $S.B$ is in $W (\TT _{A})$, and does not depend on the maximal torus $\TT _{A} $. Therefore,
$Ad(S.B) | _{\TT _{A}}$ is  of finite order, so that $(S.B)^{n} \in \TT _{A}$ for some $n$ not
bigger than $\# W$. Thus, $ B^{n} \in Z_{G}^{0} (A)$, which implies that $Ad(B) $ preserves some maximal
torus passing by $A$, and therefore $B^{n}$ is in this torus.
\end{proof}

For any given group $G$, the constant $\chi _{G}$ of this lemma is to be chosen optimal. We also prove

\begin{lemma}
If $\tilde{G}$ is a semi-simple compact group and $G=\tilde{G}/K$ where $K\subset Z_{\tilde{G}}$
is a non-trivial subgroup, then $\chi _{\tilde{G}} \leq \chi _{G}$, and $\chi _{\tilde{G}}$ divides 
$\chi _{G}$.
\end{lemma}

\begin{proof}
For the first part, we use the fact that the commutator group of a
semisimple compact Lie group is the group itself, and we let
$\tilde{A} , \tilde{B} \in \tilde{G}$ such that
$\tilde{A} .\tilde{B}.\tilde{A}^{-1 }.\tilde{B}^{-1 }=S\in K$, $S \not= Id$.
From the commutation relation it follows that, if we call $\pi$ the projection
$ \tilde{G} \ra G$, then $\pi(\tilde{A}) \not= \pi(\tilde{B})$. On the other hand,
$\pi(\tilde{A})$ and $\pi(\tilde{B})$ commute in $G$.

Consider now $\TT _{\tilde{A}}\subset \tilde{G}$, a maximal torus containing $\tilde{A}$.
Then, $\TT_{\tilde{A}}. \tilde{B} \not= \TT _{\tilde{A}}$,
$(\TT _{\tilde{A}}.\tilde{B})/K \not= \TT _{\tilde{A}}/K$,
$\pi (\TT _{\tilde{A}}.\tilde{B})$ and $\pi (\tilde{A})$ commute.

Therefore, there exist elements in $G$ who commute, but whose preimages in
$\tilde{G}$ do not.

The second statement follows from the fact that if $\tilde{A} , \tilde{B} \in \tilde{G}$
do not commute, but $A= \p (\tilde{A})$ and $B = \p ( \tilde{B})$ do commute in $G$, then
$\tilde{A} .\tilde{B}.\tilde{A}^{-1 }.\tilde{B}^{-1 }=S\in K$. Consequently,
$\tilde{A} ^{\# K} $ and $ \tilde{B} ^{\# K}$ commute.
\end{proof}

This second lemma clarifies the relation of the constant $\x _{G}$ with
the homotopy of a group, since any non-simply connected group is obtained by
factorization of a simply connected one as in the statement of the lemma. Finally, we will need the following
last lemma in this series. In order to keep the proof simple, and
since the lemma itself will be used in such a context, we use the language of toral algebras. We therefore postpone it
to the appendix \ref{appendix 1}. The reader is however encouraged to read the proof, since it actually comes up
in chapter \ref{Chapter renormalization}.

\begin{lemma} \label{parties abel. simpl. conn.}
Let $G$ be as before and suppose that it is simply connected. Then, $\chi _{G} =1$. In other words, if $G$ is simply connected,
commuting elements can be simultaneously diagonalized.
\end{lemma}

\subsubsection{Toral algebras and root-spaces} \label{root-space decomp}

Let us consider an abelian subalgebra $\mathfrak{h}$ of $g$, i.e. a
subalgebra for which the restriction of the Lie bracket to
$\mathfrak{h} \times \mathfrak{h}$ vanishes identically.
It is the Lie algebra of a torus $\mathcal{H}$ of $G$, and for this reason
it will be called a \textit{toral algebra}. If $\mathcal{H}$ is contained in $%
\mathcal{T}$, a maximal torus whose Lie algebra is $\mathfrak{t}$, then $%
\mathfrak{h\subset t}$. The Lie algebra $\mathfrak{t}$ is a maximal toral
algebra and its dimension is $w$. The relations between different (maximal)
toral algebras are analogous to those between (maximal) tori. We also give
the definition of regular and singular elements of $g$.

\begin{definition}
An element $s\in g$ will be called \textit{regular} if it is contained in a unique
maximal toral algebra, and \textit{generic} if $\exp (\R s)$ is dense in the maximal torus containing it.
If it is not regular, it will be called \textit{singular}.
\end{definition}

If $s\in g$ is regular, then $Z_{G}(\exp (\R s))$ is a maximal torus of $G$.

\bigskip

A special case of Lie group is the $n$-dimensional torus
$\mathbf{U}^{n} = \{ (z_{i}) \in (\mathbb{C}^{*}) ^{n}, \; |z_{i}|=1 \}$.
Its Lie algebra is $i \R ^{n}\subset \C ^{n}$, where $\C ^{n}$ is the Lie algebra of the
ambient group $(\mathbb{C}^{*}) ^{n}$. The kernel of the group homomorphism
$\exp :i \R ^{n}\ra \mathbf{U}^{n}$ is $2i\pi \Z ^{n}$ and every character
$\chi $ of $\mathbf{U}^{n}$ (i.e. a group homomorphism $\mathbf{U}^{n} \ra S^{1}$) can be written as
\begin{equation*}
\chi (u)=e^{\gamma (i\xi _{1},...,i\xi _{n})}
\end{equation*}%
if $u=(e^{i\xi _{1}},...,e^{i\xi _{n}})$, so that $\gamma (2i\pi m_{1},...,2i\pi
m_{n})\in 2i\pi \Z$ if all $m_{i}\in \Z$. For the same reasons, if $\TT$ is a maximal torus and $\ft $ its Lie
algebra, we find that the restriction of $\exp $ in $\ft $ defines a
group homomorphism, whose kernel is a lattice $\Gamma _{\ft }\subset 
\mathfrak{t}$. If $\chi $ is a character of $\mathcal{T}$ and $h\in 
\ft $, then
\begin{equation*}
\chi (\exp (h))=e^{2i\pi \rho (h)}
\end{equation*}%
for some $\rho \in \ft ^{\ast }$, such that $\rho (h)\in \Z$ for all $h\in \Gamma _{\mathcal{T}}$. The set of such
$\r $ is $\Gamma _{\mathcal{T}}^{\ast }$, the dual lattice of $\Gamma _{\mathcal{T}}$. The elements of
$2i\pi \Gamma _{\mathcal{T}}^{\ast }\subset \ft _{\C}^{\ast }$ (the dual space of the complexified Lie algebra
$\ft _{\C}$) are the weights of $\ft $. They are $\R$-linear mappings $\ft \ra  i \R \subset \C$. The lattice
$2i\pi \Gamma _{\ft }^{\ast }$ will be denoted by $P(G,\ft )$ or simply $P(G)$, since the lattices corresponding to
different tori are obtained by the action of an inner endomorphism.

\bigskip

Let us now fix a maximal toral algebra $\ft $ of $g$ and the corresponding torus. We will denote by $\ft _{\C}$ and
$g_{\C}$ the corresponding complexified spaces and we introduce the involution $c(x+iy)=x-iy$, where $x,y\in g$.
The Cartan-Killing form extends to a complex symmetric (and therefore not Hermitian) bilinear form.

The properties of the representation of $\mathcal{T}=\exp (\ft )$ in 
$g_{\C}$ imply the existence of a finite subset $\Delta $ of $2i\pi \Gamma _{\mathcal{T}}^{\ast }$ such that $g_{\C}$ decomposes into the sum of $\ft _{\C}$ and subspaces $E_{\rho }$ such that, for all $e_{\rho }\in E_{\rho }$ and 
$h\in \ft $,
\begin{equation*}
ad_{h}.e_{\rho }=2i\pi \rho (h)e_{\rho }
\end{equation*}
Such weights are called the roots of $G$ with respect to $\mathcal{T}$. We
note that if $\rho \in \Delta $, then $-\rho \in \Delta $ and $E_{-\rho
}=c(E_{\rho })$. If $\rho +\rho ^{\prime }$, $\rho ^{\prime }\not=-\rho $,
is a root, we have $[E_{\rho },E_{\rho ^{\prime }}]\subset E_{\rho +\rho
^{\prime }}$ and the bracket is $0$ otherwise. We will denote by $Q(g)$ the
lattice generated by the roots over $\Z$. We have clearly $Q(g)\subset P(g)$.

In the more familiar context of $SU(w+1)$, one retrieves these concepts by fixing a basis for $\C ^{w+1}$,
thus fixing the torus of diagonal matrices and the corresponding lattice of preimages of the $Id$ in $i\R^{w}$.
Each non-diaconal entry at position $(i,j)$ in $su(w+1)$, the vector space of traceless skew-symmetric $(w+1)\times (w+1)$
matrices, corresponds to a space $E_{\r }$, and the root functional simply to
$\r _{ij} : diag(2i\pi a_{k}) \mapsto (a_{i} - a_{j})$.
In particular, the number of roots is greater than $w$, and there are relations of linear dependence between them.

\bigskip

A vector $h\in \ft $ is regular if, and only if, $\rho (h)\not=0$
for all roots $\rho $, or equivalently, if $\exp (h)$ does not belong to
any of the subgroups $\chi _{\rho }^{-1}(1) \subset \ft $, where $\chi _{\rho }(\exp
(h))=e^{2i\pi \rho (h)}$ is the character whose weight is $\rho $.

For any given $\r $, there exists a subalgebra $ \R h_{\rho } \oplus \C j_{\rho} \approx su(2) $
so that the following multiplication table and orthogonality relations
hold true.
\begin{eqnarray*}
\lbrack th_{\rho },zj_{\rho }] &=&2iztj_{\rho }, \\
\lbrack zj_{\rho },wj_{\rho }] &=&2\frac{Re(iz\bar{w})}{|h_{\rho }|}%
h_{\rho },~\forall z,w\in \C
\\
\langle j_{\rho },h_{\rho }\rangle &=&\langle ij_{\rho },h_{\rho }\rangle
=\langle j_{\rho },ij_{\rho }\rangle =0 \\
\left\vert zj_{\rho }\right\vert &=&\left\vert z\right\vert \text{, and }%
\left\vert h_{\rho }\right\vert >0 \\
\lbrack j_{\rho },j_{\rho ^{\prime }}] &\in &E_{\rho +\rho ^{\prime }}
\end{eqnarray*}
Here, $h_{\r} \in \ft $, the complex plane $\C j_{\r }$ stands for the two real-dimensional
plane underlyign $E_{\r } \oplus E_{-\r}$ and we remind that
$E_{\rho +\rho ^{\prime }}=\{0\}$ if $\rho +\rho
^{\prime }$ is not a root. We stress that the complex directions $\C j_{\rho }$ complexify to 2-dimensional complex planes in $g_{\C}$.
The subalgebras thus constructed are called \textit{root spaces}.

\subsubsection{Basis of the root system} \label{Basis of the root system}

%
Using the standard properties of semi-simple compact Lie algebras, we can obtain a basis $\tilde{\Delta}$
of $\Delta $ such that if
\begin{equation} \label{linear dependence of roots}
\rho ^{\prime }=\sum_{\rho \in \tilde{\Delta}}m_{\rho ^{\prime }\rho }\rho
\end{equation}
then the $m_{\rho ^{\prime }\rho }$ are all integers of the same sign. This
results in the existence of a partition $\Delta =\Delta _{+}\cup (-\Delta
_{+})$ such that $\rho \in \Delta _{+}$ (resp. $\in -\Delta _{+}$) if, and
only if, all the integers in the above sum are positive (resp. negative).

We then have the following decomposition:
\begin{equation}
\begin{split}
g & = \ft  \oplus \bigoplus_{\rho \in \Delta _{+}} \C j_{\rho }
\\
& = \bigoplus_{\rho \in \tilde{\Delta}} \R h_{\rho }\oplus \bigoplus_{\rho \in \Delta _{+}} \C j_{\rho } 
\end{split}
\end{equation}
with
\begin{eqnarray*}
\langle j_{\r },j_{\r '}\rangle &=& 0 \\
\left\vert j_{\r }\right\vert &=& 1 \\
\langle h_{\r },h_{\r '}\rangle  &\lesseqqgtr & 0 \\
\lbrack h , j_{\r }] &=& 2i\pi \r (h) j_{\r } , \forall h \in \ft
\end{eqnarray*}

The basis $\tilde{\Delta}$ will be call a Weyl's basis for the root system
and the decomposition of $g$ as above the root system decomposition with
respect to the maximal abelian algebra $\ft$. The choice of such a basis is a simultaneous diagonalization of
all operators $ad_{h} : s \mapsto [h,s], s \in g$, for $h \in \ft$. These operators therefore commute.

A special case of mappings $\ft \ra G$ which will be of great importance in our study
is that of \textit{periodic geodesics} of the group, which are one-parameter subgroups isomorphic to $\Sp ^{1}$.
Since such mappings can be viewed as mappings $\ft \ra \mathcal{T} \subset G$, where $\mathcal{T}$ is
a maximal torus of $G$, this geodesic admits a lift in $\ft $ (the Lie algebra of $\mathcal{T}$).
Since, additionally, any basis of $\ft $ can be obtained by the action of $W(G)$ on a single one,
we can consider a fixed basis and introduce the notations
\begin{eqnarray} \label{def Er}
E_{r}(\.  ) &=&\exp (\sum\nolimits_{\rho }(2\pi r_{\rho }h_{\rho }\.  ))%
\text{ and} \\
E_{r,a}(\.  ) &=&\exp (\sum\nolimits_{\rho }(2\pi (r_{\rho }\. 
+a_{\rho })h_{\rho })
\end{eqnarray}
where $r \in \Z ^{w}$ and $a  \in \R ^{w}$ characterize the geodesic, modulo the action of $W(G)$.

\bigskip

Finally, by duality, there exists a basis $(H_{\rho })_{\rho \in \tilde{\Delta}}$ of $\ft $, such that $\rho ^{\prime }(H_{\rho })=\delta _{\rho ,\rho^{\prime }}$. We then have
\begin{equation*}
H_{\rho }=\sum_{\rho \in \tilde{\Delta}}k_{\rho ^{\prime }\rho }H_{\rho
^{\prime }}
\end{equation*}
with $k_{\rho ^{\prime }\rho }$ rational numbers of the same sign. We can
suppose that they are of the form
\begin{equation*}
\frac{l_{\rho ^{\prime }\rho }}{D}
\end{equation*}
where $D\in \N^{\ast }$ and $|l_{\rho ^{\prime }\rho }|\leq e$, with $e\in \N^{\ast }$. Clearly, $\#\tilde{\Delta}=w$, the rank of the group, while $q=\#\Delta _{+}=\frac{1}{2}(f-w)$, where $f$ is the real dimension of $g$.

\subsection{Notation and algebra in $SU(2)$ and $SO(3)$} \label{Notation and algebra in su2}

Since a certain variety of the phenomena that we will describe are
already present when $G$ is $SU(2)$ or $SO(3)$ we will construct some
concrete examples in these cases and we will compare the results obtained in this mémoire
with the ones already existing in the literature. In
order to do this, we need to describe explicitly the notions defined in the
previous section in this more concrete setting.

We bring to the attention of the reader the fact that in these concrete
examples we introduce a normalization of the Cartan-Killing form
which is not compatible with the one used in the abstract case.

\subsubsection{The group $SU(2)$}

Let us denote the matrix $S\in SU(2)$, $S=\begin{bmatrix}
t & z \\ 
-\bar{z} & \bar{t}
\end{bmatrix}
$, where $(t,z)\in \C^{2}$ and $|t|^{2}+|z|^{2}=1$, by $\{t,z\}_{SU(2)}$. The subscript will be
suppressed from the notation, unless necessary. The manifold $SU(2)$ is
naturally identified with $\Sp ^{3}\subset \C^{2}$ through $G\ni \{t,z\}_{SU(2)}\mapsto (t,z)\in \Sp ^{3}$.
The law of multiplication is pushed forward to the mapping of $\Sp ^{3}\times \Sp ^{3}\ra \Sp ^{3}$
given by $\{t_{1},z_{1}\}.\{t_{2},z_{2}\}=\{t_{1}t_{2}-z_{1}\bar{z}_{2},t_{1}z_{2}+\bar{t}_{2}z_{1}\}$.
In particular, $SU(2)$ is simply connected. Inversion is
pushed forward to the involution of $S^{3}$ $\{t,z\}\mapsto \{\bar{t},-z\}$.
The circle $\Sp ^{1}$ is naturally embedded in $G$ as the group of diagonal
matrices, which is a maximal torus of $SU(2)$. In particular $w=1$, the rank
of $SU(2)$.

\bigskip

The Lie algebra $su(2)$ is naturally isomorphic to $\R^3 \approx \R \times \C$ equipped with its vector and scalar product. The element $s=\begin{bmatrix}
it & u \\ 
-\bar{u} & -it
\end{bmatrix}
$ will be denoted by $\{t,u\}_{su(2)}\in \R \times \C$, or $\{t,\mathcal{R}u,\mathcal{I}u\}_{\mathfrak{g}}\in \R^{3}$.
The scalar product will sometimes be denoted by
\begin{equation*}
t_{1}t_{2}+\mathcal{R} (u_{1}\bar{u}_{2})=t_{1}t_{2}+\mathcal{R} u_{1}.\mathcal{R} u_{2}+
\mathcal{I} u_{1}.\mathcal{I} u_{2}
\end{equation*}
Mappings with values in $su(2)$ will be denoted by
\begin{equation*}
U(\. ) = \{ U_{t}(\. ), U_{z}(\. ) \}_{su(2)}
\end{equation*}
in these coordinates, where $U_{t}(\. )$ is a real-valued and $U_{z}(\. )$ is a complex-valued function.

The adjoint action of $h\in su(2)$ on itself is pushed-forward to twice the
vector product:
\begin{equation*}
ad_{\{1,0\}}.\{0,1,0\}=2\{0,0,1\}
\end{equation*}
plus cyclic permutations and the Cartan-Killing form, normalized by $\left\langle h,h^{\prime }\right\rangle =-\frac{1}{8}tr(ad(h)\circ
ad(h^{\prime }))$ is pushed-forward to the scalar product of $\R^{3}$. The periodic geodesics of the group for the induced
Riemannian structure are of the form $S.\exp (\{2\pi r\.  ,0\}_{su(2)}).S^{*}$, $%
S\in SU(2)$. Under this normalization, the minimal length for a geodesic
leaving from the $Id$ and arriving at $-Id$ is $\pi $, and the minimal
length of a periodic geodesic is twice as much, i.e. $2\pi $. We also find
directly that the preimages of the $Id$ in the maximal toral algebra of
diagonal matrices are points of coordinates in the lattice $2\pi \Z$.
The first critical point of the exponential mapping is $\{\pi ,0 \}$, and the critical value is $-Id $.

The adjoint action of the group on its algebra is pushed-forward to the
action of $SO(3)$ on $\R \times \C$. In particular, the diagonal matrices, of the form $\exp (\{t,0\}_{su(2)})$, fix the real direction and act by multiplication by $e^{2it}$ in the
complex direction.

Finally, the Weyl group of $SU(2)$ is naturally isomorphic to the
multiplicative group $\{1,-1\}$ and two representatives are the matrices%
\begin{equation*}
Id\text{ and }
\begin{bmatrix}
0 & 1 \\ 
-1 & 0
\end{bmatrix}
\end{equation*}
the action of the second one being a reflection with respect to the complex
plane:
\begin{equation*}
(-1)_{W}.\{t,0\}_{su(2)}=\{-t,0\}_{su(2)}
\end{equation*}

\subsubsection{The group $SO(3)$}

Since $SO(3)=SU(2)/\{\pm Id\}$, we will identify the Lie algebras of the two
groups and keep the same normalization for the Cartan-Killing form. In
particular, all non-zero vectors in $so(3)$ are regular. From the
normalization of the Cartan-Killing form, it follows that the minimal length
of a periodic geodesic, typically
\begin{equation} \label{geod SO3 non-hom}
\begin{pmatrix}
\cos (2\pi \.  ) & \sin (2\pi \.  ) & 0 \\ 
-\sin (2\pi \.  ) & \cos (2\pi \.  ) & 0 \\ 
0 & 0 & 1
\end{pmatrix}
\end{equation}
is $\pi $. This geodesic admits a lift in $SU(2)$ to a geodesic connecting
the $Id$ with $-Id$, $\exp (\{\pi \.  ,0\}_{su(2)})$, and therefore is is
not homotopic to a constant in $SO(3)$. On the other hand
\begin{equation*}
\begin{pmatrix}
\cos (4\pi \.  ) & \sin (4\pi \.  ) & 0 \\ 
-\sin (4\pi \.  ) & \cos (4\pi \.  ) & 0 \\ 
0 & 0 & 1
\end{pmatrix}
\end{equation*}
admits a lift to $E_{1}(\.  )=\exp (\{2\pi \.  ,0\}_{su(2)})$, and
therefore it is a closed geodesic, homotopic to a constant and with minimal
length. For this reason we will use the notation
\begin{equation*}
E_{1/2}(\. ) = R_{2 \pi \. } =
\begin{pmatrix}
\cos (2\pi \.  ) & \sin (2\pi \.  ) & 0 \\ 
-\sin (2\pi \.  ) & \cos (2\pi \.  ) & 0 \\ 
0 & 0 & 1
\end{pmatrix}
\end{equation*}
for geodesics as in eq. \ref{geod SO3 non-hom}.
From the above, it follows that the preimages of the $Id$ in the maximal
toral algebra of standard rotations are points of coordinates in the lattice 
$\pi \Z \subset \R $.

We also remark that $\chi _{SO(3)}=2$, and we investigate the relation of
this fact with the fact that $SO(3)=SU(2)/\{\pm Id\}$. It is a common fact
that if two matrices $A$ and $B$ in $SU(2)$ commute, they can be
simultaneously diagonalized, or in our terminology, they belong to the same
maximal torus. Since the torus of reference in this group is that of
diagonal matrices, we can rephrase this into the existence of $S\in SU(2)$
such that $S.A.S^{\ast }$ and $S.B.S^{\ast }$ belong to $\mathcal{T}$. On the
other hand, not any two commuting matrices in $SO(3)$ can be simultaneously conjugated to
the torus of standard rotations. In order to see this, we consider the torus 
$\mathcal{T}/\{\pm Id\}$ of standard rotations and an element $\pi (\{e^{i\theta },0\}) \neq Id _{SO(3)}$
in it ($\pi$ is the canonical projection $SU(2)\ra SO(3)$).
It commutes with all the matrices in $\mathcal{T}/\{\pm Id\}$, and conversely, if (a preimage of)
it commutes with a matrix in $SU(2)$, this matrix is in $\mathcal{T}$. A simple calculation shows that
\begin{equation*}
\{e^{i\theta },0\}.\{a,b\}.\{e^{-i\theta },0\}.\{\bar{a},-b\}=\{|a|+e^{2i\theta }|b|,ab(1+e^{2i\theta })\}
\end{equation*}
We see, therefore, that the equation%
\begin{equation*}
A.B.A^{\ast }.B^{\ast }=-Id
\end{equation*}
admits the solutions
\begin{eqnarray*}
A &=&\{\pm i,0\} \\
B &=&\{0,b\}
\end{eqnarray*}
with $b\in S^{1}$. This family of solutions obviously does not intersect the
standard torus, and its projection to $SO(3)$ gives the rotation
\begin{equation*}
\pi (A) =
\begin{pmatrix}
-1 & 0 & 0 \\ 
0 & -1 & 0 \\ 
0 & 0 & 1
\end{pmatrix}
\end{equation*}
which commutes with standard rotations, but also with the one parameter
family 
\begin{equation*}
\pi (B(\theta )) =
\begin{pmatrix}
\cos \theta & \sin \theta & 0 \\ 
\sin \theta & -\cos \theta & 0 \\ 
0 & 0 & -1
\end{pmatrix}
=
\begin{pmatrix}
\cos \theta & -\sin \theta & 0 \\ 
\sin \theta & \cos \theta & 0 \\ 
0 & 0 & 1%
\end{pmatrix}
.
\begin{pmatrix}
1 & 0 & 0 \\ 
0 & -1 & 0 \\ 
0 & 0 & -1%
\end{pmatrix}
\end{equation*}
for $\theta \in \lbrack 0,2\pi ]$, which does not intersect the standard torus.

Finally, we introduce the standard basis of $so(3)$,
\begin{equation*}
j_{1}=
\begin{pmatrix}
0 & 1 & 0 \\ 
-1 & 0 & 0 \\ 
0 & 0 & 0
\end{pmatrix}
\text{, } j_{2} =
\begin{pmatrix}
0 & 0 & 0 \\ 
0 & 0 & 1 \\ 
0 & -1 & 0
\end{pmatrix}
\text{, } j_{3} =
\begin{pmatrix}
0 & 0 & 1 \\ 
0 & 0 & 0 \\ 
-1 & 0 & 0
\end{pmatrix}
\end{equation*}
In the identification $su(2) \approx so(3) $, the vector $j_{1 }$ corresponds to $\{ 1,0 \}_{su(2)}$, and the plane
generated by $j_{2}$ and $j_{3 }$ to the complex plane $\C \hookrightarrow su(2) $.

\section{Calculus}

\subsection{Calculus in Lie groups} \label{Calculus in Lie groups}

Let us now define the differential operator $L:C^{\infty }(\T ,G)\ra C^{\infty }(\T ,g)$, as
\begin{equation}
L(A)(\. )=\partial A(\. ).A^{\ast }(\. )
\end{equation}
This operator uses the structure of the group in order to simplify
the expressions of derivatives. Since a Lie group is a manifold, the derivative of a $G$-valued path is a section of
its tangent bundle. The fact that the tangent bundle is trivial (i.e. trivializable in a canonical way) allows us to define the
derivative of a path as a curve in a single vector space, naturally isomorphic to the fibers, instead of a section of the tangent bundle. This is done by
using the group structure of the manifold and mapping canonically, for each $x \in \T$, the tangent space
$T_{A(x)}G \subset TG$ at $A(x) \in G$ to $g = T_{Id}G$ via the differential of left translation, $ S \mapsto S A(x)^{\ast}$.
\footnote{In fact, since the action of $G$ in the following chapter will be chosen to be left multiplication, the choice
of right multiplication in the definition of the operator $L$ slightly complicates the calculations in chapter
\ref{Chapter on energy}. However, it becomes more natural in the local theory of models, where traditionally perturbations
are chosen to act on the left.}
If the underlying group is simply $\R_{+}^{*}$, this is just a way of normalizing the derivative of the exponential as
\begin{equation*}
\phi ' (x) = \left( \dfrac{d}{dx}e^{\phi (x)} \right) e^{-\phi (x)}
\end{equation*}
instead of letting it grow along with $e^{\phi (x)}$. If the underlying group is $\Sp ^{1}$, the tangent vector
deviates from the vertical line oriented upwards when the base point is not $1 \in \Sp ^{1}$, and right multiplication rectifies
it to the standard one.

The basic properties of this operator are as follows.

\begin{enumerate}
\item $L(A(\. ).B(\. ))=a(\. )+Ad(A(\. )).b(\. )$, where we introduce a notational convention:
$L(A(\. ))=a(\. )$, $L(B(\. ))=b(\. )$ and similarly for any mapping $\T  \ra G$ (eventually $\R \ra G$).

\item Derivation of $A^{\ast }(\. ).A(\. )=Id$ gives $a^{\ast }(\. )=-Ad(A^{\ast }(\. )).a(\. )$ (notice the abuse of notation)

\item If $A(\. )=\exp (\phi (\. ))$, where $\phi :C^{\infty }(\R,\ft )$, with $\ft \subset g$ a toral algebra, then $a(\. )=\phi ^{\prime }(\. )$

\item $\partial L(A(\. ).B(\. ))=\partial a(\. )+Ad(A(\. )).\partial b(\. )+[a(\. ),Ad(A(\. )).b(\. )]$

\item If $A(\. )=\exp (U(\. ))$ is $C^{0}$-close to the identity, then
\begin{equation*}
a(\. )=\frac{\exp (ad(U(\. )))-Id}{ad(U(\. ))}\partial U(\. )
\end{equation*}
where $\frac{\exp (ad(U(\. )))-Id}{ad(U(\. ))}$ represents the
operator $\sum_{0}^{\infty }\frac{1}{(n+1)!}ad(U(\. ))^{n}$.
\end{enumerate}

Property 1 implies the following fact. Let $K$ be a subgroup of $Z_{G}$, $\tilde{G}=G/K$ and $\pi :G\ra \tilde{G}$
the canonical projection. We remind that the Lie algebras of $G$ and $\tilde{G}$ are naturally
isomorphic. Then, if $A(\. ):\T \ra G$ is $C^{1}$, $\tilde{A}(\. )=\pi (A(\. )):\T \ra \tilde{G}$ is also $C^{1}$-smooth
(possibly of period smaller than $1$) and property 1 implies that $\tilde{a}(\. )=a(\. )$,
with a slight abuse of notation. If, keeping the same
notations, we suppose that $\tilde{A}(\. )$ is given and we choose a lift 
$A(\. )$ of $\tilde{A}(\. )$ to $G$ (i.e. $\pi (A(\. ))=\tilde{A}(\. )$),
then $A(\. )$ will be at most $\#K$ periodic and $A(\. +1)=S.A(\. )$, where $S\in K$.
Then we find again that the derivatives of $A(\. )$ and $\tilde{A}(\. )$ in $g$ are equal.
Or, equivalently, the derivative in $g$ of a curve in $G$ depends on the
isomorphism class of the Lie algebra, and not the group itself, a quite natural fact, since different groups
with the same Lie algebra are locally indistinguishable.

Finally, property 4 implies the following bounds for derivatives of
products, which will be useful in the estimates of iterates of cocycles.

\begin{proposition} \label{estimates on deriv of products}
We have the following estimates concerning the derivatives.
\begin{enumerate}
\item Let $A_{i}:\T  \ra
G$, $i=1,\cdots ,n$, be smooth enough, and let $m_{0}=\max_{i}\left\Vert a_{i}\right\Vert _{\infty }$
and $m_{s}=\max_{i,0\leq \sigma \leq s}\left\Vert \partial ^{s}a_{i}\right\Vert _{\infty }$. Then we have
\begin{equation*}
\left\Vert \partial ^{s}(a_{1}...a_{n})\right\Vert _{\infty }\leq
C_{s}m_{s}(1+m_{0})^{s}n^{s+1}
\end{equation*}

\item Let $U \in C^{\infty }(\T  ,g)$ be small enough in $C^{0 }(\T  ,g)$ and $A:\T \ra G$. Then
\begin{equation*}
C_{s}^{-1}(1-\left\Vert U(\. )\right\Vert _{0})^{s}\left\Vert U(\. 
)\right\Vert _{s+1}\leq \| \partial ^{s}Le^{U(\. )}\|
_{s}\leq C_{s}(1+\left\Vert U(\. )\right\Vert _{0})^{s}\left\Vert U(\. 
)\right\Vert _{s+1}
\end{equation*}
and
\begin{eqnarray*}
\left\Vert Ad(A(\. )).U(\. )\right\Vert _{0} &=&\left\Vert U(\. 
)\right\Vert _{0} \\
\left\Vert L(\exp (Ad(A(\. )).U(\. ))\right\Vert _{s} &\leq
&C_{s}(1+\left\Vert U(\. )\right\Vert _{0})^{s}(1+\left\Vert a(\. 
)\right\Vert _{0})^{s}(1+\left\Vert a(\. )\right\Vert _{s})\left\Vert
U(\. )\right\Vert _{s+1}
\end{eqnarray*}
If we admit a uniform bound on the term $\left\Vert U(\. )\right\Vert _{0}$,
\begin{equation*}
\left\Vert L(\exp (Ad(A(\. )).U(\. ))\right\Vert _{s}\leq
C_{s}(1+\left\Vert a(\. )\right\Vert _{0})^{s}(1+\left\Vert a(\. 
)\right\Vert _{s})\left\Vert U(\. )\right\Vert _{s+1}
\end{equation*}
\end{enumerate}
\end{proposition}

\begin{proof}
\begin{enumerate}
\item Since $\partial (Ad(A(\. )).b(\. )))=Ad(A(\. )).\partial
b(\. )+[a(\. ),Ad(A(\. )).b(\. )]$, induction in $n$, convexity
inequalities and counting of the terms gives the result.

\item If we let $h\in g$, we calculate
\begin{equation*}
\partial Ad(B(\. )).h =[ b(\. ),Ad(B(\. )).h]
\end{equation*}
so that the derivative in the Lie algebra of $Inn(G)\approx G/Z_{G}$ (which
is a subalgebra of $gl(g)$ naturally isomorphic to $g$) reads
\begin{equation*}
(\partial Ad(B(\. )).Ad(B^{\ast }(\. )) =ad _{b(\. )}
\end{equation*}
and
\begin{equation*}
C^{-1}\left\Vert B\right\Vert _{s}\leq \left\Vert Ad(B(\. ))\right\Vert
_{s}\leq C\left\Vert B\right\Vert _{s}
\end{equation*}%
where the constant is universal since it depends only on the adjoint
representation of $g$ into itself\footnote{This expression is in fact not totally accurate, since in the $C^{0}$ norm we should replace the Riemannian distance from $B(\. )$ to the $Id$ with the $\min (d(B(\. ),S),S\in Z_{G})$. We keep this fact implicit in order
to keep notation simpler.}.

If we call $\tilde{U}(\. )=Ad(A(\. )).U(\. )$, we have
\begin{equation*}
Le^{\tilde{U}(\. )}=Ad(A(\. )).Le^{U(\. )}+(Id-Ad(e^{\tilde{U}(\. )}))a(\. )
\end{equation*}
Therefore,
\begin{eqnarray*}
\| Le^{\tilde{U}(\. )}\| _{s} &\leq &\| Ad(A(\. )).Le^{U(\. )}\| _{s}+
\| (Id-Ad(e^{\tilde{U} (\. )}))a(\. )\| _{s} \\
&\lesssim &\| Le^{U(\. )}\| _{s}+\| Ad(A(\. ))\| _{s}\| Le^{U(\. )}\| _{0}
+\| U(\. )\| _{0}\| a(\. )\| _{s}+\| Ad(e^{\tilde{U}(\. )})\| _{s}\| a(\. )\| _{0} \\
&\lesssim &\| e^{U(\. )}\| _{s}+(1+\| a(\. )\| _{s-1})\| e^{U(\. )}\| _{1}
+\| U(\. )\| _{0}\| a(\. )\| _{s}+\| e^{\tilde{U}(\. )}\| _{s}\| a(\. )\| _{0}
\end{eqnarray*}
In particular, for $s=1$,
\begin{eqnarray*}
\| Le^{\tilde{U}(\. )}\| _{0} &\leq &\| e^{U(\. )}\| _{1}+\| A(\. )\| _{0}\| e^{U(\. )}\| _{1}
+\| U(\. )\| _{0}\| A(\. )\| _{1} \\
&\leq &\| e^{U(\. )}\| _{1}+\| U(\. )\| _{0}(1+\| a(\. )\| _{0})
\end{eqnarray*}
and the announced estimates follow from these expressions.
\end{enumerate}
\end{proof}

For $X\in g$ small enough, let $K(X).\Delta X=(D(e^{X}).\Delta X).e^{-X}$.
Then, $K(X)$ is in $GL(g)$, since the exponential is a local diffeomorphism. This operator can be calculated by
\begin{eqnarray*}
K(X).\Delta X &=&\frac{e^{ad(X)}-Id}{ad(X)}.\Delta X \\
&=&\Delta X+\frac{1}{2}[X,\Delta X]+\frac{1}{3!}[X,[X,\Delta X]]+...
\end{eqnarray*}
In particular, $L\exp (U(\. ))=K(U(\. )).\partial U(\. )$, and
\begin{equation*}
K(-X).\Delta X=Ad(e^{-X}).K(X).\Delta X
\end{equation*}
The transposed operator with respect to the Cartan-Killing form is found by
\begin{equation*}
K(X)^{T}=K(-X)
\end{equation*}
as shows the expansion of the operator. Since $K(X)$ is invertible for small
enough $X$, we will sometimes use the notational convention%
\begin{equation*}
\widetilde{\Delta X}=K(X).\Delta X
\end{equation*}%
In fact, for $X$ small enough, $K(X)$ is close to being an isometry, since
\begin{equation*}
K(-X).K(X)=Id+\frac{1}{12}ad_{X}^{2}+O(|X|^{4})
\end{equation*}
as shows a direct calculation.

\subsection{Functional spaces}

We will consider the space $C^{\infty }(\T ^{d},g)$ equipped with the
standard maximum norms%
\begin{eqnarray*}
\left\Vert U\right\Vert ^{0} _{s} &=&\max_{\T  }\left\vert \partial
^{s}U(\. )\right\vert \\
\left\Vert U\right\Vert _{s} &=&\max_{0\leq \sigma \leq s}\left\Vert
U\right\Vert ^{0} _{\sigma }
\end{eqnarray*}
for $s\geq 0$, and the Sobolev norms
\begin{equation*}
\left\Vert U\right\Vert _{H^{s}}^{2}=\sum_{k\in \Z^{d}}(1+|k|^{2})^{s}|\hat{U}(k)|^{2}
\end{equation*}%
where $\hat{U}(k)=\int U(\. )e^{-2i\pi kx}$ are the Fourier coefficients
of $U(\. )$. The fact that the injections $H^{s +d/2}(\T ^{d},g) \hookrightarrow C^{s}(\T ^{d},g) $
and $C^{s}(\T  ^{d},g) \hookrightarrow H^{s}(\T ^{d},g)$ for all $s \geq 0$ are continuous is classical.

For mappings $\R \supset I \ra g$ we will use the norms
\begin{equation*}
\left\Vert U(\. )\right\Vert _{s,I} =\max_{0\leq \sigma \leq
s}\left\Vert \partial ^{\sigma }U(\. )\right\Vert _{L^{\infty }(I)}
\end{equation*}
where $I\subset \R$ is an interval. If $I=[0,T]$, we will replace $I$ by $T$ in these
notations.

The space $C^{\infty }(\T ^{d},G)$ is equipped with the following
norms (we remind the convention $LA(\. )=a(\. )$)
\begin{eqnarray*}
\left\Vert A\right\Vert _{0} &=&\max_{\T  }d(A(\. ),Id) \\
\left\Vert A\right\Vert ^{0} _{s} &=& \left\Vert a(\. )\right\Vert ^{0} _{s-1} \\
\left\Vert A\right\Vert _{s} &=& \left\Vert a \right\Vert _{s-1 } \\
\left\Vert A(\. )\right\Vert _{s,I} &=& \left\Vert a \right\Vert _{s,I}
\end{eqnarray*}
for $s\geq 1$. We note that in the $\left\Vert \. \right\Vert $ norms of $G$-valued maps
we omit the $C^{0}$ norm, irrelevant for the majority of the arguments since $G$ is supposed to be compact.

\bigskip

We will also use the convexity or Hadamard-Kolmogorov inequalities (see \cite{KolmIneq49}) ($U\in
C^{\infty }(\T  ^{d},g)$):
\begin{equation}
\left\Vert U(\. )\right\Vert _{\sigma }\leq C _{s, \sigma }\left\Vert
U\right\Vert _{0}^{1-\sigma /s}\left\Vert U\right\Vert _{s}^{\sigma /s}
\label{hadamard}
\end{equation}
for $0\leq \sigma \leq s$, and the inequalities concerning the composition (see \cite{KrikAst})
of functions:
\begin{equation}
\left\Vert \phi \circ (f+u)-\phi \circ f\right\Vert _{s}\leq
C_{s}\left\Vert \phi \right\Vert _{s+1}(1+\left\Vert f\right\Vert
_{0})^{s}(1+\left\Vert f\right\Vert _{s})\left\Vert u\right\Vert _{s}
\label{composition}
\end{equation}
which hold for $C^{s}$ norms.

We will use the truncation operators%
\begin{eqnarray*}
T_{N}f(\. ) &=&\sum_{|k|\leq N}\hat{f}(k)e^{2i\pi k\. } \\
\dot{T}_{N}f(\. ) &=&T_{N}f(\. )-\hat{f}(0) \\
R_{N}f(\. ) &=&\sum_{|k|>N}\hat{f}(k)e^{2i\pi k\. }\text{ and} \\
\dot{T}f(\. ) &=&Tf(\. )-\hat{f}(0)
\end{eqnarray*}
These operators satisfy the estimates
\begin{eqnarray*}
\left\Vert T_{N}f(\. )\right\Vert _{C^{s}} &\leq
&C_{s,s'}N^{(d+1)/2}\left\Vert f(\. )\right\Vert _{C^{s}} \\
\left\Vert R_{N}f(\. )\right\Vert _{C^{s}} &\leq &C_{s,s'} N^{s-s^{\prime }+d+1} \left\Vert
f(\. )\right\Vert _{C^{s^{\prime }}}
\end{eqnarray*}
The Fourier spectrum of a function will be denoted by%
\begin{equation*}
\hat{\sigma}(f)=\{k\in \Z^{d},~\hat{f}(k)\not=0\}
\end{equation*}


Finally, we will need the space of real analytic mappins $C^{\omega}_{\delta } (\T , SU(2))$, defined as
the space of analytic mappings $\T \ra SU(2) $ (the reader can think of mappings whose matrix coefficients for the adjoint
action are real analytic), and admit an analytic extension to $\T \times [-\delta i , \delta i ]$ for $\delta >0$.
The norm of this space is the analytic norm
$\| A \|_{\delta} = \sup_{x\in \T , |\epsilon | < \delta} \| A ( x + i\epsilon ) \| $. The norm on the righthand side
of the definition here is the operator norm associated to the hermitian product on $\C ^{2}$, which is legitimate since
$SU(2)$ complexifies to $SL(2, \C )$, so that the values of the mapping for complex values of the variable in the basis are
in fact in $SL(2,\C )$.

\subsection{The Nash-Moser inverse function theorem}

This theorem, stated and proved in \cite{Ham1982}, will be used in the proof
of a normal form theorem in chapter \ref{Normal form theorem}. The theorem
is stated in the context of tame mappings between Fr\'{e}chet spaces,
for which the model is, say, $\mathcal{E}=C^{\infty }(X,\R)$ spaces with $X$ a compact manifold.
Since the Nash-Moser theorem will be
used in exactly such a context, we will restrict the generality and slightly
simplify the statement. In our discussion, we give an overview of \textit{section 4.2} of 
\cite{KrikAst}, rather than referring directly to the R. Hamilton's article.

In what follows, $\mathcal{E}$ denotes a Fr\'{e}chet space and $\left\Vert
\. \right\Vert _{s}$, $s\in \mathbb{N}$ the family of its seminorms defining the usual topology of $\mathcal{E}$
by the system of neighborhoods of $0 \in \mathcal{E} $
\begin{equation*}
U_{i,j} = {x\in \mathcal{E} , \| x \|_{s} <(j+1)^{-1}, 0 \leq s \leq i}
\end{equation*}
where $i,j \in \N ^{*}$

The absence of a norm defining the topology of the space $\mathcal{E}$ makes the
definition of a differential more difficult than in Banach spaces, but
directional derivation (G\^{a}teaux derivative) is nonetheless well defined.
Let $U\subset E$ be an open subset and $E$ and $F$ be topological vector
spaces. A continuous mapping $f:U\ra F$ is called \textit{G\^{a}%
teaux-differentiable} if there exists a mapping%
\begin{equation*}
\begin{array}{cccc}
Df: & U\times E & \ra & F \\ 
& (x,h) & \mapsto & Df(x).h%
\end{array}%
\end{equation*}%
continuous jointly in $(x,h)$ and linear in $h$, and such that
the following limit exists and verifies%
\begin{equation*}
\lim_{t\ra 0}\frac{f(x+th)-f(x)}{t}=Df(x).h
\end{equation*}%
Higher order differentiability is defined inductively.

A Fr\'{e}chet space is called \textit{graded} if the semi-norms defining its
topology are increasing in strength. In the context of $C^{\infty }(\T ,g)$, this amounts to%
\begin{equation*}
\left\Vert \. \right\Vert _{s} \leq \left\Vert \. \right\Vert
_{s+1}
\end{equation*}%
which is verified by the definition of the norms. In fact, any Fr\'{e}chet
space can be made into a graded one, by a simple change in the definition of
its seminorms. A graded Fr\'{e}chet space is called \textit{tame} if there
exists a family of continuous operators $T_{N}$, indexed by $N\in \mathbb{N}%
^{\ast }$, satisfying the properties of the truncation operators introduced
just above:%
\begin{eqnarray*}
\left\Vert T_{N}e\right\Vert _{s'} &\leq &C_{s.s^{\prime }}N^{s-s^{\prime
}+r}\left\Vert e\right\Vert _{s} \\
\left\Vert R_{N}e\right\Vert _{s} &\leq &C_{s,s^{\prime }}N^{s^{\prime
}-s-r}\left\Vert e\right\Vert _{s^{\prime }}
\end{eqnarray*}%
for all $0\leq s\leq s^{\prime }$, for a uniform choice of $r\in \mathbb{N}$
and positive constants $C_{s.s^{\prime }}$, and where $R_{N}=Id-T_{N}$. In
general contexts, such operators $T_{N}$ are called \textit{smoothing
operators}, following the properties of truncation in the usual concrete
examples.

A linear application $T$ between graded Fr\'{e}chet spaces $\mathcal{E}$ and 
$\mathcal{F}$ will be called \textit{tame} if there exists $l\in \N$ and a family of positive constants $C_{s}$ such that%
\begin{equation*}
\left\Vert Te\right\Vert _{s}\leq C_{s}\left\Vert e\right\Vert
_{s+l},~\forall s\in \N ,~\forall e\in \mathcal{E}
\end{equation*}%
where we have used the same notation for seminorms in both spaces.

If $\mathcal{E}$ and $\mathcal{F}$ are graded Fr\'{e}chet spaces, $U\subset E
$ is a neighborhood of $0$ and $P:(U,0) \ra (\mathcal{F} ,0)$ is a mapping, we will say
that it is tame if there exists $l\in \N$, "the loss of derivatives", and a family of positive constants $C_{s}$ such that%
\begin{equation*}
\left\Vert Pf\right\Vert _{s}\leq C_{s}(1+\left\Vert f\right\Vert
_{s+l}),~\forall s\in \N
\end{equation*}

The following proposition of \cite{Ham1982} shows that the basic properties
of differential calculus in Banach spaces survive in this context.

\begin{proposition}
Let $X$ be a compact manifold, $E$ and $F$ two finite-dimensional vector
spaces. Then

\begin{enumerate}
\item The spaces $C^{\infty }(X,E)$ and $C^{\infty }(E,F)$, furnished with
the $C^{s}$ norms are tame Fr\'{e}chet spaces

\item Composition%
\begin{equation*}
\begin{array}{ccc}
C^{\infty }(X,E)\times C^{\infty }(E,F) & \ra  & C^{\infty }(X,F) \\ 
(f,g) & \mapsto  & g\circ f%
\end{array}%
\end{equation*}%
is a tame mapping.

\item If $g\in C^{\infty }(E,F)$ is fixed, then%
\begin{equation*}
\begin{array}{cccc}
\beta _{g}: & C^{\infty }(X,E) & \ra  & C^{\infty }(X,F) \\ 
& f & \mapsto  & g\circ f%
\end{array}%
\end{equation*}%
is $C^{\infty }$-tame, and its differential is given by%
\begin{equation*}
\begin{array}{cccc}
D\beta _{g}: & C^{\infty }(X,E) & \ra  & C^{\infty }(X,F) \\ 
& \Delta f & \mapsto  & D\beta _{g}(f).\Delta f%
\end{array}%
\end{equation*}%
where $D\beta _{g}(f).\Delta f$ represents the mapping $x\mapsto
Dg(f(x)).\Delta f(x)$.

\item If $g\in C^{\infty }(E,F)$ is fixed, then%
\begin{equation*}
\begin{array}{cccc}
\beta _{g}: & C^{\infty }(X,F) & \ra  & C^{\infty }(X,E) \\ 
& f & \mapsto  & f\circ g%
\end{array}%
\end{equation*}%
is $C^{\infty }$-tame, and its differential is given by%
\begin{equation*}
\begin{array}{cccc}
D\beta _{g}: & C^{\infty }(X,E) & \ra  & C^{\infty }(X,F) \\ 
& \Delta f & \mapsto  & (x\mapsto \Delta f(g(x)).g^{\prime }(x))%
\end{array}%
\end{equation*}

\item if $f\in C^{\infty }(E,E)$ is invertible, then there exists $U$, a
neighborhood of $f$, such that%
\begin{equation*}
\begin{array}{ccc}
U & \ra  & C^{\infty }(E,E) \\ 
g & \mapsto  & g^{-1}%
\end{array}%
\end{equation*}%
is $C^{\infty }$-tame, where $g^{-1}$ is the inverse function of $g$ in the
usual sense.
\end{enumerate}
\end{proposition}

Finally, we introduce some notation before stating the version of Hamilton's
theorem, proved in \cite{KrikAst}, that we will use.

Let $\mathcal{E}$ and $\mathcal{F}$ be two tame Fr\'{e}chet spaces, $%
U\subset \mathcal{E}$ a neigborhood of $0$, $r\in \mathbb{N}^{\ast }$ and $%
(C_{s})_{s}$ a family of positive constants indexed by $\mathbb{N}$. We
denote by $\mathbf{A}(r,(C_{s})_{s},U)$ the set of tame mappings $%
f:U\ra \mathcal{F}$ such that $f(0)=0$ and%
\begin{equation*}
\left\Vert f(x)\right\Vert _{s}\leq C_{s}(1+\left\Vert x\right\Vert
_{s+r}),~\forall s\in \mathbb{N}
\end{equation*}%
The set $\mathbf{A}^{k}(r,(C_{s})_{s},U)$, $k\in \mathbb{N}\cup \{\infty \}$%
, is formed by functions $f$ such that $D^{l}f\in \mathbf{A}(r,(C_{s})_{s},U)
$, for $0\leq l\leq k$. If there exists a tame continuous mapping $L:U\times
F\ra E$, linear in the second variable, such that for all $x\in U$, $%
Df(x)$ is invertible and of inverse $L(x)$, we will write, by abuse of
notation, $f\in \mathbf{A}^{-1}(r,(C_{s})_{s},U)$. If $f$ is both in $%
\mathbf{A}^{k}(r,(C_{s})_{s},U)$ and $\mathbf{A}^{-1}(r,(C_{s})_{s},U)$, we
will write $f\in \mathbf{A}^{k,-1}(r,(C_{s})_{s},U)$.

With this notation, we have
\begin{theorem} \label{Hamilton inverse function}
Let $f\in \mathbf{A}^{2,-1}(r,(C_{s})_{s},U)$. Then, there exist $r_{1}$, $\e _{1}$, depending only on $r$, $U$ and
the family $C_{s}$, such that $f^{-1}$, the inverse of $f$, exists and is well defined in $V_{1}$, the neighborhood
of $0 \in \mathcal{F}$ defined by $\{ \| y \| _{r_{1}} \leq \e _{1} \}$. Moreover, for
$y \in V_{1}$, $\| f^{-1}(y) \| _{s} \leq c_{s} (1 + \| y \|_{s+r_{1}})$.
\end{theorem}
The following version is a uniform one, and allows estimates on the size of the neighborhood where
inversion takes place.
\begin{theorem}
Let $f\in \mathbf{A}^{2,-1}(r,(C_{s})_{s},U)$, whose derivative $Df$ admits a tame inverse in $U$, and $M>0$ such that
$(Df)^{-1} \in \mathbf{A}(r,(M C_{s})_{s},U)$. Then, there exists $\nu >0$ (one can chose $\nu = 2 $), $r_{1}$
and $\e _{1}$, depending only on $r$, $U$ and the family $C_{s}$, such that $f^{-1}$ exists and is well defined
in a neighborhood of $0 \in \mathcal{F}$ of the form $\{ \| y \| _{r_{1}} \leq M^{-\nu } \e _{1} \}$.
In this neighborhood, $f^{-1} \in \mathbf{A}^{2}(r_{1},(M^{-\nu } C_{s})_{s})$
\end{theorem}

\section{Arithmetics, continued fraction expansion} \label{Arithmetics, continued fraction expansion}

A deep introduction into arithmetics and continued fraction expansion can be
found in \cite{KhinContFr}. A geometric interpretation of the algorithm,
followed by an introduction to K.A.M. theory can be found in \cite{ArnGeomMeth}.

Let us introduce some notation:

\begin{itemize}
\item $|||\a ||| = |\a | _{\Z } =dist(\alpha ,\Z) = \min _{\Z } |\a - l| $

\item $[\a ]$ the integer part of $\a $

\item $\{\a \}$ the fractional part of $\a $

\item $G(\a )=\{\a ^{-1}\}$, the Gauss map.
\end{itemize}

where $\a $ is a real number.

\bigskip

Consider $\a \in \T   \setminus \Q$ fixed, and let $p_{-2}=q_{-1}=0$ and $p_{-1}=q_{-2}=1$. Then $(p_{n}/q_{n})_{n\geq 0}$ is the sequence of best rational approximations of $\a $ defined recursively as follows. Let $\a _{n}=G^{n}(\a
)=G(\a _{n-1})$, $a_{n}=[\a _{n-1}^{-1}]$, \footnote{We hope that the reader will not confuse the natural numbers $a_{n}$ introduced in the continued fractions expansion with the mappings $a_{n}(\. )$ defined in chapter \ref{Chapter on energy}.}
and $\beta _{n}=\prod\limits_{0}^{n}\a _{k}$. Then the Euclidean division of $\beta _{n-2}$ by $\beta _{n-1}$ reads
\begin{equation*}
\beta _{n-2}=a_{n}\beta _{n-1}+\beta _{n}
\end{equation*}
and%
\begin{eqnarray*}
q_{n} &=&a_{n}q_{n-1}+q_{n-2} \\
p_{n} &=&a_{n}p_{n-1}+p_{n-2}
\end{eqnarray*}
$p_{n}$ and $q_{n}$ are strictly positive for $n\geq 1$ and $\beta
_{n}=(-1)^{n}(q_{n}\a -p_{n})$. We have
\begin{eqnarray*}
\frac{1}{q_{n}+q_{n+1}} &<&\beta _{n}<\frac{1}{q_{n+1}} \\
|||q_{n-1}\a ||| &<&|||k\a |||,~\forall ~0<k<q_{n}
\end{eqnarray*}
Moreover, for all $n$,
\begin{equation*}
\a _{n}\a _{n-1}<\frac{1}{2}
\end{equation*}
and consequently
\begin{equation*}
\beta _{n}<2^{-n/2}
\end{equation*}
Two basic facts which distinguish continued fractions in one-dimension from the higher dimensional analogues are the
following:
\begin{equation*}
q_{n-1}\beta _{n}  + q_{n}\beta _{n-1}  =1
\end{equation*}
and
\begin{equation*}
\begin{vmatrix}
p_{n} & p_{n-1} \\
q_{n} & q_{n-1}
\end{vmatrix}
= (-1)^{n}
\end{equation*}

The following notion is essential in K.A.M. theory. It is related with the quantification of the closeness of
rational numbers to certain classes of irrational numbers.
\begin{definition} \label{def DC}
We will denote by $DC(\gamma ,\t )$ the set of numbers $\a $ in $\T \setminus \Q$ such that for any $k\not=0$,
$|\a k|_{\Z }\geq \frac{\gamma ^{-1}}{|k|^{\t }}$. Such numbers are called \textit{Diophantine}.
The set $DC(\gamma ,\t )$, for $\t >1$ fixed and $\gamma \in \R _{+} ^{\ast}$ is of positive
Haar measure in $\T $. If we fix $\t $ and let $\gamma$ run through the positive real numbers, we obtain
$\cup_{\gamma >0} DC(\gamma ,\t )$ which is of full Haar measure.
\end{definition}

The numbers that do not satisfy any Diophantine condition are called \textit{Liouvillean}. They form a residual set of
$0$ Lebesgue measure.

This last following definition concerns the relation of the approximation of an irrational number with its continued
fractions representation.
\begin{definition} \label{def RDC}
Denote by $RDC(\gamma ,\t )$ the set of  \textit{recurrent Diophantine} numbers. It is the set of all $\a \in \T \setminus \Q$ such that $G^{n}(\a )\in DC(\gamma ,\t )$ for infinitely many $n$. It
is a set of full Lebesgue measure since the Gauss map is ergodic with respect to a smooth measure.
\end{definition}

In contexts where the parameters $\gamma $ and $\t $ are not significant, they will be omitted in the notation of both sets.
\selectlanguage{english}
\chapter{Cocycles in $\T ^{d}\times G$} \label{Generalities on cocycles}

\section{Definitions and general properties}

Let $\a =(\a _{1},...,\a _{d})\in \T ^{d}$, $d\in \N ^{\ast }$, be a topologically minimal translation. This property is, by
Kronecker's theorem, equivalent to $(\a _{1},...,\a _{d},1) \in \R ^{d+1}$ being linearly
independent over $\Z $. The translation $\T^{d} \circlearrowleft$ will sometimes be denoted by
$R_{\a }:x\mapsto x+\a \mod(\Z ^{d})$. In the greatest part of our study we will limit ourselves to the case $d=1$,
where minimality is equivalent to $\a $ being an irrational number.

If $A(\. )\in C^{s}(\T ^{d},G)$, $s\in \N \cup \infty $, then the couple $(\a ,A(\. ))$ acts on the fibered
space $\T ^{d}\times G \ra \T ^{d}$ defining a diffeomorphism by
\begin{equation*}
(\a ,A(\. )).(x,S)=(x+\a ,A(x).S)
\end{equation*}
for any $(x,S)\in \T ^{d}\times G$. We will call such an action a \textit{quasiperiodic cocycle over }$R_{\a }$ (or simply
a cocycle). The space of such actions is denoted by $SW_{\a }^{s}(\T ^{d},G)\subset Diff^{s}(\T ^{d} \times G)$ and $d$ is
the number of \textit{frequencies} of the cocycle. When $d$ and $G$ are clearly defined by the context, we will abbreviate
the notation to $SW_{\a }^{s}$. Cocycles are a class of fibered diffeomorphisms, since fibers of $\T ^{d}\times G$ are
mapped into fibers, and the mapping from one fiber to another in general depends on the base point.
A closed subspace of $SW_{\a }^{s}(\T ^{d},G)$ which will be of use will be the space $SW_{\a }^{s,1}(\T ^{d},G)$ of
cocycles \textit{homotopic to constants}. In this context, this last condition amounts to $A(\. )$ being homotopic
to constants in $G$.

The space $\bigcup\nolimits_{\a }SW_{\a }^{s}(\T ^{d},G)$, where the union is over minimal
translations, will be denoted by $SW^{s}(\T ^{d},G)$. The space $SW_{\a }^{s}(\T ^{d},G)$ inherits the topology of
$C^{s}(\T ^{d},G)$, and $SW^{s}(\T ^{d},G)$ has the standard product topology
of $\T ^{d}\times C^{s}(\T ^{d},G)$. The space $SW^{s}$ is consequently
not a complete metric space, something which makes reasoning more delicate when one
wishes to vary the frequency. In this work, the frequency is to be considered fixed. We note that cocycles are
defined over more general maps and in more general contexts of regularity
and structure of the basis and fibers, including Lie groups such as $GL(n,\R )$,
$GL(n,\C )$, $SL(n,\R )$, $SL(n,\C )$ and their classical subgroups.

If we consider a representation of $G$ on $\C ^{w^{\prime }}$, the
action of the cocycle can be also defined on $\T ^{d}\times \C ^{w^{\prime }}$, simply by replacing $S$ by a complex vector.
A particular case of representation that we will strongly use is the adjoint representation of $G$ in $g$.

We can define two projections
\begin{equation} \label{def projections coc}
\begin{array}{cccc}
\pi _{1}: & SW^{s}(\T ^{d},G) & \ra & \T ^{d} \\ 
& (\a ,A(\. )) & \mapsto  & \a
\end{array}
\end{equation}
and
\begin{equation*}
\begin{array}{cccc}
\pi _{2}: & SW^{s}(\T ^{d},G) & \rightarrow  & C^{s}(\T ^{d},G)
\\ 
& (\a ,A(\. )) & \mapsto  & A(\. )%
\end{array}%
\end{equation*}%
and we can identify $C^{s}(\T ^{d},G)$ with $\pi _{1}^{-1}(0)$. This
identification set aside, $\a $ is to be considered minimal.

The $n$-th iterate of the action is given by%
\begin{equation*}
(\a ,A(\. ))^{n}.(x,S)=(n\a ,A_{n}(\. )).(x,S)=(x+n\a 
,A_{n}(x).S)
\end{equation*}%
where $A_{n}(\. )$ represents the \textit{quasiperiodic product} of matrices equal to
\begin{equation*}
A_{n}(\. )=A(\. +(n-1)\a )...A(\. )
\end{equation*}%
for positive iterates. Negative iterates are found as inverses of positive
ones:%
\begin{eqnarray*}
(\a ,A(\. ))^{-n} &=&((\a ,A(\. ))^{n})^{-1} \\
&=&(-n\a ,A^{\ast }(\. -n\a )...A^{\ast }(\. -\a ))
\end{eqnarray*}%
The regularity needed
for the main results is in fact $C^{\infty }$, but some of the intermediate results are proved in very low regularity.

We will also need to consider the case where $G_{j}\hra G$, $j=1,2
$, and $G_{1}\cap G_{2}=Id$. If $(\a ,A_{j}(\. ))$ are cocycles in $%
\T ^{d}\times G_{j}$ over the same rotation in the basis, then we can
define a cocycle in $\T ^{d} \times G$ by their product and we have $%
SW_{\a }^{s}(\T ^{d},G_{1})\times SW_{\a }^{s}(\T ^{d}
,G_{2})\hra SW_{\a }^{s}(\T ,G)$. We will denote this
cocycle by $(\a ,A_{1}\times A_{2}(\. ))$. Clearly, the subgroup $%
G_{1}\times G_{2}\subset G$, as well as $G_{1} $ and $G_{2}$, are invariant under the action
of the product cocycle.

Cocycles appear as Poincaré mappings of \textit{quasiperiodic skew-systems}, the continuous time equivalent of cocycles.
Such systems are defined by the ODE in $\T ^{d} \times G $
\begin{equation}
\begin{cases}
\frac{d}{dx} X = F(x).X \\
\dot{x} = \omega = (\a ,1) 
\end{cases}
\end{equation}
Here, $F(\. ) : \T ^{d} \ra g $, and $\omega \in \R ^{d+1}$. The flow of this ODE satisfies
\begin{equation}
\Phi _{t} (x,S) = (x+t\omega , X(x,t).S)
\end{equation}
where the first factor is a dense geodesic in $\T^{d+1} $ and the second factor satisfies the ODE
\begin{equation}
\frac{d}{dt} X(x,t) = F(x+t\omega ).X(x,t)
\end{equation}
and we consider the fundamental solution, i.e. the one with initial values $X(x,0) = Id $. The first return map
on the vertical circle in $\T ^{d+1} = \T^{d} \times \T $ defines a cocycle in $\T ^{d} \times G$. The study of
such ODEs is an attempt to extend Floquet theory in ODEs with quasi-periodic (instead of periodic)
time-dependent coefficients, and reducibility amounts to the existence of a time-independent
conjugation to a constant vector field.

\section{Classes of cocycles with simple dynamics, conjugation}

The cocycle $(\a ,A(\. ))$ is called a constant cocycle if
$A(\. )=A\in G$ is a constant mapping. In that case,
the quasiperiodic product reduces to a simple product of matrices
\begin{equation*}
(\a ,A)^{n}=(n\a ,A^{n})
\end{equation*}
and the dynamics become easy to describe.

Another, more general, distinct class of cocycles having relatively simple
dynamics is given by the mappings $A(\. )$ taking values in a maximal
torus of $G$. Such a cocycle will be called \textit{abelian}. The interest
of abelian cocycles lies in the the existence of a discrete invariant that
distinguishes qualitatively different types of dynamics (c.f. chapter \ref{Dynamics of abelian cocycles}),
as well as in the fact that they represent the most general cocycles whose iteration can
be made explicit. More precisely, let $(\a ,A(\. ))$ be an abelian cocycle, and call $\TT $ a maximal torus
in which $A(\. )$ takes its values. Let us also fix a basis $(h_{\r })$
of $\ft $, the Lie algebra of $\TT $. With these choices,
there exists a vector $r=(r_{\r })\in \Z ^{w}$, periodic functions of zero mean value $\phi_{\r }(\. )$ and $a=\sum\nolimits_{\r } a_{\r }h_{\r } \in \ft $, such that
\begin{equation}
A(\. )=\exp (\sum\nolimits_{\r }(2\pi r_{\r }\. +\phi_{\r 
}(\. )+a_{\r })h_{\r })  \label{abelian cocycle}
\end{equation}
The iterates of such a cocycle can be calculated explicitly, namely
\begin{equation*}
(\a ,A(\. ))^{n}=(n\a ,\exp (\sum\nolimits_{\r }(2\pi r_{\r 
}n(\. +\frac{(n-1)}{2}\a )+S_{n}^{\a }\phi _{\r }(\. 
)+na_{\r })h_{\r }))
\end{equation*}
where $S_{n}^{\a }\phi (\. )$ stands for the Birkhoff sum of the function
$\phi (\. )$ over the translation by $\a $
\begin{equation*}
S_{n}^{\a }\phi (\. )=\sum_{k=0}^{n-1}\phi (\. +k\a )
\end{equation*}
Clearly, if the values of the mapping $A(\. )$ over some fibers $x$
and $x+\a $ do not commute, the calculation is no longer true.

\bigskip

The group $C^{s}(\T ^{d},G)\hra SW^{s}(\T ^{d},G)$ acts by
\textit{dynamical conjugation}: Let $B(\. )\in C^{s}(\T ^{d},G)$ and $(\a ,A(\. 
))\in SW^{s}(\T ^{d},G)$. Then we define
\begin{eqnarray*}
Conj_{B(\. )}.(\a ,A(\. )) &=&(0,B(\. ))\circ (\a ,A(\. ))\circ (0,B(\. ))^{-1} \\
&=&(\a ,B(\. +\a ).A(\. ).B^{-1}(\. ))
\end{eqnarray*}
which is in fact a change of variables within each fiber of the product
$\T ^{d} \times G$. The dynamics of $Conj_{B(\. )}.(\a ,A(\. ))$
and $(\a ,A(\. ))$ are essentially the same, since
\begin{equation*}
(Conj_{B(\. )}.(\a ,A(\. )))^{n}=(n\a ,B(\. +n\a 
).A_{n}(\. ).B^{-1}(\. ))
\end{equation*}

\begin{definition}
Two cocycles $(\a ,A(\. ))$ and $(\a ,\tilde{A}(\. ))$ in
$SW^{s}_{\a }(\T ^{d},G)$ are
called $C^{s}$-\textit{conjugate} modulo $m\in \N^{\ast }$ iff there exists
$B(\. )\in C^{s}(m\T ,G)$ such that
$(\a ,\tilde{A}(\. ))=Conj_{B(\. )}.(\a ,A(\. ))$. We will use the notation
\begin{equation*}
(\a ,A(\. ))\sim (\a ,\tilde{A}(\. ))
\end{equation*}%
to state that the two cocycles are conjugate to each other.
\end{definition}

We remark that it is indeed possible for cocycles in $SW^{s}_{\a }(\T ^{d},G)$
to be conjugate modulo $m$ with $m \geq 2$. We will encounter such examples later on.

Since the qualitative features of the dynamics are determined up to a conjugation, we are interested in
the density properties of the orbits of the classes described above and this motivates the following definition:

\begin{definition}
A cocycle will be called $C^{s}$-\textit{reducible} (modulo $m$) iff it is $C^{s}$%
-conjugate (modulo $m$) to a constant, and $C^{s}$-\textit{torus-reducible} (modulo $%
m $) iff it is $C^{s}$-conjugate (modulo $m$) to an abelian cocycle.
\end{definition}

In these definitions the conjugacy is implicitly supposed to take its values
in $G$. If we fix a unitary representation of $G \hra U(w')$, we can define a somewhat
weaker notion of (torus-)reducibility, namely reducibility in the ambient
group $U(w^{\prime })$, which is referred to as $U(w^{\prime })$ (torus-)reducibility.

Finally, since it is known that not all cocycles are reducible (e.g. generic
abelian cocycles over Liouvillean rotations, but also cocycles over
Diophantine rotations, even though this result is hard to obtain, see \cite{El2002a})
we also need the following concept.

\begin{definition} \label{def almost reducibility}
A cocycle $(\a ,A(. ))$ is said to be \textit{almost reducible} $\mod m $, $m \in \N ^{*}$,
if there exists sequence of conjugations $B_{n}(\. ) \in C^{\infty}(m \T ^{d},G)$,
such that $Conj_{B_{n}(\. )}.(\a ,A(\. ))$ becomes arbitrarily close to constants
in the $C^{\infty }$ topology, i.e. iff there exists $(A_{n})$, a sequence in $G$, such that%
\begin{equation*}
A_{n}^{\ast } \left(  B_{n}(\. +\a )A(\. )B_{n}^{\ast }(\. ) \right) \ra Id
\end{equation*}%
in $C^{\infty }(m \T ^{d},G)$. Almost torus-reducibility is defined in an analogous way.
\end{definition}

The last and weakest notion of reducibility is that of \textit{quasi-reducibility}. It is a
version of almost reducibility where we allow
$B_{n}(\. ) \in C^{\infty}(P_{n}\T ^{d},G)$ and $P_{n} \ra \infty$
and we demand that
\begin{equation*}
A_{n}^{\ast }B_{n}(P_{n} \. +\a )A(P_{n} \. )B_{n}^{\ast }(P_{n} \. )\ra Id
\end{equation*}%
in $C^{\infty }(\T ^{d},G)$. This notion is better adapted to
cases where construction of a conjugation is not possible without a loss of periodicity.
This phenomenon was present in \cite{KrikAst}, but we have been able to overcome the difficulties related
to it in chapter \ref{Local theory}.

\section{Some considerations and results on reducibility of cocycles}

Let us now study the construction of conjugacies in some cases which have
come to be considered simple. For this reason, we will allow $G$ to be a
more general Lie group than in the rest of the work.

\bigskip

The question of conjugation within the class of constant cocycles has been settled in \cite{KrikAst} (see section
\ref{Conjugation of abelian cocycles} for the statement of the theorem).
As we have already pointed out the significance of abelian cocycles, we will proceed to the study of conjugation
within this class in the aforementioned section. Let us, however, study right-away a particular case, which lies
in the core of K.A.M. theory.

The simplest case of cocycles occurs when the group $G$ is abelian. Cocycles
in $\T \times \R $ where studied for example in \cite{YocAst}
in the resolution of the linearized equation satisfied by commuting
diffeomorphisms of $\Sp ^{1}$. We also refer to \cite{ArnGeomMeth} and \cite{KatHass}. A cocycle
in this setting has the form $(\a ,\phi (\. ))$ and the dynamics are
given by%
\begin{equation*}
(x,y)\mapsto (x+\a ,\phi (x)+y)
\end{equation*}%
Suppose now that two cocycles over the same rotation $\a $ are given, $%
(\a ,\phi _{1}(\. ))$ and $(\a ,\phi _{2}(\. ))$. These
cocycles are conjugate to each other iff there exists $\psi (\. ):\T \ra \R $ such that%
\begin{equation*}
\psi (\. +\a )+\phi _{1}(\. )-\psi (\. )=\phi _{2}(\. )
\end{equation*}%
or%
\begin{equation*}
\psi (\. +\a )-\psi (\. )=\phi _{2}(\. )-\phi _{1}(\. )
\end{equation*}%
Since $\R $ is abelian, conjugating $(\a ,\phi _{1}(\. ))$ to $%
(\a ,\phi _{2}(\. ))$ is equivalent to conjugating $(\a ,\phi
(\. ))$ to $(\a ,0)$, where $\phi (\. )=\phi _{1}(\. )-\phi
_{2}(\. )$, so that this last equation reads%
\begin{equation}
\psi (\. +\a )-\psi (\. )=\phi (\. )
\label{linear cohomological eq}
\end{equation}%
This equation, whose unknown is the conjugant $\psi $ is called a \textit{linear cohomological equation}.

Periodicity of the functions implies that $\hat{\phi}(0)=0$ (or equivalently 
$\hat{\phi}_{1}(0)=\hat{\phi}_{2}(0)$) is a necessary condition for the
existence of such a $\psi $. For this reason $\hat{\phi}(0)$ is called the 
\textit{obstruction} of such an equation. If $\hat{\phi}(0)\not=0$, we can
only solve the equation%
\begin{eqnarray*}
\psi (\. +\a )-\psi (\. ) &=&\phi (\. )-\hat{\phi}(0) \\
&=&\dot{T}\phi (\. )
\end{eqnarray*}

Application of the Fourier transform shows that such a $\psi $ must satisfy,
for $k\not=0$,%
\begin{equation*}
\hat{\psi}(k)=\frac{1}{e^{2i\pi k\a }-1}\hat{\phi}(k)
\end{equation*}%
Since $\a $ is irrational (the equivalent to minimality of $R_{\a }$
for one-frequency cocycles), $e^{2i\pi k\a }-1$ is never $0$, and
therefore eq. \ref{linear cohomological eq} admits as a solution at least a "formal Fourier series",
so long as $\phi (\. )$ is, say, in $L^{2}$. The convergence of
the Fourier series of such a solution is related to \textit{small denominator} phenomena,
since $e^{2i\pi k\a }$ is arbitrarily close to $1$ for an
infinite number of $k$ (e.g. for $k$ equal to some $q_{n}$, the denominators
of continued fractions approximations).

The convergence of $\sum \hat{\psi}(k)e^{2i\pi k\. }$ is guaranteed under
some smoothness conditions on $\phi $ and/or some arithmetic conditions on $%
\a $. A relevant combination of hypotheses in our context gives the following classical lemma.
\begin{lemma}
\label{Linear cohomological eq}Let $\a \in DC(\gamma ,\tau )$ and $%
\phi(\. )\in C^{\infty }(\T ,\R )$. Then, there exists $\psi \in
C^{\infty }(\T ,\R )$ such that 
\begin{equation*}
\psi (x+\a )-\psi (x)=\phi (x)-\hat{\phi }(0)
\end{equation*}

The solution $\psi $ satisfies the estimate $\left\Vert \psi \right\Vert
_{s}\lesssim \gamma \left\Vert \phi \right\Vert _{s+\tau + 1/2}$, where the
constant is universal and depends only on $s$. The solution is unique modulo additive constants.
\end{lemma}

We remind the reader that Diophantine numbers are defined in definition \ref{def DC}, and that such a condition is necessary.
For the convenience of the reader, we include the proof of the lemma.

\begin{proof}
Using the diophantine condition on $\a $, we find that for $k \not= 0$
\begin{equation*}
| \hat{\psi}(k) | \leq \gamma |k|^{\tau} | \hat{\phi}(k) |
\end{equation*}
Therefore,
\begin{eqnarray*}
\| \psi \|_{s}^{2} & \leq & \sum_{k \in \Z}  (1+|k|)^{2s+1} | \hat{\psi}(k) |^{2} \\
&\lesssim & \gamma ^{2} \sum_{k \in \Z}  (1+|k|)^{2s+1} |k|^{2 \tau} | \hat{\phi}(k) |^{2} \\
&\lesssim & \gamma ^{2} \sum_{k \in \Z}  (1+|k|)^{2s+2\tau  +1 } | \hat{\phi}(k) |^{2} \\
&\lesssim & \gamma ^{2} \| \phi \|_{s+ \tau +1/2}^{2}
\end{eqnarray*}
where the constant may change from one line to another.
\end{proof}

We also remark that in the general case where $d \geq 1$, the same result holds if we replace the estimate by
$$\left\Vert \psi \right\Vert
_{s}\lesssim \gamma \left\Vert \phi \right\Vert _{s+\tau + d/2}$$

On the other hand, if $\a $ is Liouvillean, no solution exists in any function space for generic
$C^{\infty }$ (or even $C^{\omega}$) functions.
This kind of arithmetic obstructions to reducibility is therefore already present in the simplest of algebraic settings,
and we will encounter in the guise of \textit{resonances} in the dynamics in the fibers in chapter \ref{Local theory}.

\bigskip

A second kind of obstructions to reducibility is the one related to
homotopy. Cocycles conjugate to each other are also homotopic to each
other, independently of the arithmetic properties of the rotation in the basis.
In order to be precise, let us suppose that $(\a ,A_{i}(\. ))$%
. $i=1,2$, are continuous and conjugate to each other, and that $B(\. )$
is continuous. The path in $C^{0}(\T ^{d},G)$%
\begin{equation*}
B(\. +(1-t)\a )A_{2}(\. )B^{-1}(\. )
\end{equation*}%
deforms continuously $A_{1}(\. )$ to $B(\. )A_{2}(\. )B^{-1}(\. )
$. Since the homotopy group of a compact Lie group $G$ is that of $\mathcal{T%
}\hra G$, where $\TT $ is a maximal torus, we can deform
continuously $A_{2}(\. )$ to $\tilde{A}(\. )$ and $B(\. )$ to $%
\tilde{B}(\. )$ which they take values in $\TT $. As a
consequence, both $A_{1}(\. )$ and $A_{2}(\. )$ are homotopic to $%
\tilde{A}(\. )$.
In particular, a reducible cocycle is homotopic to constants, and so are almost reducible cocycles.

An example of a cocycle which is not reducible and not almost reducible for
this exact reason is%
\begin{equation*}
(\a ,R_{2\pi \. })\in SW^{\infty }(\T ^{d},SO(3))
\end{equation*}%
for any minimal $\a $. This cocycle, as well as $C^{0}$-small perurbations of it, are not reducible. Its second iterate,%
\begin{equation*}
(2\a ,R_{4\pi (\. +\a /2)})
\end{equation*}%
admits a lift in $SW^{\infty }(\T ^{d},SU(2))$, which is $(2\a ,E_{1}(\. +\a /2))$ in the notation that we have adopted. Since $SU(2)
$ is simply connected, there is no topological obstruction to the
reducibility of this cocycle, but it is in fact not reducible. The reason is
a non-trivial theorem by R. Krikorian (cf. theorem \ref{compactness of
iterates of reducible cocycles}) which is still of topological nature, but
this time the topological space is $C^{\infty }(\T ^{d},SU(2))$ and
not the group $SU(2)$ itself. Since the reason for the non-reducibility of $(2\a ,E_{1}(\. +\a /2))$
is not a discrete invariant depending continuously on the mapping in the fibers, one could expect
(and it is in fact true, cf. theorem \ref{RK global density}) that arbitrarily small perturbations of this
cocycle become almost reducible.

One could call such obstructions geometric
ones, since one of the results that we will prove is that, at least for
one-frequency coycles, cocycles which are not conjugate in any way close to
constant ones are those whose dynamics are similar to those of some periodic geodesic of $G$.

\bigskip

The third type of obstructions is of purely dynamical nature and they were
constructed in \cite{El2002a}. The obstruction in this case is unique
ergodicity of the dynamics in $\T ^{d}\times SO(3)$, where the
preserved measure is the product of the Haar measures on the basis and the
fibers. Since a constant cocycle $(\a ,R_{\vartheta })$ admits many invariant
measures, unique ergodicity is incompatible with reducibility.
The construction of these cocycles by a K.A.M. scheme shows in fact that
non-reducible cocycles are not exotic objects that live far from the well
understood constant cocycles, since they are in fact almost reducible.

\bigskip 

We close this chapter by stating some results in the existing litterature
concerning the reducibility of quasiperiofic skew-systems in compact Lie groups. Some of these theorems
are stated in the continuous-time case, but their proofs work equally well in the discrete-time case.

Firstly, let us give two positive results in the local setting, by
R.\ Krikorian, the first one being on the density of reducible cocycles in
the neighborhood of constants, proved in \cite{KrikAst}.

\begin{theorem}
\label{RK local density}
Let $G$ be a compact semisimple Lie group $\omega \in
DC_{d}(\gamma ,\sigma )$, i.e. such that%
\begin{equation*}
|(k,\omega )|\geq \frac{\gamma ^{-1}}{|k|^{\sigma }},~k\in \Z %
^{d}\backslash \{0\}
\end{equation*}%
Then, there exist $\varepsilon _{0},s_{0}>0$ and a natural number $\chi _{G}$%
, depending only on the group $G$, such that, for all $A\in g$ and $F\in C^{\infty }(%
\T ^{d},g)$ satisfying $\left\Vert F\right\Vert _{s_{0}}\leq
\varepsilon _{0}$ and every $\varepsilon ,s>0$, there exists $F^{\prime }\in
C^{\infty }(\chi _{G}\T ^{d},g)$ such that $(\omega /2\pi
,A+F^{\prime }(\. ))$ is reducible $\mod\chi _{G}$ and satisfies $%
\left\Vert F-F^{\prime }\right\Vert _{s}\leq \varepsilon $. If $G=SO(3)$, or 
$SU(w+1)$, one can take $\chi _{G}=1$. In other words, reducible systems are
dense in%
\begin{equation*}
\{A+F(\. ),~A\in g,~F\in C^{\infty }(\T ^{d},g),~\left\Vert
F\right\Vert _{s_{0}}\leq \varepsilon _{0}\},~g=so(3),~su(w+1)
\end{equation*}
\end{theorem}

A step of the proof of the theorem is the proof of quasi-reducibility of systems satisfying
the same smallness condition.

The density properties are quite delicate and (as shows the proof)
reducibility of any given system cannot be concluded, at least
in the K.A.M. constructive sense.
For this reason, such theorems are in general
coupled with theorems where one embeds an individual system or cocycle
in a family of such systems depending on parameters, and then studies the measure-theoretic
abundance of reducible systems with respect to the parameters. The first such theorems,
proofs of which are in the scope of classical K.A.M. theory, asserted that for generic one-parameter
families there is reducibility for a set of positive measure in the parameter space. Moreover, the measure
tends to be full as the size of the perturbation goes to $0$ (the "K.A.M. tongues").
For this part of the theory, one can consult \cite{El1988} and \cite{KrikAst}.
These theorems have been improved to prevalence theorems, i.e. theorems asserting reducibility
in full measure for generic one-parameter families. This improvement has been possible only for
analytic cocycles, in contrast with positive measure theorems who have also been proved in the $C^{\infty}$ case (\cite{KrikAst}).

The prevalence theorem that we cite next was proved in \cite{Krik99}.

\begin{theorem} \label{Krik prevalence}
Let $G$ and $\omega $ as above, $h>0$, $\Lambda \subset \R $ an
interval and $A\in g$, generic. Then, there exists $\epsilon _{0}$,
depending on $\Lambda $, $A$, $\omega $, $h$, such that, if $F(\. )\in
C_{h}^{\omega }(\T ^{d},G)$ satisfies $|F|_{h}<\epsilon _{0}$, then
for $a.e.$ $\lambda \in \Lambda $, $(\omega /2\pi ,\lambda A+F(\. ))$ is
reducible $\mod\chi _{G}$. If $G=SO(3)$, or $SU(w+1)$, one can take $\chi _{G}=1$
\end{theorem}

The local picture as we know it is completed by H. Eliasson's theorem, which we have already discussed.

\begin{theorem}
\label{El ergodic}
Let $\omega $ as above and let $h>0$. Then, there exists a constant $%
C=C(h,\sigma )$ such that the set of $F(\. )\in C_{h}^{\omega }(\T %
^{d},so(3))$ for which the system $(\omega /2\pi ,F(\. ))$ has a unique
invariant measure is a $G_{\delta }$-dense set.
\end{theorem}

We also cite the following theorem of qualitative nature, proved in \cite{KrikAst} (theorem \textit{2.2.3})

\begin{theorem}
\label{compactness of iterates of reducible cocycles}Let $G$ be a compact
connected semisimple Lie group, $\a \in \T ^{d}$ a minimal
translation and $(\a ,A(\. ))\in SW^{\infty }(\T ^{d},G)$,
such that $\{(\a ,A(\. ))^{n}\}_{n\in \Z}$ be precompact in
$SW^{\infty }(\T ^{d},G)$ for the $C^{\infty }$
topology. Then, there exists $A_{0}\in G$ and
$B(\. )\in C^{\infty }(\R^{d}/\tilde{\chi}_{G}\Z^{d},G)$ such that
\begin{equation*}
A(\. )=B(\. +\a )A_{0}B^{\ast }(\. )
\end{equation*}%
The integer $\tilde{\chi}_{G}$ can be chosen equal to $c_{G}\chi _{G}$, and
if $\chi _{G}=1$ it can be chosen equal to $1$.
\end{theorem}
We remind that $\chi _{G}$ is a constant depending only on $G$.
We also remark that the inverse is trivially true, and we give the immediate
corollary

\begin{corollary}
A cocycle $(\a ,A(\. ))\in SW^{\infty }(\T ^{d},G)$ is
reducible, possibly modulo $\tilde{\chi}_{G}$, if, and only if, there
exists $q\in \Z^{\ast }$ such that $(\a ,A(\. ))^{q}$ is reducible.
\end{corollary}

As for the problem of reducibility in an ambient group and intrinsic reducibility,
C. Chavaudret has shown in \cite{Chav11} the following theorem.
\begin{theorem}
Consider a skew-system in $\T ^{d}\times G$, where $G$ is $Sp(n,\R )$ (the real symplectic group of dimension $n$),
$SL(n,\R )$ (the group of real $n\times n$ matrices of determinant $1$), $O(n)$ (the $n$-dimensional orthogonal group),
or $U(n)$. Suppose, moreover, that the system is $GL(n,\C )$-reducible. Then, it is $G$-reducible $\mod 2$. It is
reducible $\mod 1 $ when $G = U(n)$.
\end{theorem}
The equivalence of $U(w^{\prime })$- and $G$-reducibility is unknown.

Finally, we conclude with the results obtained more recently in the global setting, i.e. where there
is no assumption of closeness to constants for the studied system. Although it was preceded by some
preliminary results by K. Fraczek (\cite{Fra2000}), the first general result of global density is due
to R. Krikorian (\cite{Krik2001}):
\begin{theorem}
\label{RK global density}
There exists $\Sigma \subset \T $, of full measure such that, for $\a \in \Sigma$, reducible cocycles
are dense in $SW^{\infty}_{\a }(\T ,SU(2))$ in the $C^{\infty}$ topology.
\end{theorem}

The set $\Sigma$ as it is defined is slightly thinner than $RDC$ (see definition \ref{def RDC}). The proof of the theorem
relies in the renormalization of dynamics, which brakes down the proof of the theorem to the local study of a collection
of models (periodic geodesics of $SU(2)$), followed by the actual study of the models using K.A.M. techniques.
%
%

A corollary of the proof of the theorem is the quantization of the length of rescaled iterates of the cocycle, for rotations
in $\Sigma$ and in regularity $C^{2}$. This result was generalized by K. Fraczek in \cite{Fra2004}:

\begin{theorem}
\label{Fra degree}
For every $(\a ,A(\. ))\in SW^{1}(\T ,SU(2))$, there exists a
measurable and bounded function $\psi :\T \rightarrow su(2)$ such that%
\begin{equation*}
\frac{1}{n}LA_{n}(\. )=\frac{1}{n}a_{n}(\. )\rightarrow \psi (\. )
\end{equation*}%
Lebesgue $a.e.$. Moreover, $\psi $ is invariant by the dynamics of the
cocycle,%
\begin{equation*}
Ad(A(\. ))\psi (\. )=\psi (\. +\a )
\end{equation*}%
and therefore $|\psi (\. )|$ is $a.e.$ a constant.
\end{theorem}

This constant was defined to be the \textit{degree} of the cocycle and in the same paper it was shown that it is \textit{quantized}.
\begin{theorem}
\label{Fra quantif degree}
If $(\a ,A(\. ))\in SW^{2}(\T ,SU(2))$, then its degree is in $2 \pi \mathbb{N}$.
\end{theorem}

We remark that no arithmetic conditions, other than irrationality, are assumed for $\a $, and that from the proof
(which uses renormalization) it follows that the degree as defined by K. Fraczek generalizes the length of
"renormalization representatives" as used by R. Krikorian in his proof.

Subsequently, K.Fraczek showed that
\begin{theorem}
\label{Fra measurable inv}
If $(\a ,A_{i}(\. ))\in SW^{2}(\T ,SU(2))$, $i=1,2$, are measurably conjugate to each other
then their degrees are equal.
\end{theorem}

Finally, K.Fraczek refined the local study of obstructions and showed that if the rotation in the basis is
$RDC$ and the cocycle is of degree $r$, it is conjugate to the obstruction, thus obtaining the following global density theorem.
\begin{theorem}
\label{Fra global density}
Reducible cocycles are dense in $SW^{\infty}_{\a }(\T ,SU(2))$ in the $C^{\infty}$ topology
if $\a \in RDC$.
\end{theorem}

\section{Statement of the main theorem and plan of the proof}

After having established the basic vocabulary, we can state the main theorem of this m\'{e}moire.
\begin{theorem} \label{Global density theorem}
Let $\a \in RDC$ and $(\a ,A(\. ))\in SW^{\infty }(\T ,G)$. If $1\leq m \leq \x _{G} $ is the minimal natural number
such that $(\a ,A(\. ))^{m}$ is homotopic to constants, then $(\a ,A(\. ))^{m}$ is accumulated by reducible cocycles
in $SW_{m \a }^{\infty }(m\T ,G)$.
\end{theorem}
In other words, for such $\a $, any cocycle $(\a ,A(\. ))$ in $\T \times G$ has an iterate which is
accumulated by reducible cocycles and the maximal number of iterations needed in order to satisfy this property depends
only on the group $G$. We remind that the definition of the constant $\x _{G}$ is given by lemma \ref{parties abeliennes}.

We can also give a version of the theorem concerning the cocycle itself, and not perturbations of it. Moreover,
its statement does not involve any loss of periodicity, only iteration for reasons of homotopy.
\begin{theorem} \label{Global almost torus reducibility theorem}
Let $\a \in RDC$, and $(\a ,A(\. ))\in SW^{\infty }(\T ,G)$ and $m$ as above. Then, there exists $\x \in \N ^{*}$,
$\x | m$, such that $(\a ,A(\. ))^{\x }$ is almost torus-reducible in $SW_{\x \a }^{\infty }(\T ,G)$.
\end{theorem}

The proof of these theorems, based on intermediate results some of which have an interest on their own, occupies the
remaining of this mémoire. It combines techniques used in the proof of theorems \ref{RK global density},
\ref{Fra global density} and \ref{RK local density}, which nonetheless have had to be adapted in order to be
applicable in this more general context. Let us give a general plan of the proof.
\subparagraph*{Abelian cocycles}
Firstly, in chapter \ref{Dynamics of abelian cocycles}, we study briefly the basic examples of non-reducible cocycles
obtained easily from theorem \ref{compactness of iterates of reducible cocycles}. These are cocycles
$C^{\infty }$-conjugate to ones of the form $(\a , E_{r }(\. ))$, where $E_{r }(\. )$ is a $1$-periodic geodesic of
$G$ (for the notation, see subsection \ref{Basis of the root system}). If  $(\a ,A(\. ))$ is such a cocycle, we find that by
\begin{equation} \label{inv curve heur}
\lim \dfrac{1}{q_{n}} L A_{q_{n}} (\. )
\end{equation}
with $q_{n}$ as in section \ref{Arithmetics, continued fraction expansion}, we can define a smooth curve
$\ba (\. ) : \T \ra g \setminus \{ 0 \} $ which is invariant under the dynamics:
\begin{equation*}
Ad (A(\. )).\ba(\. ) = \ba(\. +\a) 
\end{equation*}
The invariance of the curve and the minimality of $\a $ imply that the class of the vectors $\ba(\. )$ in $g \mod Inn(g)$
is well defined, and that the curve is in fact traced in a sphere in $g$ (since the adjoint action of $G$ is an isometry).
For such cocycles, in fact, this class is essentially the degree $r \in \Z^{w} $ of the geodesic $E_{r }(\. )$.

Conversely, the
existence of such an invariant curve $\ba(\. )$ implies that, given any vector $s \in g$ in the class of $\ba(\. )$, we can
conjugate the given cocycle to $(\a ,\tilde{A}(\. ))$ commuting with $s$. Two cases are possible : either the vector $s$ is
\textit{regular} (in the algebraic sense of the first chapter), or it is \textit{singular}. In the fist case, since only one
maximal torus $\TT$ passes by $s$, $Ad (\tilde{A}(\. )).s =s$ implies that $\tilde{A}(\. )$ takes values in $\TT$, i.e.
$(\a ,A(\. ))$ is torus-reducible. However, if $s$ is singular, torus-reducibility cannot be concluded by an algebraic reasoning.
Such an argument can only produce what we will call a \textit{splitting} of the dynamics in proposition \ref{construction splitting}.
The splitting consists of two subgroups, $G_{+}$ and $G_{0}$, with
$G_{+}$ abelian and given by the intersection of all maximal tori
passing by $s$, and with $G_{0}$ being the largest subgroup of $G$ commuting with $G_{+}$ and intersecting with $G_{+}$ only at the
$Id$. If the vector $s$ is singular, the cocycle $(\a ,A(\. ))$ can be conjugated to a cocycle taking values in
$G_{0} \times G_{+} \hra G$, which is not abelian. If the invariant curve is obtained as in eq. \ref{inv curve heur},
we can additionally conclude that the part of the dynamics in $G_{+}$ carries all the linear growth of the derivatives
of the cocycle, and the derivatives in the $G_{0}$-part grow sublinearly.
\subparagraph*{The energy of a cocycle}
In chapter \ref{Chapter on energy} we begin the general study of cocycles. We drop the assumption that the cocycle
$(\a ,A(\. ))$ is torus-reducible, but we define, following K. Fraczek, an invariant curve of regularity only $L^{2}(\T ,g)$.
Due to the curve's low regularity, we cannot derive conclusions finer than the fact that if it is \textit{non-degenerate}
(i.e. if it is not equal to $0$ for $a.e. \, x \in \T$), the cocycle is not reducible. On the other hand, since the invariance
relation holds $a.e.$, we can still define the type of $a.e.$ vector in the image of $\ba(\. )$ as the degree of the cocycle. The
length of these vectors is an invariant of the dynamics (just as the degree), which we will call energy. There is a shift in
terminology with respect to K. Fraczek's work. There, the degree of a
cocycle in $\T \times SU(2) $ is defined as a numerical
quantity (which corresponds to our energy), since classes of vectors
in $su(2) \approx \R^{3}$ modulo $Inn(su(2)) \approx SO(3)$
are spheres centered in the origin, thus completely determined by
their radius. In more general groups, however, the length
of vectors in the image of $\ba(\. )$ is an insufficient invariant, and therefore we need to keep the term \textit{degree}
for their type. Subsequently, we show that higher-order derivatives (if they exist) carry no additional information, since
$\frac{1}{n^{\sigma +1} }\partial ^{\sigma }LA_{n}(\. )\ra 0$ in $L^{2}$, provided that
$\partial ^{\sigma }LA(\. ) \in L^{2 } $.

The most important properties of the energy are those proven in Proposition \ref{properties of energy} : the energy
of the $n$-th iterate is nothing else but $n$ times the energy of the cocycle, and the energy of the cocycle $(\a ,A(\. ))$
is not bigger than the energy of the path $A(\. ) : \T \ra G$ (i.e. the $L^{2}$ norm of its derivative), or its length
(i.e. the $L^{1}$ norm of its derivative). These $L^{2}$ and $L^{1}$ norms decrease as $n$ increases due to two causes.
One is the averaging, already present in abelian groups : if $\phi : \T \ra \R$, then its Birkhoff sums converge to
its mean value, $\hat{\phi} (0) $, which is certainly smaller than $\| \phi \| _{L^{i}}, i=1,2$. The other is
the non-alignment of $Ad (A (\. )).a_{n-1} (\. +\a )$ with $a (\. )$
(cf. the invariance relation satisfied by $\ba (\. )$), which results in an at most linear growth of lengths, due to
the triangle inequality.
\subparagraph*{Renormalization}
The low regularity of the \textit{limit object of the dynamics}, the curve $\ba(\. )$, calls for a more sophisticated means of
study of the dynamics. This means is known as the renormalization scheme, and we present and apply it in chapter
\ref{Chapter renormalization}, following R. Krikorian. In this context, renormalization could
be motivated by the observation that, thanks to the invariance of the curve,
\begin{equation} \label{invariance of curve heur}
Ad (A_{(-1)^{n}q_{n}}(\. )).\ba(\. ) = \ba(\. +\beta _{n})
\end{equation}
where the quantities $q_{n}$ and $\beta _{n}$ are given by the continued fractions algorithm (see section
\ref{Arithmetics, continued fraction expansion}). This fact, together with the fact that the passage in the limit
in eq. \ref{invariance of curve heur} is still justified in $L^{2}$, shows that for large $n$, and around every point $\nu $ of
Lebesgue continuity of $\ba(\. )$, $A_{(-1)^{n}q_{n}} (\. )$ looks like a piece of geodesic of speed $(-1)^{n}q_{n}\ba(\nu )$. Let
us fix $0$ as such a point for definiteness.

However, the fact that
$A_{(-1)^{n}q_{n}} (\. ) \approx \exp ( (-1)^{n}q_{n} \ba(0 ) \. ) $
does not imply, at least not directly, that
$(\a , A(\. )) \sim (\a , \exp ( \ba(0 )\. )) $. In order to establish such a
fact, we need to consider $A_{(-1)^{n}q_{n}} (\. )$ 
coupled with the preceding iterate in the continued fractions algorithm, $A_{-(-1)^{n}q_{n-1}} (\. )$.
For $n$ large enough, $A_{-(-1)^{n}q_{n-1}} (\. )$ obviously
also approaches the limit object. The couple of cocycles
\begin{equation*}
(\beta _{n-1}, A_{-(-1)^{n}q_{n-1}} (\. )) \text{ and }
(\beta _{n}, A_{(-1)^{n}q_{n}} (\. ))
\end{equation*}
is to be considered as a linearly independent couple of commuting cocycles
which completely describes the dynamics of the original cocycle. Indeed, it is obtained by lifting the continued fractions
algorithm and interpreting division at each step as iteration of the commuting pair
\begin{equation*}
(1, Id) \text{ and }
(\a ,A(\. ))
\end{equation*}
Here, $1$ codes the initial scale of the torus $\T = \R / \Z$, $\a $ the dynamics
in the basis, $A(\. )$ the dynamics in the fibers. Finally, $Id$ codes the
$1$-periodicity of the mapping $A(\. )$, as well as the fact that the
space in which the dynamics takes place is $\T \times G$,
i.e. that fibers $\{ x \} \times G $ and $\{ x+1 \} \times G $ are
isomorphic via the identity mapping. Both couples
$((1, Id),(\a ,A(\. )))$ and
$((\beta _{n-1}, A_{-(-1)^{n}q_{n-1}} (\. )),
(\beta _{n}, A_{(-1)^{n}q_{n}} (\. )))$ act naturally on
$\R \times G$. They generate the same $\Z ^{2}$ action, modulo
a change of basis of generators, coded by a matrix in $GL(2,\Z)$.

After $n$ steps of the algorithm of continued fractions applied on the
initial couple $((1, Id),(\a ,A(\. )))$, we obtain
$((\beta _{n-1}, A_{-(-1)^{n}q_{n-1}} (\. )),
(\beta _{n}, A_{(-1)^{n}q_{n}} (\. )))$. Let us interpret the
quantities defining the second one.
In a scale reduced to $\beta _{n-1}$,
the dynamics of the cocycle $(\a , A(\. ))$ are represented by the
cocycle $(\beta _{n}, A_{(-1)^{n}q_{n}} (\. ))$,
but acting on a space whose fiber $\{ x + \beta _{n-1} \} \times G $ is isomorphic to
$\{ x \} \times G $ by $S \ra A_{-(-1)^{n}q_{n-1}} (x ).S$ (and not
by the $Id$).
We remark that the distortion of the fibers accumulates as $x$ increases,
so that $x$ in this context is more naturally considered as a real-valued variable.

We can, however, zoom in, in order to change the
scale of the dynamics back to $1$. The zoom, which makes sense since
the cocycles act on $\R \times G$, gives the commuting couple
\begin{equation*}
(1, \tilde{A}_{-(-1)^{n}q_{n-1}} (\. )) \text{ and }
(\a _{n}, \tilde{A}_{(-1)^{n}q_{n}} (\. ))
\end{equation*}
Finally, we can restore the fibers by a conjugation
$(0,B(\. )) : \R \ra G$ such that
\begin{equation*}
B(\. ) = \tilde{A}_{-(-1)^{n}q_{n-1}} (\. +1 ).\tilde{A}_{-(-1)^{n}q_{n-1}} (\. ))^{*}
\end{equation*}
We then obtain a commuting couple of the type
\begin{equation*}
(1, Id) \text{ and }
(\a _{n}, \check{A}_{n} (\. ))
\end{equation*}
which we call a \textit{renormalization representative} of
$(\a, A(\. ))$. A renormalization representative is a cocycle over a
different rotation ($\a_{n}$ instead of $\a$), and of shorter period
($\beta_{n-1}$ instead of $1$) than the original one. It represents,
nonetheless, completely the original cocycle in the following sense.
Conjugation of the representative, or arbitrarily small perturbations
of it translate back to conjugations of the original cocycle.
Consequently, all reducibility and accumulation by reducible cocycles
properties are shared by the two objects.

The important point comes from our first observation in this paragraph
(cf. the formula in eq. \ref{invariance of curve heur}). In our
vocabulary, it means that around a typical point $\nu \in \T $, the two cocycles
\begin{equation*}
(1, \tilde{A}_{-(-1)^{n}q_{n-1}} (\. )) \text{ and }
(\a _{n}, \tilde{A}_{(-1)^{n}q_{n}} (\. ))
\end{equation*}
defining the commuting pair at the $n$-th step of the algorithm
converge to the same limit object of the dynamics, $\ba (\nu )$. This allows
us to obtain a renormalization representative as an expression of
the limit object:
\begin{equation*}
\check{A}_{n} (\. ) \approx \exp (\ba (\nu )\. )
\end{equation*}
Since the representative is a $1$-periodic cocycle, we obtain the
quantization of the degree (theorem \ref{quant deg}):
\begin{equation*}
e^{\ba (\nu )} = Id, \, a.e. \, \nu \in \T
\end{equation*}
We have hidden some problems concerning homotopy and the algebraic
type of the degree (regular or singular), but roughly the convergence
of the
renormalization scheme shows the following thing. If we are willing to
allow the arithmetic properties of the rotation to deteriorate, we can
simplify the expression of our cocycle to a perturbation of its degree.
Then, any conjugation or perturbation that we construct for the
representative (therefore in a distorted phase-space and locally in
$x$) is translated into a conjugation or perturbation of the original
cocycle, well defined for all $x \in \T $ and $1$-periodic.

The homotopy class of a cocycle manifests itself, when significant,
in the context of lemma \ref{parties abeliennes}. Therefore, it is
significant only if the degree is not a regular vector and only if
the cocycle is non-homotopic to constants within $G_{0}$. The way to
cope with this issue is to consider a sublattice of the action, and
thus commuting pairs of the type
\begin{equation*}
(\x, Id) \text{ and }
(\x \a _{n}, \check{A}_{n} (\. ))
\end{equation*}
as renormalization representatives. The number $\x \in \N ^{*}$ is
such that "$A(\. )$ is homotopic to constants in the part of the group
that matters".

The preceding arguments sum up essentially to saying that the ergodicity of the dynamics in the basis results in the
cancellation of non-commutativity in the fibers. Moreover, they interpret the energy $(\a ,A(\. ))$ as the length
of the path coding the dynamics of the $q_{n}$-th iterate of $(\a ,A(\. ))$ in the scale $\beta _{n-1}$.
Moreover, they show how commutativity appears robustly when the lengths in the proper scale are in accordance with
the asymptotic characteristics of the dynamics: initially, non-commutativity makes lengths decrease, until the
ergodic theorem cancels the non-commutativity. Finally, let us remark that this derivation shows that the fact that the
group $G$ is not commutative does not affect the asymptotic properties of the dynamics and only marginally
(appearing through homotopy) the models of the dynamics: the
quantization of the degree is the same as for
cocyles in $\T \times \T ^{w} $ and the models of dynamics are (non-abelian) perturbations of geodesics in
$\T ^{w} \hra G $.

Finally, the observation that conjugation between cocycles acts on the limit curve by algebraic conjugation (see the proof
of proposition \ref{properties of energy}), can serve as a motivation for K. Fraczek's theorem on the measurable invariance
of the degree, which we generalize and strengthen in terms of its
regularity assumptions (theorem \ref{thm measurable invariance}).
Since the invariant curve defining the degree of a $C^{1}$ cocycle is measurable, it is reasonable to expect that measurable
conjugation of $C^{1}$ cocycles, which preserves the properties of their limit curves, should preserve the degree.
We do not know, however, a proof of this fact not using
renormalization, although this property is not related to it.

\subparagraph*{A priori estimates for perturbations of regular cocycles} \label{plan a priori estimates}
The effect of non-commutativity is exploited in chapter \ref{A priori estimates on perturbations of regular geodesics}.
At this point, we suppose that renormalization of a given cocycle is
not bound to simplify its dynamics, i.e. that the cocycle
$(\a ,A(\. ))$ is given as a perturbation of a geodesic $(\a ,E_{r,a}(\. ))$
whose derivative is equal to the degree of $(\a ,A(\. ))$.
Under this assumption, we obtain some necessary conditions satisfied by
the perturbation, summed up in lemma \ref{A priori estimates}.

In order to motivate these estimates, let us restrain ourselves to the case where $G = SU(2)$ and consider
a cocycle whose dynamics in the fibers are given by a perturbation of $E_{r} (\.) =\exp (\{2\pi r\. ,0\}_{su(2)})$,
with $r \in \N ^{*}$ (for the notation, see subsection \ref{Notation and algebra in su2}). We remark that a perturbation
in the diagonal direction (i.e. in the direction commuting with the privileged direction of the dynamics) dies out thanks
to the additive ergodic theorem. We suppose, therefore, that the
perturbation is non-zero only in the complex direction of
$su(2) \approx \R \times \C$, and write the perturbation in the form $U(\. ) : x \mapsto \{0 ,U_{z} (\. ) \}_{su(2)}$, with
$U_{z} (\. ) : \T \ra \C $ small enough. Therefore, we consider a cocycle of the form $(\a ,E_{r,a}(\. )e^{U(\. )})$,
with such a mapping $U(\. )$, restrained in the plain orthogonal to $e_{r}$.

Then again, in this simplified context, two possibilities can be distinguished. Either $0 \in \C $
is in the bounded component of the complement of the curve $U_{z}(\. )$, or in the unbounded one, depending on the
significance of the constant part of the perturbation over its non-constant part.

For the first case, let us suppose moreover that the perturbation has
only one frequency, and thus of the form
$U_{z}(\. )= \e e^{2i \pi k \.} $ with $k \neq 0$, so that $ U(\. )= \{0, \e e^{2i \pi k \.} \}$.
\footnote{The coefficient $\e $ should in fact be chosen in $\C ^{\ast}$, but its absolute value is the significant part,
as the imaginary part only introduces a phase in the perturbation, which by the geometric argument given below should
se easily seen to be irrelevant. The parameter $\e $ remains, nonetheless a $2$-dimensional real parameter for considerations
concening the dimension of obstructions.}
In this case, $\partial U(\. )= \{0, 2i\pi \e k e^{2i \pi k \.} \}$ is orthogonal to
$U(\. )$, and thus does not commute with it. Transfer of $\partial e^{U(\. )}$ to $g =T_{Id}$ gives the first order term
$\partial U(\. )$ and the second order term $\frac{1}{2} [U(\. ),\partial U(\. )]$.
This second term is, by elementary vector algebra in $\R ^{3}$, parallel to $e_{r} = L E_{r,a}(\. )$, i.e. in the real
line of $\R \times \C$, and its value is $\{ 2\pi \e ^{2} k, 0 \} $.

Since $[ \. , \. ] \equiv 2 \. \times \. $, we obtain that if
$k > 0$ (i.e. if the circle of radius $\e$ in $\C$ is traversed in the direction
of $e_{r}$) all terms contribute positively in the calculation of the $L^{2}$ norm
of the derivative. If, however, $k < 0$, the bracket $\{ 2\pi \e ^{2} k, 0 \} $ is in the opposite direction than $e_{r}$ :
the path $A(\. )$ is traversed with a speed equal to $ 2\pi (r- \e ^{2} |k|)$, constantly smaller than $2\pi r$ at least in the
diagonal direction. On the other hand, the term in the complex direction, $\partial U(\. )$, contributes to the $L^{2}$ norm by
$( 2\pi k \e ) ^{2} $. Therefore, if $k \leq -2r$, this drift in
the real direction is compensated by the movement within the complex one. On the other hand, for $-2r+1 \leq k <0$, the
contribution remains negative. The same calculation when all Fourier modes are authorized in the perturbation shows that,
at least in the order of $\e ^{2}$, the frequencies are not mixed, so
that the conclusion remains valid: if the perturbation
rotates clockwise in $\{0 \} \times \C$ around $\R$ sufficiently
slowly, the counter-clockwise rotation of the geodesic
does not compensate it, and the cocycle cannot be of degree $r$.

Up to now there is no dynamics in the estimates, but it comes into effect when the perturbation is of the second type,
i.e. if $0$ is in the unbounded component of the complement of $U_{z}(\. )$. In this case, the constant part is the principal
part of the perturbation, so let us suppose that the perturbation is actually constant :
\begin{equation*}
A (\.) =E_{r} (\.) e^{\{ 0, z\} }
\end{equation*}
where $z \in \C ^{*}$ is small. The path is merely a (right) translate of the geodesic $E_{r} (\.) $, so that is length
is that of the geodesic. Surprisingly enough, in the second iterate of this cocycle the previous phenomenon appears:
\begin{eqnarray*}
A (\. +\a).A (\. ) &=& E_{r} (\. +\a) e^{\{ 0, z \} }.E_{r} (\.) e^{\{ 0, z\} } \\
&=&  E_{2r} (\. ) E_{r} (\a ) e^{\{ 0, z\  e^{-4i\pi r \.} \} }.e^{\{ 0, z\} }
\end{eqnarray*}
The constants $E_{r}(\a )$ and $e^{\{ 0, z\} } $ are insignificant, and the previous reasoning applies with $r $ replaced
by $2r$. In both cases, the fact that estimation of the length is continuous in the $C^{1}$ topology completes the picture
when $G = SU(2)$. For definiteness and comparability with the general case, we remind that not $r$, but $2r$ corresponds to
$r_{\r }$ in the notation for general Lie groups.

Finally, since $su(2)$ is embedded in any semi-simple compact Lie algebra,
and since the interaction of different such subalgebras is of higher order, the result remains true for any
compact algebra $g$.

These estimates allow a preliminary description of the configuration in
$SW^{\infty}_{\a}(\T ,G)$ of the conjugacy classes
of periodic geodesics (section \ref{Density of zero-energy cocycles}),
for any $\a \in \T \setminus \Q $. From the estimates,
we obtain directly that every model $(\a , E_{r}(\. ))$ is in the
boundary of an open set (in the $C^{1}$ topology) of cocycles of energy
strictly smaller than $|e_{r}|$. Since any cocycle
$(\a ,A(\. ))$ has renormalization representatives
$\mod \x $ arbitrarily close to such a geodesic, we can perturb such a renormalization representative and obtain a cocycle of
smaller energy. This perturbed cocycle is the renormalization representative of a cocycle
$(\x \a ,A'(\. )) \in SW^{\infty} (\T ,G) $. However, due to the distortion-restoration of the fibers and the rescaling,
we cannot relate the smallness of the perturbation of the renormalization representatives to the distance of the cocycles
$(\a ,A(\. ))$ and $( \a ,A'(\. ))$. We can say, nonetheless, that $A(\. )$ can be continuously deformed to
$A^{\prime}(\. )$, and that, if we go deep enough into renormalization, we can make the perturbation of the
representative arbitrarily small. 

As a conclusion, we can say that the different conjugacy classes are
not too far from each other, unlike their discrete
repartition in $SW^{\infty}(\T ,\T ^{w} )$. Additionally, it is
precisely the non-abelian character of $G$ which destroys this
structure.
On the other hand, this argument cannot be pushed any further, and a
more precise conclusion demands a study of the
renormalization representatives as cocycles. If some renormalization
representative of $(\a ,A(\. ))$ can be conjugated arbitrarily close to
the limit object, since the distortion of the fibers and the rescaling factor have been kept constant, the same follows for
the cocycle $(\a ,A(\. ))$ by inverting the renormalization scheme.
Likewise, if the renormalization representative is accumulated by
cocycles reducible to the limit object, we can conclude that the
property of accumulation by lower energy cocycles holds equally for the
cocycle $(\a ,A(\. ))$ itself.

\subparagraph*{Perturbations of regular geodesics}
The local study of the limit objects starts in chapter \ref{Normal form theorem} with the local study of regular geodesics.
From this chapter on, unlike with the previous ones, we suppose that the rotation $\a $ satisfies some recurrent
Diophantine condition (see definition \ref{def RDC}), so that K.A.M. and K.A.M.-like methods are applicable.

In this chapter we suppose that the cocycle
$(\a_{n},E_{r,a_{n}}(\. ).e^{U_{n}(\. )})$ is given,
where $\a_{n} \in DC$ with \textit{a priori} fixed constants,
$U_{n}(\. ) : \T \ra g$ satisfies some smallness condition in $C^{\infty}$ and $E_{r,a_{n}}(\. )$ is a regular geodesic.
The hypothesis that $\a_{n } \in DC $ is satisfied by the definition of $RDC$ if we consider a subsequence in the renormalization
scheme. Moreover, we suppose that the degree of the renormalization representative of some cocycle is $r$, and therefore equal to the
degree of $(\a_{n },E_{r,a_{n }}(\. ))$. 

These assumptions pose the problem as that of the study of perturbations of regular geodesics over Diophantine rotations,
the perturbations being of the same degree as the geodesics. In order to simplify the notation and since we have established
the connection of this setting with the global problem, we suppose that the cocycle is given in the form
$(\a ,E_{r,a }(\. ).e^{U (\. )})$ is given, where $\a \in DC$.

The regularity of $E_{r,a}(\. )$ implies that for every embedding $su(2) \hra g$ corresponding to a root $\r $,
there exists an integer $r_{\r } \neq 0$ such that $Ad(E_{r,a}(\. )).j_{\r } = 2i\pi r_{\r } j_{\r}$. So, we suppose,
with no serious loss of generality \footnote{The only loss is that we are forced to consider only $r_{\r } \in 2\Z^{*} $
rather than in $\Z^{*}$, but this does not affect our discussion. The discussion remains true of one formally substitutes
$r= 1/2$.}
that $G = SU(2)$, and that $E_{r,a}(\. ) = \{ e^{2i\pi r ( \. +a)}, 0 \}_{SU(2)}$.

The a priori estimates described just above indicate that only a part
of the spectrum of the perturbation is significant. It corresponds to
the part of the perturbation of the the limit object which cannot be
evened out by the additive ergodic theorem.

\bigskip

Before connecting this fact with the local theory of regular geodesics,
let us make a small digression and state a more general remark on the
local theory of models of cocycles. Let us suppose that the cocycle
\begin{equation*}
(\a ,A(\. ))
\end{equation*}
is a model of dynamics in which we are interested, and let us consider a perturbation of it in the form
\begin{equation*}
(\a ,A(\. ) e^{U(\. )})
\end{equation*}
where $U(\. ) : \T \ra g$ is small. What one would wish in the context of the local theory of $(\a ,A(\. ))$ would be to
find a conjugation, hopefully of the order of $U(\. )$, conjugating the perturbed cocycle to the exact model. The
conjugant is supposed to be of the form $\exp (Y(\. ))$, where again $Y(\. )$ is supposed to be small. The equation satisfied
by such a $Y(\. )$ is a non-linear one, namely
\begin{equation*}
e^{Y(\. +\a )}.A(\. ).e^{U(\. )}.e^{-Y(\. )}=A(\. )
\end{equation*}
or
\begin{equation*}
e^{Ad(A^{\ast }(\. )). Y(\. +\a )}.e^{U(\. )}.e^{-Y(\. )}=Id
\end{equation*}
\textit{The exact model acts on the conjugant by algebraic conjugation along the translation by $\a $}.
Since this equation is non-linear, we opt for the solution of the
linearized one and therefore renounce in single-handedly conjugating
the perturbed cocycle to the exact model. The linearized equation reads
\begin{equation} \label{linearized local eq heur}
Ad(A^{\ast }(\. )).Y(\. +\a )-Y(\. )=-U(\. )
\end{equation}
By solving this equation, we can only reduce $(\a ,A(\. ))$ to
\begin{equation*}
e^{Y(\. +\a )}.A(\. ).e^{U(\. )}.e^{-Y(\. )}=A(\. )e^{U'(\. )}
\end{equation*}
with $U'(\. )$ much smaller than the initial pertrubation $U(\. )$
(we say that it is "of second order"). This will hopefully (i.e. under
some relevant assumptions) give rise to an iterative scheme which will
settle the question. Unfortunately, even this is not possible, at least
not in general, as, for a given $A(\. )$, the image of the operator
\begin{equation} \label{def operator heur}
Y(\. ) \mapsto Ad(A^{\ast }(\. )).Y(\. +\a )-Y(\. )
\end{equation}
may not be closed in $C^{\infty } (\T ,g)$, or its closure may not be
equal to $C^{\infty } (\T ,g)$. In the first case, some first order
phenomena can not be reduced with good estimates by a conjugation of
the same order (we will call such phenomena \textit{resonances} in
chapter \ref{Local theory})\footnote{Exact resonances belong actually
to the second case of obstructions.} . In the second case, the
perturbation in
the complementary space cannot be reduced by a conjugation of the
same order. We will see in the case of regular cocycles that such
perturbations cannot be reduced by conjugations of any order.
The space complementary to the image of the operator constitutes the
\textit{dynamical obstructions} to reducibility to the exact model.
Perturbations supported in the obstructions are the exit door to the
conjugacy class of the model. This is quite straightforward in the case
of regular geodesics, and quite trickier in the case of constant
cocycles.

\bigskip

Let us now end our digression and return to the special case under consideration. The conjugation equation, studied in section
\ref{Preliminary transformations and notation}, then reads
\begin{equation*}
Ad(E_{r,a}^{\ast }(x)).Y(x+\a )-Y(x)=-U(x)
\end{equation*}
which, if we note $Y(\. ) = \{ Y_{t} (\. ),Y_{z} (\. ) \}_{su(2)} $, decomposes into
\begin{eqnarray*}
Y_{t}(\. +\a )-Y_{t}(\. ) &=&-U_{t} (\. ) \\
e^{-4i\pi r (\. +a) }.Y_{z }(\. +\a )-Y_{z }(\. ) &=&-U_{z }(\. )
\label{lin eq c heur}
\end{eqnarray*}

In the first equation, small divisor phenomena appear, but they are related only to $\a $, so that the assumption that
$\a \in DC $ settles this part of the perturbation (see lemma \ref{Linear cohomological eq}). We remind that the only
topological obstruction to the solution of the equation concerns the mean value of the rhs function, which is an abelian
constant, so that its dynamics are of the same type as those of
$(\a,E_{r,a}(\. ))$, only with a different $a$.

The second equation reads, in the space of frequencies,
\begin{equation}
e^{2i\pi ( (k+2r)\a - 2a)}\hat{Y}_{z }(k+2r)-\hat{Y}_{z }(k)=\hat{U}_{z }(k)  \label{dif eq heur}
\end{equation}
for all $k\in \Z $. It is a difference equation of step $2r$ in the space of rapidly decaying sequences, so that its free
parameters are $2r$ terms of the sequence. Since these parameters,
which are in fact the obstructions, are freely chosen, we can place
them in the interval $\{ -2r+1, \cdots ,0 \} $. We thus establish
that $(\a,E_{r,a}(\. ))$ perturbed by obstructions will have different
dynamics if the obstructions are strong enough: in that case, the a
priori estimates show that the cocycle cannot be of degree $r$. This
observation gives rise to the K.A.M.
scheme of \cite{Krik2001}, a variant of which we use
in chapter \ref{Perturbations of singular geodesics}. This method,
sufficient for the proof of the density theorem, shows that
if the cocycle $(\a, A(\. ))$ is close enough to $(\a,E_{r,a}(\. ))$, the smallness of the perturbation depending on $\a \in DC$,
and if $(\a, A(\. ))$ is of degree $r$, then it is smoothly conjugate to $(\a,E_{r,a'}(\. ))$. A fundamental reason for
the convergence of the scheme (apart from the legitimate assumption on the degree) is the absence of small divisor phenomena
in the dynamics in the fibers: there are no such phenomena in eq. \ref{dif eq heur} and as a consequence the conjugation
not only always exists, but it also is of the order of the perturbation.

This last observation triggers a more general question. Since the core of the difficulties of K.A.M. theory is absent in this
context, can we say more? This is done by dropping the assumption on the degree of $(\a, A(\. ))$, thus authorizing non-zero
obstructions.

The answer could be seen to be affirmative in the following way. Let us assume that $(\a,E_{r,a}(\. )e^{U(\. )})$ is given,
and let us write $U(\. ) = U_{cob}(\. )+U_{ob}(\. )$, where $U_{cob}(\. )$ is in the image\footnote{\textit{Cob} stands
for coboundary and \textit{Ob} for obstructions} of the operator introduced in eq. \ref{def operator heur},
and $U_{ob}(\. )$ is in a complementary space, and it is spectrally supported in $\{ -2r+1, \. s ,0 \}$. Then, there
exists $Y(\. ) \in C^{\infty } (\T ,g)$ such that
\begin{equation*}
Ad(E_{-r,-a}(x)).Y(x+\a )-Y(x)=-U_{cob}(x)
\end{equation*}
Moreover, as we will see in section \ref{Local reduction lemmas regular geodesics}, $Y(\. )$ satisfies good (tame) estimates with
respect to $U(\. )$, where the constants and the loss of derivatives depend only on $\a $, and therefore are constant
throughout the scheme. Consequently, we can write
\begin{eqnarray*}
E_{r,a}(\. )e^{U(\. )} &=& E_{r,a}(\. )e^{U_{cob}(\. )+U'(\. )}e^{U_{ob}(\. )} \\
&=& e^{-Y(\. +\a)} E_{r,a' }(\. )e^{U'(\. )} e^{U_{ob}(\. )} e^{Y(\. )}
\end{eqnarray*}
In these equations, $U'(\. )$ may change from one line to another, but it is quadratic with respect to $U(\. )$ and
$Y(\. )$, and $a' =a+\hat{U}_{t}(0)$. The next step, however, has to be different, since, in general, $U_{ob}(\. )$ 
will be of the order of $U(\. )$. This takes us to the study of perturbations of models of the type
\begin{equation*}
(\a ,E_{r,a }(\. ).e^{P (\. )})
\end{equation*}
where $P (\. )$ takes values in the space of obstructions
\footnote{We keep implicit the dependence of the obstructions with
respect to the abelian constant $a$, since for $a'$ close to $a$ the space of obstructions of $(\a ,E_{r,a' }(\. ))$ is a
finite-dimensional vector space, canonically isomorphic to the obstructions of $(\a ,E_{r,a }(\. ))$. In order to see this,
is suffices to check that, in the case where $G=SU(2)$ we can normalize $a$ to $0$ by introducing a translation
$x\mapsto x+ \lambda $, and then verify that, in proposition \ref{solution of difference equation}, the obstructions
$\Gamma $ behave well under such translations for $\lambda $ seen now as a parameter close to $0$.
Finally, the dimension depends only on $r$.}
and satisfies some fixed a priori bound, and $U (\. )$ is at least quadratic with respect to $P (\. )$. Then, conjugation
in the neighborhood of this model reads
\begin{equation*}
e^{Y(\. +\a )}.E_{r,a }(\. ).e^{U(\. )}.e^{P(\. )}.e^{-Y(\. )}=E_{r,a' }(\. ).e^{U'(\. )}.e^{P'(\. )}
\end{equation*}
where $Y(\. )$ is of the order of $U(\. )$, $U'(\. )$ is quadratic with respect to $U(\. )$, and $P'(\. )$ is close to
$P(\. )$ up to a correction of the order of $U(\. )$.

Given the a priori bounds on $P(\. )$, and anticipating that $Y(\. )$ be of the order of $U(\. )$, we find that, up to quadratic
terms, the preceeding equation is equivalent to
\begin{equation*}
e^{Y(\. +\a )}.E_{r,a }(\. ).e^{U(\. )}.e^{-Y(\. )}.e^{P(\. )}=E_{r,a' }(\. ).e^{U'(\. )}.e^{P'(\. )}
\end{equation*}
where we have used the fact that the operator $Id - Ad(\exp (Y(\. )))$ is of the order of
$Y(\. )$ (by the estimates on the Hausdorff-Campbell formula of section \ref{estimate of quadratics}), and thus
$Ad(\exp (Y(\. ))).P(\. ) = P(\. )$, plus terms of the order of $U(\. )$ which we have incorporated in $U(\. )$.
\footnote{There is a hidden classical fixed point argument here. In the final proof that we give it does not appear in this
way, but at this point the reader can think of this trick as a trial-and-error procedure, where the terms coming from
$[Id - Ad(\exp (Y(\. )))].P(\. )$ are gradually incorporated in $U(\. )$ until equilibrium is reached. The absence of small
divisors guarantees convergence.}
Rewriting this last equation in a more familiar form gives
\begin{equation*}
e^{Y(\. +\a )}.E_{r,a }(\. ).e^{U(\. )}.e^{-Y(\. )}=E_{r,a' }(\. ).e^{U'(\. )}.e^{P'(\. )-P(\. )}
\end{equation*}
which is exactly the case treated previously:
\begin{equation*}
e^{Y(\. +\a )}.E_{r,a }(\. ).e^{U(\. )}.e^{-Y(\. )}=E_{r,a' }(\. ).e^{U'(\. )}.e^{U_{Ob}(\. )}
\end{equation*}
and the step is concluded if we call $P'(\. )=P(\. ) +U^{\prime}_{Ob}(\. )$, $a' =a+\hat{U}_{t}(0)$, and $U'(\. )$ the
quadratic term. Iteration is possible, since if the a priori bounds are
satisfied with some margin by $P(\. )$, then they are equally satisfied
with some margin by $P'(\. )$, and the diagonal constant $a'$ is not
too far from $a$. Therefore, the scheme whose iterative step we just
described should converge and produce a conjugation $\tilde{Y}(\. )$
satisfying
\begin{equation*}
e^{\tilde{Y}(\. +\a )}.E_{r,a }(\. ).e^{U(\. )}.e^{-\tilde{Y}(\. )}.e^{P(\. )}=E_{r,a' }(\. ).e^{\tilde{P}(\. )}
\end{equation*}
where $\tilde{P}(\. )$ takes its values in the space of obstructions.

The scheme in the way we presented it looks more than a classical fixed-point argument than a K.A.M. scheme, and this is
true in a way, since the loss of derivatives and the constants entering the estimates depend only on $r$ and $\a$, which
are kept constant throughout the scheme, and small divisors are only a secondary phenomenon. The actual proof, however,
is a convergent classical K.A.M. scheme, or its equivalent, the Hamilton inverse function theorem (theorem
\ref{Hamilton inverse function}). The exactness of the statements of
theorems \ref{Normal Form Thm} and
\ref{NK global reducibility to regular geodesics} is precisely due to
the secondary effect of small divisor phenomena in this context.

\subparagraph*{Perturbations of singular geodesics} Since the natural place of chapter \ref{Perturbations of singular geodesics},
at least from the dynamical point of view, is actually before chapter \ref{Local theory}, let us try to motivate the need for
a more accurate local theory of constant cocycles, which made chapter \ref{Local theory} necessary. We also hope that it will
serve the reader as an introduction for the difficulties of the K.A.M. scheme finally applied in chapter
\ref{Perturbations of singular geodesics}.

Singular cocycles stand in the middle between regular ones and cocycles of $0$ energy, in the sense that their obstruction
to reducibility, the invariant curve $\ba (\. )$ is non-degenerate, but in the same time weaker than the obstruction of a
regular cocycle. The first manifestation of this fact is already present in renormalization, since a singular
cocycle can be renormalized to a commuting pair of the type
\begin{equation*}
(1 , \tilde{C} _{0}^{(n)} (\. )) \text{ and } (\a _{n}, E_{r}( \.) \tilde{A} _{0}^{(n)} (\. ) )
\end{equation*}
where $\tilde{A} _{0}^{(n)} (\. ) $ and $\tilde{C} _{0}^{(n)} (\. ) $ are close to constants commuting with
$e_{r} = LE_{r} =\ba (0 )$. However, unlike for regular cocycles,
in general we cannot bring this pair to the canonical form
\begin{equation*}
(1 , Id) \text{ and } (\a _{n}, E_{r}( \.) \tilde{A} _{0}^{(n)} (\. ) )
\end{equation*}
unless we iterate $\x _{0}$ times first, where $\x _{0}$ is the constant $\x _{G_{0}}$ related to $G_{0}$, the largest
semisimple subgroup of $G$ not commuting with $\ba (0)$. We can, nonetheless, always obtain a commuting pair of the type
\begin{equation*}
(\x _{0} , Id) \text{ and } (\x _{0} \a _{n}, E_{\x _{0}r} \tilde{A} _{0}^{(n)} (\. ) )
\end{equation*}
with $\tilde{A} _{0}^{(n)} (\. )$ close to a constant $\tilde{A} _{0}^{(n)}$ belonging to the same maximal torus as
$E_{r} (\. )$. By rescaling by a factor of $\x _{0}$ and introducing some obvious notation, we can bring ourselves
to the case where the $1$-periodic coycle $(\a _{n}, E_{r} (\. ) \tilde{A} _{0}^{(n)} e^{U_{n}(\. )} )$ is given. The
$RDC$ condition on $\a $ allows us to assume that $\a _{n} \in DC$, so by still simplifying the notation we obtain the
cocycle $(\a , E_{r} (\. ) A e^{U(\. )} )$ with $U(\. ) \in C^{\infty} (\T ,g)$ arbitrarily small and $\a \in DC$.

A second manifestation of the intermediate character of singular cocycles is the fact that they can be obtained as
renormalization representatives of positive energy perturbations of regular cocycles, something that should become clear
to the reader in the body of the mémoire.

This intermediate character makes itself clearer in their local theory, i.e. in the reduction of
cocycles of the form $(\a , E_{r} (\. ) A e^{U(\. )} )$ of degree $r$, with $e_{r} \in g$ a singular vector.
In this case, the exact model that we consider is the normal form $(\a , E_{r} (\. ) A )$, and the conjugation
equation \ref{linearized local eq heur} then reads
\begin{equation*}
Ad(A^{\ast }E_{r}^{\ast } (\. )).Y(\. +\a )-Y(\. )=-U(\. )
\end{equation*}
We can directly partition roots into two sets. The roots $\r \in \D _{+}$ such that $[e_{r}, j_{\r }] \neq 0$ called
$I^{(+)}$, and the roots satisfying $[e_{r}, j_{\r }] = 0$, called $I^{(0)}$. The fact that the geodesic is singular
implies that both sets are non-empty: The emptiness of $I^{(+)}$ corresponds to a degenerate curve $\ba (\. )$, while the
emptiness of $I^{(0)}$ corresponds to a regular geodesic. The fact that $ E_{r} (\. ) $ and $A$ belong to the same maximal
torus allows us to write the equations separately in each eigenspace. The equation for the coordinates in the torus
is the same as the one for $Y_{t} (\. )$ in \ref{lin eq c heur}, while the one in the direction $\r \in I^{(+)} $ is of the
same type as that for $Y_{z} (\. )$. If, however, $\r \in I^{(0)} $,
then the direction $j_{\r }$ commutes with $e_{r}$, so that only the
constant part of the normal form acts on the conjugation via
the operator associated to the cocycle $(\a , E_{r} (\. ) A )$ as in
eq. \ref{def operator heur}. This is the characteristic
of the perturbation theory of constant cocycles and the source of small
divisor phenomena in the fibers.

The same observation holds in fact for the effect of perturbations of the normal form in the directions $j_{\r }$. If $\r \in I^{(+)} $,
the positive energy part of the dynamics, any perturbation of the normal form $(\a , E_{r} (\. ) A )$ in the
direction $j_{\r }$ with suitably chosen frequencies has energy strictly smaller than $|e_{r}|$. On the other hand,
any small enough perturbation in  $j_{\r }$ with $\r \in I^{(0)} $, the $0$ energy part of the dynamics, does not
affect the energy, since it is evened out by the Ergodic theorem. Continuity of such estimates with respect to the
$C^{1}$ norm allows us to conclude that if $U(\. )$ is $\e $-small in $C^{1}$ with $\e $ small enough, then a perturbation
of the order of $\e $ is sufficient in order to obtain a cocycle of smaller energy. One only needs to chose
\begin{equation*}
(\a , E_{r} (\. ) A e^{U(\. )} e^{z j_{\r }} )
\end{equation*}
with $|z| $ big enough, but of the order of $\e $, and
$\r \in I^{(+)}$. In other words, a theorem that can allow the proof
of the global density theorem is the almost reducibility to normal
forms, but with a controllable growth of the conjugations with respect
to the decay of perturbations. Let us state what we need more clearly.
Suppose that for $(\a , E_{r} (\. ) A e^{U(\. )} )$
of degree $r$ and with $U(\. )$ small enough, there exists a sequence of conjugations $B_{n} (\. ) :\T \ra G$ such that
\begin{equation*}
Conj_{B_{n}}(\a , E_{r} (\. ) A e^{U(\. )} ) = (\a , E_{r} (\. ) A_{n} e^{U_{n}(\. )} )
\end{equation*}
with $A_{n}$ some sequence of constants commuting with $E_{r} (\. )$, and $U_{n}(\. ) \ra 0$ in $C^{\infty}$. Then,
conjugation is equivalent to
\begin{eqnarray*}
E_{r} (\. ) A e^{U(\. )}  &=& B_{n}^{\ast }(\. +\a) E_{r} (\. ) A_{n} e^{U_{n}(\. )} B_{n}(\. ) \\
&=& B_{n}^{\ast }(\. +\a)  E_{r} (\. ) A_{n} B_{n}(\. ) e^{Ad(B_{n}^{\ast}(\. )). U_{n}(\. )}
\end{eqnarray*}
This is an instance of a more general fact. If $(\a , A(\. ))$ is a
model of dynamics and
$(\a , \tilde{A} (\. )) = Conj_{B(\. )}(\a , A(\. ))$ is another
expression of the same model, then
\begin{equation*}
(\a , \tilde{A}(\. ) e^{Ad(B(\. )).U(\. )} ) = Conj_{B(\. )}(\a , A(\. )
e^{U(\. )})
\end{equation*}
We can say that \textit{dynamical conjugation acts on perturbations by algebraic conjugation}.

Therefore we can perturb $(\a , E_{r} (\. ) A_{n} e^{U_{n}(\. )} )$
and obtain a cocycle of smaller energy. If
$Ad(B_{n}^{\ast}(\. )). U_{n}(\. ) \ra 0$ in $C^{\infty}$, then the
perturbation of the initial cocyle $(\a , E_{r} (\. ) A e^{U(\. )} )$
is also arbitrarily small. This is a necessary condition for the
density of reducible cocycles.

The independence of positive and $0$-energy phenomena, at least in the first order, motivates a scheme where the initial
cocycle $(\a , E_{r} (\. ) A e^{U(\. )} )$ is supposed to be of degree $r$, and the perturbation is small enough. Then,
at each step the conjugation in the positive and $0$-energy parts are applied separately for the linearized problem,
so that $(\a , E_{r} (\. ) A_{n} e^{U_{n}(\. )} )$ is conjugated to
\begin{equation*}
(\a , E_{r} (\. ) A_{n} e^{ObU_{n}(\. )} e^{U_{n+1}(\. )} )
\end{equation*}
where $ObU_{n}(\. )$ are the positive energy obstructions of $U_{n}(\. )$, as defined in the study of regular geodesics.
Then, since $U_{n+1}(\. )$ is quadratic with respect to $U_{n}(\. )$, $ObU_{n}(\. )$ must also be quadratic, otherwise
the cocycle $(\a , E_{r} (\. ) A_{n} e^{ObU_{n}(\. )} e^{U_{n+1}(\. )} )$, which is conjugate to
$(\a , E_{r} (\. ) A_{n} e^{U_{n}(\. )} )$ would not be of degree $r$, a contradiction. The scheme is therefore
ready to be iterated.

Such a strategy, however, would demand the almost reducibility theorem for perturbations of constants, which was
not available. Its weaker equivalent, the almost quasi-reducibility theorem, is sufficient for the proof of the local
density theorem by the following trick, used by R. Krikorian
\footnote{We present a simplified version of it, for the sake of comprehension of the exposition.}.
The $1$-periodic cocycle $(\a ,Ae^{U(\. )}))$, sufficiently
close to the constant $(\a ,A)$, after $n$ steps of the K.A.M. scheme is conjugated by
$B_{n} (\. ) \in C^{\infty} (cD^{n}\T ,G )$ to a $D^{n}$-periodic cocycle
$(\a ,A_{n}e^{U_{n}(\. )}))$, where $U_{n}(D^{n}\. ) \ra 0$ in $C^{\infty}$ exponentially fast. Then, the mapping
$U_{n}(\. )$ is proved to be increasingly well approximated by $1$-periodic mappings $U_{n}^{\prime}(\. )$, something
with is used in order to obtain a new sequence of conjugations $B_{n}^{\prime}(\. ) \in C^{\infty} (\x _{G}\T ,G )$,
reducing the cocycle $(\a ,Ae^{U(\. )}))$ to $(\a ,A_{n}e^{\tilde{U}_{n}(\. )}))$, where still $\tilde{U}_{n}(\. )$
goes to $0 \in C^{\infty} (\x _{G} \T ,g)$ exponentially fast, while conjugations grow polynomially. Then,
embedding of the cocycle $(\a ,A_{n}e^{\tilde{U}_{n}(\. )}))$ in a suitably chosen one-parameter family and
use of the reducibility in positive measure guarantees that there exists a parturbation of $(\a ,A_{n}e^{\tilde{U}_{n}(\. )}))$
of the order of $\tilde{U}_{n}(\. )$ which is reducible. The conjugation reducing this last cocycle is of the order of
the perturbation.

%

A relevant observation is that in general there exist $\r \in I^{(+)} $ and $\r ' \in I^{(0)} $
such that $[j_{\r }, j_{\r '}] \neq 0$, for example the vectors $j_{1 }$ and $j_{2}$ in the example of a singular
cocycle in the next section. As a result, the independence of the reduction in the positive ($I^{(+)}$) and $0$
($I^{(0)}$) energy parts of the dynamics holds only in the first order, and the loss of periodicity in conjugation
within the $G_{0 }$ propagates to the rest of the perturbation. This makes the procedure of regaining periodicity
more delicate. Suppose that we have obtained the $D^{n}$-periodic cocycle
$(\a , E_{r} (\. ) A_{n} e^{U_{n}(\. )} )$, with $U_{n}(\. )$ very small. Imitation of R. Krikorian's trick for regaining
periodicity would amount to writing the mapping in the form
\begin{equation*}
E_{r} (\. ) A_{n} e^{U_{n}(\. )} = E_{r} (\. ) A_{n} e^{U_{n}^{(0)}(\. )} e^{U_{n} ^{(+)}(\. )}
\end{equation*}
where $U_{n}^{(0)}(\. ) \approx \pi _{g_{0}} U_{n}(\. )$ and $U_{n} ^{(+)}(\. )$ has a $0$ projection in $g_{0}$.
By the local theory, there exists a perturbation of the cocycle $(\a ,A_{n} e^{U_{n}^{(0)}(\. )})$ of the order of
$U_{n} (\. )$ such that the new cocycle, $(\a ,A_{n} e^{\tilde{U} _{n}^{(0)}(\. )})$ is reducible and
the conjugation $\tilde{Y} _{n}^{(0)}(\. )$ is of the order of $\tilde{U} _{n}^{(0)}(\. )$. The cocyle
\begin{equation*}
(\a , E_{r} (\. ) A_{n} e^{\tilde{U} _{n}^{(0)}(\. )} e^{U_{n} ^{(+)}(\. )})
\end{equation*}
can then be conjugated to the one defined by the mapping
\begin{equation*}
E_{r} (\. ) A_{n}^{\prime} e^{Ad(\tilde{Y} _{n}^{(0)}(\. )).U _{n} ^{(+)}(\. )})
\end{equation*}
The non-commutativity of the positive and $0$-energy parts, however, results in the perturbation being quadratic, but
non-zero, in the $0$-energy part. Therefore, only one step of reduction can be made without having to deal with
small divisor phenomena. It should be possible to iterate the procedure
via a fixed point argument, but it is obviously a complication in the
proof.

%

We think that problems arising from the loss of periodicity in the
local theory of constant cocycles are not impossible to overcome.
Neverthelesss, the complications in the proof of the global density
theorem, such us the need for a priori estimates on perturbations of
singular geodesics, where the perturbations have arbitrarily long
periods, or the adaptation of the argument using the reducibility in
positive measure would make the argument very complicated. Finally, the
existent local theory does not fit into, and we were not able to make
it fit into, the point of view that we have taken up concerning the
role of obstructions in the dynamics. These reasons motivated the
content of chapter \ref{Local theory}, which we will resume now.

\subparagraph*{Revisiting the local theory} We now treat the case which, with respect to the dynamics of the
initial cocycle, corresponds to the curve $\ba (\. )$ being degenerate to $0 \in g$, i.e. the case of $0$ energy. Under this
assumption, the initial cocycle $(\a ,A(\. ))$ can be renormalized to commuting pairs of the form
\begin{equation*}
(\beta _{n-1}, \tilde{C}^{(n)})\text{ and }(\beta _{n}, \tilde{A} ^{(n)})
\end{equation*}
up to an arbitrarily small error, and therefore has renormalization representatives modulo $\x _{G}$ arbitrarily close
$(\x _{G}  \a _{n}, \tilde{A} _{0}^{(n)} )$, where $\tilde{A} _{0}^{(n)} \in G$ is a constant. These cocycles are
$\x _{G}$-periodic, but we can rescale the torus $\T $ and see them as $1$-periodic cocycles over the rotation $\a_{n}$.
Again, the $RDC$ condition on $\a $ implies that for a subsequence of such renormalization representatives, $\a _{n}$ is in
$DC$, which we will suppose from now on. Finally, we simplify the notation and suppose that we are given the cocycle
$(\a ,Ae^{U(\. )})$, with $\a \in DC$, $A\in G$, and $U(\. ): \T \ra g$ satisfying some smallness condition.

The reason for the differences between the theory of perturbations of
periodic geodesics and that of constants is the importance of small
divisor phenomena. These phenomena are in the core of the local theory,
and the complications that they cause are not a failure of the methods.
The proof of theorem \ref{RK local density} relies on a generalized
K.A.M. scheme, following the results of H. Eliasson (see theorem
\ref{El ergodic}). In this context a cocycle of the form $(\a ,Ae^{U(\. )})$ is given, with $U(\. ): \T \ra g$ small enough.
Then, the local conjugation equation \ref{linearized local eq heur} reads
\begin{equation*}
Ad(A^{\ast }).Y(\. +\a )-Y(\. )=-U(\. )
\end{equation*}
The operator associated the constant cocycle $(\a , A)$ always has a positive codimension image in $C^{\infty} (\T ,g)$.
This codimension is bounded by $f$, the real dimension of $g$, equal to $w+2q$, where $w$ is the rank and $q$ is the number
of positive roots. The first $w$ constraints, always present, come from the toral coordinates, where the reduction
equation is as for $Y_{t} (\. )$ in eq. \ref{lin eq c heur}. These obstructions, just as in the local study of regular
geodesics, are not important, since they are always constants, and thus of the same type as the exact model. The rest
$2q$ constraints come from the complex directions in $g$. For the direction $j_{\r }$ the equation reads
\begin{equation*}
e^{-2i\pi a_{\r }}.Y_{\r }(\. +\a )-Y_{\r }(\. ) = -U_{\r }(\. )
\end{equation*}
or, in the frequency space,
\begin{equation} \label{small divisors heur}
(e^{-2i\pi (a_{\r } -k\a )} -1).\hat{Y}_{\r } (k)= -U_{\r }(k )
\end{equation}
Clearly, this equation is not solvable for the mode $k_{\r } \in \Z $ if it satisfies
\begin{equation} \label{def resonance heur}
a_{\r } -k_{\r }\a  \in \Z
\end{equation}
These obstructions are not always are present, since such a $k_{\r }$ not always exists. If the $k_{\r }$ satisfying eq.
\ref{def resonance heur} (which is automatically unique, since $\a \notin \Q$)) is equal to $0$, the obstruction
is of the same type as the constant model around which we have linearized. Therefore, in this case the obstruction is
not important. On the other hand, if $k_{\r } \neq 0$, then the corresponding obstruction is a non-constant perturbation
and thus of a different type than $(\a ,A)$. For this reason, we distinguish this case and call such a root $\r $ a
\textit{resonant root}. A constant in $G$ is resonant if it has a resonant root.

The first results in the direction of theorem \ref{Krik prevalence} were obtained by systematically ignoring resonant constants,
and this approach allowed the proof of reducibility for a positive measure set in the parameter space. It is based on classical
K.A.M. theory which in this context amounts to supposing that up to a rapidly growing truncation order $N_{n}$, the roots
$a_{\r }$ are $K_{n}$-away from resonances, with $K_{n}$ comparable to $N_{n}$. Then, for frequencies up to $N_{n}$ the equation
\ref{small divisors heur} can be solved with good estimates, while if some $C^{s_{0}}$ norm ($s_{0}$ is fixed throughout the
scheme) is small enough with respect to $N_{n}$, the frequencies not eliminated form a quadratic term. In this procedure,
one needs at each step, by imposing transversality conditions to the one-parameter family, to assure that the non-resonance
conditions are satisfied for a sufficiently large set of parameters.

If the cocycle is not embedded in such a family, in order to make reduction possible, one has to be able to cope with
resonant constants. In this part, we follow and adapt the proof of theorem \ref{RK local density}, a corollary of
the proof of which is local quasi-reducibility (see definition \ref{def almost reducibility}). The proof of this theorem was directly
based on the approach taken up by H. Eliasson. In section \ref{reduction section} we reinterpret H. Eliasson's reduction
of resonances as a reduction of obstructions. More precisely, let us
suppose that at the $n$-th step of the algorithm we have obtained a
cocycle $(\a ,A_{n}e^{U_{n}(\. )})$ by successive conjugations acting on $(\a ,A.e^{U(\. )})$. Then, either the constant
$A_{n}$ is sufficiently away (i.e. $K_{n}$-away) from resonances up to the order of truncation $N_{n}$, or not.
In the first case, all obstructions are constant, so the preceding
procedure works. In the second case, we partition
the roots into three classes. The ones for which the obstruction is constant, called $I_{0}$ in lemma
\ref{reduction of resonant modes}, those for which the obstruction is non-constant, $I_{r}$, and those for which there is
no obstruction, $I_{d}$. In the corollary (cor. \ref{reduction to cst}) to this lemma, we prove that if the perturbation
$U_{n}(\. )$ of the constant cocycle $(\a ,A_{n})$ were spectrally supported in the obstructions, then the cocycle
$(\a ,A_{n}.e^{U_{n}(\. )})$ would be reducible, but by a conjugation far from the identity. Our interpretation of
the role of the obstructions calls us to rewrite the cocycle in the form
\begin{equation*}
(\a ,A_{n}.e^{U_{n}(\. )}) = (\a ,A_{n}.e^{Ob_{n}U_{n}(\. )}.e^{U_{n}^{\prime}(\. )})
\end{equation*}
where $Ob_{n}U_{n}(\. )$ stands for the obstructions of $U_{n}(\. )$ with respect to conjugation in the neighborhood of
$(\a ,A_{n})$ (i.e. to the operator \ref{def operator heur}), and $U_{n}^{\prime}(\. )$ has only second order terms
in the space of obstructions. Corollary \ref{reduction to cst} suggests that $(\a ,A_{n}.e^{Ob_{n}U_{n}(\. )})$ is a good
candidate for the model around which we linearize, and the technical part of the scheme confirms it: if we center the
truncation on the obstructions by means of the operator $T_{N_{n}}^{(\bk )}$ defined in section \ref{The local conjugation
lemma, local theory} (see eq. \ref{truncation operators local theory} and \ref{decentered truncation}) conjugation works
like in the presence of only constant obstructions. In this way, we can conjugate
$(\a ,A_{n}.e^{Ob_{n}U_{n}(\. )}.e^{U_{n}^{\prime}(\. )})$ to $(\a ,A_{n}.e^{Ob_{n}U_{n}(\. )}.e^{\tilde{U}_{n+1}(\. )})$
with a close-to-the-identity conjugation, where $\tilde{U}_{n+1}(\. )$
is quadratic with respect to $U_{n}(\. )$. The
estimates show that reduction of the principal part of the new
perturbation, $Ob_{n}U_{n}(\. )$, does not change the
quadratic character of $\tilde{U}_{n+1}(\. )$, so that, by conjugating with the reduction of obstructions, we obtain the cocycle
$(\a ,A_{n}.e^{U_{n+1}(\. )})$ with $U_{n+1}(\. )$ much smaller than $U_{n}(\. )$, and iteration is possible. The iteration
of this K.A.M. scheme proves directly the almost reducibility of any cocycle sufficiently close to a constant one, and
the comparison of the (polynomial) rate of growth of conjugants versus the (exponential) decay of the norms of perturbations
proves the local density of reducible cocycles. This is the content of corollary \ref{corollary for iteration}, which
is the conjugation lemma used in the iteration. Convergence of the scheme (which immediately proves local almost reducibility,
i.e. theorem \ref{NK almost reducibility}) is proved in section \ref{The iteration local theory}, and the section is concluded
with the proof of the (already known) local density theorem (theorem \ref{NK local density}).

In fact, we could skip the reduction of the obstructions at each step and construct a K.A.M. scheme for the reduction
of any given cocycle (satisfying the smallness conditions) arbitrarily close to reducible ones by a product of converging
conjugations. This is another natural concept of almost reducibility. Then, comparison of the norm of the conjugant
reducing the limit cocycles to constants with the decay of the perturbation would yield the same results. This would be in fact
more faithful to our interpretation of the role of obstructions, but the construction of such a scheme is needlessly
complicated.

\section{A concrete example} \label{A concrete example}

In order to make clear the proof of global density theorem, we will take up the study of an example in $\T \times SU(3)$.
It is a simply connected group, for which $\chi _{SU(3)} = 1$, so we avoid problems related to homotopy, which will be
coped with in section \ref{Density of zero-energy cocycles}. For $SU(3)$ we have $w=2$, $q=3$ and $f=8$, for the rank,
the number of roots and the real dimension of $su(3)$.

Let us introduce the notations
\begin{eqnarray*}
j_{1} &=&
\begin{bmatrix}
0 & 1 & 0 \\ 
-1 & 0 & 0 \\ 
0 & 0 & 0
\end{bmatrix}, \,
j_{2} =
\begin{bmatrix}
0 & 0 & 0 \\ 
0 & 0 & 1 \\ 
0 & -1 & 0
\end{bmatrix}, \,
j_{3} =
\begin{bmatrix}
0 & 0 & 1 \\ 
0 & 0 & 0 \\ 
-1 & 0 & 0
\end{bmatrix} \\
h_{1} &=&
\begin{bmatrix}
i & 0 & 0 \\ 
0 & -i & 0 \\ 
0 & 0 & 0
\end{bmatrix}, \,
h_{2} =
\begin{bmatrix}
0 & 0 & 0 \\ 
0 & i & 0 \\ 
0 & 0 & -i
\end{bmatrix}
\end{eqnarray*}
We also note that $[j_{l},ij_{l}]=2h_{l}$, for $l=1,2,3$, and normalize the Cartan-Killing norm so that the norms of all
vectors are equal to $1$. Finally, we introduce the corresponding numbering of roots,
\begin{eqnarray*}
\r _{1} &:& h  \mapsto ( t_{1} -t_{2} ) \\
\r _{2} &:& h  \mapsto ( t_{2} -t_{3} ) \\
\r _{3} &:& h  \mapsto ( t_{1} -t_{3})
\end{eqnarray*}
for
\begin{equation*} h = 2i \pi
\begin{bmatrix}
t_{1} & 0 & 0 \\ 
0 & t_{2} & 0 \\ 
0 & 0 & t_{3}
\end{bmatrix}
\end{equation*}
with $t_{1}+t_{1}+t_{3}=0$. Then, $\{ \r_{1} , \r_{2} \}$ form the Weyl's basis, $\r_{3} = \r_{1} + \r_{2}$ and
$\{ \r_{1} , \r_{2}, \r_{3} \}$ are the positive roots. For a mapping $U(\. ) : \T \ra su(3)$ we will use the coordinates
\begin{equation*}
U(\. )=\sum \nolimits _{i=1,2} U_{i}^{\ft} (\. )h_{i} + \sum \nolimits _{i=1,2,3} U_{i} (\. ) j_{i}
\end{equation*}

Let, now, $A(\. ) : \T \ra SU(3) $ be such that $\| LA (\. ) \|_{L^{1}} \geq |e_{(1,3)}| $, where
\begin{equation*}
e_{(1,3)} =
\begin{bmatrix}
2i \pi & 0 & 0 \\ 
0 & 4i \pi & 0 \\ 
0 & 0 & -6i \pi
\end{bmatrix}
= 2\pi h_{1} +6\pi h_{2}
\end{equation*}
For this vector, $\r _{1} (e_{(1,3)}) = -1$, $\r _{2} (e_{(1,3)}) = 5$, $\r _{3} (e_{(1,3)}) = 4$. Since all these
integers are non-zero, the vector is regular.

By proposition \ref{properties of energy}, the condition on the length of the path $A(\. ) : \T \ra SU(3) $ is necessary
for the cocycle $(\a ,A(\. ) ) $ to be of degree $e_{(1,3)}$, and it is satisfied if, say, $\| LA (\. ) \|_{L^{1}} \geq 200 \pi$.
\subparagraph*{The energy of the cocycle}
Let us suppose that the sequence $\ba _{n} (\. ) = LA_{n} (\. )$ converges to a non-degenerate curve
$\ba (\. ) \in L^{2} (\T ,g)$, and suppose moreover that $a.e.$ vector in the image of $\ba (\. )$ is a regular vector in
$su(3)$. At this point, the results proved in chapter \ref{Chapter on energy}, namely corollary
\ref{Diagonalization of regular cocycles} and lemma \ref{higher order} allow us to make some preliminary conclusions. The
corollary shows that the cocycle can be measurably diagonalized. Finally, the lemma shows that $n^{-(s+1)}\partial ^{s} LA_{n} (\. )$
goes to $0$ in $L^{2}$. Estimation of the decay of norms in $L^{2}$ uses the fact that
$\| \partial ^{s} LA_{n} (\. )\|_{\infty}$ grow no faster than $n^{s+1}$, by lemma \ref{growth lemma}.
\subparagraph*{Renormalization of a regular cocycle}
The next step in the study of the dynamics of the cocycle is renormalization of the dynamics, which we will follow under the
assumption that $(\a ,A(\. ) ) $ is a regular cocycle. The convergence of the scheme (theorem \ref{Quantization of degree})
shows that, for almost every $\nu \in \T$, the cocycle $(\a ,A(\. ) ) $ can be renormalized to commuting pairs
\begin{eqnarray*}
(1,\tilde{C}_{\nu }^{(n)}(\. )) \\
(\a _{n},\tilde{A}_{\nu }^{(n)}(\. ))
\end{eqnarray*}
(notice that we have already introduced the rescaling factor) satisfying the following property.
Both mappings $\tilde{C}_{\nu }^{(n)}(\. )$ (which codes the distortion of the fibers) and $\tilde{A}_{\nu }^{(n)}(\. )$
(which codes the rescaled dynamics in the distorted phase space) approach $\ba (\nu )$:
\begin{eqnarray*}
\tilde{A}^{(n)}_{ \nu}(x) &=&\tilde{A}^{(n)}_{\nu } (0 ).\exp ((-1)^{n} q_{n} \beta _{n-1} \ba (\nu )x+O(\varepsilon _{n}(\nu ))) \\
\tilde{C}^{(n)}_{ \nu}(x) &=&\tilde{C}^{(n)}_{\nu } (0 ).\exp (-(-1)^{n} q_{n-1} \beta _{n-1} \ba (\nu )x+O(\varepsilon _{n}(\nu )))
\end{eqnarray*}
where both constants $\tilde{C}^{(n)}_{\nu } (0 )$ and $\tilde{A}^{(n)}_{\nu } (0 )$ commute with $\exp (\R \ba (\nu ) )$.
Since we assume that $\ba (\nu )$ is a regular vector, we have that $\r_{i} (\ba (\nu )) \neq 0$, for $i=1,2,3$, which
is clearly equivalent to the three diagonal entries being distinct. Elementary linear algebra the shows that both
constants $\tilde{C}^{(n)}_{\nu } (0 )$ and $\tilde{A}^{(n)}_{\nu } (0 )$ have to be diagonal.

Commutation of the couple implies directly that
\begin{equation*}
\tilde{C}^{(n)}_{\nu } (0 ).\tilde{A}^{(n)}_{\nu } (0 )\exp ( \ba (\nu )x)
=\tilde{A}^{(n)}_{\nu } (0 )\exp ( \ba (\nu )(x+1).\tilde{C}^{(n)}_{\nu } (0 )
\end{equation*}
Using the commutation of the constants with $\ba (\nu )$, we find that
\begin{equation*}
\tilde{C}^{(n)}_{\nu } (0 ).\tilde{A}^{(n)}_{\nu } (0 )\exp ( \ba (\nu ) )
=\tilde{A}^{(n)}_{\nu } (0 ).\tilde{C}^{(n)}_{\nu } (0 )
\end{equation*}
up to an arbitrarily small error. Since both constants are diagonal, the above equation implies that
$\exp ( \ba (\nu ) )= Id$, in accordance with theorem \ref{quant deg}.

If we normalize the commuting couple, we obtain a representative which
is arbitrarily close to
\begin{equation*}
(\a _{n} ,\tilde{A}^{(n)}_{\nu } (0 ) e^{\ba (\nu )x} )
\end{equation*}
This is the content of theorem \ref{Quantization of degree} in the case
$G= SU(3)$ and with the additional hypothesis that $\ba (\nu )$ is
regular.
\subparagraph*{Perturbations of regular geodesics}
In order to be able to illustrate the local study of regular geodesics more clearly, we reintroduce the assumption that
the degree of $(\a, A(\. ))$ is $e_{(1,3)}$, so that the renormalization representatives have the form
\begin{equation*}
(\a _{n} ,E_{(1,3),a_{n}}(\. ) e^{U(\. )}) )
\end{equation*}
where
\begin{equation*}
E_{(1,3),a_{n}}(\. ) =
\begin{pmatrix}
e^{2i \pi \.} & 0 & 0 \\ 
0 & e^{4i \pi \.} & 0 \\ 
0 & 0 & e^{-6i \pi \. }
\end{pmatrix}
.
\begin{pmatrix}
e^{2i \pi a_{1,n}} & 0 & 0 \\ 
0 & e^{2i \pi (a_{2,n} - a_{1,n})} & 0 \\ 
0 & 0 & e^{-2i \pi a_{2,n}}
\end{pmatrix}
\end{equation*}
And we also suppose that $n$ is such that $\a _{n} \in DC$.

There are exactly three embeddings of $su(2)$ in $su(3)$ sending the positive root of $su(2)$ to a positive one in $su(3)$,
corresponding to the three positive roots of $su(3)$. Since $[e_{(1,3)},j_{1}]=-2i\pi j_{1}$, $[e_{(1,3)},j_{2}]=10i\pi j_{2}$
and $[e_{(1,3)},j_{3}]=8i\pi j_{3}$, theorem \ref{Normal Form Thm} implies that any small enough perturbation of this periodic
geodesic can be conjugated to a cocycle of the form
\begin{equation*}
(\a_{n} E_{(1,3),a_{n}+a_{n}^{\prime}}(\. ). e^{P(\. )})
\end{equation*}
with $P(\. )$ of the form
\begin{equation*}
\exp \left(
\hat{P}_{1}(0) j_{1} + \sum \nolimits _{-9}^{0} \hat{P}_{2}(k)e^{2i\pi k \.} j_{2}
+ \sum \nolimits _{-7}^{0} \hat{P}_{k}(k)e^{2i\pi k \.}j_{3}
\right)
\end{equation*}
where the $\hat{P} _{i} (k)$ are small complex numbers, and $a_{n}^{\prime} $ is close to $0$. Then, the a priori estimates
of the chapter \ref{A priori estimates on perturbations of regular geodesics} show that the cocycle is of degree $(1,3)$
if, and only if, all the $\hat{P} _{i} (k)$ in this previous expression are $0$.

Moreover, they show that if these coefficients are non-zero, the cocycle is of energy strictly less than $|e_{(1,3)}|$.
Even though it is possible to perturb this cocycle to a regular one (if we perturb it by $\hat{P} _{1} (0) \neq 0$),
we will examine the perturbation by $\hat{P} _{2} (0) \neq 0$. This perturbation is in the subalgebra associated to
the root $\r _{2}$, so that is will lead renormalization  a cocycle of degree $e_{(1,l)}$, with $l$ equal to $0,1$ or $2$.
Let us suppose that the perturbed cocycle is in fact of degree $e_{(1,2)}$, where
\begin{equation*}
e_{(1,2)} =
\begin{bmatrix}
2i \pi & 0 & 0 \\ 
0 & 2i \pi & 0 \\ 
0 & 0 & -4i \pi
\end{bmatrix}
= 2\pi h_{1} +4\pi h_{2}
\end{equation*}
We remark that the complex direction $\C j_{1}$ commutes with this vector, unlike the other two complex directions.

\subparagraph*{Renormalization of a singular cocycle} The hypothesis that the degree of the new perturbed cocycle
$(\a_{n} ,A_{1} (\. ))$ is $e_{(1,2)}$ implies that it can be renormalized to commuting pairs of the form ($m > n $)
\begin{eqnarray*}
\tilde{A}^{(m)}_{1, \nu}(x) &=&\tilde{A}^{(m)}_{1, \nu } (0 ).\exp ((-1)^{m} q_{n} \beta _{m-1} \ba _{1} (\nu )x+O(\varepsilon _{m}(\nu ))) \\
\tilde{C}^{(m)}_{1 \nu}(x) &=&\tilde{C}^{(m)}_{1, \nu } (0 ).\exp (-(-1)^{m} q_{n-1} \beta _{m-1} \ba _{1} (\nu )x+O(\varepsilon _{m}(\nu )))
\end{eqnarray*}
where both constants $\tilde{C}^{(m)}_{\nu } (0 )$ and $\tilde{A}^{(m)}_{\nu } (0 )$ commute with $\exp (\R \ba _{1} (\nu ) )$
and we have called $\ba _{1} (\. )$ the invariant curve of the cocycle $(\a ,A_{1} (\. ))$.

Clearly, commutation of the two constants with $\ba _{1} (\nu ) \approx e_{(1,2)}$ no longer implies that
the constants are diagonal. They are constrained, however, in the subgroup
\begin{equation*}
\begin{pmatrix}
SU(2) & 0 \\ 
0 & 1 
\end{pmatrix}
\times
\exp (\R e_{(1,2)}) = G_{0} \times G_{+}
\end{equation*}
which is the splitting (cf. definition \ref{def splitting} and theorem \ref{construction splitting} ) associated to $e_{(1,2)}$.

Then, as in the case of a regular cocycle, we can conjugate the pair to
\begin{eqnarray*}
(1 , \tilde{C}^{(m)}_{\nu } (0 ) ) \\
(\a _{m} , \tilde{A}^{(m)}_{\nu } (0 ).\exp ( \ba _{1} (\nu )x) )
\end{eqnarray*}
plus error terms. Since this is still a commuting pair, we have, as in the regular case,
\begin{equation*}
\tilde{C}^{(m)}_{\nu } (0 ).\tilde{A}^{(m)}_{\nu } (0 )\exp ( \ba _{1} (\nu ))
=\tilde{A}^{(m)}_{\nu } (0 ).\tilde{C}^{(m)}_{\nu } (0 )
\end{equation*}
Since both constants take values in the splitting $G_{0} \times G_{+}$ and the factor $G_{+}$ is abelian,
the commutator of the constants is in $SU(2)$. Therefore we obtain that $\ba _{1} (\nu )$ has to be a preimage of the $Id$.
Then, the preceding equation simplifies to
\begin{equation*}
\tilde{C}^{(m)}_{\nu } (0 ).\tilde{A}^{(m)}_{\nu }(0 )
=\tilde{A}^{(m)}_{\nu } (0 ).\tilde{C}^{(m)}_{\nu } (0 )
\end{equation*}
up to an arbitrarily small error. Since, now, the semisimple factor $G_{0}$ of the splitting is in fact $SU(2)$, we
obtain once again that the constants are both diagonal. For general groups, though, this is not true and in order
to gain this property we may need to consider powers of the constants, a complication explained in the proof of theorem
\ref{Quantization of degree}.

Following the procedure already described, we obtain the renormalization representative
\begin{equation*}
(\a _{m} ,\tilde{A}^{(n)}_{1, \nu } (0 ) e^{\ba _{1} (\nu )x} )
\end{equation*}
only with a different constant $\tilde{A}^{(m)}_{1, \nu } (0 )$, which is on the same torus as $\ba _{1} (\nu )$.
This is the content of theorem \ref{Quantization of degree} in the case $G= SU(3)$ and for a singular $\ba _{1} (\nu )$.

\subparagraph*{Perturbations of singular geodesics} This leads us to the local study of singular geodesics. Here,
we suppose that the given cocycle is of the form
\begin{equation*}
(\a _{m} , E_{(1,2)} (\. ) A e^{U(\. )})
\end{equation*}
with $A \in G$ on the same torus as $E_{(1,2)} (\. )$, and $U(\. ) \in C^{\infty} (\T ,g)$ satisfying some smallness
condition depending on $\a _{m} \in DC$. We also suppose that the cocycle is of degree $(1,2)$. Then, the iterative procedure
described in chapter \ref{Perturbations of singular geodesics} allows us to conjugate this cocycle arbitrarily close to cocycles
in normal form:
\begin{equation*}
(\a _{m} , E_{(1,2)} (\. ) A_{k} e^{U_{k}(\. )})
\end{equation*}
with $A_{k} \in G$ on the same torus as $E_{(1,2)} (\. )$, and $U_{k}(\. )$ converging exponentially fast to
$0 \in C^{\infty} (\T ,g)$, while the conjugations grow only polynomially. Then, we need only to perturb the cocycle like
\begin{equation*}
(\a _{n} , E_{(1,2)} (\. ) A_{k} e^{U_{k}(\. )}e^{z_{k} j_{i}} )
\end{equation*}
where $z_{k} \in \C ^{\ast }$ is of the order of $U_{k}(\. )$ and $i=2,3$, in order to obtain a cocycle of smaller enery.

We remark that, since in general we cannot hope
to conjugate the cocycle to the normal form $(\a _{n} , E_{(1,2)} (\. ) A_{k})$, we cannot control the way in which
the energy will decrease : all models $E_{(j,k)}$ with $j=0,1$ and $k=0,1,2$ are possible, with the exception of $(1,2)$.

\subparagraph*{Renormalization of $0$ energy cocycles} After a finite number of perturbations (the maximal number of
perturbations  is bounded by a number depending only on the degree of the cocyle) $(\a ,A(\. ))$), all of which can be made
arbitrarily small, and renormalization of the perturbed cocycles and conjugation of renormalization representatives, we will
obtain a cocycle of $0$ energy. This cocyle can be renormalized to commuting pairs of the type ($l >m$)
\begin{eqnarray*}
(1 , \tilde{C}^{(l)}_{\nu } (0 ) ) \\
(\a _{l} , \tilde{A}^{(l)}_{\nu } (0 ))
\end{eqnarray*}
plus error terms. Then, commutation of the pair shows that the two constants commute, up to an arbitrarily small error.
In $SU(3)$ this implies that they are both diagonal, so that we can obtain the renormalization representative of the form
\begin{equation*}
(\a _{l} , A e^{U(\. )})
\end{equation*}
In general groups, however, commutativity is gained after iteration.

\subparagraph*{Perturbations of constant cocycles} This is the last step in our study, at least from the dynamical point of view,
and it is described in chapter \ref{Local theory}. Here, we suppose that $\a _{l} \in DC$, and theorem \ref{NK local density}
allows us to perturb the renormalization representative $(\a _{l} , A e^{U(\. )})$ to $(\a _{l} , A e^{U'(\. )})$, which is
reducible.

Since the induction was finite and the perturbations at each step could be made arbitrarily small, the global density theorem
follows.

\selectlanguage{english}
\chapter{Dynamics of abelian cocycles} \label{Dynamics of abelian cocycles}

In this chapter we attempt a survey of the models of cocycles whose dynamics can be
characterized as simple. We try to point out what remains true when we embed these
"toy cocycles" in the larger class that we seek to study in the rest of this mémoire.

In order to simplify the presentation, we suppose that the regularity of the mappings is $C^{\infty }$,
but most of the proofs use very low regularity. The results
will be extended in weaker regularity in the following chapters, where the
precise statements will be given, but in the proofs we will use the notions,
techniques and observations that arise naturally in this phenomenological part of the study. We note nonetheless that
the regularity typically needed in the proofs is no more than $C^{2}$ for the cocycles and for conjugations.

Since, in particular, minimal translations on the torus are also (uniquely)
ergodic, some of the proofs work also in the measurable case, with the
necessary modifications of the proofs and of the statements of the results.

Finally, since not all proofs work when $d \geq 2$ (e.g. those using the Denjoy-Koksma inequality, see below),
and since the conclusions are significant precisely in the case
$d=1$, we suppose, as for the greatest part of the remaining of this thesis, that
the cocycle has only one frequency.

 \section{Cocycles in $SW (\T ,\Sp ) $}

The simplest case of a cocycle in a compact (but not semisimple) Lie group is that of a cocycle in
$\T \times \Sp = \T \times \Sp ^{1}$. Any given $(\a , A(\. )) \in SW (\T, \Sp ) $ can be written in the form
\begin{equation*}
(\a , E_{r } (\. )e^{2i \pi (\phi (\. )+ \theta )})
\end{equation*}
with $\phi : \T \ra \Sp $ such that $\int \phi = 0 $ and $\theta \in \R $, and where, for $r \in \Z $,
we have noted by $E_{r} (\. ) : \Sp \circlearrowleft$ the map $e ^{2i \pi \theta} \mapsto e ^{2i \pi r \theta} $.

Conjugation in $SW (\T, \Sp )$ is related with the solution of a linear cohomological equation
(see lemma \ref{Linear cohomological eq}), and conjugation between constant cocycles
is realized by conjugations of the type $E_{r} (\. )$.

Such cocycles admit a first classification by homotopy, for if $r' \neq r $, then the cocycles
$(\a , E_{r } (\. ))$ and $(\a , E_{r' } (\. ))$ are not conjugate, and the choice of non-trivial $\phi $
and $\theta$ clearly cannot change this fact. Even more, if $\a $ satisfies some diophantine condition, $r $
and $\theta + \a \Z $ completely determine the class of conjugacy of the cocycle.

However, since the source of interest in $SW (\T, \Sp ) $ comes from the canonical embedding $\Sp \hra SU(2)$,
classification by homotopy is not the good starting point. In fact, even in view of the embedding
$\Sp \mod \{ \pm 1 \} \hra SO(3)$, where there are homotopy issues, this proves to be insufficient,
since the second iterate of a cocycle non-homotopic to constants in $SW (\T ,SO(3)) $ becomes homotopic to constants,
whereas this is not true in $SW (\T ,\Sp ) $. The concept apt for generalization proves to be the fact that
non-trivial homotopy in $SW (\T ,\Sp ) $ is equivalent to linear growth of the derivatives of its iterates. This can be done by observing that
the rate of growth of derivatives equals the degree of the mapping in
absolute value.

In order to obtain it, let us consider a particular case for the differential operator $L $ introduced in
\ref{Calculus in Lie groups}. We identify $\R $ with $ T_{1} \Sp $ and trivialize $T \Sp $ by
$ T \Sp \ni (e^{2i \pi \theta} , s ) \mapsto  (e^{2i \pi i \theta} , e^{-2i \pi \theta } s ) \in \Sp \times \R $.
Application of the operator $L$ amounts to projecting the differential to the second coordinate in the trivial
tangent bundle. A simple calculation then shows that $\frac{1}{n} L A_{n}(\. ) \ra 2\pi r $, where $A(\. )$ is as
in the beginning of the paragraph.

If we consider $SW (\T , \Sp ) \hra SW (\T , SU(2) ) $
as the subgroup of diagonal cocycles, this cocycle assumes the form
$(\a , E_{r}(\. ))$, where the notation is now as in paragraph \ref{Notation and algebra in su2}.
The relics of homotopy is this type of growth of the derivative, which is in the direction of the vector
$\{1,0 \} _{su(2)}$. Factorization of $SU(2)$
by $\pm Id $ and consideration of $SW (\T , \Sp ) \hra SW (\T , SO(3) ) $ shows that in this case
the allowed rates of growth of derivatives are in $\frac{1}{2} 2\pi \Z $ where the factor $\frac{1}{2}$ is precisely
the cardinal of $\{ \pm Id \} $.

In a first attempt to exit from the abelian world, we can conjugate the cocycle
$(\a , A(\. )) \in SW (\T , \Sp ) \hra SW (\T , SU(2) ) $ by an arbitrary mapping
$B(\. ) : \T \ra SU(2) $, and call the cocycle thus obtained $(\a , \tilde{A}(\. ))$.
Since $(\a , \tilde{A}_{n}(\. )) = (\a , B(\. +n\a )A_{n}(\. )B^{*}(\. ))$ and using the properties of
the operator $L$, we discover that
\begin{eqnarray*}
\tilde{a}_{n}^{\ast }(\. ) = L (\tilde{A}_{n}^{\ast }(\. ))
&=&b(\. )+Ad(B(\. )).a_{n}^{\ast }(\. )+Ad(B(\. )A_{n}^{\ast }(\. )).b^{\ast }(\. +n\a ) \\
&=&-n.Ad(B(\. )).2\pi r \{ 1,0 \}_{su(2)} + o(n)
\end{eqnarray*}
(we remind the notational convention following which $a_{n}(\. )=LA_{n}(\ .)$ and $b(\. )= LB(\. )$).

The study in the following section is just a higher-dimensional analogue of what we just saw.

\section{Dynamics of torus-reducible cocycles} \label{dynamics of
torus-reducible cocycles}

\subsection{A first definition of the degree}

Let $(\a ,A(\. ))$ be a $C^{\infty }$-torus-reducible cocycle in
$SW^{\infty }(\T ,G)$, $\TT $ a maximal torus and $(h_{\r })_{\r \in \tilde{\Delta }}$
a basis of $\ft$. Therefore, $A(\. )$ is of the type
 \begin{equation}
A(x)=B(x+\a ).\exp (\sum\nolimits_{\r \in \tilde{\D }}2\pi (r_{\r }x+\phi _{\r 
}(x)+a_{\r })h_{\r }).B^{-1}(x)  \label{coc deg 0}
\end{equation}
with $\phi _{\r }(\. ):\T \ra \R$ of mean value $0$, $a_{\r }\in \T $, $r\in \Z^{w}$ and $B(\. ):\T \ra G$.

After reminding the notation $a_{n}(\. ) = L A_{n}(\. )$, where the operator $L$ is defined in section
\ref{Calculus in Lie groups}, we can readily prove
\begin{proposition} \label{quantif of length for torus-reduc}
Let $(\a , A(\. )) \in SW^{1}$ be $C^{1}$ torus-reducible. Then the length of
the $n$-th iterate normalized by $n$ is quantized:
\begin{equation*}
\frac{1}{n} | a_{n}(\. ) | \ra 
| \sum\nolimits_{\r }2\pi r_{\r }h_{\r } |
\end{equation*}
uniformly and $r_{\r} \in \Z, \forall \r \in \tilde{\D }$. The limit does not depent on the conjugation
$B(\. )$ and is therefore an invariant under $C^{1}$ conjugation.
\end{proposition}

The definition of this quantity will be extended to all $C^{1}$ cocycles and called "energy" in the next chapter.
The important point will be that the quantization does not change when we consider non-abelian cocycles.

\begin{proof}
Taking the derivative
of the formula for the $n$-th iterate gives
\begin{equation*}
a_{n}(x)=b(x+n\a )+Ad(B(x+n\a )).\sum\nolimits_{\r }2\pi r_{\r 
}(n+S_{n}^{\a }\phi _{\r }^{\prime }(x))h_{\r }-Ad(A_{n}(x)).b(x)
\end{equation*}
(we remind the notational convention following which $a_{n}(\. )=LA_{n}(\. )$ and $b(\. )= LB(\. )$).

If $r=0$, we find that
\begin{equation*}
a_{n}(x)=b(x+n\a )+Ad(B(x+n\a )).\sum\nolimits_{\r }S_{n}^{\a 
}\phi _{\r }^{\prime }(x)h_{\r }-Ad(A_{n}(x)).b(x)
\end{equation*}%
so that $\frac{1}{n}a_{n}(\. )$ tends uniformly to $0$ by unique ergodicity of $R_{\a }$ with respect to
the Haar measure on $\T$.

If, now, $r\in \Z^{d}\setminus \{0\}$, we find that%
\begin{equation}
a_{n}(x)=Ad(B(x+n\a )).\sum\nolimits_{\r }2\pi r_{\r 
}(n+S_{n}^{\a }\phi _{\r }^{\prime }(x))h_{\r }+c_{n}(x)
\label{eq der}
\end{equation}%
where$\left\Vert c_{n}(\. )\right\Vert _{L^{\infty }}\leq 2
\left\Vert b(\. )\right\Vert _{L^{\infty }}$, so
that
\begin{equation*}
\| a_{n}(x)-Ad(B(x+n\a )).\sum\nolimits_{\r } 2\pi r_{\r 
}nh_{\r }\| _{L^{\infty }}=o(n)
\end{equation*}
which implies the statement since the adjoint action is an isometry of $g$.
\end{proof}

A special class of (torus-)reducible cocycles in the case $r=0$ is the one consisting of those reducible to a
constant in $Z_{G}$. If we let $(\a ,A(\. ))=(\a ,B(\. +\a 
).C.B^{\ast }(\. ))$, with $C \in Z_{G}$ be such a cocycle, we have%
\begin{equation*}
a_{n}(\. )=b(\. +n\a )-Ad(B(\. +n\a ).B^{\ast }(\. 
)).b(\. )
\end{equation*}%
and therefore%
\begin{equation}
|\partial ^{s}a_{q_{n}}(\. )|=O(\beta _{n})  \label{denjoy}
\end{equation}%
where the constant depends on $s$ and $B(\. )$.

Calculation of higher order derivatives (still in the case $r=0$)
and a similar calculation show that higher derivatives share the same property:
\begin{eqnarray}
\frac{1}{n}\partial ^{s}a_{n}(x) &\ra &0\text{, and} \\
|\partial ^{s}a_{q_{n}}(\. )| &\leq &C_{s},~\forall ~s\geq 0
\end{eqnarray}

In the special case where $n = q_{k}$, a denominator of a best rational approximation,
the Denjoy-Koksma inequality (see \cite{Herm79}), gives estimates on the speed of convergence.
According to this inequality, for any function $\varphi (\. ) : \T \ra \R $ of bounded variation,
\begin{equation*}
|S^{\a }_{q_{k}} \varphi (\. ) - q_{k} \int _{\T } \varphi (\. ) | \leq
Var (\varphi)
\end{equation*}
where $Var (\varphi)$ stands for the total variation of $ \varphi (\. )$. Application of
this inequality in our case and use of the fact that  $\int _{\T } \phi _{\r }^{\prime } = 0$ gives
\begin{equation*}
|a_{q_{n}}(\. )|\leq C \max \| \phi _{\r } \|_{1}
\end{equation*}

\bigskip
Similar estimations in the case where $r \not= 0 $ give only
\begin{equation*}
\frac{1}{n^{2}}\partial ^{s}a_{n}(\. )\ra 0
\end{equation*}%
uniformly.

\bigskip

Examination of the proof of the previous proposition shows that $C^{1}$ torus-reducibility
implies in fact more than the existence of a discrete invariant when $r \not= 0$:
for large $n$,
\begin{equation} \label{definition bold derivatives}
\ba _{n}(\. )=\frac{1}{n}a_{n}(\. )
\end{equation}
is uniformly bounded away from $0$. It follows from eq. (\ref{eq der}) that if $n_{k}\a \ra 0$ in $\T $, then
\begin{equation*}
\ba _{n_{k}}(\. )\ra \ba (\. ) =  Ad(B(\. )).\sum\nolimits_{\r }2\pi r_{\r }h_{\r }
\end{equation*}
uniformly.
Moreover, whenever $n_{k}\a \ra \gamma $ in $\T $, $\ba %
_{n_{k}}(\. )\ra Ad(B(\. +\gamma )).\sum\nolimits_{\r }2\pi
r_{\r }h_{\r }$, that is to $\ba (\. + \gamma )$.

Therefore, $B(\. )$ satisfies%
\begin{equation}
Ad(B^{*}(\. )). \ba (\. )=\sum\nolimits_{\r }2\pi r_{\r }h_{\r 
}  \label{deriv a}
\end{equation}
but this equation in general does not determine $B(\. )$, nor the
dynamics of the cocycle, as we will see later on.

As for $\ba (\. )$, we can prove, in fact, something more:

\begin{lemma}
\label{inv curve}Let $(\a , A(\. )) \in SW^{1}$ be $C^{1}$ torus-reducible, and let $r \not= 0$.
Then, $\frac{1}{q_{n}}a_{q_{n}}(\. )$ converges uniformly (in $C^{0}$)
to a $C^{1}$ continuous curve $\ba (\. ) : \T \ra g$.

The curve $\ba (\. )$ satisfies the invariance
relation
\begin{equation}
Ad(A(\. )).\ba (\. )=\ba (\. +\a ) \label{curve}
\end{equation}
\end{lemma}

\begin{proof}
We need only to prove the invariance relation.
Let $n\in \N$. The cocycles $(\a ,A(\. ))$ and $(\a ,A(\. ))^{q_{n}}$
commute, since they are iterates of the same diffeomorphism. We have,
therefore, $A(\. +q_{n}\a ).A_{q_{n}}(\. )=A_{q_{n}}(\. +\a 
).A(\. )$. Taking the derivative of this formula gives%
\begin{equation*}
Ad(A(x+q_{n}\a )).a_{q_{n}}(x)-a_{q_{n}}(x+\a 
)=Ad(A_{q_{n}}(x+\a )).a(x)-a(x+q_{n}\a )
\end{equation*}%
If we divide both sides by $q_{n}$ and let $n$ go to infinity, we obtain the
announced formula by uniform convergence.
\end{proof}

We can immediately prove

\begin{corollary}
The modulus of $\ba (\. )$ is constant. Moreover, the image of
$\ba (\theta )$ under $Inn(g)$ is a constant subset of $g$, not depending on $\theta$.
\end{corollary}

\begin{proof}
Using the invariance relation (\ref{curve}) we find%
\begin{equation*}
\left\vert \ba (\. +\a )\right\vert =\left\vert Ad(A(\. )).%
\ba (\. )\right\vert =\left\vert \ba (\. )\right\vert
\end{equation*}%
Since $R_{\a }$ is minimal on the torus and $\left\vert \ba (\. 
)\right\vert $ is invariant under $R_{\a }$, it is constant. The
relation (\ref{curve}) itself shows that, for a fixed $h\in g$, the set%
\begin{equation*}
\{\theta \in \T ,~\exists ~B\in G~s.t.~Ad(B).\ba (\theta )=h\}
\end{equation*}%
is invariant under $R_{\a }$. If for $h,h'$ the corresponding sets intersect,
the two vectors can be conjugated in $Inn(g)$, so that minimality of $R_{\a }$
allows us to conclude.
\end{proof}

We remark that the fist statement of the corollary is exactly the definition of the
energy of the cocycle given in \ref{quantif of length for torus-reduc}, but its proof in
this context gives an additional reason (the invariance of the curve under the dynamics)
which is the property that we will use when we examine the inverse problem of determining
the dynamics under the assumption of the existence of $\ba (\. )$.

Finally, let us suppose that $(\a , A_{i}(\. )) \in SW^{1}$,
$i=1,2$, are conjugate by $\tilde{B}(\. ) \in C^{1}$ and that $(\a , A_{1}(\. ))$ satisfying the hypotheses of
the previous lemma. Then, $(\a , A_{2}(\. ))$ also satisfies these hypotheses, and
the curves $\ba _{i}(\. )$ associated to $(\a , A_{i}(\. ))$ as in the lemma satisfy
\begin{equation*}
\ba _{1}(\. ) = Ad (\tilde{B}(\. )) \ba _{2}(\. )
\end{equation*}
as shows a direct calculation. Therefore, the class of $\ba (\. )$ in $g$ modulo $Inn(g)$
is in fact a second quantized invariant under $C^{1}$ conjugation. This invariant takes values in the
lattice of preimages of the $Id$ in $g$, modulo the action of $Inn(g)$. The definition of this invariant
will be extended to what we will call the \textit{degree} in the next chapter.

The conclusion of chapter \ref{Chapter renormalization} is, in fact,
that the two quantized invariants defined in this section
keep the same quantization in their values when aptly extended to all $C^{1}$ cocycles. Chapter 6
already implies that, unlike for cocycles in $\T \times \T ^{w}$, the conjugacy classes of
abelian models for cocycles in $\T \times G$ should not be expected to form a space
of discrete classes in $SW^{1}$.

\subsection{The failure of the converse derivation \label{introduction of splitting}}

We have seen that $C^{1}$ torus reducibility implies the existence of some nice
objects, which are invariant under conjugation of the dynamics.

Let us now suppose that a cocycle $(\a ,A(\. ))$ is given such that
the normalized sequence $\ba _{q_{n}}(\. ) =\frac{1}{q_{n}}a_{q_{n}}(\. )$ tends
uniformly to a $C^{\s }$ curve $\ba (\. )$, with $\s \geq 1$.
Inspection of the proof of lemma \ref{inv curve} shows that in fact the curve $\ba (\. )$
satisfies automatically the conclusions of the lemma, and in particular $|\ba _{q_{n}}(\. )|$
tends to a positive constant $C$. Let us suppose that $C>0$, in which case we say that
$\ba (\. )$ is \textit{non-degenerate}.

Lemma \ref{inv curve} and its corollary allow us to transform algebraic conjugations
acting on the invariant curve into dynamic conjugations acting on the
cocycle in the following way.

\begin{lemma}
\label{invariant curve}If there exists a (non-degenerate) curve $\ba %
(\. )$ in $g$ satisfying eq. \ref{curve}, there
exists $h\in g \setminus \{ 0 \}$ a vector in the Lie algebra and a $c_{G}$-periodic
mapping $B(\. )$ such that $(\a ,\tilde{A}(\. ))=Conj_{B(\. 
)}(\a ,A(\. ))$ satisfies
\begin{equation*}
Ad(\tilde{A}(\. )).h=h
\end{equation*}
\end{lemma}

\begin{proof}
Let us consider $I_{\ba }=\{\ba (\theta ),\theta \in \T %
\} $ and $h\in g$ such that $Ad(B)\ba (\theta )=h$ for some $\theta
\in \T $ and $B\in G$. The set of vectors of $g$ that can be
conjugated to $h$ is given by $Ad(\check{G}).h$, where $\check{G}=G/Z_{G}$
and the Lie algebra of $\check{G}$ is naturally isomorphic to $g$. Moreover,
if $h^{\prime }=Ad(B).h$ for some $B\in \check{G}$, then for any $B^{\prime
}\in \mathcal{Z}_{h}=\{S\in G,Ad(S).h=h\}$,%
\begin{equation*}
h^{\prime }=Ad(B.B^{\prime }).h
\end{equation*}%
Finally, if $B_{1}$ and $B_{2}$ are such that $h^{\prime }=Ad(B_{i}).h$, $i=1,2$%
, then $B=B_{1}B_{2}^{\ast }$ satisfies%
\begin{equation*}
Ad(B).h=h
\end{equation*}%
so that $B\in \mathcal{Z}_{h}$. Therefore, $Ad(\check{G}).h$ is naturally
isomorphic to $G/ \mathcal{Z}_{h}$.

The set $\mathcal{Z}_{h} = Z_{G}(\exp (\R h))$ is a
subgroup of $\check{G}$. Let us call $\tilde{g}$ the Lie algebra of $%
\mathcal{Z}_{h}$. We remark that the manifold (in fact symmetric space) $G/ \mathcal{Z}%
_{h}$ has, at each point, a tangent space canonically isomorphic to $g/%
\tilde{g}$ (the quotient is taken in the category of vector spaces): this is
true at the image of $Id$ under the projection $\pi :\check{G}\ra 
G/ \mathcal{Z}_{h}$, and any other coset is obtained by left
multiplication in $\check{G}$, as $S.\mathcal{Z}_{h}=S.(Id.\mathcal{Z}_{h})$%
. The projection is locally a submersion in the neighborhood of the $Id$, as
for all $s\in g$,%
\begin{equation*}
d\pi _{Id}(s)= s- \pi _{\tilde{g}}s
\end{equation*}%
where $\pi _{\tilde{g}}$ is the orthogonal projection on $\tilde{g}$ with respect to the Cartan-Killing form.
Therefore, the kernel of $d\pi _{Id}$ is $\tilde{g}$. It follows that $\pi $
is globally a sumbersion and the kernel of the differential at $S\in \check{G%
}$ is $(R_{S})_{\ast }\tilde{g}$, where $R_{S}:H\mapsto SH$ is the right
multiplication in $\check{G}$.

Consequenlty, since the curve $\ba (\. ):\T \ra g$
takes values in $Ad(\check{G}).h \approx G/ \mathcal{Z}_{h}$, it
admits a lift $\check{B}:\R \ra \check{G}$ of regularity $%
C^{\s }$%
\begin{equation*}
\ba (\. )=Ad(\check{B}(\. )).h
\end{equation*}%
Since $\ba (\. )$ is periodic,%
\begin{equation*}
Ad(\check{B}^{\ast }(\. +1)\check{B}(\. )).h=h
\end{equation*}%
and therefore $D(\. )=\check{B}^{\ast }(\. +1)\check{B}(\. )$ takes
values in $\mathcal{Z}_{h}$.

Anticipating lemma \ref{normalization lemma} (for the proof of which the semisimplicity of $G$ is irrelevant),
we write $D(\. )=\check{D}(\. +1)\check{D}^{\ast }(\. )$, with
$\check{D}(\. )\in C^{\s }(\R ,\mathcal{Z}_{h})$. The mapping
\begin{equation*}
\tilde{B}(\. )=\check{B}(\. )\check{D}(\. )
\end{equation*}%
is in $C^{\s }(\T ,\check{G})$, since%
\begin{eqnarray*}
\tilde{B}(\. +1) &=&\check{B}(\. +1)\check{D}(\. +1) \\
&=&\check{B}(\. +1)D(\. )\check{D}(\. ) \\
&=&\check{B}(\. )\check{D}(\. )
\end{eqnarray*}%
Since $Ad(\check{D}(\. )).h=h$, we have%
\begin{equation*}
\ba (\. )=Ad(\tilde{B}(\. )).h
\end{equation*}%
and, $\tilde{B}(\. )$ admits a $c_{G}$-periodic lift $B(\. )\in C^{\s }(c_{G} \T,G)$ satisfying the same
equation owing to the fact that
\begin{equation*}
B^{\ast }(\. +1)B(\. )\in Z_{G}
\end{equation*}
Using, now the invariance relation satisfied by $\ba (\. )$, we see that
\begin{eqnarray*}
Ad(A(\. ))\ba (\. ) &=&\ba (\. +\a ) \\
Ad(A(\. )B(\. )).h &=&Ad(B(\. +\a )).h \\
Ad(B^{\ast }(\. +\a )A(\. )B(\. )).h &=&h
\end{eqnarray*}%
Therefore, the lemma is proved.
\end{proof}

The description of the dynamics of the cocycle $(\a ,\tilde{A}(\. ))$ is given by the following proposition.

\begin{proposition} \label{construction splitting}
Let $(\a ,A(\. ))$ be a cocycle in $SW(\T ,G)$ satisfying the hypothesis of lemma \ref%
{invariant curve}. Then, there exists an abelian subgroup $G_{+} = \TT _{h}$, and a semisimple subgroup $G_{0}$
($G_{0}$ may be trivial) commuting with $G_{+}$ and such that $G_{0} \cap G_{+} = Id$ satisfying the
following properties. The cocycle $(\a ,A(\. ))$ can be conjugated $\mod c_{G}$ to a cocycle of the form
\begin{equation*}
(\a ,\tilde{A}(\. )) = (\a ,\tilde{A}_{0}(\. )\times \tilde{A}_{+}(\. ))  \label{splitting}
\end{equation*}%
The cocycle $(\a ,\tilde{A}_{+}(\. )) \in SW(\T ,G_{+})$ is abelian and
\begin{equation*}
\frac{1}{n} \tilde{a}_{+,n}(\. ) \ra h
\end{equation*}
and the cocycle $(\a ,\tilde{A}_{0}(\. )) \in SW(\T ,G_{0})$ is such that
\begin{equation*}
\frac{1}{n} \tilde{a}_{0,n}(\. ) \ra 0
\end{equation*}
The conjugation can be chosen of the same regularity as the invariant curve.
\end{proposition}

\begin{corollary}
If $G=SU(2)$, any cocycle satisfying the hypothesis of lemma \ref%
{invariant curve} can be conjugated to an abelian one by a $2$-periodic transfer
function.
\end{corollary}

\begin{proof} [Proof of the Corollary]
Since every torus in $SU(2)$ is a maximal torus, and since the hypothesis of lemma implies that
$G_{+}$ is non-trivial, we obtain that $G_{0} = \{ Id \}$. Therefore, in the notation of the theorem,
$(\a ,\tilde{A}(\. ))$ is abelian.
\end{proof}

This proposition motivates the following definition.
\begin{definition} \label{def splitting}
Let $(\a ,A(\. )) \in SW(\T ,G)$ and suppose that there exist
$G_{0}$ and $G_{+}$ as in theorem \ref{construction splitting} and $B(\. )$ such that
$(\a ,\tilde{A}(\. )) = Conj_{B(\. )} (\a ,A(\. ))$ takes values in 
$G_{0}\times G_{+} \hra G$.
We will then say that $(\a ,A(\. ))$ admits a \textit{splitting} of the dynamics
\end{definition}

\begin{proof} [Proof of Proposition \ref{construction splitting}]
Let $h$ and $B(\. )$ as in the proof of the previous lemma.

If $h$ is a regular element of the algebra, the condition $Ad(\tilde{A}(\. )).h=h$
is equivalent to $\tilde{A}(\. )$ taking values on the unique maximal torus
passing by $h$. Therefore, in this case $G_{+}$ is a maximal torus and $G_{0}=\{ Id\} $.

Let us now suppose that $h$ is singular. Since $Ad(\tilde{A}(\. ))$ fixes $h$,
$\tilde{A}(\. )$ takes values in $\mathcal{Z}_{h}=Z_{G}(\exp (\R h))$.
As $g$ is the Lie algebra of a compact group, it follows from
Theorem \ref{Dieud Theorem Comp Gr} that $\mathcal{Z}_{h}$ is a Lie group
whose Lie algebra $\tilde{g} = \ker (ad_{h})$ splits in the direct sum of $g_{+}$, the intersection of all
toral algebras containing $h$, and $g_{0} = [\tilde{g} , \tilde{g}]$, the Lie algebra of
$\mathcal{Z}_{h}/\exp (g_{+})$. By construction, $g_{+}$ and $g_{0}$ commute.

A concrete description of the decomposition can be obtained via a root space decomposition
of $g$ associated to a torus containing $g_{+}$. Then, $g_{0}$ is generated by the vectors $h_{\r }$
and $j_{\r }$ corresponding to roots $\r $ such that $\r (h) =0$, and therefore
by constuction $g_{0}$ is semi-simple. Then, $g_{+}$ is $g_{0}^{\bot} \cap \tilde{g}$.

Finally, $G_{0}\cap G_{+}=Id$ where $G_{0}=\exp (g_{0})$ and $G_{+}=\exp (g_{+})$.
If $w_{0}$ and $w_{+}$ are the respective ranks of
$G_{0}$ and $G_{+}$, we have $w_{0}+w_{+}=w$ and $G_{+}$ is a torus of dimension $0<w_{+}\leq w$
and equality holds iff $h$ is regular.

Since $(\a ,\tilde{A}(\. ))\in SW^{s}(\T ,\tilde{G})$, and $\tilde{g}=g_{0}\oplus g_{+}$,
we can easily verify that the normalized derivatives of $(\a ,\tilde{A}_{+}(\. ))$ converge to $h$, and those of 
$(\a ,\tilde{A}_{0}(\. ))$ converge to $0$:
\begin{eqnarray*}
L \tilde{A} _{q_{n}} (\.) &=& b^{*}(\. +q_{n} \a) + Ad(B^{*} (\. + q_{n} \a)).a _{q_{n}}(\.)
+Ad(B^{*} (\. +q_{n} \a)A_{q_{n}}(\.) ) b(\.) \\
&=& q_{n} Ad(B^{*} (\. )) \ba (\.) + o (q_{n}) \\
&=& q_{n} h + o (q_{n})
\end{eqnarray*}
Since $L \tilde{A} (\.) = a_{+} (\.) + a_{0} (\.)$, and the two components are iterated separately thanks to
the commutativity of $G_{+}$ and $G_{0}$, we obtain that
\begin{eqnarray*}
\frac{1}{q_{n}}(a_{+})_{q_{n}} (\.) \ra h \\
\frac{1}{q_{n}}(a_{0})_{q_{n}} (\.) \ra 0
\end{eqnarray*}
\end{proof}

Heuristically, the existence of an invariant curve obtained as the limit of normalized derivatives
allows us to abelianize only a part of the dynamics. The fact that the part of the
dynamics controlled by the invariant curve is exactly the part where there is linear growth of the
derivatives allows us to conclude that in the remaining part derivatives grow sublinearly. Sublinear growth
of the derivatives is true for almost reducible cocycles, and, as we
will see in chapter \ref{Chapter renormalization}, the converse
is true, in the sense that sublinear growth characterizes dynamics of
close-to-constant cocycles. On the other hand,
it is already clear that linear growth of the derivatives is the
characteristic of non-local dynamics.

A concrete example of a curve $\ba (\.)$ for which we cannot conclude torus-reducibility can be seen
in $SU(4)$. If the points of the curve belong to $Inn(g).h$, where
\begin{equation*}
h =
\begin{bmatrix}
\lambda i & 0 & 0 & 0 \\ 
0 & -\lambda i & 0 & 0 \\ 
0 & 0 & 0 & 0 \\ 
0 & 0 & 0 & 0
\end{bmatrix}
\end{equation*}
then commutation with $h$, which is the only exploitable information, does not provide anything on the
dynamics within the subgroup
\begin{equation*}
\begin{bmatrix}
Id_{2} & 0 \\ 
0 & SU(2)
\end{bmatrix}
\end{equation*}
which is the $G_{0}$ factor of the corresponding decomposition. The factor $G_{+}$ is given by the
subgroup image of
\begin{equation*}
\begin{bmatrix}
(\lambda _{1}+ 2\lambda _{2}) i & 0 & 0 & 0 \\ 
0 & -\lambda _{1} i & 0 & 0 \\ 
0 & 0 & -\lambda _{2} i & 0 \\ 
0 & 0 & 0 & -\lambda _{2} i
\end{bmatrix}
\end{equation*}
under the exponential mapping, where $\lambda _{1}, \lambda _{2} \in \R$. We remark that the only diagonal vector
in $g_{0}$ of the Weyl basis of $su(4)$ is
\begin{equation*}
\begin{bmatrix}
0 & 0 & 0 & 0 \\ 
0 & 0 & 0 & 0 \\ 
0 & 0 & i & 0 \\ 
0 & 0 & 0 & -i
\end{bmatrix}
\end{equation*}
while the vectors
\begin{equation*}
\begin{bmatrix}
0 & 0 & 0 & 0 \\ 
0 & i & 0 & 0 \\ 
0 & 0 & -i & 0 \\ 
0 & 0 & 0 & 0
\end{bmatrix}
\text{ and }
\begin{bmatrix}
i & 0 & 0 & 0 \\ 
0 & 0 & 0 & 0 \\ 
0 & 0 & 0 & 0 \\ 
0 & 0 & 0 & -i
\end{bmatrix}
\end{equation*}
are neither in $g_{0}$ nor in $g_{+}$. The only vector of the Weyl's basis which is in $g_{+}$ is the remaining one,
\begin{equation*}
\begin{bmatrix}
i & 0 & 0 & 0 \\ 
0 & -i & 0 & 0 \\ 
0 & 0 & 0 & 0 \\ 
0 & 0 & 0 & 0
\end{bmatrix}
\end{equation*}
which accounts for the fact that the corresponding complex direction
\begin{equation*}
\begin{bmatrix}
0 & 1 & 0 & 0 \\ 
-1 & 0 & 0 & 0 \\ 
0 & 0 & 0 & 0 \\ 
0 & 0 & 0 & 0
\end{bmatrix}
\end{equation*}
is the only one in the root-space decomposition commuting with $\g0$.

Concrete examples exist, in fact, in any group of dimension higher than $SU(2)$, namely in $SU(3)$,
where vectors of the type
\begin{equation*}
\begin{bmatrix}
2\lambda i & 0 & 0 \\ 
0 & -\lambda i & 0 \\ 
0 & 0 & -\lambda i
\end{bmatrix}
\end{equation*}
commute with the subgroup
\begin{equation*}
\begin{bmatrix}
1 & 0 \\ 
0 & SU(2)
\end{bmatrix}
\end{equation*}
and the corresponding abelian algebra $g_{+}$ is generated by the vector
\begin{equation*}
\begin{bmatrix}
2\lambda i & 0 & 0 \\ 
0 & -\lambda i & 0 \\ 
0 & 0 & -\lambda i
\end{bmatrix}
\end{equation*}
itself.

\bigskip

If we anticipate proposition \ref{conjugation of abelian cocycles}, we find directly that
$(\a ,A(\. ))$ is torus-reducible if, and only if,
$(\a ,\tilde{A}_{0}(\. ))$ is torus-reducible in $G_{0}$.

\bigskip

Since $\tilde{A}_{+}(\. )$ takes its values in a torus and it is
periodic, it is of the form%
\begin{equation*}
\tilde{A}_{+}(\. )=\sum\nolimits_{\r }2\pi h_{\r }(r_{\r }\. 
+\phi _{\r }(\. ))
\end{equation*}%
where the $\phi _{\r }(\. )$ are periodic. By resolving a linear cohomological equation in
$G_{+}$, we find that if $\a $ is Diophantine, $(\a ,\tilde{A}(\. ))$
can be conjugated to a cocycle of the form
\begin{equation*}
\left( \a ,\tilde{A}_{0}(\. )\times (E_{r}(\. ).A) \right)
\end{equation*}
where $E_{r}(\. )=\exp (\sum\nolimits_{\r }2\pi r_{\r }h_{\r 
}\. )$ and $A\in G_{+}$. In the particular case where $\tilde{A}_{0}(\. ) = A_{0} \in G_{0} $,
such a splitting will be called a \textit{normal form}.

The construction of the splitting of the dynamics shows that in general we
cannot expect to avoid phenomena of coexistence of close-to constants
cocycles with far-from-constants cocycles. The two cases where this does not
occur is either when iteration and normalization converge to constants, or
when the vector obtained by the construction is regular, and it should be
expected that this dichotomy is generic under some reasonable Diophantine
conditions. Since every non-zero vector in $su(2)$ is regular, the result in 
$SU(2)$ assumes a more elegant form, namely that the existence of an invariant curve implies
the existence of a $2$-periodic conjugation to an abelian cocycle.
This form, however, cannot be obtained in full generality for any given
group $G$.

\section{Conjugation of abelian cocycles} \label{Conjugation of abelian cocycles}

The action of conjugacies on the class of constant cocycles has been
exhaustively described in \cite{KrikAst} (Prop. \textit{2.5.9}):

\begin{proposition} \label{RK conj constant cocycles}
Let $(\a ,A_{i})$, $i=1,2$ be two constant cocycles, conjugate by $%
B(\. )\in C^{\infty }(m\T ^{d},G)$, where $B(0)=Id$ and $m\in \N^{\ast }$. Then, for any character $\chi _{\r }$, $\chi _{\r 
}(A_{1}).\chi _{\r }(A_{2}^{-1})\in \exp (2i\pi /m \Z)$, and $B(\. )$ is a group homomorphism $m \T ^{d}\ra G$.
\end{proposition}

Since conjugation between abelian cocycles is less restrictive, the result is less precise. On the other hand
it is somewhat more elegant, since it is related to the action of Weyl's group, which is a finite group.
\begin{proposition}
\label{conjugation of abelian cocycles}Let $(\a ,\tilde{A}(\. ))$ and 
$(\a ,A(\. ))$ in $SW(\T ,\TT )\hra SW(\T ,G)$ such that
\begin{equation*}
(\a ,\tilde{A}(\. ))=Conj_{B(\. )}(\a ,A(\. ))
\end{equation*}%
and consider the corresponding splittings $\tilde{G}_{0}\times \tilde{G}_{+}\hra G$ and $G_{0}\times G_{+}\hra G$. Then%
\begin{eqnarray*}
(\a ,\tilde{A}_{+}(\. )) &=&Conj_{B(\. )}(\a ,A_{+}(\. )) \\
(\a ,\tilde{A}_{0}(\. )) &=&Conj_{B(\. )}(\a ,A_{0}(\. ))
\end{eqnarray*}%
The action of $B(\. )$ on $G_{+}\ra \tilde{G}_{+}$ projects naturally
to the action of a representative of $W(G_{+}) \subset W(G)$.
\end{proposition}

The following corollary is immediate.
\begin{corollary}
The integer
vector $r$ which distinguishes qualitatively different models of dynamics
of abelian cocycles is determined up to the action of $W(G)$.
\end{corollary}

Before passing on to the proof of the proposition, we remark that the hypothesis that
$\tilde{A}(\. )$ and $A(\. )$ take values on the same torus is not restrictive, since
any two tori in $G$ are obtained by a (constant algebraic) conjugation acting on a standard one.

\begin{proof} [Proof of Proposition \ref{conjugation of abelian cocycles}]
Let us suppose that
\begin{equation*}
(\a ,\tilde{A}(\. ))=Conj_{B(\. )}(\a ,A(\. ))
\end{equation*}%
holds, for a certain $B(\. ) \in C^{0} (\T,G)$.

The mappings $(\a ,A(\. ))$ and $(\a ,\tilde{A}(\. ))$ take values in
the same maximal torus $\TT $, whose toral algebra is $\mathfrak{t}$.
In this case, $B(\. )$ is determined up to left and right multiplication
with a constant in $\TT $, as shows the formula%
\begin{equation*}
\tilde{A}(\. )=B(\. +\a ).A(\. ).B^{-1}(\. )
\end{equation*}%
Let, now, $h,h^{\prime }\in \mathfrak{t}$ be any two vectors. Then, the same
formula, together with the fact that the restriction of the adjoint action
of both $A(\. )$ and $\tilde{A}(\. )$ in $\mathfrak{t}$ is the
identity, gives directly%
\begin{equation*}
\langle Ad(B(\. +\a )).h,h^{\prime }\rangle =\langle Ad(B(\. 
)).h,h^{\prime }\rangle
\end{equation*}
In particular, if we let these two vectors run through a basis of $\mathfrak{t}$,
we see by the minimality of $R_{\a }$ that the restriction of
$Ad(B(\. ))$ in $\mathfrak{t}$ is in fact a constant.

If we denote by $B_{\r ,h}(\. )j_{\r }$ the projection of the image
of $Ad(B(\. )).h$ on the plane $\C j_{\r }$, and also use the notation
$Ad(\tilde{A}(\. )).j_{\r }=\exp (\tilde{a}_{\r }(\. )).j_{\r }$, we find that%
\begin{equation*}
B_{\r ,h}(\. +\a )=\exp (- \tilde{a}_{\r }(\. ))B_{\r ,h}(\. )
\end{equation*}
In particular, for all $\r $ and $h$, we have $|B_{\r ,h}(\. )|=C$, a constant.
Therefore, if $C \not= 0$, $B_{\r ,h}(\. ) : \T \ra C. \Sp ^{1}$ admits a lift $C\exp (b_{\r ,h}(\. ))$
which satisfies
\begin{equation*}
b_{\r ,h}(\. +\a )-b_{\r ,h}(\. )=\tilde{a}_{\r }(\. )~%
\mod  2i\pi \Z
\end{equation*}
We now disinguish two cases. If $\tilde{a}_{\r }(\. ): \R \ra i \R$
is of the form $2i\pi (\tilde{r}_{\r }\. +\varphi _{\r }(\. ))$ with
$r_{\r }\in \Z^{\ast }$ (i.e. if $h_{\r } \notin g_{0} $) and $\varphi _{\r }(\. )$ periodic, then $b_{\r ,h}(\. )$
cannot be a lift of a periodic function, unless it is constant and therefore
$B_{\r ,h}(\. )$ must be equal to $0$. On the other hand, if $\tilde{r}_{\r }=0$, the equation may or not
admit smooth or even continuous solutions depending on the arithmetic properties of $\a $,
but in the particular case where $\hat{\varphi}_{\r }(0)\in \a \Z^{\ast }$
(such constants are called resonant) it may also admit solutions of the form%
\begin{equation*}
b_{\r ,h}(\. )=2i\pi r\. 
\end{equation*}
which correspond to mappings $c_{G} \T \ra G$
If we admit that the conjugant be $m$-periodic, the resonant constants can
belong to $\a \frac{1}{m} \Z^{\ast }$, and then the equation admits non-local solutions of the form
\begin{equation*}
b_{\r ,h}(\. )=2i\pi \frac{r}{m}\. 
\end{equation*}%
a phenomenon already observed in \cite{KrikAst}. We remark that there is no
loss of periodicity apart from the phenomena observed in conjugation between
constant cocycles.

Therefore, $\langle Ad(B(\.)) h_{\r^{\prime}} , j_{\r}\rangle$ can be non-zero only if
$ Ad(\tilde{A} (\.)) j_{\r} = C \exp (2i\pi (\varphi _{\r ^{\prime} }(\. )))$ with
$C >0$. Since the roles of $\tilde{A} (\.)$ and $A(\.)$ are symmetric,
the same holds for $A(\.)$.
\bigskip

If we now turn to the study of $B_{\r ,\r '} (\.) = \langle Ad(B(\.)) j_{\r} , j_{\r '}\rangle$, we find that
\begin{equation*}
B_{\r ,\r '}(\. +\a )=\exp (\tilde{a}_{\r ' }(\. )- a_{\r }(\. ))B_{\r ,\r '}(\. )
\end{equation*}
Consequently, $B_{\r ,\r '}(\. )$ can be non-zero only if $\tilde{a}_{\r ' }(\. )- a_{\r }(\. )$
is a periodic function, which amounts to $\tilde{r}_{\r '} = r_{\r}$.

We conclude that the class of conjugant $B(\.)$ is well defined in $W(G_{+})\approx W(\tilde{G%
}_{+})\subset W(G)$, where $G_{0} \times G_{+} \hra G$ and
$\tilde{G}_{0} \times \tilde{G}_{+} \hra G$ are the splittings corresponding
to $(\a ,A(\. ))$ and $(\a ,\tilde{A}(\. ))$.
Since any two different maximal tori are obtained by an
algebraic conjugation in $G$ and all the groups $W(\TT )$ are
naturally isomorphic to $W$, independently of such conjugations, we have
proved the proposition
\end{proof}

Therefore, the additional symmetries of the torus, owing to the fact that it
is embedded in a Lie group $G$, result in the existence of conjugations
between models which are not conjugate otherwise. The action of Weyl's group
reduces the number of abelian models of dynamics, by taking into account the
symmetries of $\TT$ induced by the action of $G$. For example, the
cocycles $(\a , E_{r}(\. ))$ and $(\a , E_{-r}(\. ))$, $r \in \N ^{*}$
are not conjugate in $SW(\T , \Sp )$, but they are conjugate in
$SW(\T ,SU(2))$.

\bigskip

If the rotation in the basis is Diophantine, the vectors $r$ and $(a_{\r 
})_{\r \in \tilde{\Delta}}$ completely describe the dynamics of an abelian
cocycle up to conjugation (cf. eq. (\ref{abelian cocycle}) for the
notation), as shows the resolution of a classical linear cohomological
equation. If the rotation is Liouvillean, we cannot conclude anything more
than the density in $C^{\infty }$ topology in the space of abelian cocycles
of the orbits of models of the type%
\begin{equation*}
(\a ,\exp (\sum\nolimits_{\r }(2\pi (r_{\r }\. +a_{\r 
})h_{\r })
\end{equation*}%
Since any basis of a maximal torus of $G$ can be obtained by the action of $%
W(G)$ on a single one, we can consider a fixed basis and use the
notations 
\begin{eqnarray} \label{normal forms of cocycles}
E_{r}(\. ) &=&\exp (\sum\nolimits_{\r }(2\pi r_{\r }h_{\r }\. ))%
\text{ and} \\
E_{r,a}(\. ) &=&\exp (\sum\nolimits_{\r }(2\pi (r_{\r }\. 
+a_{\r })h_{\r })
\end{eqnarray}

\bigskip

In the particular case where $G=SU(2)$, the possible abelian models are
exhausted by%
\begin{equation*}
(\a ,\exp (\{2\pi r\. +\phi (\. )+a,0\})
\end{equation*}%
where $r\in \Z$. The action of Weyl's group identifies $r$ with $-r$ (and $\phi (\. )$
and $a$ with $-\phi (\. )$ and $-a$) in the space of abelian cocycles, so
that we need only consider $r\in \N$ in the description of the abelian models.
\selectlanguage{english}
\chapter{The energy of a cocycle in $\T \times G$} \label{Chapter on energy}

\section{Introduction}

In the previous chapter (proposition \ref{quantif of length for torus-reduc}),
we observed the phenomenon of quantization of the asymptotic normalized length of torus-reducible cocycles in
$\T \times G$ over any irrational rotation. This phenomenon was first observed in  \cite{Krik2001} in his study
of cocycles in $\T \times SU(2)$ (cf. also \cite{Fra2000} for a similar result for cocycles over the golden ratio).
In this more restricted algebraic context, R. Krikorian proved the quantification of this
same quantity for cocycles over rotations in the set $\s $ of theorem \ref{RK global density} as
an intermediate result of the proof of the global density of reducible cocycles in the non-local case. The
proof uses the local density theorem, theorem \ref{RK local density}, and the
renormalization scheme that we will be presented and implemented in
chapter \ref{Chapter renormalization}. Under some
arithmetic conditions, somewhat stricter than a simple recurrent Diophantine
one, renormalization was proved to converge towards either constant
cocycles, or the normal forms $(\a _{n},E_{r}(\. +\theta ))$, where
$E_{r}(\. )=\exp (2i\pi rh\. )$ and $\theta \in \T $ is a
constant and $r \in \N ^{*}$, which are geodesics of $SU(2)$
and thus a special case of abelian cocycles. These normal forms, called
obstructions, are not reducible by virtue of theorem
\ref{compactness of iterates of reducible cocycles}. They are
proved to be $C^{\infty }$-accumulated by reducible cocycles by establishing
a dichotomy: if a cocycle in the neighborhood of a normal form is not
reducible to it, renormalization converges to a normal form of smaller
length, and cocycles of the latter type form a dense set in the neighborhood
of normal forms. Induction and use of the local version of the theorem
concludes the proof, and the quantization of the asymptotic length of $C^{2}$
cocycles has been proved under the arithmetic condition imposed.

This result was generalized in \cite{Fra2004} (cf. theorems \ref{Fra quantif degree} and
\ref{Fra global density}). There, it was proved that
the conjugation to the obstructions under the condition defining
$\Sigma $
implies conjugation under a simple recurrent Diophantine condition
\footnote{This result can be strengthened by establishing a normal form theorem in the
neighbourhood of $(\a ,E_{r}(.))$ with a Diophantine $\a $, which
asserts that reducibility to $(\a ,E_{r}(.))$ is of codimension $2r$,
and that the dichotomy proved by R.\ Krikorian persists under the weaker
artithmetic condition. This will be the object of chapter \ref{Normal form theorem}}.
Moreover, it was proved by K. Fraczek with the use
of ergodic theorems and renormalization that the quantization phenomenon is
observed for $C^{2}$ cocycles independently of any arithmetic conditions,
and the asymptotic normalized length, or energy, of the cocycle was defined
to be its degree. Additionally, it was proved, using a technique similar to
the one used in order to obtain proposition \ref{splitting}, that every cocycle
in $\T \times SU(2)$ of positive degree is measurably
torus-reducible. This result cannot be generalized in more general groups,
as it is based essentially on the fact that the tori of $SU(2)$ are of
dimension $1$, or, in other terms, the fact that phenomena of positive
degree cannot be observed in coupling with local phenomena.

Here, we extend these results to cocycles in compact groups, where the
complexity of the geometry forces a change in terminology. The fact that a
maximal toral algebra in $su(2)$ is the real line in $\R \times \C$,
results in the classification of the different abelian models by an invariant
in $2\pi \N$, as already discussed, equal to the mean speed of a reparametrized
periodic geodesic. However, this invariant is insufficient for more
general groups. For example, the cocycles obtained by the geodesics
$E_{r_{i}}(\. )$ for $r_{1}=(3,0,4)$ and $r_{2}=(5,0,0)$, i.e.
\begin{equation*}
\begin{bmatrix}
e^{6i\pi \. } & 0 & 0 & 0 \\ 
0 & e^{-6i\pi \. } & 0 & 0 \\ 
0 & 0 & e^{8i\pi \. } & 0 \\ 
0 & 0 & 0 & e^{-8i\pi \. }%
\end{bmatrix}%
\text{ and }%
\begin{bmatrix}
e^{10i\pi \. } & 0 & 0 & 0 \\ 
0 & e^{-10i\pi \. } & 0 & 0 \\ 
0 & 0 & 1 & 0 \\ 
0 & 0 & 0 & 1%
\end{bmatrix}%
\end{equation*}%
in $SU(4)$, who have the same speed (or degree, according to K. Fraczek's
terminology), cannot be conjugated over any irrational rotation as shows
Proposition \ref{conjugation of abelian cocycles}. Therefore, we keep the
term degree for the vector in $2\pi \Z^{w}$, defined modulo the action of
Weyl's group, i.e. modulo permutations of its coordinates and changes of sign, and call its length energy.

\section{Definition and basic properties of the energy} \label{Definition and basic properties of the energy}

Let $(\a ,A(\. ))\in SW^{s}$, with $\a \in \T \backslash \Q$, $s\geq 1$, and
\begin{equation*}
a(\. )=LA(\. ) = \partial A (\. ).A^{\ast }(\. ) \in C^{s-1}(\T ,g)
\end{equation*}
(see section \ref{Calculus in Lie groups}). Since the 
$n$-th iterate of the cocycle for $n \geq 1$ is given by $(n\a ,A_{n}(\. ))=(n\a 
,A(\. +(n-1)\a )...A(\. ))$, the derivative of its inverse reads
\begin{equation*}
a_{n}^{\ast }(\. )=LA_{n}^{\ast }(\. )=a^{\ast }(\. )+Ad(A^{\ast
}(\. )).a_{n-1}^{\ast }(\. +\a )
\end{equation*}
as imply the properties of the differential operator $L$ (see section
\ref{Calculus in Lie groups}). Therefore, if we define the unitary operator%
\begin{equation*}
\begin{array}{cccc}
\mathcal{U}: & L^{2}(\T ,g) & \ra & L^{2}(\T ,g) \\ 
& b(\. ) & \mapsto & Ad(A^{\ast }(\. )).b(\. +\a )%
\end{array}%
\end{equation*}%
we have $a_{n}^{\ast }(\. )=\sum_{0}^{n-1}\mathcal{U}^{k}a^{\ast }(\. 
) $. We note that this operator is in fact associated to the cocycle $(\a ,A(\. ))$ and should be noted by $\mathcal{U}_{(\a ,A(\. ))}$, but the simplified notation will be used, since the cocycle is to be considered fixed.

By the von Neumann ergodic theorem, there exists $\ba_{+}^{\ast
}(\. )\in L^{2}(\T ,g)$, $\mathcal{U}$-invariant, such that
\begin{equation*}
\ba_{+}^{\ast }(\. )=\lim_{n\ra \infty }\frac{1}{n}%
a_{n}^{\ast }(\. )
\end{equation*}%
in $L^{2}$ and Lebesgue a.e. $U$-invariance implies that $|\ba _{+}^{\ast }(\. )|$ is constant a.e., since%
\begin{equation*}
|\ba_{+}^{\ast }(\. )|=|Ad(A^{\ast }(\. )).\ba_{+}^{\ast
}(\. +\a )|=|\ba_{+}^{\ast }(\. +\a )|
\end{equation*}%
and $\a $ is an ergodic translation on the torus.

In the same way we define $\ba_{-}^{\ast }(\. )$ as%
\begin{equation*}
\ba_{-}^{\ast }(\. )=-\lim_{n\ra \infty }\frac{1}{n}%
a_{-n}^{\ast }(\. )
\end{equation*}

Since $A_{-n}(\. +n\a )A_{n}(\. )=Id$, we have%
\begin{eqnarray*}
a_{n}^{\ast }(\. ) &=&-Ad(A_{n}^{\ast }(\. )).a_{-n}^{\ast }(\. 
+n\a ) \\
&=&-\mathcal{U}^{n}.a_{-n}^{\ast }(\. )
\end{eqnarray*}
and consequently we find that
\begin{equation*}
\| \frac{1}{n}a_{n}^{\ast }(\. )-\ba_{+}^{\ast }(\. 
)\|_{L^{2}}=\| -\frac{1}{n}a_{-n}^{\ast }(\. )-\mathbf{a%
}_{+}^{\ast }(\. )\| _{L^{2}}
\end{equation*}%
which implies that%
\begin{equation*}
\ba^{\ast }(\. )=\ba_{+}^{\ast }(\. )=\ba%
_{-}^{\ast }(\. )
\end{equation*}%
is unambiguously defined.

In a similar way, the fact that $A_{n}^{\ast }(\. )A_{n}(\. )=Id$, and
therefore $a_{n}^{\ast }(\. )=-\mathcal{U}^{n}.a_{n}(\. -n\a )$,
implies that
\begin{equation*}
\frac{1}{n}a_{n}(\. )\ra -\ba^{\ast }(\. )
\end{equation*}%
in $L^{2}$ when $n\ra \pm \infty $, and we are justified to define%
\begin{equation*}
\ba(\. )=\lim_{n\ra \pm \infty }\frac{1}{n}a_{n}(\. )=-%
\ba^{\ast }(\. )
\end{equation*}%
As before, $\ba(\. )$ satisfies the invariance equation \ref{curve}
\begin{equation*}
Ad(A(\. )).\ba(\. )=\ba(\. +\a )
\end{equation*}
and $|\ba(\. )|$ is constant $a.e.$. The invariance equation implies
that the measurable families of subalgebras $g_{+}(\. )$ and $g_{0}(\. 
)$, associated to $\ba(\. )$ as in section \ref{dynamics of
torus-reducible cocycles}, satisfy the same invariant relations $a.e.$. The
proofs are obtained as in the afore-mentioned section just by replacing
minimality and continuity by ergodicity and measurability, respectively.

\begin{definition}
Let $(\a ,A(\. ))$ be a $C^{1}$ cocycle. The number%
\begin{equation*}
\left\Vert \ba(\. )\right\Vert _{L^{2}}=\lim_{n\ra \pm
\infty }\frac{1}{n}\left\Vert a_{n}(\. )\right\Vert _{L^{2}}=|\ba%
(\. )|~a.e.
\end{equation*}%
will be called the \textit{energy} of $(\a ,A(\. ))$ and
denoted by $\emph{en}(\a ,A(\. ))$. The set where
$|\ba(\. )|=\emph{en}(\a ,A(\. ))$ is invariant by $R_{\a }$
and thus of full measure.
\end{definition}

We remark that the derivative of $A(\. )$ in $g$ and the inner
endomorphism $Ad(A^{\ast }(\. ))$ (and therefore the operator $\mathcal{U}
$) depend only on the Lie algebra $g$ and not on the topology of the group $G$ having $g$ as its Lie algebra.
Theorem \ref{compactness of iterates of reducible cocycles} implies directly that a cocycle of positive energy is
not reducible. There exist however cocycles of $0$ energy which are not reducible, as we will see in section
\ref{Density of zero-energy cocycles}.

Since the invariance of the curve gives directly that $Ad(A_{q_{n}}(\. )).%
\ba(\. )=\ba(\. +(-1)^{n}\beta _{n})$, the following
corollary is immediate

\begin{corollary}
$\ba(\. )$ and $A_{q_{n}}(\. )$ asymptotically commute:%
\begin{equation*}
Ad(A_{q_{n}}(\. )).\ba(\. )\ra \ba(\. )~\text{%
in~}L^{2}
\end{equation*}
\end{corollary}

We observe that the mapping $\ba(\. )$ generalizes the one
defined for $C^{1}$-torus-reducible cocycles, with uniform convergence
replaced by convergence in $L^{2}$, and the invariance of the degree in the
latter case persists in this more general setting:

\begin{proposition} \label{properties of energy}
The energy of a cocycle is invariant by conjugation in $C^{1}(\T ,G)$. Moreover,%
\begin{eqnarray*}
\emph{en}\left( \a ,A(\. )\right) ^{n} &=&|n|.\emph{en}\left( \a 
,A(\. )\right) ,\forall n\in \Z \\
\emph{en}\left( \a ,A(\. )\right) &\leq &\left\Vert a(\. 
)\right\Vert _{L^{i}},~i=1,2,\infty
\end{eqnarray*}
\end{proposition}

\begin{proof}
Let $B(\. )\in C^{1}(\T ,G)$. Then, calling $(\a ,\tilde{A}%
(\. ))=Conj_{B(\. )}(\a ,A(\. ))$,%
\begin{equation*}
\tilde{a}_{n}(\. )-Ad(B(\. +n\a )).a_{n}(\. )=b(\. +n\a 
)+Ad(B(.+n\a ).A_{n}(\. )).b^{\ast }(\. )
\end{equation*}%
so that%
\begin{equation*}
\left\vert |\tilde{a}_{n}(\. )|-|a_{n}(\. )|\right\vert \leq Cst
\end{equation*}%
and $\emph{en}(\a ,\tilde{A}(\. ))=\emph{en}(\a ,A(\. ))$.

The second relation follows directly from the fact that%
\begin{equation*}
\left( \a ,A_{n}(\. )\right) ^{k}=\left( \a ,A(\. )\right)
^{nk}
\end{equation*}%
and the third from the triangle inequality.
\end{proof}

\begin{corollary} \label{0 energy for KAM regime}
The ergodic cocycles in $\T \times SO(3)$ constructed in \cite%
{El2002a} are of zero energy.
\end{corollary}

In fact, we can prove that all cocycles in the regime of the local theory as
in \cite{Krik95} and \cite{KrikAst} are of zero energy, despite the eventual
loss of periodicity.

We also note the following simple fact.

\begin{lemma}
Let $(\a ,A(\. ))=(\a ,A_{1}(\. )\times A_{2}(\. ))$. Then,%
\begin{equation*}
\emph{en}((\a ,A(\. )))=\emph{en}((\a ,A_{1}(\. )))+\emph{en}%
((\a ,A_{2}(\. )))
\end{equation*}
\end{lemma}

In view of these properties of the invariant curve and the discussion in section \ref{introduction of splitting},
we are motivated to give the following definition

\begin{definition}
A cocycle (of positive energy) will be called \textit{regular} if, and only if, the curve $\ba
(\. )$ consists a.e. of regular vectors, and \textit{singular} otherwise. This
property is invariant by $C^{1}$-conjugation.
\end{definition}

We remind that a vector in $g$ is called regular if, and only if, it is
contained in a single maximal torus, so that a cocycle being regular or
singular is an algebraic property of its dynamics, and should not be
confused with the regularity of the cocycle, i.e. the regularity of the
mapping $\T \ra G$ defining the dynamics in the fibers.

The following corollary of Proposition \ref{construction splitting} is immediate:

\begin{corollary} \label{Diagonalization of regular cocycles}
If the cocycle $(\a ,A(\. ))$ is regular and of positive energy, then it can be measurably
conjugated to an abelian cocycle. In particular, all cocycles in $\T \times SU(2)$ of positive energy have this property.
\end{corollary}

The hypothesis of the corollary is optimal as shows the following example.
Let $\a \in DC$, $E_{r}(\. )$ be a periodic geodesic in $SU(2)$ and $%
(\a ,A\exp (F(\. ))$ be a non-reducible cocycle in $SU(2)$. It is
well known that such cocycles exist, and they are even abundant in the
neighborhoods of constants (see, e.g., \cite{El2002a}). The hypothesis that $%
\a \in DC$ implies that such a cocycle is not torus-reducible, since for
a Diophantine rotations reducibility and torus-reducibility coincide.
Finally, we consider $SU(2)\times SU(2)\hookrightarrow SU(4)$ and define $%
(\a ,A(\. ))=(\a ,E_{r}(\. )\times A\exp (F(\. ))$, which
is a cocycle in $\T \times SU(4)$. This cocycle is of positive
energy, but not regular, and clearly not $C^{\infty}$-torus-reducible, by virtue of
theorem \ref{conjugation of abelian cocycles}.

This is, in fact the best possible generalization of K. Fraczek's theorem which states that all cocycles of 
positive energy in $\T \times SU(2)$ can be measurably conjugated to abelian ones, obtained as a corollary
of theorem \ref{Fra degree} of the same paper.

If we try to adapt the arguments in section \ref{dynamics of torus-reducible cocycles} in the measurable context,
which is, in fact, imitating K. Fraczek's proof of the afore-mentioned result, but for a singular invariant curve,
we obtain a void statement. We cannot conclude anything more than the existence of a measurable conjugacy and the
corresponding (measurable) splitting
\begin{eqnarray*}
\tilde{G}_{0}\times \tilde{G}_{+} &\hookrightarrow &G \\
(\a ,\tilde{A}(\. )) &=&(\a ,\tilde{A}_{0}(\. )\times \tilde{A}%
_{+}(\. ))
\end{eqnarray*}%
where $(\a ,\tilde{A}_{+}(\. ))$ is "of positive energy" and $(\a ,\tilde{A}_{0}(\. ))$ is in general
non-abelian and "of zero energy".

The definition of the energy can be extended in a rather awkward way to some
measurable abelian cocycles by using Birkhoff's ergodic theorem. More precisely,
let us consider the case where $G=SU(2)$ and a measurable abelian cocycle
$(\a , A(\. ))$. If the mapping $A(\. )$ admits a lift in $g$ of the form
$\exp (2i \pi (r \. +\phi (\. )) h)$ with $r \in \mathbb{Z}^{*}$ and
$\phi (\. ) : \T \ra \mathbb{R}$ an integrable function such that
$\int \phi (\. ) = 0$, \footnote{This is in fact not any more restrictive than the existence of 
such a lift, since, if we consider $A(\. + \gamma)$ (such a choice of $\gamma$ is possible
since $r \not= 0$) instead of $A(\. )$, we can eliminate $\int \phi (\. )$
without affecting the dynamics.} then the $n$-th iterate is calculated as in the abelian case:
\begin{equation*}
A_{n}(\. )=\exp ((2\pi r(n \. +\frac{n(n-1)}{2}\a +S_{n}^{\a }\phi (\. ))h))
\end{equation*}
Clearly, normalization by $n$ of the lift of $A_{n}(\. )$ in $g$ and passage to
the limit gives $2\pi r$, which is what we would want to define as the degree of
$(\a , A(\. ))$. In fact this generalization is cyclic in the sense that
it supposes that the degree can be defined, as the existence of a lift of this form
is not guaranteed in the measurable context.

This generalization cannot in general be applied to $(\a 
,\tilde{A}_{+}(\. ))$, which is abelian by construction, and thus we do
not know if the cocycle $(\a ,\tilde{A}_{+}(\. ))$ remains a cocycle
of positive energy under this extended definition. On the other hand, there
is no known extension to (non-abelian) cocycles of regularity lower than $%
H^{1}$, and it is highly probable that there does not exist any, so that the
statement that $\emph{en}(\a ,\tilde{A}_{0}(\. ))=0$ is devoid of any
meaning, unless $(\a ,\tilde{A}_{0}(\. ))$ is found to be at least
measurably torus-reducible. However, since in $SU(2)$ exactly one of the factors of the
splitting is non-trivial (which is not always the case, even if it should be
so "generically", in any given compact group), and the conclusion that all
positive energy cocycles are measurable torus-reducible is important, even though
it does not survive in more general contexts.

\bigskip

Finally, let us revisit the expression $a_{n}^{\ast }(\. )=\sum \nolimits _{0}^{n-1}\mathcal{U}^{k}a^{\ast }(\. )$.
Using the triangle inequality, as well as the fact that $\mathcal{U}$ is a unitary operator, we find that
\begin{equation*}
\left\Vert a_{n}^{\ast }(\. )\right\Vert _{L^{2}}\leq n\left\Vert a^{\ast
}(\. )\right\Vert _{L^{2}}
\end{equation*}%
It is a classical fact that equality holds if, and only if, all the vectors
composing the sum are colinear. In this particular case, this amounts to the
existence of a positive constant $\lambda $ such that%
\begin{equation*}
\mathcal{U}a^{\ast }(\. )=\lambda a^{\ast }(\. )
\end{equation*}%
Since $| \lambda | =\left\Vert \mathcal{U}a^{\ast }(\. )\right\Vert
_{L^{2}}/\left\Vert a^{\ast }(\. )\right\Vert _{L^{2}}$ (we assume that $%
\left\Vert a^{\ast }(\. )\right\Vert _{L^{2}}\not=0$), we find that $\lambda =1$, and
the ergodic sum reduces to%
\begin{equation*}
a_{n}^{\ast }(\. )=na^{\ast }(\. )
\end{equation*}%
Therefore, in the case of maximal growth of the derivatives of the iterates of $(\a , A(\. ))$, $%
\ba^{\ast }(\. )=a^{\ast }(\. )$, and $\emph{en(}\a 
,A(\. ))=\| a^{\ast }(\. )\| _{L^{2}}$. Reinserting
the expression for the operator $\mathcal{U}$ we obtain the relation%
\begin{equation*}
a^{\ast }(\. )=Ad(A^{\ast }(\. )).a^{\ast }(\. +\a )
\end{equation*}%
or%
\begin{equation*}
a^{\ast }(\. +\a )=-a(\. )
\end{equation*}%
Thus,
\begin{eqnarray*}
a_{2}(\. ) &=&a(\. +\a )+Ad(A(\. +\a )).a(\. ) \\
&=&2a(\. +\a )
\end{eqnarray*}%
and inductively%
\begin{equation*}
a_{n}(\. )=na(\. +n\a )
\end{equation*}%
Since%
\begin{equation*}
\frac{1}{n}a_{n}(\. )=a(\. +n\a )\ra \ba(\. )
\end{equation*}%
and $\a $ is minimal, this can be true only if $a(\. )$ is constant
and equal to, say, $h$. The assumption that $A(\. )$ is $H^{1}$ implies,
therefore, that it is of the form%
\begin{equation*}
\exp (h\. ).A_{0}
\end{equation*}%
The condition that $\mathcal{U}a(\. )=a(\. )$ and the periodicity of $%
A(\. )$ imply respectively that%
\begin{eqnarray*}
Ad(A_{0}).h &=&h \\
\exp (h) &=&Id
\end{eqnarray*}%
and we have proved

\begin{theorem}
The operator $\mathcal{U}$ associated to a cocycle $(\a ,A(\. ))$ of
regularity $H^{1}$ admits $a(\. )$ as an eigenvector with positive
eigenvalue if, and only if, the cocycle is a periodic geodesic, and
therefore $C^{\infty }$. The eigenvalue is equal to $1$.
\end{theorem}

It seems reasonable to speculate that the attractors of the dynamics in $SW^{1}(\T,G)$
should be the periodic geodesics of the group, which is in fact proved to be
true in the next chapter.

\bigskip

Another distinct case of the action of the operator $\mathcal{U}_{(\a 
,A(\. ))}$ is the one associated to abelian cocycles, which is slightly
more general than that of the periodic geodesics. For such a cocycle, we have%
\begin{equation*}
Ad(A(\. )).a(\. +\a )=a(\. +\a )
\end{equation*}%
Supposing that this relation is verified, we find that the sum defining $%
\ba(\. )$ is reduced to a simple ergodic sum over $R_{\a }$,
so that $\ba(\. )=\hat{a}(0)$, a constant. This constant is fixed
by $Ad(A(\. ))$, as we have seen in the previous chapter.

\section{Higher-order derivatives}

Let us start by stating and proving a lemma on the growth of $C^{s}$ norms
of the iterates of a smooth cocycle:

\begin{lemma}
\label{growth lemma}The $n$-th iterate of the cocycle $(\a ,A(. ))\in SW^{s}$ satisfies the estimates%
\begin{eqnarray*}
\left\Vert \partial ^{\s }a_{n}(\. )\right\Vert _{0} &\leq &C_{\s 
}n^{\s +1}(1+\left\Vert a(\. )\right\Vert _{\s })\left\Vert
a(\. )\right\Vert _{0}^{\s +1}\text{, or} \\
\left\Vert A_{n}(\. )\right\Vert _{s} &\leq
&C_{s}n^{s}(1+\left\Vert a(\. )\right\Vert _{s})\left\Vert a(\. 
)\right\Vert _{0}^{s}
\end{eqnarray*}
where the constant $C_{\s }$ depends only on $\s $, and $0\leq
\s \leq s-1$. Therefore, the normalized derivatives $\frac{1}{n^{\s 
+1}}\partial ^{\s }a_{n}(\. )$ are bounded uniformly for $n$.
\end{lemma}

\begin{proof}
The proof is obtained by applying Proposition \ref{estimates on deriv of
products} and using the translation invariance of the norms.
\end{proof}

The result given above can be considerably refined:

\begin{lemma}
\label{higher order}Let $1\leq \s \leq s-2$. Then $\frac{1}{n^{\s +1}%
}\partial ^{\s }a_{n}(\. )\ra 0$ in $L^{2}$, as $n\ra 
\pm \infty $.
\end{lemma}

\begin{proof}
The estimates of Lemma \ref{growth lemma} and convexity inequalities imply
that it is sufficient to prove the lemma for $\s =1$, and the fact that $%
\ba(\. )=-^{\ast }(\. )$ implies that it is
sufficient to prove the result for $\frac{1}{n^{2}}\partial a_{n}^{\ast
}(\. )$.

The properties of the differential operator $L$ imply that
\begin{equation*}
\partial a_{n}^{\ast }(\. )=\sum\limits_{0}^{n-1}\mathcal{U}%
^{k-1}\partial a^{\ast }(\. 
)+\sum\limits_{k=0}^{n-2}\sum\limits_{l=k+1}^{n-1}[\mathcal{U}^{k}a^{\ast
}(\. ),\mathcal{U}^{l}a^{\ast }(\. )]
\end{equation*}

As has already been proved, $\frac{1}{n}\sum_{0}^{n-1}\mathcal{U}^{k-1}\partial a^{\ast }(\. )$ converges in $L^{2}$,
and therefore
\begin{equation*}
\frac{1}{n^{2}}\sum_{0}^{n-1}\mathcal{U}^{k-1}\partial a^{\ast }(\. )
\end{equation*}
tends to 
$0$ as $n\ra \infty $, and, since $[\. ,\. ]$ is
antisymmetric, we have to prove that%
\begin{equation*}
\frac{1}{n^{2}}\sum\limits_{k=0}^{n}\sum\limits_{l=0}^{k}[\mathcal{U}%
^{k}a^{\ast }(\. ),\mathcal{U}^{l}a^{\ast }(\. )]\ra 0\text{ in }L^{1}(\T ,g)
\end{equation*}
which is the object of lemma \ref{summation} in Appendix \ref{appendix 1}.

Since the sequence $(\frac{1}{n^{2}}\partial a_{n}^{\ast }(\. ))_{n}$ is
uniformly bounded in $L^{\infty }$, we obtain the convergence in $L^{2}$,
and thus we can prove the same result for higher derivatives. The case $%
n\ra -\infty $ is treated analogously.
\end{proof}

The two lemmas of this section imply that the essential information on the asymptotic behaviour of the
dynamics of a cocycle should be contained in at most the first derivative, since the mere
existence in $L^{2}$ of a derivative of a higher order implies that, after
rescaling, it converges to $0$ in the same space. This is the object of the
next chapter.
\selectlanguage{english}
\chapter{$\Z^{2}$ actions on $\R \times G$ and the renormalization scheme} \label{Chapter renormalization}

In this chapter, we present and apply the main tool used in the study of
non-local phenomena of the dynamics of cocycles in $\T \times G$,
where $G$ is typically $SL(2,\R)$ or $SU(2)$. It can be described essentially
as a division algorithm, since it is the outcome of the lift of the euclidean division algorithm for
the continued fractions expansion to the dynamics of a cocycle over an irrational rotation.
The lift of this algorithm divides the limit object of the dynamics
obtained in the previous chapter, i.e. the invariant curve $\ba (\. )$, by the group
of iterates of the cocycle. Since, however, the limit object lives in a low-regularity space
($L^{2}$), the algorithm should be expected to diverge (at least a priori) in $C^{\infty}$.

Another interpretation of the function of renormalization in the study of quasiperiodic cocycles
could be the reduction of the study of general global objects to the local study of perturbations of normal forms,
obtained as the basins of attraction of renormalization. The property used in the
reduction of the global study to the different local ones is the preservation of the
quality of dynamics under renormalization, and it is precicely due to the fact that renormalization divides the 
limit object of the dynamics by the group of iterates of the cocycle.

An application of a weaker, yet renormalization-like, algorithm to the study of cocycles
in $\T \times SU(2)$ dates back to \cite{Rych92}, with some interesting results and questions
for research, but the application is restricted to the golden ratio, and it is based on the very fact
that, for the golden ratio $\varphi $, iteration of $R_{\varphi }$ coincides with the continued
fractions algorithm. This scheme was used also in \cite{Fra2000}, where it was used in the proof of
the quantization of the degree of a cocycle, as it is defined in the same paper.

\bigskip

Renormalization was introduced in the study of cocycles in the same product space in \cite{Krik2001},
under a form which is applicable to any irrational rotation, and the convergence results where proved
for a full measure set of frequencies. The algorithm was proved to
converge for $C^{2}$ cocycles over such rotations, and then the local study of the normal forms towards
which the renormalization converges gave the first global density result for $C^{\infty }$ cocycles in
$\T \times SU(2)$. In the study of the convergence of renormalization, two functionals where introduced.
The one based on $L^{2}$ norms (called $J^{(2)}$) was used to measure the convergence and the second,
the length functional (called $J^{(1)}$), was proved to take discrete values in a way that generalizes
the result of K. Fraczek, obtained for cocycles over the golden ratio.

These results where generalized in manifold directions in \cite{Fra2004}. In
a first time, the algorithm was proved to converge independently of any
arithmetic conditions on the (irrational) rotation, for $C^{1}$ regularity.
The two functionals introduced by R. Krikorian where proved to be
alternative definitions of the degree, a quantity that can be defined
without the use of renormalization and for $C^{1}$ cocycles over any
irrational rotation. In particular, the regularity needed for the
equivalence was only $C^{1}$ for $J^{(2)}$, but $C^{2}$ for $J^{(1)}$. Since
it is $J^{(1)}$ that is used in the identification of the normal forms
towards which renormalization converges, the quantization of the degree was
proved in regularity $C^{2}$. Re-examination and improvement of the local
study made in \cite{Krik2001} loosened the arithmetic condition for the
global density of reducible cocycles to a $RDC$. Finally, it was proved that 
$C^{2}$ cocycles that are measurably conjugate have the same degree.

\bigskip

In the presentation of the algorithm, we have adapted the one in \cite{FK2009}, used for the study of cocycles
in $\T \times SL(2,\R)$. The proof of the convergence of renormalization is based on K. Fraczek's approach
as it is given in \cite{Fra2004}, but the notation is adapted to that used in \cite{Krik2001}.

The generalizations of the results are essentially two-fold, the first
direction being of algebraic nature and the second one concerning
regularity. By its definition, the algorithm does not depend on the fibers $G$
of the product space, but only on the fact that the cocycle has only one
frequency, since continued fractions algorithms are less efficient when more
frequencies are involved. The tool is therefore ready to be applied to the study of
cocycles in any compact Lie group, where the complexity of the
observable phenomena outside the case studied so far is revealed (see the discussion on the
splitting of the dynamics, cf. section \ref{introduction of splitting}). The second direction
is based on a rather simple observation, which establishes the equivalence
of the three different possible definitions of the degree for $C^{1}$ cocycles,
and this allows us to prove the quantization of the degree
for cocycles of such regularity. In fact, observation of the proof shows
that the real regularity needed is $H^{1}$ and this regularity represents the threshold under which the degree of a
cocycle cannot be defined. For this, see \cite{Kat01} for the study of the "parabolic map" of $\T ^{2}$, given by
$(x,y) \ra (x+ \a , x+ y) $. It is the additive version of the dynamics of periodic geodesics in $SU(2)$. Therein it is
established that all invariant distributions of this mapping, other than the Haar measure on the torus, are in
$H^{-1} (\T ^{2})$, the space of distributions in $\T ^{2}$ that can be continuously extended to functionals on
$H^{1}(\T ^{2} )$.

Summing up, in this chapter we prove that the degree of a cocycle in $SW^{1}(\T ,G)$, for any compact
Lie group $G$, admits a definition in a way that naturally extends the degree as defined for
cocycles in $SW^{1}(\T,SU(2))$. This degree is the topological degree of a periodic geodesic,
viewed as a mapping of $\T$ into a maximal torus of $G$, and
therefore takes discrete values. The periodic geodesics of the group are the
normal forms which arise as topological obstructions to reducibility, thus
marking the passage from local to global. Finally, the measurable invariance
of the degree is proved to be true in regularity $C^{1}$ and persists in
this more general algebraic setting.

\section{Synopsis of the chapter}

After having introduced the notion of $\Z ^{2}$ actions and defined the renormalization algorithm for
such actions, we generalize the definition of the energy of a cocycle to actions and subsequently
study the convergence of the algorithm in order to obtain the following key theorems to the proof
of the global density theorem.

The first one concerns the \textit{quantization} of the length of the vectors belonging to the invariant
curve $\ba (\. )$ (see section \ref{Definition and basic properties of the energy} for the notation),
which is the same as for abelian cocycles:
\begin{theorem} \label{quant deg}
Let $(\a ,A(\. ))\in SW^{1}(\T ,G)$, with $G$ a semisimple compact Lie group and $w$ the dimension of its maximal tori.
Then, the energy of the cocycle satisfies
\begin{equation*}
\emph{en}(\a ,A(\. ))=2\pi \| \sum\nolimits_{\rho \in \tilde{%
\Delta}}r_{\r }h_{\r }\| ,~r\in \N^{w}
\end{equation*}%
where the vector $h = 2\pi \sum\nolimits_{\rho \in \tilde{\Delta}}r_{\rho
}h_{\rho } \in g$ is a preimage of the $Id$ and $\ba (\nu ) \in Inn(g).h$ for $a.e. \, \nu \in \T$. 
\end{theorem}

This theorem, whose proof is in p. \pageref{proof quant def}, motivates the following definition

\begin{definition}
The vector $r\in \N^{w}$, defined modulo the action of Weyl's group and invariant under $C^{1}$
conjugation, is the degree of $(\a ,A(\. ))$. It is denoted by $\deg ((\a ,A(\. )))$ and it satisfies
\begin{equation*}
\exp (2\pi \sum\nolimits_{\rho \in \tilde{\Delta}}r_{\rho }.h_{\rho })=Id
\end{equation*}
\end{definition}

The degree of an abelian cocycle coincides with the topological degree of
the corresponding mapping $E_{r} (\. ):\T \ra \TT $, which justifies the terminology.

The following theorems concern the convergence of renormalization \textit{per se}. For the definition
of the unknown quantities in their statements (mostly the operators $(\Lambda _{\x _{0}})_{\ast }$ and
$\mathcal{\tilde{R}}^{n}_{\nu }$ acting on $\Phi$), we address the reader to section \ref{Renormalization of actions}.

The first theorem asserts that renormalization of a positive energy cocycle $(\a ,A(\. ))$, homotopic to constants, converges
towards geodesics of speed equal to the degree of $(\a ,A(\. ))$:
\begin{theorem} \label{Quantization of degree}
Let $(\a ,A(\. ))\in SW^{s}(\T ,G)$, $2\leq s\leq \infty $, homotopic to the $Id$, and $\emph{en}(\a ,A(\. ))>0$.
Then, for $a.e.$ $\nu \in \T $, there exists a sequence of conjugations $D_{n,\nu }(\. )\in H^{s-1}(\R,G)$
satisfying the following properties: $Conj_{D_{n,\nu }(\. )} \mathcal{\tilde{R}}^{n}_{\nu }\Phi$ is normalized
and $Conj_{D_{n,\nu }(\. )} \mathcal{\tilde{R}}^{n}\Phi (e_{2})(\. )$ is arbitrarily close to
\begin{equation*}
A_{\nu }^{(n)}.\exp (\ba (\nu )(\. ))
\end{equation*}%
in $H^{s-1}(\T ,G)$, with $A_{\nu }^{(n)}\in \mathcal{Z}_{\nu }$. The set of such $\nu \in \T $
is invariant under $R_{\a }$.
\end{theorem}
We have called $ \mathcal{Z}_{\nu }= Z_{G}(\exp (\R \ba (\nu)))$.
For cocycles non-homotopic to constants, we have the following version of the theorem.
\begin{theorem} \label{Quantization of degree non-hom}
Let $(\a ,A(\. ))\in SW^{s}(\T ,G)$, $2\leq s\leq \infty $ and $\emph{en}(\a ,A(\. ))>0$.
Then, for $a.e.$ $\nu \in \T $, there exists a sequence of conjugations $D_{n,\nu }(\. )\in H^{s-1}(\R,G)$
satisfying the following properties: $Conj_{D_{n,\nu }(\. )} \mathcal{\tilde{R}}^{n}_{\nu }\Phi$
is normalized, and $Conj_{D_{n,\nu }(\. )}\mathcal{\tilde{R}}^{n}\Phi (e_{2})(\. )$ is arbitrarily close to
\begin{equation*}
A_{\nu }^{(n)}.\exp ( \ba (\nu )\. ) . A_{\nu , 0 }^{(n)} (\. )
\end{equation*}
in $H^{s-1}(\T ,G)$. Here $A_{\nu , 0 }^{(n)} (\. ) : \T \ra G_{0 }$ is a geodesic of minimal length in its
homotopy class, of order $\x $ in the homotopy group of $G_{0}$,
and $A_{\nu }^{(n)}\in \mathcal{Z}_{\nu }$ anti-commutes with $A_{\nu , 0 }^{(n)} (\. )$:
$A_{\nu }^{(n)}.A_{\nu , 0 }^{(n)} (\. ).(A_{\nu }^{(n)})^{*}=A_{\nu , 0 }^{(n)} (\. )^{*}$. The renormalized action
$Conj_{D_{n,\nu }(\. )}\Lambda _{\x } . \mathcal{\tilde{R}}^{n}_{\nu }\Phi$, which corresponds to the
renormalization representative of the $\x$-th iterate of $(\a , A(\. ))$, satisfies the conclusions of the
previous theorem. The constant $\x $ satisfies $0 \leq \x \leq \x _{0}$,
with $\x _{0}$ the constant of lemma \ref{parties abeliennes} corresponding to $G_{0}$, the $0$-energy component
of the splitting associated to $(\a ,A(\. ))$. This constant is taken equal to $1$ if $G_{0} = \{ Id \}$.
\end{theorem}
Thus, homotopy appears in the part of the dynamics which is of $0$ energy, if it is not trivial.

The second couple of theorems concerns the cocycles of $0$, and asserts that renormalization of $0$-energy
cocycles converges toward constant cocycles, again modulo iteration by $\x $ times:
\begin{theorem} \label{Renormalization 0 energy}
Let $(\a ,A(\. ))\in SW^{s}(\T ,G)$, $2\leq s\leq \infty $, homotopic to the $Id$, and
$\emph{en}(\a ,A(\. ))=0$. Then, for every $\nu \in \T $, there
exists a sequence of conjugations $D_{n,\nu }(\. )\in H^{s-1}(\R,G)$ such that
$Conj_{D_{n,\nu }(\. )} \mathcal{\tilde{R}}^{n}\Phi $ is normalized and
$Conj_{D_{n,\nu }(\. )}\mathcal{\tilde{R}}^{n}\Phi (e_{2})(\. )$ is arbitrarily close to a constant.
\end{theorem}
The version for cocycles non-homotopic to constants reads
\begin{theorem} \label{Renormalization 0 energy non-hom}
Let $(\a ,A(\. ))\in SW^{s}(\T ,G)$, $2\leq s\leq \infty $ such that $(\a ,A(\. ))^{m }$ is homotopic to the $Id$. and 
$\emph{en}(\a ,A(\. ))=0$. Then, for every $\nu \in \T $, there
exists a sequence of conjugations $D_{n,\nu }(\. )\in H^{s-1}(\R,G)$ such that
$Conj_{D_{n,\nu }(\. )} \mathcal{\tilde{R}}^{n}\Phi $ is normalized and
$Conj_{D_{n,\nu }(\. )} \mathcal{\tilde{R}}^{n}\Phi (e_{2})(\. )$ is arbitrarily close to $A_{0}.\tilde{A}(\. )$. Here,
$\tilde{A}(\. )$ a geodesic of minimal length in the homotopy class of
$A(\. )$ and $A_{0}\in G $ anti-commutes with $\tilde{A}(\. )$.
The action $Conj_{D_{n,\nu }(\. )}\Lambda _{m } . \mathcal{\tilde{R}}^{n}\Phi $
is arbitrarily close to a normalized constant action.
\end{theorem}

We then proceed to a brief discussion on the preservation of the homotopy under renormalization
%
%
before extending a result known since \cite{Fra2004}, namely the dependence
of the degree of the cocycle (now defined as a vector and not as a numerical quantity) on the
rotation in the basis to this more general context.

The degree of a $C^{1}$-cocycle can be well defined, we have seen that
it results immediately from the definitions that $C^{1}$ conjugations leave
the degree invariant. However, using the convergence of renormalization, we
can prove a stronger proposition:

\begin{proposition} \label{thm measurable invariance}
Let $(\a ,A_{i}(\. ))\in SW^{1}(\T ,G)$, $i=1,2$, be conjugate
by $B(\. ):\T \ra G$ and let us suppose that $B(\. )$
is measurable. Then%
\begin{equation*}
\deg (\a ,A_{1}(\. ))=\deg (\a ,A_{2}(\. ))
\end{equation*}
\end{proposition}
The proof of this proposition occupies section \textit{5.8}.

\section{$\Z^{2}$ actions}

In order to cope with the expected divergence in $C^{\infty }$ of the approximation of the limit object $\ba (\. )$,
we will have to examine the dynamics in finer scales, which sums up to considering tori
$\gamma ^{-1} \T = \R / \gamma ^{-1} \Z$, where $\gamma $ goes to infinity in a controlled way, related
to the arithmetic properties of $\a$. In order to avoid such non-uniform objects,
we consider cocycles in $SW^{s}(\R,G)$ and express periodicity in a more abstract way,
which nonetheless proves to be practical for the study of the dynamics.

Let, therefore, $(\a ,A(\. ))$ and $(\beta ,B(\. ))$ be commuting cocycles in $SW^{s}(\R,G)$,
with $\a ,\beta \in (0,1]$, not necessarily irrational. Commutation of these cocycles
sums up to the following relation satisfied by the mappings $A(\. )$ and $B(\. )$
\begin{equation*}
B(\. +\a ).A(\. )=A(\. +\beta).B(\. )
\end{equation*}
We can then define an \textit{abelian action} of $\Z^{2}$ on $\R\times G \, $
\footnote{and likewise on the product of $\R$ with any space on which $G$ acts.} by
\begin{equation*}
\Phi \left( k,l\right) .\left( x,S\right) =(\a ,A(\. ))^{k}\circ
(\beta ,B(\. ))^{l}.\left( x,S\right)
\end{equation*}%
The space of such actions will be denoted by $\Lambda ^{s}(\R,G)$ and we will say that
the action $\Phi $ is generated by $(\a ,A(\. ))$ and $(\beta ,B(\. ))$.
We will denote by $e_{1}=(1,0)$ and $e_{2}=(0,1)$ the standard basis of the
module $\Z^{2}$. We will also use the notation
\begin{equation*}
\begin{pmatrix}
(\a ,A(\. )) \\ 
(\beta ,B(\. ))
\end{pmatrix}
\end{equation*}
for such an action $\Phi $.

\begin{definition}
An action will be called \textit{constant} if $\Phi \left( k,l\right) $ is a constant
cocycle for every $(k,l)\in \Z^{2}$. It will be called \textit{normalized} if
$\Phi (e_{1})=(1,Id)$, in which case $\pi _{2}(\Phi (e_{2}))$ is automatically $1$-periodic.
A cocycle $(\a ,A(\. ))$ with $A(\. )$ $1$-periodic is thus naturally identified with the
action generated by $(1,Id)$ and $(\a ,A(\. ))$. We will say that this action is \textit{associated}
to the cocycle $(\a ,A(\. ))$.
\end{definition}
An action is constant iff its generators are constant, and we remind this fact since we will examine actions through their
generators, and not abstractly.


The space $C^{s}(\R,G)$ acts on $\Lambda ^{s}$ by the generalization of conjugation
in $SW^{s}(\T ,G)$: if $B(\. )\in C^{s}(\R,G)$, then $\Phi ^{\prime }=Conj_{B(\. )}\Phi $ is given by
\begin{equation*}
\Phi ^{\prime }(k,l)=Conj_{B(\. )}(\Phi (k,l)),~\forall ~(k,l)\in \Z^{2}
\end{equation*}
and this gives rise to the following natural definitions.

\begin{definition}
Two actions will be called \textit{conjugate} if they belong to the same orbit under
the action of conjugacies.

An action will be called \textit{reducible} if it is conjugate to a normalized
constant action, and \textit{torus-reducible} if it is conjugate to a normalized
action such that $\pi _{2}(\Phi (e_{2}))$ takes its values on a torus of $G$%
, which with no loss of generality can be assumed to be a fixed torus $%
\mathcal{T}$.
\end{definition}
Naturally, two actions are conjugate iff their generators are, but with transfer functions in $\R \ra G$.

%

\bigskip

If $d(\. )$ is a (semi)metric on a space of applications taking values in 
$G$, then it induces a (semi)metric on $\Lambda ^{s}$ defined by%
\begin{equation*}
\max_{i=1,2}d(\pi _{2}(\Phi _{1}(e_{i})),\pi _{2}(\Phi _{2}(e_{i})))
\end{equation*}%
i.e. the $\max $ of the distances of their generators. In the special case
of the $C^{s}(I)$ topology, where $I\subset \R$ or $\T $ and $s\geq 0$, we will use the notations%
\begin{eqnarray*}
\left\Vert \Phi \right\Vert _{s,I}^{0} &=&\max_{i=1,2}\left\Vert \pi _{2}(\Phi
(j_{i}))\right\Vert _{s,I} \\
d_{s,I}(\Phi _{1},\Phi _{2}) &=&\max_{i=1,2}\left\Vert \pi _{2}(\Phi
_{1}(j_{i}))-\pi _{2}(\Phi _{2}(j_{i}))\right\Vert _{s,I}
\end{eqnarray*}%
and%
\begin{eqnarray*}
\left\Vert \Phi \right\Vert _{s,I} &=&\max_{1\leq \s  \leq
s}\left\Vert \Phi \right\Vert _{\s  ,I} \\
d_{s,I}(\Phi _{1},\Phi _{2}) &=&\max_{1\leq \s  \leq s}d_{\s 
,I}(\Phi _{1},\Phi _{2})
\end{eqnarray*}

We can now prove

\begin{lemma}\label{normalization lemma}
Let $\Phi \in \Lambda ^{s}(\R,G)$ such that $\pi _{1}(\Phi (1,0))=1$. Then there exists $B(\. )\in C^{s} (\R,G)$ such
that $Conj_{B(\. )}\Phi $ is normalized. $B(\. )$ satisfies, for $1\leq \s  \leq s-1$ and $T\in \N^{\ast }$, the estimates
\begin{equation*}
\| B(\. )\|_{\s  ,T} \leq C_{\s } T^{\s }(1+\left\Vert \pi _{2}(\Phi (e_{1})\right\Vert _{\s ,T})
\left\Vert \pi _{2}(\Phi (e_{1}))\right\Vert _{0,T}
\end{equation*}
and the normalized action satisfies
\begin{equation*}
\| Conj_{B(\. )}\Phi (e_{2})\| _{s,1} \leq C_{s}
(1+\left\Vert \pi _{2}(\Phi (e_{1}))\right\Vert _{1,2})\left\Vert \pi
_{2}(\Phi (e_{2})\right\Vert _{s,1}^{\max }+\left\Vert \pi _{2}(\Phi
(e_{1})\right\Vert _{\s  ,2}^{\max }\left\Vert \pi _{2}(\Phi
(e_{1}))\right\Vert _{1,1}
\end{equation*}

\begin{proof}
In order to prove the first part, we need to solve the equation
\begin{equation}
B(\. +1)=B(\. ).C^{\ast }(\. )  \label{normal}
\end{equation}%
where $C(\. )=\pi _{2}(\Phi (1,0))$. Let us fix $B_{0}(\. )\in C^{s}([0,1],G)$ such that
\begin{eqnarray*}
B_{0}(0) &=&Id~\text{and }\partial ^{\s  }B_{0}(0)=0,~1\leq \s  \leq s
\\
\partial ^{\s  }B_{0}(\. )|_{t=1} &=&\partial ^{\s  }(B_{0}(\.
).C^{\ast }(\. ))|_{t=0},~0\leq \s  \leq s \\
\left\Vert B_{0}(\. )\right\Vert _{\s  ,1} &\leq &C_{\s 
}\left\Vert C(\. )\right\Vert _{\s  ,1},~1\leq \s  \leq s-1
\end{eqnarray*}%
Then, define for all $j\in \Z$
\begin{equation*}
B(\. +j)=B_{0}(\. ).C_{j}^{\ast }(\. )
\end{equation*}%
It can be verified inductively that $B(\. )$ is $C^{s}(\R,G)$ and by its very definition it solves
the equation (\ref{normal}). Moreover, if $\s  ,T\in \N$, it satisfies
\begin{equation*}
\left\Vert B(\. )\right\Vert _{\s  ,[T,T+1]}=\left\Vert B_{0}(\.
).C_{T}(\. )\right\Vert _{\s  ,1}
\end{equation*}%
and the estimate can be proved by imitating the proof of lemma \ref{growth
lemma}.
\end{proof}
\end{lemma}

Since $B(\. )\in C^{s}(\R,G)$ is $1$-periodic if $Conj_{B(\. )}(1,Id)=(1,Id)$, the following lemma is
an immediate consequence of the definitions and the previous lemma.

\begin{lemma} \label{relation of actions with cocycles}
A cocycle is (torus-)reducible if, and only if, the $\Z^{2}$ action associated to it is (torus-)reducible.
Likewise, it is accumulated by (torus-)reducible cocycles if, and only if, the corresponding action is.
\end{lemma}

We remark that, in view of the proof of the previous lemma, an action is
torus-reducible if it is conjugate to an action $\Phi ^{\prime }$ such that
$\pi _{2}(\Phi ^{\prime }(e_{1}))$ and $\pi _{2}(\Phi ^{\prime }(e_{2}))$
take their values on the same torus, and $\pi _{1}(\Phi ^{\prime }(e_{1}))=1$.

However, unlike with actions associated to cocycles in, say $\T \times SU(w+1)$,
we cannot always normalize a constant action to another constant one
for the following reason. If%
\begin{equation*}
\Phi =%
\begin{pmatrix}
(1,C) \\ 
(\a ,A)%
\end{pmatrix}%
\end{equation*}%
is a constant action, then $C$ and $A$ commute. However, if $G$ is not simply connected, then they need not belong to
the same maximal torus. A natural conjugation normalizing the cocycle is $B(\. )=\exp (-H\. )$,
where $\exp (H)=C$. But, unless $Ad(A).H=H$,
\begin{equation*}
\Phi ^{\prime }=%
\begin{pmatrix}
(1,Id) \\ 
(\a ,\exp (-H(\. +\a )).A.\exp (H\. ))
\end{pmatrix}%
\end{equation*}
is not constant, but with a good choice of the conjugant it is a geodesic of minimal length in its class, as shows the
proof of lemma \ref{parties abel. simpl. conn.}. Normalization can be achieved by considering the action given by a
sub-lattice of $\Z ^{2}$. If $(\a ,\exp (-H(\. +\a )).A.\exp (H\. ))^{m}$ is homotopic to constants, then there exists a
preimage $H' \in g $ of $C$ and $0 \leq \x \leq \x _{G}$ such that $Ad(A^{\x }).H'=H'$. Therefore, if we consider
the action $\tilde{\Phi}$ generated by $(1,C)$ and $(\a , A) ^{\x }$, and conjugate by $B'(\. )=\exp (-H' \. )$,
we obtain the constant action
\begin{equation*}
Conj_{B'(.\ )} \tilde{\Phi}=
\begin{pmatrix}
(1,Id) \\ 
(\x \a ,\exp (-\x H' \a ).A^{\x })
\end{pmatrix}
\end{equation*}
Therefore, we have the following lemma.
\begin{lemma} \label{covering}
Let $\Phi $ be a constant action with generators $(1,C)$ and $(\a ,A)$. This action can be normalized to a constant action
iff $A$ and $C$ belong to a common maximal torus. If not, it can be normalized to an action $\Phi '$ associated to a cocycle
$(\a , A(\. ))$ non-homotopic to constants, and more precisely to a geodesic of minimal length in its class.
The minimal $\x \in \N $ such that $C $ and $A^{\x }$ belong to a common maximal torus and the order of $A(\. )$ in
the homotopy group of $G$ are the same number.

The action generated by $(1,C)$ and $(\a ,A)^{\x }$, corresponding to the sub-action $\Z \oplus \x \Z \hra \Z ^{2}$
can be normalized to a constant action.
\end{lemma}
This last action is $\Lambda _{\x }. \Phi$ in the notation of the next section. Since normalized actions are
associated to cocycles conjugate to each other (cf. proof of lemma
\ref{invariance of degree under conjugation of actions}, and since homotopy is preserved by fibered conjugation,
a different choice of conjugant will not solve the problem. Either we do not normalize the action to a constant one,
or we consider an iterate of the cocycle, the action associated to which can be normalized to a constant one.

\section{Renormalization of actions} \label{Renormalization of actions}

We now present the basic tool for the study of the dynamics of cocycles via
their identification with the corresponding $\Z^{2}$-action. Given the irrational number $\a $,
we generate an infinity of base-changes of the module $\Z^{2}$, and these base-changes
are pulled back to changes of generators of the $\Z^{2}$- action corresponding to $(\a ,A(\. ))$.
These successive base-changes, given by the continued fractions algorithm applied to $\a $, provide us
with a subsequence of iterates of $(\a ,A(\. ))$, which, being adapted to $\a $,
reflects its arithmetic properties, and whose derivatives converge to $\ba (\. )$.

We now define the actions that are used in the renormalization:

\begin{itemize}
\item $GL(2,\Z) $ acts on $\Lambda ^{s}$ by change of basis of the module $\Z^{2}$. If $\Phi \in \Lambda ^{s}$ and
$P\in GL\left( 2,\Z\right) $, then $P_{\ast }\Phi (k,l)=\Phi (P^{-1}.(k,l)),~\forall ~(k,l)\in \Z^{2}$, i.e.
$P_{\ast }\Phi $ is generated by $P^{-1}.e_{1}$ and $P^{-1}.e_{2} $

\item $\R$ acts by translations in the basis: for $\nu \in \R$: 
\begin{equation*}
\pi _{2}(T_{\nu }\Phi (k,l))=\pi _{2}(\Phi ((k,l))(.+\nu ),~\forall
~(k,l)\in \Z^{2}
\end{equation*}
and we remind that $\pi _{2}$ is the projection on the mapping defining the cocycle, defined right after $\pi _{1}$,
the projection on the frequency (cf. eq. \ref{def projections coc}).
\item $\R_{+}^{\ast }$ acts by dilatations of the basis:\ if $\lambda \in \R_{+}^{\ast }$, then
\begin{eqnarray*}
\pi _{1}(M_{\lambda }\Phi (k,l)) &=&\lambda ^{-1}\pi _{1}(\Phi ((k,l))(\. ) \\
\pi _{2}(M_{\lambda }\Phi (k,l)) &=&\pi _{2}(\Phi ((k,l))(\lambda \. ),~\forall ~(k,l)\in \Z^{2}
\end{eqnarray*}

\item We also define the action of the semigroup $\N ^{*}$ by the restriction to the a sub-lattice. If
$m \in \N ^{*}$, the action $\Lambda _{m }. \Phi $ is generated by $\Phi (e_{1}) $ and $\Phi (m e_{2})$.
\end{itemize}

We note that the first three actions commute with each other, since dilatations and translations commute with iteration.
Conjugation commutes with base changes, while it satisfies the following invariance properties
\begin{eqnarray*}
T_{\nu }\circ Conj_{B(\. )} &=&Conj_{B(\. +\nu )}\circ T_{\nu } \\
M_{\lambda }\circ Conj_{B(\lambda \.)} &=&Conj_{B(\. )}\circ M_{\lambda }
\end{eqnarray*}
and he equalities
\begin{equation*}
\left\Vert M_{\lambda }\Phi \right\Vert _{s,T}=\lambda ^{s}\left\Vert
M_{\lambda }\Phi \right\Vert _{s,\lambda ^{-1} T}
\end{equation*}%
follow directly from the definitions.

Since the loss of information described by lemma \ref{covering} is inherent, we give the following definition
\begin{definition} \label{def. lattice}
Given a $\Z^{2}$ action $\Phi $ and $m\in \N^{\ast }$, the action $\Lambda _{m} .\Phi $ will be called the
$m$-lattice of $\Phi $. An action $\Phi ^{\prime }$ such that $\Lambda _{m} .\Phi$ is normalized will be called
normalizable modulo $m$.
\end{definition}

It can be seen directly that the action $\Phi $ is (torus-)reducible (resp.
accumulated by reducible actions) iff $M_{\lambda }P_{\ast }\Phi $ is
(torus-)reducible (resp. accumulated by reducible actions) for some (and
thus for all) $P\in GL(2,\Z)$, and $\lambda =\pi_{1}(P_{\ast }\Phi (1,0))$, and the same holds for $T_{\nu }$. We can
therefore generalize the different notions of reducibility of $\Z^{2}$-actions as defined and used above to
\begin{definition}
An action $\Phi$ will be called reducible if it is conjugate to a constant
action $\Phi^{\prime }$ such that $\pi _{2}(\Phi^{\prime }(e_{1}))=Id$,
and torus-reducible if it is conjugate to an action $\Phi ^{\prime }$ such
that $\pi _{2}(\Phi ^{\prime }(e_{1}))$ and $\pi _{2}(\Phi ^{\prime
}(e_{2})) $ take their values on the same maximal torus $\TT \subset G$.

It will be called (torus-)reducible modulo $m \in \N ^{*}$ iff
$\Lambda _{m} . \Phi $ is (torus-)reducible.
\end{definition}

Given an action $\Phi $ such that $\pi _{1}(\Phi (e_{1}))=1$, the $n$-th
renormalized action can now be defined as%
\begin{equation*}
\mathcal{R}^{n}\Phi =M_{\beta _{n-1}}\circ (Q_{n})_{\ast }\Phi
\end{equation*}%
where, following the notations introduced in section \textit{1.3},
\begin{equation*}
Q_{n}=%
\begin{pmatrix}
p_{n} & p_{n-1} \\ 
q_{n} & q_{n-1}%
\end{pmatrix}%
\end{equation*}%
is in $GL( 2,\Z)$ and the $Q_{n}$ satisfy the recursive relation
\begin{equation*}
Q_{n}=Q_{n-1}%
\begin{pmatrix}
a_{n} & 1 \\ 
1 & 0%
\end{pmatrix}%
=Q_{n-1}A_{n}
\end{equation*}%
so that%
\begin{equation*}
\mathcal{\tilde{R}}^{n}\Phi =M_{\a _{n-1}}\circ (A_{n})_{\ast }\mathcal{R%
}^{n-1}\Phi
\end{equation*}%
For the action $\Phi $ associated to the cocycle $(\a ,A(\. ))$,
\begin{eqnarray*}
\mathcal{\tilde{R}}^{n}\Phi (e_{2}) &=&(\a _{n},A_{(-1)^{n}q_{n}}(\beta
_{n-1}\. ))=(\a _{n},A^{(n)}(\beta _{n-1}\. ))=(\a _{n},\tilde{%
A}^{(n)}(\. )) \\
\mathcal{\tilde{R}}^{n}\Phi (e_{1}) &=&(1,A_{(-1)^{n}q_{n-1}}(\beta
_{n-1}\. ))=(1,A^{(n-1)}(\beta _{n-1}\. ))=(1,\tilde{C}^{(n)}(\. ))
\\
(\beta _{n},A^{(n)}(\. )) &=&(\beta _{n-1},A^{(n-1)}(\.
))^{-a_{n}}\circ (\beta _{n-2},A^{(n-2)}(\. ))
\end{eqnarray*}
We will also use the notations%
\begin{eqnarray*}
\mathcal{\tilde{R}}_{\nu }^{n}\Phi &=&T_{-\nu }\mathcal{\tilde{R}}%
^{n}(T_{\nu }\Phi ) \\
\mathcal{R}_{\nu }^{n}\Phi &=&T_{-\nu }(Q_{n})_{\ast }(T_{\nu }\Phi
)
\end{eqnarray*}%
and omit the subscript when $\nu =0$.

The actions $\mathcal{\tilde{R}}_{\nu }^{n}\Phi $ can be normalized in a
non canonical way: there exists $B(\. )=B_{n,\nu }(\. )\in
C^{\infty }(\R,G)$ such that $Conj_{B(\. )}\mathcal{\tilde{R}}_{\nu }^{n}\Phi $ is
normalized.
\begin{definition} \label{def renormalization representative}
For an action $\Phi $ associated to the cocycle $(\a 
,A(\. ))$, we will call such a $Conj_{B(\. )}\mathcal{\tilde{R}}_{\nu }^{n}\Phi (e_{2})$ a
\textit{renormalization representative} of $(\a ,A(\. ))$.
\end{definition}
The choice of a renormalization representative is non-canonical, since it depends on the choice of $B(\. )$, but since different
renormalization representatives are cocycles conjugate to each other (cf. proof of lemma \ref%
{invariance of degree under conjugation of actions}), the indeterminacy is
not important for the results that we seek to obtain. However, we will see that when renormalization goes deep enough,
so that $\ba (\. )$ is approximated well enough, it indicates a natural way of normalizing the actions.

Similarly, we call a renormalization representative $\mod m$ the cocycle 
$Conj_{B(\. )} \Lambda _{m} .\mathcal{\tilde{R}}_{\nu }^{n}\Phi (e_{2})$, when this last action
is normalized $\mod m$.

From the above, it follows that

\begin{proposition}
The cocycle $(\a ,A(\. ))$ is (torus-)reducible (resp. accumulated by (torus-)reducible cocycles) if there exist
$n$ and $\nu $ such that $\mathcal{R} _{\nu }^{n}\Phi $ is (torus-)reducible (resp. accumulated by
(torus-)reducible cocycles), where $\Phi $ is the action associated to $(\a ,A(\. ))$.
\end{proposition}

We remark that the action associated to a cocycle is (torus-)reducible $\mod m$, iff the $m$-th iterate
of the cocycle is (torus-)reducible. We also need to make clear that renormalization does not commute
with the action of $\N ^{*} $. This is because the continued fractions of $m\gamma$ and those of $\gamma $ are
not related. They become related if we place $m \gamma $ in $m \T$, which we will do in the proof of our
main theorem.

Using lemmas \ref{growth lemma} and \ref{higher order}, we can prove the a priori estimates

\begin{proposition} \label{a priori estimates renormalization}
For the action $\Phi $ associated to the cocycle $(\a ,A(\. ))$, we
have%
\begin{eqnarray*}
\left\Vert \mathcal{\tilde{R}}^{n}_{\nu}\Phi \right\Vert _{\s  ,\R}^{\max } &\lesssim &K_{\s  } \\
\left\Vert \partial ^{\s  }\tilde{c}^{(n)}(\. )\right\Vert _{L^{2}(I)}
&\ra &0 \\
\left\Vert \partial ^{\s  }\tilde{a}^{(n)}(\. )\right\Vert _{L^{2}(I)}
&\ra &0
\end{eqnarray*}%
for all $1\leq \s  \leq s-1$ and any compact interval $I$ and for any $\nu \in \T $.
\end{proposition}

\begin{proof}
As we have seen, for any $n\in \Z ^{\ast }$, the derivative of $A_{n}(\. )$ satisfies
\begin{equation*}
\| a_{n}(\. )\| _{0}\leq n\left\Vert a(\.
)\right\Vert _{0}
\end{equation*}%
In particular, $\| a^{(n)}(\. )\| _{0} = \| LA^{(n)}(\. )\| _{0}\leq
q_{n}\left\Vert a(\. )\right\Vert _{0}$, and since%
\begin{equation*}
\tilde{a}^{(n)}(\. )=\beta _{n-1}a^{(n)}(\beta _{n-1}\. )
\end{equation*}%
and $q_{n}\beta _{n-1}<1$, the result is proved for the $C^{0}$ norms of the
first derivatives. Similarly,%
\begin{equation*}
\left\Vert \partial a_{n}(\. )\right\Vert _{0}\lesssim n\left\Vert
\partial a(\. )\right\Vert _{0}+n^{2}\left\Vert a(\. )\right\Vert
_{0}^{2}
\end{equation*}%
Since, now, $\partial \tilde{a}^{(n)}(\. )=\beta _{n-1}^{2}\partial
a^{(n)}(\beta _{n-1}\. )$, the estimate follows as before.

Finally, the convergence to $0$ in $L^{2}(I)$ for $\partial ^{\s  }\tilde{%
a}^{(n)}(\. )$, $\s  \geq 1$, follows from the fact that%
\begin{eqnarray*}
\partial \tilde{a}^{(n)}(\. ) &=&\beta _{n-1}^{2}\partial a^{(n)}(\beta
_{n-1}\. ) \\
&=&q_{n}^{2}\beta _{n-1}^{2}\partial \ba ^{(n)}(\beta _{n-1}\. )
\end{eqnarray*}%
and $\partial \ba ^{(n)}(\. )\ra 0$ in $L^{2}(\T )$,
so that $\partial \ba ^{(n)}(\beta _{n-1}\. )$ in $L^{2}(I)$, for
any compact interval $I$.

The estimates on $\tilde{c}^{(n)}(\. )$ and its derivatives are obtained
in exactly the same way, and we just remark that the numerical factor that
gains convergence for $\tilde{c}^{(n)}(\. )$ (as does $q_{n}\beta _{n-1}$
for $\tilde{a}^{(n)}(\. )$) is now $q_{n-1}\beta _{n-1}=\a 
_{n-1}q_{n-1}\beta _{n-2}$, which is in fact more than what is needed in
order to establish convergence. The factor $\a _{n-1}$ will be
interpreted as a scaling factor in the proof of the convergence of the
scheme.
\end{proof}

\subsubsection*{A remark}
A. Avila's theory of $SL(2,\C )$ cocycles (cf. \cite{Av09a}), gives
immediately the following quantization theorem for analytic cocycles
in $\T \times SU(2)$. Before stating it, we remind that $SU(2)$
complexifies to $SL(2,\C )$ and remark that, if $A(\. )\in
C_{\delta }^{\omega }(\T ,SU(2))$, i.e. if $A(\. )$ is analytic
and admits a holomorphic extension to the band $| \Re(z)|<\delta $, then%
\begin{equation*}
\frac{d}{d\e  }\left\Vert A(x_{0}+\e  i)\right\Vert
|_{\e  =0^{+}}=|\pi _{\ft }(a(x_{0}))|
\end{equation*}%
where $\ft $ is a maximal torus \footnote{In fact in the case where $A(x_{0}) = \pm Id$ one should
consider the maximum of such projections on all maximal tori, which is equal to the norm of the derivative}
passing by $A(x_{0})$ and $\left\Vert \. \right\Vert $ stands for the operator norm associated to
the Hermitian norm on $\C ^{2}$. This relation holds for the
following simple reason. If we assume that%
\begin{equation*}
A(0)=\{A_{\ft }(0),0\}
\end{equation*}
is diagonal and $A_{\ft }(x)=\ exp(is_{0}+is_{1}x+O(x^{2}))$, then
the increment of the norm as we pass to the complex plane is given by the
increment of the real part of $A_{\ft }(0+i\e  )$ which is
precisely $s_{1}$. The general case follows.

\begin{theorem}
The energy of analytic cocycles in $\T \times SU(2)$ is quantized in $2\pi \N $.
\end{theorem}

\begin{proof}
Suppose that the curve $\ba (\. )$ is non-degenerate, so that the
cocycle is measurably diagonalizable. Moreover, for $a.e.$ point $x_{0}$,%
\begin{equation*}
\frac{1}{q_{n}}\frac{d}{d\e  }\left\Vert A_{q_{n}}(x_{0}+\e 
i)\right\Vert |_{\e  =0^{+}}\underset{n\ra \infty }{\ra }%
|\ba (x_{0}))|
\end{equation*}%
This fact implies that the acceleration of a cocycle, as defined in \cite%
{Av09a}, is equal to its energy as has been defined herein. In the same
paper by A. Avila, it is shown that this quantity is quantized in $2\pi 
\mathbb{N}$.
\end{proof}

As we will see later on, this fact persists in more general contexts and in
lower regularity.

\section{Energy of actions}

Using lemma \ref{normalization lemma}, we can define the energy of a $\Z^{2}$ action, in a way that generalizes the definition given for cocycles:

\begin{definition}
The energy of a normalized action $\Phi $ equals the energy of $\Phi (e_{2})$%
, i.e.%
\begin{equation*}
\emph{en}(\Phi )=\emph{en}(\Phi (e_{2}))
\end{equation*}%
If $M_{\lambda }\Phi $ is normalized, then%
\begin{equation*}
\emph{en}(\Phi )=\lambda ^{-1}.\emph{en}(M_{\lambda }\Phi (e_{2}))
\end{equation*}
\end{definition}

The pertinence of the definition is established by the following lemma.

\begin{lemma}
\label{invariance of degree under conjugation of actions}The definition of
the energy does not depend on the choice of the conjugacy that normalizes
the action, provided that it is $C^{1}$.
\end{lemma}

\begin{proof}
Let $B_{i}(\. )\in C^{1}(\R,G)$, $i=1,2$ be such that $\Phi _{i}^{\prime }=Conj_{B_{i}(\. )}\Phi $ be normalized.
Then, if we let $B(\. )=B_{1}(\. ).B_{2}^{\ast }(\.) $, it satisfies $Conj_{B(\. )}(1,Id)=(1,Id)$, and therefore
it is $1$-periodic. Since $\Phi _{1}^{\prime }=Conj_{B(\. )}\Phi _{2}^{\prime }$,
which implies that $\Phi _{1}^{\prime }(e_{2})=Conj_{B(\. )}\Phi
_{2}^{\prime }(e_{2})$, and the energy of a cocycle is invariant under the
action of $C^{1}$ periodic conjugations (by proposition \ref{properties of energy}), the result has been proved.
\end{proof}

The following corollary is immediate.

\begin{corollary}
The energy of an action is invariant under $C^{1}$ conjugation of actions.
\end{corollary}

The energy of an action is related with the renormalization scheme in the
following way (notice that we need only consider renormalization of
normalized actions):

\begin{lemma}
The sequences $d_{n}^{i}$ given by%
\begin{equation*}
d_{n}^{i}=\beta _{n}\| a^{(n-1)}(\. )\| _{L^{i}(\T )}+\beta _{n-1}\| a^{(n)}(\. )\| _{L^{i}(\T )}
\end{equation*}%
for $i=1,2$ are decreasing and converge to $\emph{en}(\Phi )$. As a
consequence, the energy of an action is invariant under renormalization:%
\begin{equation*}
\emph{en}(\Phi )=\emph{en}(\mathcal{\tilde{R}}_{\nu }^{n}\Phi )
\end{equation*}%
for all $n\in \N^{\ast }$ and $\nu \in \T $. Moreover, the energy of a cocycle is
equal to the energy of any of its renormalization representatives.
\end{lemma}

\begin{proof}
By virtue of the recurrence relation satisfied by $\mathcal{R}^{n}\Phi $ and
the triangle inequality, we find that%
\begin{equation*}
\| a^{(n)}(\. )\| _{L^{i}}\leq \| a^{(n-2)}(\. )\| _{L^{i}}+a_{n}\| a^{(n-1)}(\. )\| _{L^{i}}
\end{equation*}%
Since $\beta _{n-2}=a_{n}\beta _{n-1}+\beta _{n}$, we find%
\begin{equation*}
\beta _{n-1}\| a^{(n)}(\. )\| _{L^{i}}+\beta _{n}\| a^{(n-1)}(\. )\| _{L^{i}}\leq \beta _{n-1}\| a^{(n-2)}(\. )\| _{L^{i}}
+\beta _{n-2}\| a^{(n-1)}(\. )\| _{L^{i}}
\end{equation*}%
and the sequences $d_{n}^{i}$ are decreasing. Since $\frac{1}{q_{n}}%
\| a^{(n)}(\. )\| _{L^{i}}\ra \emph{en}(\Phi )$ for $i=1,2$, and since $\beta _{n-1}q_{n}+\beta _{n}q_{n-1}=1$, we
immediately obtain the proof of the first part of the lemma.

The second part follows from the first and the invariance of the energy under smooth enough conjugation.
\end{proof}

This quantity was introduced in \cite{Krik2001} and was used in the proof of
the convergence of the renormalization scheme. The observation that this
quantity equals the degree is due to K.\ Fraczek. Another important quantity
used in the calculation of the energy of an action is its original
definition by R. Krikorian as the length functional $J^{1}$ defined as
follows. Let us consider the sequence%
\begin{eqnarray*}
J_{n}^{1}(\nu ) &=&\int_{\nu }^{\nu +\beta _{n-1}}|a^{(n)}(\.
)|+\int_{\nu }^{\nu +\beta _{n}}|a^{(n-1)}(\. )| \\
&=&\int_{0}^{1}|\tilde{a}_{\nu }^{(n)}(\. )|+\int_{0}^{\a _{n}}|%
\tilde{a}_{\nu }^{(n-1)}(\. )|
\end{eqnarray*}%
We can prove immediately

\begin{lemma}
For $\nu \in \T $ fixed, $(J_{n}^{1}(\nu ))_{n}$ is decreasing, and $%
(J_{n}^{1}(\nu ))_{n}$ converges uniformly to the energy of the action.
\end{lemma}

\begin{proof}
We calculate%
\begin{eqnarray*}
J_{n}^{1}(\nu ) &=&\int_{\nu }^{\nu +\beta _{n-1}}|a^{(n)}(\.
)|+\int_{\nu }^{\nu +\beta _{n}}|a^{(n-1)}(\. )| \\
&\leq &\int_{\nu }^{\nu +\beta _{n-1}}\sum_{1}^{a_{k}}|a^{(n-1)}(\.
+\beta _{n-2}-j\beta _{n-1})|+\int_{\nu }^{\nu +\beta
_{n-1}}|a^{(n-2)}(\. )|+\int_{\nu }^{\nu +\beta _{n}}\left\vert
a^{(n-1)}(\. )\right\vert \\
&\leq &\int_{\nu }^{\nu +\beta _{n-2}}|a^{(n-1)}(\. )|+\int_{\nu }^{\nu
+\beta _{n-1}}|a^{(n-2)}(\. )|=J_{n-1}^{1}(\nu )
\end{eqnarray*}%
which proves the first part.

We call
\begin{equation*}
J^{1}(\nu ) = \lim _{n \ra \infty} J_{n}^{1}(\nu )
\end{equation*}
which is defined for all $\nu \in \T$.

Let us now observe that $A^{\ast }(\.
+(-1)^{n}q_{n}\a ).A^{(n)}(\. )=A^{(n)}(\. -\a ).A^{\ast}(\. )$, so that%
\begin{equation*}
\left\vert |a^{(n)}(\. )|-|a^{(n)}(\. -\a )|\right\vert \leq \left\vert
a^{\ast }(\. )\right\vert +\left\vert a^{\ast }(\.
+(-1)^{n}q_{n}\a )\right\vert \leq 2\left\Vert a^{\ast }(\.
)\right\Vert _{L^{\infty }}
\end{equation*}
and we find%
\begin{eqnarray*}
|J_{n}^{1}(\nu )-J_{n}^{1}(\nu +\a )| &\leq &\left\vert \int_{\nu }^{\nu
+\beta _{n-1}}|a^{(n)}(\. )|-|a^{(n)}(\. -\a )|\right\vert +\left\vert
\int_{\nu }^{\nu +\beta _{n}}|a^{(n-1)}(\. )|-|a^{(n-1)}(\.
-\a )|\right\vert \\
&\leq &2\left\Vert a^{\ast }(\. )\right\Vert _{L^{\infty }}(\beta
_{n-1}+\beta _{n})
\end{eqnarray*}%
and, by the definition of $J^{1}(\nu )$, $J^{1}(\nu +\a )=J^{1}(\nu )$,
for all $\nu \in \T $. Since $J^{1}(\. )$ is defined as a
pointwise limit of a decreasing sequence of continuous functions, and the
invariance relation holds everywhere, $J^{1}(\. )=J^{1}$, a constant. We
now find that%
\begin{eqnarray*}
\| J_{n}^{1}(\. )\| _{L^{1}} &=&\int_{\T }J_{n}^{1}(\nu )d\nu \\
&=&\beta _{n-1}\| a^{(n)}(\. )\| _{L^{1}}+\beta _{n}\| a^{(n-1)}(\. )\| _{L^{1}}
\end{eqnarray*}%
by Fubini's theorem and the invariance of the Haar-Lebesgue measure on the
torus. The Lebesgue monotone convergence theorem and the previous lemma
conclude the proof of convergence of $J_{n}^{1}(\. )$ in $L^{1}$. Since $(J_{n}^{1}(\. ))$ is
a sequence of continuous functions on a compact space converging pointwise
to a continuous function, the convergence is in fact uniform.
\end{proof}

The proof of this lemma is essentially as in \cite{Fra2004}, except for this last observation, which is the reason
for the improvement of our results in terms of regularity. The functional $J^{1}(\. )$ is the quantity
used in the proof of the convergence of the renormalization scheme, where it appears as the length
of the path of $\mathcal{R}^{n}_{\nu} \Phi$ after the natural normalization, and by this observation we relate it
\textit{a priori} to the energy of the cocycle, whereas in K. Fraczek's article they were
related \textit{a posteriori}.

The factors in the definition of $J_{n}^{1}(\. )$ in this last lemma, are exactly
the respective lengths of the curves $\tilde{R}^{n}_{\nu}.e_{1} (\. )$ and $\tilde{R}^{n}_{\nu}.e_{2} (\. )$
for times in $[0,\a_{n}]$ and $[0,1]$ respectively. The factor $\a_{n}$ appears as the rescaling factor
which we encountered in the proof of \ref{a priori estimates renormalization}. For these reasons, the argument
$\nu$ is to be considered as the base point of renormalization, i.e. the point around which we can
rescale and obtain a good image of the global dynamics. The fact that the choice of the point
is irrelevant is justified by the integration in the proof, which connects the pointwise information of the
functional $J_{n}^{1}(\. )$ with the global (in $\T $) information of $d_{n}^{1}$.

\section{Convergence of the scheme}

As indicates the conclusion of the previous section, the sequence%
\begin{equation*}
u_{n}(\nu )=(-1)^{n}\int_{\nu }^{\nu +\beta _{n-1}}a^{(n)}(\.
)-(-1)^{n}\int_{\nu }^{\nu +\beta _{n}}a^{(n-1)}(\. )
\end{equation*}%
should be expected to converge to $\ba (\nu )$. The study of this fact, which in other words is the
invariance of the limit object of the dynamics under renormalization, is the object of
the section. We begin by the following proposition.
\begin{proposition}
\label{convergence}The sequence $(u_{n}(\. ))$ converges to $\ba %
(\. )$ in $L^{i}$, $i=1,2$.
\end{proposition}
\begin{proof}
Using the notation $\frac{1}{q_{n}}a^{(n)}(\. )=\ba ^{(n)}(\. )$%
, we have%
\begin{eqnarray*}
|\int_{\nu }^{\nu +\beta _{n-1}}a^{(n)}(\. )| &\leq &q_{n}\int_{\nu }^{\nu +\beta _{n-1}}|\ba ^{(n)}(\. )| \\
&\leq &q_{n}\beta _{n-1}\sup_{n}\| \ba ^{(n)}(\. )\| _{L^{\infty }}
\end{eqnarray*}%
which is finite, as has already been proved. Thus, $(u_{n}(\. ))$ is
uniformly bounded in $L^{\infty }$.

Using the fact that $q_{n}\beta _{n-1}+q_{n-1}\beta _{n}=1$, we find%
\begin{equation*}
u_{n}(\nu )-\ba (\nu )=(-1)^{n}\int_{\nu }^{\nu +\beta
_{n-1}}a^{(n)}(\. )-q_{n}\beta _{n-1}\ba (\nu )-(-1)^{n}\int_{\nu
}^{\nu +\beta _{n}}a^{(n-1)}(\. )-q_{n-1}\beta _{n}\ba (\nu )
\end{equation*}%
Let us assume for simplicity that $n$ is even. Then%
\begin{eqnarray*}
|\int_{\nu }^{\nu +\beta _{n-1}}a^{(n)}(y)dy-q_{n}\beta _{n-1}\ba (\nu
)| &\leq &q_{n}\int_{\nu }^{\nu +\beta _{n-1}}|\ba ^{(n)}(y)-\ba %
(\nu )|dy \\
&\leq &q_{n}\int_{\nu }^{\nu +\beta _{n-1}}|\ba ^{(n)}(y)-\ba %
^{(n)}(\nu )|+|\ba ^{(n)}(\nu )-\ba (\nu )|dy \\
&\leq &\frac{1}{\beta _{n-1}}\int_{\nu }^{\nu +\beta _{n-1}}|\ba %
^{(n)}(y)-\ba ^{(n)}(\nu )|+|\ba ^{(n)}(\nu )-\ba (\nu )|
\end{eqnarray*}%
and, since $\ba ^{(n)}(\. )\ra \ba (\. )$ a.e., we
need only to prove the same thing for the first term. This done in lemma \ref%
{egorov} in Appendix \ref{appendix 1}.

The other term can be estimated in the same way in both arguments, and
likewise the case of an odd $n$ in the second one and therefore, $%
(u_{n}(\. ))$ converges almost everywhere to $\ba(\. )$. The
bounded convergence theorem concludes the proof of convergence in $L^{1}$,
and the uniform boundedness of the sequence grants automatically convergence in $L^{2}$.
\end{proof}

\begin{corollary}
For $a.e.$ $\nu \in \T $,%
\begin{equation*}
\int_{\nu }^{\nu +\beta _{n-1}}a^{(n)}(\. )-(-1)^{n}q_{n}\beta _{n-1} \ba(\nu )\ra 0
\end{equation*}%
and the set where there is convergence, denoted by $\mathcal{E}$, is invariant by $R_{\a }$.
\end{corollary}

Let us point out that for all $\nu \in \T$,
\begin{eqnarray*}
|u_{n}(\nu )| &=&|\int_{\nu }^{\nu +\beta _{n-1}}a^{(n)}(\. )-\int_{\nu
}^{\nu +\beta _{n}}a^{(n-1)}(\. )| \\
&\leq &\int_{\nu }^{\nu +\beta _{n-1}}|a^{(n)}(\. )|+\int_{\nu }^{\nu
+\beta _{n}}|a^{(n-1)}(\. )| \\
&=&J_{n}^{1}(\nu )
\end{eqnarray*}%
This inequality implies directly that%
\begin{equation*}
\lim \sup |u_{n}(\nu )|  \leq \emph{en}(\a ,A(\. )),~\forall ~\nu \in 
\T 
\end{equation*}%
and, if the energy of the cocycle is $0$, $u_{n}(\nu )$ converges to $0$ in $%
C^{0}$ topology.

\bigskip

We now use the commutation of the generators and their asymptotic commutation with $\ba (\. )$.
We begin by remarking that in the general case where $\gamma (\. )$
is a curve defined locally around $0$ and taking values in $g$, the integral of the ODE
\begin{eqnarray*}
L\Gamma (\. ) &=&\gamma (\. ) \\
\Gamma (0) &=&\Gamma _{0}
\end{eqnarray*}%
does not satisfy%
\begin{equation*}
\Gamma (\nu )=\exp (\int_{0}^{\nu }\gamma (\. )).\Gamma _{0}
\end{equation*}%
unless $\int_{0}^{x}\gamma (\. )$, or $\Gamma (x)$, and $\gamma (x)$
commute for all $x\in \lbrack 0,\nu ]$. Such a relation is verified up to an
error in $L^{1}$ by $\ba ^{(n)}(\. )$ and $A^{(n)}(\. )$:%
\begin{equation} \label{assypmtotic commutation}
Ad(A^{(n)}(\. )).\ba ^{(n)}(\. )=\ba ^{(n)}(\. )+\e _{n, \. }
\end{equation}%
where $\e _{n, \. }\in L^{1}$ is of the order of $\| \ba ^{(n)}(\. )-(-1)^{n}\ba (\. )\|
_{L^{1}}+\| \ba (\. +\beta _{n})-\ba (\. )\| _{L^{1}}$. This fact, along with the continuous dependence of
solutions of ODE's and the fact that $a_{n}(\. ) = LA_{n}(\. )$, implies that
\begin{eqnarray*}
\tilde{A}^{(n)}(\nu +x) &=&\tilde{A}^{(n)}(\nu ).\exp (\int_{\nu }^{\nu
+\beta _{n-1}x}a^{(n)}(\. )+O(\e  _{n,\nu }(x ))) \\
\tilde{C}^{(n)}(\nu +x) &=&\tilde{C}^{(n)}(\nu ).\exp (\int_{\nu }^{\nu
+\beta _{n-1}x}a^{(n-1)}(\. )+O(\e  _{n,\nu }(x )))
\end{eqnarray*}%
for $x\in \lbrack 0,1]$. In particular, 
\begin{subequations}
\label{limit}
\begin{eqnarray}
\tilde{A}^{(n)}(\nu +1) &=&\tilde{A}^{(n)}(\nu ).\exp (\int_{\nu }^{\nu
+\beta _{n-1}}a^{(n)}(\. )+O(\e  _{n,\nu }(1 ))) \\
\tilde{C}^{(n)}(\nu +\a _{n}) &=&\tilde{C}^{(n)}(\nu ).\exp (\int_{\nu
}^{\nu +\beta _{n}}a^{(n-1)}(\. )+O(\e  _{n,\nu }(\a _{n })))
\end{eqnarray}
\end{subequations}
and, using the preceding calculations (proposition \ref{convergence} and its corollary), we find 
\begin{eqnarray*}
\tilde{A}^{(n)}(\nu +1) &=&\tilde{A}^{(n)}(\nu ).\exp ((-1)^{n}q_{n}\beta
_{n-1}\ba (\nu )+O(\e  _{n,\nu }(1 ))) \\
\tilde{C}^{(n)}(\nu +\a _{n}) &=&\tilde{C}^{(n)}(\nu ).\exp
(-(-1)^{n}q_{n-1}\beta _{n}\ba (\nu )+O(\e  _{n,\nu }(\a _{n} )))
\end{eqnarray*}
These formulas can be written as
\begin{eqnarray*}
\tilde{A}^{(n)}(\nu +1) &=&\tilde{A}^{(n)}(\nu
).\exp (K_{n}(\nu )+O(\e  _{n}(\nu ))) \\
\tilde{C}^{(n)}(\nu +\a _{n}) &=&
\tilde{C}^{(n)}(\nu ).\exp (L_{n}(\nu )+O(\e  _{n}(\nu )))
\end{eqnarray*}
where we call
\begin{eqnarray*}
K_{n}(\nu ) &=&(-1)^{n}q_{n}\beta _{n-1}\ba(\nu ) \\
L_{n}(\nu ) &=& -(-1)^{n}q_{n-1}\beta _{n}\ba (\nu )
\end{eqnarray*}

If the cocycle is smoother, we can sharpen the convergence result using the
following lemma

\begin{lemma}
Let $(\a ,A(\. ))\in SW^{s}(\T ,G)$, $2\leq s\leq \infty $. If $\emph{en}(\a ,A(\. ))>0$, there exists a full measure
set $\mathcal{E}_{s}\subset \T $ such that for each $\nu \in \mathcal{E}_{s}$,
\begin{eqnarray*}
\tilde{a}^{(n)}(\. )-(-1)^{n}q_{n-1}\beta _{n}\ba (\nu )
&\ra &0\text{ uniformly in }[\nu -1,\nu +1] \\
|\ba (\nu )| &=&\emph{en}(\a ,A(\. )) \\
\partial ^{\s  }\tilde{a}^{(n)}(\. ) &\ra &0\text{ uniformly
in }[\nu -1,\nu +1]\text{, }1\leq \s  \leq s-2 \\
\frac{1}{\beta _{n-1}}\int_{\nu -1}^{\nu +1}|\partial ^{s-1}\tilde{a}%
^{(n)}(\. )| &\ra &0
\end{eqnarray*}%
We stress that there is no uniformity in $\nu $ or $\s  $, but only in
each interval $[\nu -1,\nu +1]$. The sets $\mathcal{E}_{s}$ are invariant
under $R_{\a }$.

If $\emph{en}(\a ,A(\. ))=0$, we have%
\begin{equation*}
\partial ^{\s  }\tilde{a}^{(n)}(\. )\ra 0\text{ uniformly in }%
[\nu -1,\nu +1]\text{, }0\leq \s  \leq s-2
\end{equation*}
for all $\nu \in \T $.
\end{lemma}

\begin{proof}
We can apply lemma \ref{egorov} of appendix \ref{appendix 1} to $\frac{1}{%
q_{n}^{s}}\partial ^{s-1}a^{(n)}(\. )$, which has been proved to converge
to $0$ in $L^{2}$ and obtain a full measure set of in $\T $ for which%
\begin{equation*}
\frac{1}{\beta _{n-1}q_{n}^{s}}\int_{\nu -\beta _{n-1}}^{\nu +\beta
_{n-1}}|\partial ^{s-1}a^{(n)}(\. )|\ra 0
\end{equation*}%
holds. If $\E_{s}$ is the intersection of such sets for $0 \leq \s  \leq s $ with the set where $|%
\ba (\. )|=\emph{en}(\a ,A(\. ))$, it is of full measure and
the rest of the properties follow from the convergence results already
proved. The invariance of the sets $\mathcal{E}_{s}$ follows from the
$\mathcal{U}$-invariance of the limit curve $\ba (\. )$.
\end{proof}

Under these assumptions, the mismatch function $\e  _{n, \nu }(\. )$ in eq. (\ref{limit}) is in
fact $C^{s-1}$ in the dynamic variable $x$ for every fixed $\nu \in \T $, and tends to $0$ in $H^{s-1}$,
with the same local convergence properties as in the previous lemma.

We can now prove the quantization theorem:

\begin{proof} [Proof of theorem \ref{quant deg}] \label{proof quant def}
Let us suppose for the moment that $0\in \E_{s}$, for simplicity in notation.
For $n$ even and big enough,
\begin{equation*}
\tilde{C}^{(n)}(x)=\tilde{C}^{(n)} (0) .\exp (L_{n}(0).(x)+O(\e _{n, 0 } (x)))
\end{equation*}
uniformly around $0$, with $\tilde{C}^{(n)} (0) \in \mathcal{Z}_{0}=Z_{G}(\exp (\R\ba (0)))$
up to a small error, i.e. $Ad(\tilde{C}_{0}^{(n)}).\ba (0)=\ba (0)+ O(\e_{n})$
and the same holds for $\tilde{A}^{(n)}(x)$. The fact that the generators of the renormalized action commute implies that
\begin{equation*}
\exp (\ba (0)+O(\e _{n}))) =\tilde{C}^{(n)} (0) \tilde{A}^{(n)} (0) (\tilde{C}^{(n)} (0))^{*} (\tilde{A}^{(n)} (0))^{*}
\end{equation*}
where we have used that $q_{n} \beta _{n-1} + q_{n-1} \beta _{n} =1 $. Since the commutator of two constants in
$\mathcal{Z}_{0} = G_{0} \times G_{+} $ is in $G_{0}$, and since $\exp (\ba (0)) \in G_{+}$, we immediately obtain that
\begin{equation*}
\exp (\ba (0) ) \in G_{0} \cap G_{+} = \{ Id \}
\end{equation*}
Since the choice of the base-point of renormalization is not significant, we have proved
theorem \ref{quant deg}.
\end{proof}
Let us now proceed to the proof of theorems \ref{Quantization of degree}
and \ref{Renormalization 0 energy}, and their analogues for cocycles non-homotopic to constants.
\begin{proof}[Proof of Theorems \ref{Quantization of degree} and \ref{Quantization of degree non-hom}]
The conjugation
\begin{equation*}
\tilde{D}_{n}(x)=\exp (-L_{n} (0) \frac{x(x-1)}{2})
\end{equation*}%
taking values in $\exp (\R \ba (0))$ reduces the action $\mathcal{\tilde{R}}^{n}\Phi $
to the one generated by
\begin{eqnarray*}
&&(1,\tilde{C}^{(n)} (0).\exp (O(\e _{n}))) \\
&&(\a _{n},\tilde{A}^{(n)} (0).\exp (\ba (0)(\. )+O(\e_{n})))
\end{eqnarray*}
where again $\tilde{A}^{(n)} (0) \in \mathcal{Z}_{0}$ is a different constant
and we have used the fact that $L_{n}(0)$ and $K_{n}(0)$ are collinear with $\ba(0)$,
as well as the commutation relation \ref{assypmtotic commutation}. Let us recall the
splitting of $\mathcal{Z} _{0} = G_{0} \times G_{+} \hra G$ (cf. section
\ref{dynamics of torus-reducible cocycles}) which corresponds to $\ba (0)$.
We can decompose, up to a small error, $\tilde{C}^{(n)} (0)$ into $C_{+}.C_{0}$, where $C_{0}\in G_{0}$
and $C_{+}\in G_{+}$. Then, if $H_{+}\in g_{+}$ is a preimage of $C_{+}$,
the conjugant $\exp (-H_{+}\. )$ reduces the action to%
\begin{eqnarray}
&&(1,C_{0}.\exp (O(\e _{n}))) \\
&&(\a _{n},\tilde{A}^{(n)} (0) .\exp (\ba (0)(\. )+O(\e_{n})))
\end{eqnarray}
with $\tilde{A}^{(n)} (0)$ replaced by $\tilde{A}^{(n)} (0) \exp (-H_{+}\a _{n})$.
We remark that the first generator is a perturbation of a constant cocycle in $\T \times G_{0}$.
Let us also split $\tilde{A}^{(n)}(0)$ into $A_{+}.A_{0}$, as we have done
with $\tilde{C}_{0}^{(n)}$.

In order to normalize the action, we have to reduce the constant $C_{0}$ to
the $Id$, and the natural conjugant is $\exp (-H_{0}\. )$, where $H_{0}\in g_{0}$
is a preimage of $C_{0}$. If $A_{0}$ and $C_{0}$ are on a common maximal torus in $G_{0}$,
this conjugation conjugates the action to
the one generated by
\begin{eqnarray}
&&(1,\exp (O(\e _{n}))) \\
&&(\a _{n},\exp (-H_{0}\a _{n} )\tilde{A}^{(n)} (0) .\exp (\ba (0)(\. )+O(\e_{n})))  \label{normalized}
\end{eqnarray}
and a close-to-the $Id$ can normalize this action to an action associated to a cocycle of the same form.

If, however, $C_{0}$ and $A_{0}$, the proof of lemma \ref{parties abel. simpl. conn.} shows that we can
choose $H_{0}\in g_{0}$, a preimage of $C_{0}$, satisfying the following condition. The vector $H_{0}$
decomposes into $H_{1} + H_{2}$, where $Ad(A_{0}).H_{1}=H_{1}$ and $Ad(A_{0}).H_{2}=-H_{2}$, and
then $\exp (2H_{2})= Id$. We choose $H_{2}$ of minimal length, and the minimal $\x$,
$1 \leq \x \leq \x _{0} \leq \x _{G}$, such that $Ad(A^{\x}).H_{2} = H_{2}$. Then conjugation by
$\exp (-H_{0}\. )$ reduces the action to the one generated by
\begin{eqnarray}
&&(1,\exp (O(\e _{n}))) \\
&&(\a _{n},\exp (-H_{0}\a _{n} )A_{+} .\exp (\ba (0)(\. )). \exp (-H_{0}\. ).A_{0}.\exp (H_{0}\. )  \exp(O(\e_{n})))
\end{eqnarray}
As shows the proof of lemma \ref{parties abel. simpl. conn.}, $\exp (-H_{0}\. ).A_{0}.\exp (H_{0}\. )$
is a periodic geodesic in $G_{0}$, it is non homotopic to constants, and its order in the homotopy group
of $G_{0}$ is precisely $\x$. Consequently, if, as previously normalize this action by a close-to-the $Id$
conjugation reducing the first generator to $(1,Id)$, we obtain the first part of theorem
\ref{Quantization of degree non-hom}. The second part follows from the choice of $\x $ made above.

%
%
%
%
%

Translating this construction by $\nu $, composing with the algebraic conjugation
$D(\nu )$ that diagonalizes $\ba (\nu )$, and modifying accordingly
the notation, we can obtain the same result for any $\nu \in \E_{s}$. This proves theorems
\ref{Quantization of degree} and \ref{Quantization of degree non-hom} in their full generality.
\end{proof}

\begin{proof}[Proof of Theorems \ref{Renormalization 0 energy} and \ref{Renormalization 0 energy non-hom}]

Repeating the arguments in the $0$ energy case, and remembering that under
this assumption $J_{n}^{1}(\. )\ra 0$ uniformly proves theorem \ref{Renormalization 0 energy}.

\end{proof}

We close this section with the following remark.

\begin{remark}
We stress that if the cocycle is regular, the constant $\x _{G}$ is
irrelevant, as follows from the proof. In particular, if $G=SO(3)$, this
constant equals $2$, but all non-zero vectors in $g=so(3)\approx su(2)$ are
regular. We will examine the consequences of this phenomenon in section \ref%
{Density of zero-energy cocycles}.
\end{remark}

\section{Preservation of homotopy under renormalization}

From the above arguments we can draw the rather surprising conclusion that,
in the global as in the local theory, the greatest part of the analysis of
the dynamics of a cocycle is made in the isomorphism class of the Lie
algebra of the group. In particular, for a given cocycle $(\a ,A(\.
))\in SW^{s}(\T ,G)$, the definition of the energy and the invariant
curve do not depend directly on the group, but only on $a(\. )$ (which
however should be the derivative of a closed path $\T \ra G$)
and $Ad(A(\. ))\in C^{s}(\T ,Inn(g))$, which both remain the same
if we replace $G$ by $G/K$, where $K\subset Z_{G}$ and $(\a ,A(\. ))$
by $(\a ,\pi (A(\. )))\in SW^{s}(\T ,G/K)$, where $\pi
:G\ra G/K$ is the canonical projection. In particular, since the
functionals measuring the convergence of the renormalization are calculated
in the same way, they depend only in the same indirect way on the group.
These properties are not necessarily conserved when we pass to a covering of 
$G$, and we will examine how this affects the study of the dynamics.

The difference between the groups becomes important in the solution of the equation
\begin{equation*}
\exp (h)=Id
\end{equation*}
When we renormalize cocycles in $\T \times (G/K)$, a finer lattice in $\Z^{w}$ is admissible
than the one admissible for cocycles in $\T \times G$, and this is the case even for regular cocycles.
In the case of close-to-a-constant and singular cocycles, the constant $\x _{G}$ becomes important. Since
the constants $\x _{G}$ and $\x _{G/K}$ are different, we have
to iterate a different number of times in order to obtain the necessary
commutativity for the normalization of the renormalized action.

A manifestation of the first phenomenon is the quantization of $E_{r}(\. )$
in $SU(2)$, where the levels of energy allowed are in $2\pi \N$, whereas the quantization in $SO(3)$ is in $\pi \N$. The factor $1/2$ is clearly due to the cardinal of the center of $SU(2)$, equal to $\{\pm Id\}$,
which in this case is also the reason why $\x _{SU(2)}=1$, while $\x _{SO(3)}=2$.

\bigskip

It is nonetheless possible to renormalize a cocycle without losing track of
its homotopy class, but this seems to be obtained in a more natural, but
still not canonical, way by passing on to a group with simpler homotopy
\footnote{A similar way to the one we have chosen is to associate a symbol to the
action associated to the cocycle. The symbol codes the homotopy class of the
path in $G$. It is invariant by the all the actions used in renormalization,
except for the action of $\N^{\ast }$, which accounts for the fact that composition of a path with
itself simplifies its homotopy. However, we find this presentation
needlessly abstract, and we prefer the one given in the text.}.

Let, therefore, $(\a ,A(\. ))\in SW^{s}(\T ,G)$ and $\tilde{G}$ be the universal cover of $G$, so that
$G=\tilde{G}/K$, where $K\subset Z_{\tilde{G}}$. Then, $A(\. )$ admits a lift $\tilde{A}(\. )$ in $\tilde{G}$
which is at most $\#K$-periodic. There exists, however, $S\in K$ such that $\tilde{A}(\. +1)=S.\tilde{A}(\. )$.
The choice of $S$ depends on the chosen lift, but the
smallest integer $m$ such that $S^{m}=Id$ is in fact a characteristic of the
homotopy class of $A(\. )$. Therefore, even though the couple $((1,Id),(\a ,\tilde{A}(\. )))$
does not define a $\Z^{2}$- action on $\T \times \tilde{G}$, it satisfies the commutativity relation
\begin{equation}
(1,Id)\circ (\a ,\tilde{A}(\. ))=(0,S)\circ (\a ,\tilde{A}(\.
))\circ (1,Id)  \label{anticommutation}
\end{equation}

\bigskip

Let us firstly study cocycles satisfying this weaker commutation relation,
and let two cocycles $U,V\in SW^{s}(\T ,\tilde{G})$ commute modulo $S\in Z_{\tilde{G}}$:
\begin{equation*}
U\circ V=(0,S)\circ V\circ U
\end{equation*}%
These cocycles do not define a $\Z^{2}$ action, but we can always define a function%
\begin{equation*}
\begin{array}{cccc}
\Phi _{l}: & \Z^{2} & \ra & SW^{s}(\T ,\tilde{G}) \\ 
& (m,n) & \longmapsto & U^{m}\circ V^{n}%
\end{array}%
\end{equation*}%
or%
\begin{equation*}
\begin{array}{cccc}
\Phi _{r}: & \Z^{2} & \ra & SW^{s}(\T ,\tilde{G}) \\ 
& (m,n) & \longmapsto & V^{n}\circ U^{m}%
\end{array}%
\end{equation*}%
which is not a group homomorphism. This property can be regained if we
factorize $\tilde{G}$ modulo $K$, or if we restrict it to the submodule $(\#K.\Z)^{2}$. It is equally regained if we consider the adjoint action of $\tilde{G}$ on $g$ and the induced $\Z^{2}$ action.

Let us examine $\Phi _{l}$. A matrix in $GL(2,\Z)$ does not act on it in a canonical way. It can be nonetheless verified that the change of basis of $\Z^{2}$ by a matrix in $GL(2,\Z)$ preserves the relation modulo an inversion of $S$ which is not
significant, as it corresponds to the change of orientation of the path induced by renormalization. If
\begin{equation*}
\begin{pmatrix}
a & b \\ 
c & d%
\end{pmatrix}%
\in GL(2,\mathbb{Z})
\end{equation*}%
and $\tilde{U}=U^{a}\circ V^{c}$ and $\tilde{V}=U^{d}\circ V^{d}$, then%
\begin{equation*}
\tilde{U}\circ \tilde{V}=(0,S^{\pm 1})\circ \tilde{U}\circ \tilde{V}
\end{equation*}%
according to the sign of the determinant. Therefore, if $P\in GL(2,\Z)$, the action $\tilde{\Phi}=P_{\ast }\Phi _{l}$ is well defined and%
\begin{equation*}
\tilde{\Phi}_{l}(e_{1})\circ \tilde{\Phi}_{l}(e_{2})=(0,S^{\pm 1})\circ 
\tilde{\Phi}_{l}(e_{2})\circ \tilde{\Phi}_{l}(e_{1})
\end{equation*}%
so that the commutation relation is preserved, and the only reason for which
it is non-canonical is the original choice of $\Phi _{l}$ instead of $\Phi
_{r}$. Additionally, it can be verified directly that the renormalization
and normalization of the action $\tilde{\Phi}$ associated to the couple $%
((1,Id),(\a ,\tilde{A}(\. )))$ as if it were a $\Z^{2}$ action preserve this relation, in view of the fact that $S\in Z_{%
\tilde{G}}$:%
\begin{equation*}
\tilde{R}_{\nu }^{n}\tilde{\Phi}(e_{1})\circ \tilde{R}_{\nu }^{n}\tilde{\Phi}%
(e_{2})=(0,S^{(-1)^{n}})\circ \tilde{R}_{\nu }^{n}\tilde{\Phi}(e_{2})\circ 
\tilde{R}_{\nu }^{n}\tilde{\Phi}(e_{1})
\end{equation*}%
and this property persists after conjugation of actions. For example, if $%
(1,C(\. ))$ and $(\a ,A(\. ))$ satisfy%
\begin{equation*}
C(\. +\a ).A(\. )=S.A(\. +1).C(\. )
\end{equation*}%
and $B:\R \ra G$ normalizes $(1,C(\. ))$ to $(1,Id)$, then $\tilde{A}(\.
)=B(\. +\a )A(\. )B^{\ast }(\. )$ satisfies%
\begin{equation*}
\tilde{A}(\. +1)=S.\tilde{A}(\. )
\end{equation*}

Since $S$ codes the homotopy class of the cocycle, we find that the
renormalization in fact preserves homotopy, in contrast with simple
iteration of the cocycle.

Finally, since all the estimates concerning the measurement of the
convergence of the scheme depend only on the isomorphism class of $g$, this
renormalization of non-abelian actions has the same characteristics as the
renormalization of the action associated to $(\a ,A(\. ))\in SW^{s}(\T ,G)$.

\section{Dependence of the degree on the frequency \label{dependence of degree on the rotation}}

In this section we simplify and generalize an argument given in \cite%
{Fra2004}. It was used in the construction of two cocycles $(\a 
_{i},A(\. ))\in SW^{\infty }(\T ,SU(2))$, $i=1,2$, given by the
same mapping $A(\. ):\T \ra SU(2)$, but different
rotations in the basis satisfying $\a _{1}-\a _{2}\not=1/2$, and
such that $\emph{en}(\a _{1},A(\. ))\not=\emph{en}(\a _{2},A(\. ))$. Since any compact group contains
subgroups locally isomoprhic to $SU(2)$, we present the argument in the context of $SU(2)$,
and use the fact that it is of perturbative nature in order to obtain the same result in any compact group.

\bigskip

From the results so far, it is clear that if $\left\Vert a(\.
)\right\Vert _{L^{1}}<2\pi $ then $\emph{en}(\a ,A(\. ))=0$,
independently of $\a $. Moreover, a simple calculation shows that, for
$E_{r}(\. )=\{e^{2i\pi r\. },0\}_{SU(2)}$, $r\in \N^{\ast }$, then
$\emph{en}(\a ,E_{r}(\. +\theta ))=2\pi r$, again independently of $\a $,
and of $\theta \in \R$. However, if $A(\. )=E_{r}(\. )A_{0}$, where $A_{0}$ is
a non-diagonal constant, then $\left\Vert a_{2}(\. )\right\Vert_{L^{1}}<4\pi r$,
so that $\emph{en}(\a _{1},A(\. ))<2\pi r$ (see \cite{Krik2001} or the following chapter).
Thus, our goal should be the construction of a mapping $A(\. ):\T \ra SU(2)$
such that $(\a _{1},A(\. ))$ is conjugate to $(\a _{1},E_{r}(\. ))$
whereas $(\a _{2},A(\. ))$ is conjugate to $(\a _{2},E_{r}(\. )A_{0})$, with $A_{0}$ as before.

Let us therefore consider $A(\. )=\Phi (\. +\a 
_{1}).E_{r}(\. ).\Phi (-\. )$, where $\Phi (\.
):\T \ra \mathcal{T}^{\prime }$ is a (non-trivial) homomorphism
($\mathcal{T}^{\prime }$ is a maximal torus, not necessarily the standard
one). By construction, $(\a _{1},A(\. ))\sim (\a _{1},E_{r}(\. ))$, and
therefore $\emph{en}(\a _{1},A(\. ))=2\pi r$. On the other hand,%
\begin{equation*}
Conj_{\Phi (\.)} (\a _{2},A(\. )) = \Phi (\. +\a _{2}).A_{0}.E_{r}(\. ).\Phi(-\. )
\end{equation*}%
where $A_{0}=\Phi (\a _{1}-\a _{2})$, so that $%
(\a _{2},A(\. ))\sim (\a _{2},A_{0}E_{r}(\. ))$, and $\emph{en}%
(\a _{2},A(\. ))<2\pi r$, provided that $A_{0}$ is non-diagonal. This
last property is satisfied, say, if $\Phi (\a _{1}-\a _{2})\not\in G_{+} \times G_{0}$.

In particular, we see that the property that the cocycle $(\a ,A(\.
)) $ be (torus-)reducible should be expected to depend both on the rotation and the
mapping defining the cocycle.

\bigskip

Following the remark in the beginning of the section, similar constructions can be
carried out with $SU(2)$ replaced by any compact Lie group giving

\begin{theorem}
Let $\a _{1},\a _{2}\in \T \backslash \Q$ be such for some torus homomorphism
$\Phi :c_{G} \T \ra G$, $\Phi (\a _{1}-\a _{2})\not\in Z_{G}$. Then, there exists a mapping
$A(\. ):\T \ra G$ such that $\emph{en}(\a _{1},A(\. ))\not=\emph{en}(\a _{2},A(\. ))$.
\end{theorem}

The condition on the $\a _{i}$ is satisfied if $\a _{1}-\a 
_{2}\not\in \frac{1}{c_{G}\psi _{G}} \Z$ where $c_{G}$ is the cardinal of $Z_{G}$ and $\psi _{G}$ is a constant in $\N^{\ast }$ depending only on the group $G$.

\section{Proof of measurable invariance of the degree (prop. \ref{thm measurable invariance})}

In this section, we prove proposition \ref{thm measurable invariance}. There is a similar
proof in \cite{Fra2004}, but for cocycles of regularity $C^{2}$.  The improvement of the
result in terms of regularity is granted by the improved results for the
convergence of renormalization. There is also the slight complication due to
the fact that for general compact groups the degree is a vector and not a number as in $SU(2)$.

\bigskip

Let us call $\Phi _{i}$, $i=1,2$ the actions associated to the cocycles $ (\a ,A_{i}(\. ))$.
If we use the fact that $\mathcal{\tilde{R}}_{\nu }^{n}\Phi _{2}=Conj_{\tilde{B}_{\nu }(\. )}
\mathcal{\tilde{R}}_{\nu }^{n} \Phi _{1}$, we find that
\begin{equation*}
\tilde{B}_{\nu }(\. +\a _{n})\tilde{A}_{\nu ,1}^{(n)}(\. +1)\tilde{C}
_{\nu ,1}^{(n)}(\. )^{\ast }\tilde{B}_{\nu }^{\ast }(\.)=\tilde{A}%
_{\nu ,2}^{(n)}(\. +1 )\tilde{C}_{\nu ,2}^{(n)}(\. )^{\ast },~a.e.
\end{equation*}%
We have seen (cf. formula \ref{limit} and those that follow it) that for almost every $\nu $ we have
$$A_{\nu ,1}^{(n)}(\. +1 )C_{\nu ,1}^{(n)}(\. )^{\ast }\ra A_{\nu ,1}^{(n)}\exp ((-1)^{n}%
\ba _{1}(\nu )\. )(C_{\nu ,1}^{(n)})^{\ast }$$
uniformly in $[-1,1]$, where $A_{\nu ,1}^{(n)}$ and $C_{\nu ,1}^{(n)}$ are constants, and
similarly for $\tilde{A}_{\nu ,2}^{(n)}(\. )\tilde{C}_{\nu
,2}^{(n)}(\. )^{\ast }$.

The set of $\nu $ such that both $\tilde{B}_{\nu }(x+\a _{n})$ and $%
\tilde{B}_{\nu }(x+1)$ tend to $\tilde{B}_{\nu }(0)$ for $a.e.$ $x\in
\lbrack -1,1]$ is of full measure. Therefore, $\ba _{1}(\nu )$ and $\ba _{2}(\nu )$
are $a.e.$ algebraically conjugate and thus%
\begin{equation*}
\deg ((\a ,A_{1}(\. )))=\deg ((\a ,A_{2}(\. )))
\end{equation*}

\selectlanguage{english}
\chapter{A priori estimates on perturbations of regular geodesics} \label{A priori estimates on perturbations of regular geodesics}

As has been shown in the previous chapter, the renormalization scheme converges, independently of
any arithmetical conditions, to a model of dynamics given by $(\a 
,E_{r}(\. ).A)$, where $r\in \Z ^{w}$ and $A\in G$ commutes with $%
E_{r}(\. )$. However, convergence is gained by the introduction of the scaling factors $\beta _{n-1} ^{-1}$,
since otherwise convergence is possible only in $L^{2}$, where lives $\ba (\. )$, the limit object of the dynamics.
This is actually the way lemma \ref{relation of actions with cocycles} becomes useful in  the study of the dynamics,
since it shows the limitations of renormalization. We cannot conclude the desired density properties, unless we
stop renormalization after a finite number of steps, and then take up the local study of the limit cocycles,
but viewed as cocycles and not as actions any more. Therefore, the local theory of geodesics is the next natural step in our
study.

The non-triviality of $G_{+}$ implies the existence of an obstruction to the reducibility of the cocycle
and is the core of the passage from local to global phenomena. In this chapter we begin the study
of the case where this obstruction is maximal. By maximal we mean that $%
e_{r}=L(E_{r}(\. ))$ is a regular vector in $g$, i.e. it belongs to a
unique maximal torus which we will call $\ft $. Since the property
of a cocycle being regular is invariant under conjugation and
renormalization, fact that follows directly from the invariance of the degree
under renormalization (cf. Proposition \ref{convergence}), we can study the dynamics of
perturbations of periodic geodesics whose derivative is a regular vector in $g$,
without loss of generality. Another characterization of regular cocycles is
the triviality of the zero-energy component $G_{0}$ for the splitting associated to
a positive energy cocycle.

The convergence of the renormalization and the invariance of the degree
under renormalization allow us to pose the problem in the following
perturbative setting. We can suppose that the cocycle $(\a ,A(\. ))$
is of degree $r$, with $|r|>0$, and a $C^{\infty }$ perturbation of a regular
periodic geodesic of degree $r$. In other words, we suppose that $A(\. )$
can be written in the form%
\begin{equation*}
E_{r}(\. ).A.\exp (U(\. ))
\end{equation*}%
where $e_{r}=L(E_{r}(\. ))$ and $A\in G$ commute (i.e. $Ad(A).e_{r}=e_{r}$%
), and $U(\. )$ is small. A smallness condition of the type $\| U(\.) \|_{H^{1}} \leq C$,
where $0<C<1$ is some fixed constant will in fact be sufficient. The fact that $e_{r}$ is regular implies that
$E_{r}(\. )$ and $A$ take values on the same maximal torus $\T$.
Then, the choice of $\tilde{\Delta}\subset \Delta_{+} $, a basis for the roots
of $g$, with respect to the maximal abelian algebra $\ft $, the Lie
algebra of $\mathcal{T}$ allows us to write $E_{r}(\. )$ in the form $%
\exp (\sum 2\pi r_{\r }h_{\r }\. )$ and $A$ in the form $\exp (\sum
2\pi a_{\r }h_{\r })$ and use the notation%
\begin{equation*}
E_{r,a}(\. )=E_{r}(\. ).A
\end{equation*}%
as in eq. \ref{def Er}. Therefore, the models of cocycles which we are interested in are written in the form
\begin{equation*}
(\a ,E_{r,a}(\. )\exp (U(\. )))
\end{equation*}%
with $U(\. )=\sum\nolimits_{\r \in \tilde{\Delta}}U_{\r }^{\mathfrak{%
t}}(\. )h_{\r }+\sum\nolimits_{\r \in \Delta _{+}}U_{\r }(\. 
)j_{\r }$, small enough in $C^{\infty }(\T ,g)$.

\bigskip

It turns out that some estimates on the perturbation can be deduced, under
the assumption that $(\a ,E_{r,a}(\. )\exp (U(\. )))$ is of degree 
$r$. It has already been proved (cf. Proposition \ref{properties of energy})
that $\left\Vert \ba _{n}(\. )\right\Vert _{L^{i}}\geq \emph{en}(\a ,A(\. ))$, for $i=1,2$ and for
all $n\in \mathbb{N}^{\ast }$. In particular, $\left\Vert a(\. 
)\right\Vert _{L^{1}}\geq \emph{en}(\a ,A(\. ))$, and this gives a
first estimate, which is however inadequate for the following reason. Let $%
(a,E_{r}(\. ).e^{\{0,z\}})$ be a constant perturbation of a periodic
geodesic in $SU(2)$. We find directly that $\| L(E_{r}(\. 
).e^{\{0,z\}}) \| _{L^{i}}=2\pi r$. However, it is known since \cite{Krik2001}
that such a cocycle, and even small enough perturbations in $C^{1}
$ (the smallness depending on $|z|$ and $\a ^{-1}$) can be of energy
strictly smaller than $2\pi r$. The reason for this is the non-commutativity
of $e^{\{0,z\}}$ and $E_{r}(\. )$, which results in the second iterate
being of energy strictly smaller than $4\pi r$. We point out that, since the invariant curve $\ba (\.)$
has been defined by simple iteration, and not renormalization, we can obtain the needed estimates
using the derivative of $(2\a , A_{2}(\. ))$, and not by estimating the functional $J^{(1)}$ on
$R^{1} \Phi$, where again $\Phi $ is the $\Z ^{2}$ action associated to the cocycle $(\a , A(\. ))$.
This is one of the drawbacks of the techniques used in \cite{Krik2001}, since it results in
an implicit dependence of the smallness condition on the frequency. The lack of uniformity results in
the set $\Sigma$ of theorem \ref{RK global density} being a strict subset of $RDC$.

Further iteration gives expressions that are too complicated for the
estimates, so we will restrict the estimations in those of the first two
iterates. They prove to be sufficient, since the influence of the first order terms is already important
after two iterations, and higher-order phenomena are outside the scope of reduction schemes.

\section{Notation}

As mentioned before, the perturbation%
\begin{equation*}
U(\. )=\sum\nolimits_{\r \in \tilde{\Delta}}U_{\r }^{\ft %
}(\. )h_{\r }+\sum\nolimits_{\r \in \Delta _{+}}U_{\r }(\. 
)j_{\r }
\end{equation*}%
is supposed to be small in a topology to be made precise. We will mostly use
is expression in Fourier coefficients%
\begin{equation*}
U(\. )=\sum\nolimits_{\r \in \tilde{\Delta}}\sum\nolimits_{k\in 
\Z }\hat{U}_{\r }^{\ft }(k)e^{2i\pi k\. }h_{\r 
}+\sum\nolimits_{\r \in \Delta _{+}}\sum\nolimits_{k\in \Z }\hat{U%
}_{\r }(k)e^{2i\pi k\. }j_{\r }
\end{equation*}%
so that%
\begin{equation*}
\partial U(\. )=\sum\nolimits_{\r \in \tilde{\Delta}%
}\sum\nolimits_{k\in \Z }2i\pi k\hat{U}_{\r }^{\ft %
}(k)e^{2i\pi k\. }h_{\r }+\sum\nolimits_{\r \in \Delta
_{+}}\sum\nolimits_{k\in \Z }2i\pi k\hat{U}_{\r }(k)e^{2i\pi
k\. }j_{\r }
\end{equation*}

Since $U(\. )$ is to be seen as a perturbation of $E_{r,a}(\. )$, it
is reasonable to impose that $\sum\nolimits_{\r \in \tilde{\Delta}}\hat{U}%
_{\r }^{\ft }(0)h_{\r }=0$, i.e. that there is no constant part
in the perturbation in the torus $\ft $. This can be obtained by
using the following fact. We can write%
\begin{equation*}
\exp (U(\. ))=\exp (\sum\nolimits_{\r \in \tilde{\Delta}}\hat{U}_{\r 
}^{\ft }(0)h_{\r }).\exp (\tilde{U}(\. )-\sum\nolimits_{\r 
\in \tilde{\Delta}}\hat{U}_{\r }^{\ft }(0)h_{\r })
\end{equation*}%
where $\tilde{U}(\. ) - \dot{T} U(\. )$ is of second order. The mapping%
\begin{equation*}
h\longmapsto \int \pi _{\ft }(\exp ^{-1}(e^{h}.e^{\tilde{U}(\. 
)}))
\end{equation*}%
is a local diffeomorphism for $\tilde{U}(\. )$ taking values in a
neighborhood of $0$ in $\ft $, which gives the result.

Finally if $E_{r,a}$ is a regular geodesic (see eq. \ref{def Er}) and $e_{r}=LE_{r,a}(\. )$,
for any $\r \in \Delta _{+} $ there exist $\tilde{r}_{\r }\in \Z ^{*}$ such that
\begin{equation*}
\lbrack e_{r},j_{\r }]=2i\pi \tilde{r}_{\r }.j_{\r }
\end{equation*}%
and $\left\Vert \tilde{r}\right\Vert \leq C.\emph{en} (\a , A(\. ))$, where the constant
depends only on the group $G$ and we call $\tilde{r} = (r_{\r })_{\r \in \Delta_{+}} \in \Z ^{q}$.

\section{Synopsis of the chapter}

In this chapter, we obtain an estimate similar to the one used in \cite{Krik2001} in the control of
the low frequencies of a perturbation of the normal form of degree $r$, under the assumption that
the perturbed cocycle is still of the same degree $r$.

In the context of this m\'{e}moire, the estimates are valid if the cocycle is a perturbation of a regular
geodesic. As can be seen if one follows the calculations, the result remains true if the normal form is
a singular geodesic coupled with a constant in $G_{0}$ and the perturbation is $0$ in $g_{0}$. Let us, however,
state the main result of the chapter in the first, more transparent context.
\begin{lemma} \label{A priori estimates}
There exists a positive constant $C$, depending only on $r$, such that if
$(\a ,E_{r,a}(\. ).e^{U(\. )})$ is of degree $r$, where $U(\. )$ is small enough in $H^{1}$, then it satisfies
\begin{equation*}
\| \Lambda _{\tilde{r}}^{0}U(\. )\| _{L^{2}}\leq
C\left\Vert (Id-\Lambda _{\tilde{r}})\partial U(\. )\right\Vert _{L^{2}}
\end{equation*}
\end{lemma}
The truncation operator $\Lambda _{\tilde{r}}^{0}$ is defined in def. \ref{definition of obstructions for regular geodesics}.
It is a truncation in low frequencies, corresponding to the phenomenon that we described in the plan of the proof (p. 
\pageref{plan a priori estimates}), only in a more complicated geometry. Firstly, we use the fact that interactions between
the different subalgebras isomporphic to $su(2)$ are of higher order, in order to break the estimates into a superposition
of estimates in simpler geometry. In this context, we can say that if in $(su(2))_{\r }$, the subalgebra corresponding to
the root $\r $, the perturbation is too eccentric (i.e. the non-abelian constant part is dominant), or if the perturbation
spins around $ h_{\r }$ too fast in the negative direction, then the perturbation in $(su(2))_{\r }$
contributes negatively to the energy of the cocycle. And this is exactly the description of the part of
$\Lambda _{\tilde{r}}^{0}$ concerning $(su(2))_{\r }$. If, now, the perturbed cocycle is assumed to be of degree $r$, then
the contributions of the different parts of the perturbation must sum up to a non-negative contribution, and this grants
the a priori estimate of the lemma.

Using these estimates, we can obtain a first theorem describing (with an inadequate precision)
the configuration of the conjugacy classes of periodic geodesics. Let us introduce some notation
before stating the theorem proved above. Let $%
(\a ,A(\. ))\in SW^{s}(\T ,G)$, $s\geq 1$, and let $1\leq \x \leq \x _{G}$ be the smallest positive
integer such that $A_{\x }(\. )$ is homotopic to constants.
Let also $\Phi $ be the $\Z ^{2}$ action associated to $(\a ,A(\. ))$. Then, we have
\begin{theorem} \label{density of 0 energy}
Let $G$ be a compact semi-simple Lie group and $(\a ,A(\. ))\in
SW^{s}(\T ,G)$, $s\geq 1$, and let $\x $ as above. Then, if $\emph{en} (\a, A(\. )) >0$,
$(\a ,A(\. ))^{\x }$ has renomalization representatives $\mod \x $ which are
arbitrarily close to cocycles of smaller energy. All cocycles of zero energy
in the homotopy class of $(\a ,A(\. ))$ satisfy the following
property: renormalization $\mod \x $ of the $\x $-lattice associated to such cocycles
converges to constant actions.
\end{theorem}
The proof of this theorem occupies section \ref{Density of zero-energy cocycles}. Since the proof
relies heavily on the phenomenon of non-commutativity exploited in section \ref{Estimation of the energy of the second iterate}, it is clear
that this property is due to the fact that $G$ is non-commutative.

This theorem is a density theorem for $0$-energy actions, and a density theorem for close-to constant normalized
actions, modulo a finite covering of the torus. However, it does not assert that reducible and almost-reducible
cocycles are dense in $SW^{\infty}(\T, G)$ modulo $\x _{G}$, for the following reason. Let $(\a ,A(\. ))$ be a cocycle
as in the theorem and $\Phi $ be the associated action. The theorem only asserts that for $a.e.$ $\nu \in \T$,
$\mathcal{\tilde{R}}_{\nu }^{(n)}\Phi $ converges to models
accumulated by lower energy cocycles. If we wish to obtain a density
theorem, we have to be able to normalize $\mathcal{\tilde{R}}_{\nu
}^{(n)}\Phi $ to $\Phi ^{\prime }$ and to prove that $\Phi ^{\prime }(e_{2})$
is accumulated by lower energy cocycles, and then apply lemma \ref{relation of actions with cocycles}.
This, however, is a stronger assertion.

The theorem, which holds for all irrational rotations, only allows us to
construct perturbations of $\mathcal{\tilde{R}}_{\nu }^{(n)}\Phi $ which are
of smaller energy, but cannot be made arbitrarily small unless we allow $n$ to grow. It allows us,
nonetheless, to conclude a fundamental difference between cocycles in $SW^{\infty}(\T, \T ^{w})$, which
is abelian and not semi-simple, and $SW^{\infty}(\T, G)$, with $G$ compact semisimple. We saw in
the introductory study of abelian cocycles (chapter 3) that these cocycles are classified by homotopy,
and therefore form discrete classes, and iteration of a path $\T \ra \T ^{w}$ non-homotopic to constants
will not make it homotopic to constants. On the other hand, when we study abelian cocycles in $G$,
and therefore consider $SW^{\infty}(\T, \T ^{w}) \hra SW^{\infty}(\T, G)$, with $w$ the rank of $G$,
we find that the limit object of the dynamics is quantized in the same lattice as in $SW^{\infty}(\T, \T ^{w})$,
but the classes are no longer discrete (eventually modulo iteration).

The situation is intermediate in compact groups which are not semi-simple. Let us consider, for example, the cocycle in
$\T \times U(2)$ defined by
\begin{equation*}
A(\. )=
\begin{bmatrix}
e^{2i \pi k \.} & 0 \\ 
0 & e^{2i \pi k \.}
\end{bmatrix}
\end{equation*}
over any irrational, with $k \in \Z ^{*}$. This cocycle admits a non-trivial limit curve in $u(2)$ consisting
of the sole vector 
\begin{equation*}
\begin{bmatrix}
2i \pi k & 0 \\ 
0 & 2i \pi k 
\end{bmatrix}
\end{equation*}
We remark that $A(\. )$ takes values in the center of $U(2)$, which is naturally isomorphic to $\Sp ^{1}$. Since, now,
$U(2)$ locally around the $Id$ is isomorphic to a neigborhood of $(1,Id)$ in $\Sp ^{1} \times  SU(2)$, we can readily see
that all small enough perturbations of $(\a, A(\. ) )$ have the same limit curve. Clearly, the reason for this
is that the limit object of the dynamics of $(\a, A(\. ) )$ lives in the direction in which the semisimplicity
criterion fails for $u(2)$.

\section{Estimation of the energy of the path $A(\. )$}

A direct calculation shows that%
\begin{equation*}
a(\. )=L(A(\. ))=e_{r}+Ad(E_{r,a}(\. )).u(\. )
\end{equation*}%
where $u(\. )=L \exp (U(\. ))$, so that the estimates of the energy of the path $A(\. )$ in $G$ are
significant only when $U(\. )$ is non-constant. The remaining case will
be studied later in this section. Using the $Ad$-invariance of the
Cartan-Killing form we find%
\begin{equation*}
\left\Vert a(\. )\right\Vert _{L^{2}}^{2}=\emph{en}^{2}+2\int _{\T}
\left\langle e_{r},u(\. )\right\rangle +\left\Vert u(\. )\right\Vert
_{L^{2}}^{2}
\end{equation*}%
and the assumption that the cocycle is of degree $r$ implies
\begin{equation*}
2\int _{\T} \left\langle e_{r},u(\. )\right\rangle +\left\Vert u(\.
)\right\Vert _{L^{2}}^{2}\geq 0
\end{equation*}

\bigskip

Using the expansion of the operator $L$ and the fact that $U(\. )$ is to
be considered small, we have the estimate%
\begin{eqnarray*}
u(\. ) = Le^{U(\. )} &=&\partial U(\. )+\frac{1}{2}[U(\. ),\partial U(\. 
)]+O(\left\vert U(\. )\right\vert ^{2}\left\vert \partial U(\. )\right\vert ) \\
&=&\tilde{u}(\. )+O(\left\vert U(\. )\right\vert ^{2}\left\vert
\partial U(\. )\right\vert )
\end{eqnarray*}

We now find that
\begin{eqnarray*}
2\left\langle e_{r},\tilde{u}(\. )\right\rangle  &=&2\left\langle
e_{r},\partial U(\. )\right\rangle +\left\langle e_{r},[U(\. 
),\partial U(\. )]\right\rangle  \\
&=&2\left\langle e_{r},\pi _{\ft }\partial U(\. )\right\rangle
+\left\langle e_{r},[U(\. ),\partial U(\. )]\right\rangle  \\
&=&2\left\langle e_{r},\pi _{\ft }\partial U(\. )\right\rangle
+\left\langle [e_{r},U(\. )],\partial U(\. )\right\rangle  \\
&=&2\left\langle e_{r},\pi _{\ft }\partial U(\. )\right\rangle
+\sum_{\r \in \Delta _{+}}Re(2i\pi \tilde{r}_{\r }U_{\r }(\. 
)\partial \bar{U}_{\r }(\. ))
\end{eqnarray*}
and therefore%
\begin{equation*}
2\int \left\langle e_{r},\tilde{u}(\. )\right\rangle =\int Re%
(\sum\nolimits_{\r \in \Delta _{+}}4\pi ^{2}\tilde{r}_{\r 
}\sum\nolimits_{k\in \Z }k|\hat{U}_{\r }(k)|^{2})
\end{equation*}%
The error term can be bounded by%
\begin{eqnarray*}
\int \left\vert U(\. )\right\vert ^{2}\left\vert \partial U(\. 
)\right\vert  &\leq &\left\Vert U\right\Vert _{L^{\infty }}\int \left\vert
U(\. )\right\vert \left\vert \partial U(\. )\right\vert  \\
&\leq &\left\Vert U\right\Vert _{L^{\infty }}\left\Vert U\right\Vert
_{L^{2}}\left\Vert \partial U\right\Vert _{L^{2}} \\
&\lesssim &\left\Vert U\right\Vert _{H^{1}}^{3}
\end{eqnarray*}%
since $H^{1}$ is continuously embedded in $L^{\infty }$. Consequently,%
\begin{equation*}
2\int \left\langle e_{r},u(\. )\right\rangle =\sum_{\r \in \Delta
_{+}}4\pi ^{2}\tilde{r}_{\r }\sum\nolimits_{k\in \Z }k|\hat{U}%
_{\r }(k)|^{2}+O(\left\Vert U\right\Vert _{H^{1}}^{3})
\end{equation*}

\bigskip

We also have%
\begin{eqnarray*}
\left\Vert \tilde{u}\right\Vert _{L^{2}}^{2} &=&\left\Vert \partial U(\. 
)\right\Vert _{L^{2}}^{2}+\left\Vert [U(\. ),\partial U(\. 
)]\right\Vert _{L^{2}}^{2} \\
&=&\left\Vert \partial \pi _{\ft }U(\. )\right\Vert
_{L^{2}}^{2}+\sum_{\r \in \Delta _{+}}\left\Vert \partial U_{\r }(\. 
)\right\Vert _{L^{2}}^{2}+O(\left\Vert U\right\Vert _{H^{1}}^{4})
\end{eqnarray*}%
Therefore,%
\begin{equation*}
2\int \left\langle u(\. ),e_{r}\right\rangle +\left\Vert u\right\Vert
_{L^{2}}^{2}\leq \sum_{\r \in \Delta _{+}}4\pi ^{2}\tilde{r}_{\r 
}\sum_{k\in \Z }k|\hat{U}_{\r }(k)|^{2}+\left\Vert \partial \pi _{%
\ft }U(\. )\right\Vert _{L^{2}}^{2}+\sum_{\r \in \Delta
_{+}}\left\Vert \partial U_{\r }(\. )\right\Vert
_{L^{2}}^{2}+O(\left\Vert U\right\Vert _{H^{1}}^{3})
\end{equation*}%
Since%
\begin{equation*}
\sum_{\r \in \Delta _{+}}\left\Vert \partial U_{\r }(\. )\right\Vert
_{L^{2}}^{2}=\sum_{\r \in \Delta _{+}}\sum_{k\in \Z }4\pi ^{2}k^{2}|%
\hat{U}_{\r }(k)|^{2}
\end{equation*}%
and%
\begin{equation*}
\sum_{\r \in \Delta }4\pi ^{2}\tilde{r}_{\r }\sum_{k\in \Z }k|%
\hat{U}_{\r }(k)|^{2}+\sum_{\r \in \Delta _{+}}\left\Vert \partial
U_{\r }(\. )\right\Vert _{L^{2}}^{2}=4\pi ^{2}\sum_{\r \in \Delta
_{+}}\sum_{k\in \Z }(\tilde{r}_{\r }+k)k|\hat{U}_{\r }(k)|^{2}
\end{equation*}%
we find that%
\begin{equation*}
2\int \left\langle u(\. ),e_{r}\right\rangle +\left\Vert u\right\Vert
_{L^{2}}^{2}\leq 4\pi ^{2}\sum_{\r \in \Delta _{+}}\sum_{k\in \Z }(%
\tilde{r}_{\r }+k)k|\hat{U}_{\r }(k)|^{2}+\left\Vert \partial \pi _{%
\ft }U(\. )\right\Vert _{L^{2}}^{2}+O(\left\Vert U\right\Vert
_{H^{1}}^{3})
\end{equation*}%
The fact that $(\tilde{r}_{\r }+k)k<0$ for $k$ between $0$ and $-\tilde{r}%
_{\r }$ motivates the following definition.

\begin{definition} \label{definition of obstructions for regular geodesics}
Let $\tilde{r}_{\r }\in \Z ^{\ast }$, $\r \in \Delta _{+}$. We
will call $I_{\r }=\{1,2,...-\tilde{r}_{\r }-1\}$ if $\tilde{r}_{\r }<0
$, and $I_{\r }=\{-1,-2,..,-\tilde{r}_{\r }+1\}$ if $\tilde{r}_{\r }>0$%
. We also let $\Lambda _{\tilde{r}}$ and $\Lambda _{\r }$ be the
truncation operators%
\begin{eqnarray*}
\Lambda _{\tilde{r}}U(\. ) &=&\sum\nolimits_{\r \in \Delta
_{+}}\Lambda _{\r }U_{\r }(\. )j_{\r } \\
&=&\sum\nolimits_{\r \in \Delta _{+}}\sum\nolimits_{k\in I_{\r }}\hat{U%
}_{\r }(k)e^{2i\pi k\. }j_{\r }
\end{eqnarray*}

We will also use the non-homogeneous truncation operator%
\begin{eqnarray*}
\Lambda _{\tilde{r}}^{0}U(\. ) &=&\sum\nolimits_{\r \in \Delta
_{+}}\Lambda _{\r }^{0}U_{\r }(\. )j_{\r } \\
&=&\Lambda _{\tilde{r}}U(\. )+\sum\nolimits_{\r \in \Delta _{+} }\hat{U}%
_{\r }(0)j_{\r }
\end{eqnarray*}
The image of the operator $\Lambda _{\tilde{r}}^{0}$ in $L^{2}$ will be denoted by $\mathcal{P}_{\tilde{r}}$.
\end{definition}

\bigskip

Using these notations, we have%
\begin{eqnarray*}
2\int \left\langle u(\. ),e_{r}\right\rangle +\left\Vert u\right\Vert
_{L^{2}}^{2} &\leq &-4\pi ^{2}\left\Vert \Lambda _{\tilde{r}}U(\. )\right\Vert
_{L^{2}}^{2}+\sum\nolimits_{\r \in \Delta }(\tilde{r}_{\r 
}+1)^{2}\left\Vert (Id-\Lambda _{\r })\partial U_{\r }(\. 
)\right\Vert _{L^{2}}^{2} \\
&&+\left\Vert \partial \pi _{\ft }U(\. )\right\Vert
_{L^{2}}^{2}+O(\left\Vert U\right\Vert _{H^{1}}^{3})
\end{eqnarray*}

The condition that the cocycle is of degree $r$ implies that%
\begin{equation}
4\pi ^{2}\left\Vert \Lambda _{r}U(\. )\right\Vert _{L^{2}}^{2}\leq
\sum\nolimits_{\r \in \Delta }(\tilde{r}_{\r }+1)^{2}\left\Vert
(Id-\Lambda _{\r })\partial U_{\r }(\. )\right\Vert
_{L^{2}}^{2}+\left\Vert \partial \pi _{\ft }U(\. )\right\Vert
_{L^{2}}^{2}+O(\left\Vert U\right\Vert _{H^{1}}^{3})  \label{low.frequencies}
\end{equation}%
or, in a more compact form

\begin{lemma}
Let $(\a ,E_{r,a}(\. )\exp (U(\. )))$ be of degree $r$. Then,
there exists a positive constant such that if $\left\Vert U\right\Vert
_{H^{1}}$ is small enough%
\begin{equation*}
\left\Vert \Lambda _{\tilde{r}}U(\. )\right\Vert _{L^{2}}\lesssim
\left\Vert (Id-\Lambda _{\tilde{r}})\partial U(\. )\right\Vert _{L^{2}}
\end{equation*}%
The constant depends only on the degree.
\end{lemma}

We remark that this inequality does not give any information on $\hat{U}(0)$.

\section{Estimation of the energy of the second iterate} \label{Estimation of the energy of the second iterate}

For the second iterate of $(\a ,A(\. ))=(\a ,E_{r,a}(\. 
).e^{U(\. )})$, we have%
\begin{equation*}
(2\a ,A_{2}(\. ))=(2\a ,E_{r,a}(\. +\a ).e^{U(\. 
+\a )}.E_{r,a}(\. ).e^{U(\. )})
\end{equation*}%
and therefore%
\begin{equation*}
a_{2}(\. )=a(\. +\a )+Ad(A(\. +\a )).a(\. )
\end{equation*}%
so that%
\begin{equation*}
\left\Vert a_{2}(\. )\right\Vert _{L^{2}}^{2}=2\left\Vert a(\. 
)\right\Vert _{L^{2}}^{2}-2\int \left\langle a^{\ast }(\. +\a 
),a(\. )\right\rangle
\end{equation*}%
We have already seen that%
\begin{equation*}
a(\. )=e_{r}+Ad(E_{r,a}(\. )).(\partial U(\. )+\frac{1}{2}[U(\. 
),\partial U(\. )])+O(|U(\. )|^{2}.|\partial U(\. )|)
\end{equation*}%
and in the same way we find%
\begin{equation*}
-a^{\ast }(\. )=\partial U(\. )-\frac{1}{2}[U(\. ),\partial U(\. 
)]+e_{r}+[e_{r},U(\. )]-\frac{1}{2}[U(\. ),[e_{r},U(\. 
)]]+O(|U(\. )|^{3}+|U(\. )|^{2}.|\partial U(\. )|)
\end{equation*}%
Since we have already estimates on $\| \dot{T}U\| _{L^{2}}$%
, we will split $U(\. )$ in $\hat{U}(0)+\dot{T}U(\. )=U_{0}+\dot{U}%
(\. )$ where $U_{0}=\hat{U}(0)$ is orthogonal to $e_{r}$. With these
notations, we find that%
\begin{eqnarray*}
\left\Vert a_{2}(\. )\right\Vert _{L^{2}}^{2} &\leq &2\left\Vert a(\. 
)\right\Vert _{L^{2}}^{2}+2\emph{en}^{2}-\left\Vert [e_{r},U(\. 
)]\right\Vert _{L^{2}}^{2}+2\int \left\langle [e_{r},U_{0}],Ad(E_{r,a}(\. 
)).\partial \dot{U}(\. )\right\rangle + \\
&&2\int \left\langle \partial \dot{U}(\. +\a )+[e_{r},\dot{U}(\. 
+\a )],Ad(E_{r,a}(\. )).\partial \dot{U}(\. )\right\rangle
+O(\left\Vert U\right\Vert _{H^{1}}^{3})
\end{eqnarray*}%
so that%
\begin{eqnarray*}
\left\Vert a_{2}(\. )\right\Vert _{L^{2}}^{2} &\leq &4\emph{en}^{2}+2\int
\left\langle \partial \dot{U}(\. )+[e_{r},\dot{U}(\. )],\partial \dot{U%
}(\. )\right\rangle -\left\Vert [e_{r},U(\. )]\right\Vert _{L^{2}}^{2}+
\\
&&2\int \left\langle \partial \dot{U}(\. +\a )+[e_{r},U(\. +\a 
)],Ad(E_{r,a}(\. )).\partial \dot{U}(\. )\right\rangle +O(\left\Vert
U\right\Vert _{H^{1}}^{3}) \\
&\leq &4\emph{en}^{2}+C_{1}\| \partial \dot{U}(\. )\|
_{L^{2}}^{2}-\left\vert[e_{r},U_{0}]\right\vert (\left\vert
[e_{r},U_{0}]\right\vert -\| \partial \dot{U}(\. )\|
_{L^{2}})+O(\left\Vert U\right\Vert _{H^{1}}^{3}) \\
&\leq &4\emph{en}^{2}+C_{1}\| \partial \dot{U}(\. )\|
_{L^{2}}^{2}-C_{2}\left\vert [e_{r},U_{0}]\right\vert (\left\vert
[e_{r},U_{0}]\right\vert -\| \partial \dot{U}(\. )\|
_{L^{2}})
\end{eqnarray*}%
if $\left\Vert U\right\Vert _{H^{1}}$ is small enough. We find, therefore,
that

\begin{lemma}
Let $(\a ,E_{r,a}(\. ).e^{U(\. )+U_{0}})$ be a perturbation (small
enough in $H^{1}$) of the cocycle $(\a ,E_{r,a}(\. ))$. Let also,
without loss of generality, $\hat{U}(0)=0$ and $%
U_{0}\not=0$. Then, there exists a constant $C_{0}$ depending only on $r$
such that if%
\begin{equation*}
\left\Vert \partial U\right\Vert _{L^{2}}\leq C_{0}|U_{0}|
\end{equation*}%
then $(\a ,E_{r,a}(\. ).e^{U_{0}}e^{U(\. )})$ is of energy
strictly smaller than that of $(\a ,E_{r,a}(\. ))$.
\end{lemma}

This was first proved in \cite{Krik2001}, but the constant implicitly
depended on $\a $.

\section{End of the proof of lemma \ref{A priori estimates}}

Using the estimates obtained so far, we can prove, lemma \ref{A priori estimates}:

\begin{proof} [Proof of lemma \ref{A priori estimates}]
The estimate on the $L^{2}$ norm of the derivative of the first iterate
implies that%
\begin{equation*}
\left\Vert \Lambda _{r}U(\. )\right\Vert _{L^{2}}\leq C^{\prime
}\left\Vert (Id-\Lambda _{r})\partial U(\. )\right\Vert _{L^{2}}
\end{equation*}%
for some constant $C^{\prime }$. Since%
\begin{equation*}
\| \Lambda _{r}^{0}U(\. )\| _{L^{2}}^{2}=\left\Vert
\Lambda _{r}U(\. )\right\Vert _{L^{2}}^{2}+|\hat{U}(0)|^{2}
\end{equation*}%
$\left\Vert \Lambda _{r}^{0}U(\. )\right\Vert _{L^{2}}\geq M\left\Vert
(Id-\Lambda _{r}^{0})\partial U(\. )\right\Vert _{L^{2}}$ for $M$ big
enough would imply that $\left\Vert \partial U\right\Vert _{H^{1}}\leq
C_{0}|U_{0}|$, which contradicts the hypothesis.
\end{proof}

\section{Further estimates}

Use of the invariance of the degree under $C^{1}$ conjugation allows us in
fact to control some of the low frequencies in the spectrum of the
perturbation. It follows immediately from the definitions that, for $B(\. 
)=\exp (U(\. ))$,%
\begin{eqnarray*}
Conj_{B(\. )}(\a ,E_{r,a}(\. ).e^{U(\. )}) &=&(\a 
,e^{U(\. +\a )}.E_{r,a}(\. )) \\
&=&(\a ,E_{r,a}(\. ).e^{V(\. )})
\end{eqnarray*}%
with $V(\. )=Ad(E^{*}_{r,a}(\. )).e^{U(\. +\a )}$, so that%
\begin{equation*}
\hat{V}_{\r }(k)=e^{2i\pi k\a }\hat{U}_{\r }(k+2r_{\r })
\end{equation*}%
Inserting this relation in the estimates obtained in the previous sections
gives the estimate%
\begin{equation*}
\| \Lambda _{-\tilde{r}}^{0}U\| _{L^{2}}\lesssim \| (Id-\Lambda _{-\tilde{r}}^{0})\partial U\| _{L^{2}}
\end{equation*}%
More generally, since conjugation by a periodic geodesic $E_{r^{\prime }}$
acts by translation on the frequencies of $U(\. ) \,$ \footnote{%
It also changes the constant $A$, by it is not important since the constant
does not enter in the estimates.}, we can see conjugation as an action on
the truncation operator. If, now, $\Lambda $ is a translate of $\Lambda _{%
\tilde{r}}^{0}$ under this action, then we have the same type of estimate,
with $\Lambda $ in the place of $\Lambda _{-\tilde{r}}^{0}$.

\bigskip

Finally, the conjugation $D(\. )=\exp (U(\. )/2)$ imitates the effects
of two iterations and normalization of the derivative, since%
\begin{eqnarray*}
Conj_{D(\. )}(\a ,E_{r,a}(\. ).e^{U(\. )}) &=&(\a 
,e^{U(\. +\a )/2}.E_{r,a}(\. ).e^{U(\. )/2}) \\
&=&(\a ,E_{r,a}(\. ).e^{V(\. )/2}.e^{U(\. )/2})
\end{eqnarray*}%
with $V(\. )$ as previously, since%
\begin{eqnarray*}
(\a ,E_{r,a}(\. ).e^{U(\. )})^{2} &=&(2\a ,E_{r,a}(\. 
+\a ).e^{U(\. +\a )}.E_{r,a}(\. ).e^{U(\. )}) \\
&=&(2\a ,E_{r,a}(2\. +\a ).e^{V(\. )}e^{U(\. )})
\end{eqnarray*}%
A calculation shows, however, that the effect of non-commutativity of $e_{r}$
with the constant part of the perturbation is not seen in $L(E_{r,a}(\. 
).e^{V(\. )/2}.e^{U(\. )/2})$. As a result, estimation of the energy
of the path of $E_{r,a}(\. ).e^{V(\. )/2}.e^{U(\. )/2}$ gives only
the trivial estimate%
\begin{equation*}
\left\Vert u(\. )\right\Vert _{L^{2}}^{2}\geq |\int \langle u^{\ast
}(\. +\a ),Ad(E_{r,a}(\. )).u(\. )\rangle |
\end{equation*}%
which shows, nonetheless, that the use of Cauchy-Schwartz inequality in
the estimation of the energy of the path of the second iterate does not
result into too rough estimates.

\section{Proof of theorem \ref{density of 0 energy}} \label{Density of zero-energy cocycles}

Let $E_{r,a}(\. )$ be a periodic geodesic in $G$, not necessarily
regular, and the corresponding splitting $G_{0}\times G_{+}\hra G$.
If the zero energy component is not trivial, and $A_{0} e^{U_{0}(\. )} \in C^{\infty} (\T , G_{0})$
is a perturbation of a constant $A_{0} \in G_{0}$, it is clear that $(\a , E_{r,a}(\. )A_{0} e^{U_{0}(\. )})$
is of the same degree, for an irrational $\a$, since renormalization of the
perturbed cocycle will converge to some $E_{r,a^{\prime }}(\. )$. On the
other hand, the preceding estimates show that a small perturbation
spectrally supported in some $I_{\r }$ with $\r (e_{r})\not=0$, and
therefore not commuting with $E_{r,a}(\. )$, decreases the energy of the
cocycle. The same thing holds for renormalization representatives of cocycles non-homotopic to constants.
If $A \in G_{0} $ anti-commutes with $\exp (H \. ) : 2\T \ra G_{0} $, then the $G_{0}$-valued
cocycle is of energy $0$. Therefore, no $G_{0}$-valued perturbation can change the energy of the cocycle.
We also point out that the perturbed cocycle is in the same
homotopy class as $E_{r,a}(\. )$.

Therefore, the picture is the following. If we renormalize a cocycle in a
given class of homotopy, the scheme converges towards an action within the
same class, which can be of positive or zero energy. Let us begin the study
with the two simplest non-abelian groups, $SU(2)$ and $SO(3)$, where all
geodesics are regular, but $SO(3)$ is not simply connected and $\x 
_{SO(3)}=2$, while $SU(2)$ is simply connected and $\x _{SU(2)}=1$.

\subsubsection{Cocycles in $SU(2)$ and $SO(3)$}

The positive energy abelian models in $SU(2)$, all of them regular, are
given by the periodic geodesics $E_{r}(\. )$, $r\in \mathbb{N}^{\ast }$%
\begin{equation*}
\begin{pmatrix}
e^{2i\pi r\. } & 0 \\ 
0 & e^{-2i\pi r\. }%
\end{pmatrix}%
\end{equation*}%
These models, after projection to $SO(3)$, become the positive energy models
which are homotopic to the identity:%
\begin{equation*}
\begin{pmatrix}
\cos (4\pi r\. ) & \sin (4\pi r\. ) & 0 \\ 
-\sin (4\pi r\. ) & \cos (4\pi r\. ) & 0 \\ 
0 & 0 & 1%
\end{pmatrix}%
\end{equation*}%
Perturbation of such a model, say by a small constant in non-standard
configuration (which we will note by $A$ throughout this paragraph) reduces
the energy. In the special case of $E_{1}(\. )$, such a perturbation will
yield a cocycle which, after renormalization, will converge to
close-to-constant dynamics. Consideration of its homotopy class (or,
equivalently, the fact that it admits a $1$-periodic lift to $SU(2)$), implies that in
$SO(3)$ as in $SU(2)$, it can be normalized modulo $1$, and passage to a $2$-lattice
is not necessary. Therefore, in both groups, such models are accumulated by cocycles
of $0$ energy modulo $1$.

In $SO(3)$, however, we also have the class of geodesics non-homotopic to
constants. given by the periodic geodesics $E_{r+1/2}(\. )$, $r\in 
\mathbb{N}$:%
\begin{equation*}
\begin{pmatrix}
\cos (2\pi (2r+1)\. ) & \sin (2\pi (2r+1)\. ) & 0 \\ 
-\sin (2\pi (2r+1)\. ) & \cos (2\pi (2r+1)\. ) & 0 \\ 
0 & 0 & 1%
\end{pmatrix}%
\end{equation*}%
Such geodesics, seen in $SU(2)$, emanate from the $Id$, and after having
passed $r$ times by the $Id$ end in $-Id$. Clearly, the second iterate of
such a cocycle is homotopic to constants. The same calculations as before
show that a perturbation of $E_{r+1/2}(\. )$ with $r>0$ by $A$ drives the
renormalization towards some $E_{r^{\prime }+1/2}(\. )$, with $r^{\prime
}<r$. A remarkable phenomenon is encountered when we reach the normal form $%
E_{1/2}(\. )$, which is the minimal energy model of dynamics
simultaneously non-reducible and non-homotopic to constants. Once again, the
estimations show that renormalization of $E_{1/2}(\. ).A$ will converge
towards constant actions, or equivalently that for $n$ big enough $\mathcal{%
\tilde{R}}^{n}\Phi $ (the action associated to a $(\a ,E_{1/2}(\. ).A)
$) will be close to a constant action. The consideration of its homotopy
class shows that in this case we will have to iterate $\mathcal{\tilde{R}}%
^{n}\Phi (e_{2}) $ before normalizing the action. Therefore, $(\Lambda _{2})_{\ast }\mathcal{%
\tilde{R}}^{n}\Phi $, and consequently the second iterate of $E_{1/2}(\. 
).A$, is in the class of local dynamics $\mod 2$, whereas $E_{1/2}(\. ).A$ itself
cannot be. In fact, $\Phi $ is expected to converge towards constants like,
say%
\begin{eqnarray*}
\pi _{2}(\mathcal{\tilde{R}}^{n}(e_{1})) &=&%
\begin{pmatrix}
-1 & 0 & 0 \\ 
0 & -1 & 0 \\ 
0 & 0 & 1%
\end{pmatrix}
\\
\pi _{2}(\mathcal{\tilde{R}}^{n}(e_{2})) &=&%
\begin{pmatrix}
1 & 0 & 0 \\ 
0 & -1 & 0 \\ 
0 & 0 & -1%
\end{pmatrix}%
\end{eqnarray*}%
who commute, but are not on the same maximal torus. Normalization of the
action itself drives it back to the neighborhood of translate of $E_{1/2}(\. )$, since
the natural choice is the conjugant $E_{1/2}(\frac{1}{2}\. )$, and%
\begin{equation*}
E_{1/2}(\frac{1}{2}(\. +\a )).%
\begin{pmatrix}
1 & 0 & 0 \\ 
0 & -1 & 0 \\ 
0 & 0 & -1%
\end{pmatrix}%
.E_{1/2}(\frac{1}{2}\. )^{T}=E_{1/2}(\frac{1}{2}\a )%
\begin{pmatrix}
\cos (2\pi \. ) & \sin (2\pi \. ) & 0 \\ 
\sin (2\pi \. ) & -\cos (2\pi \. ) & 0 \\ 
0 & 0 & -1%
\end{pmatrix}%
\end{equation*}%
On the other hand, the square of both constants is the $Id$, so that $%
(\Lambda _{2})_{\ast }\mathcal{\tilde{R}}^{n}\Phi $ can be normalized
without any problem. The factor $2$ in $\Lambda _{2}$ in this case is due to
the fact that $\x _{SO(3)}=2$. In fact, this very example shows that there
exist cocycles in $\T \times SO(3)$ which become reducible after
iteration, since, if we call%
\begin{eqnarray*}
A(\. ) &=&%
\begin{pmatrix}
\cos (2\pi \. ) & \sin (2\pi \. ) & 0 \\ 
\sin (2\pi \. ) & -\cos (2\pi \. ) & 0 \\ 
0 & 0 & -1%
\end{pmatrix}
\\
&=&%
\begin{pmatrix}
\cos (2\pi \. ) & -\sin (2\pi \. ) & 0 \\ 
\sin (2\pi \. ) & \cos (2\pi \. ) & 0 \\ 
0 & 0 & 1%
\end{pmatrix}%
.%
\begin{pmatrix}
1 & 0 & 0 \\ 
0 & -1 & 0 \\ 
0 & 0 & -1%
\end{pmatrix}%
\end{eqnarray*}%
then, for any irrational $\a $, $\pi _{2}(2\a ,A_{2}(\. ))$ is
equal to%
\begin{equation*}
\begin{pmatrix}
\cos (2\pi \a ) & -\sin (2\pi \a ) & 0 \\ 
\sin (2\pi \a ) & \cos (2\pi \a ) & 0 \\ 
0 & 0 & 1%
\end{pmatrix}%
\end{equation*}%
However, $(\a ,A(\. ))$ itself is not reducible for any $\a $,
since $A(\. )$ is non-homotopic to the $Id$.

\bigskip

Schematically, we can say that renormalization of $E_{1/2}(\. ).A$
converges towards constant actions, but normalization without iteration
drives us back away from constant cocycles. Iteration and normalization
allow us to conjugate the dynamics to a model close to a constant (which in fact is in the
conjugacy class of $(\a , Id)$). We can say
that, if we do not authorize iteration, the geodesic 
$E_{1/2}(\. )$ takes up the role of constants in its homotopy class.

\subsubsection{Cocycles of smaller energy in the neighborhood of regular
geodesics}

In this case, the value of $\x _{G}$ is irrelevant with the normalization
of the action, since we can conjugate the action to a normalized one and
arbitrarily close to some $(\a _{n},E_{r,a}(\. ))$, without
iterating. Any small enough constant non-diagonal perturbation of such a
model reduces the energy, but in general it may lead renormalization to
singular dynamics and not directly to constants. For example, renormalization
of%
\begin{equation*}
\begin{pmatrix}
e^{2i\pi \. } & 0 & 0 \\ 
0 & e^{4i\pi \. } & 0 \\ 
0 & 0 & e^{-6i\pi \. }%
\end{pmatrix}%
.\exp 
\begin{pmatrix}
0 & 0 & 0 \\ 
0 & 0 & z \\ 
0 & -\bar{z} & 0%
\end{pmatrix}%
\end{equation*}%
in $SU(3)$, with $z\in \C $ small, may converge to $%
\begin{pmatrix}
e^{2i\pi \. } & 0 & 0 \\ 
0 & e^{2i\pi \. } & 0 \\ 
0 & 0 & e^{-4i\pi \. }%
\end{pmatrix}%
$, which is singular. In the neighborhood of a singular geodesic $\x _{G_{0}}$
becomes relevant, since normalization without iteration may, as in $SO(3)$,
drive the algorithm away from constant actions. Therefore, in general the
local picture in the neighborhood of regular geodesics is that such models
are accumulated by cocycles of smaller energy, which however may be
conjugate (close) to the corresponding geodesics or (close) to constant
actions only modulo $1\leq m\leq \x _{G}$.

The picture is simpler in groups like $SU(w+1)$ where the constant equals
one, and therefore all cocycles of positive energy are accumulated by
cocycles of smaller energy, which can be conjugated (close) to the
corresponding geodesics or (close) to constant actions modulo $1$.

\subsubsection{Cocycles of smaller energy in the neighborhood of singular
geodesics}

In the case of a singular cocycle of positive energy, renormalization and
conjugation of actions is not affected until it reaches the form of eq. \ref%
{normalized}, which is (after some simplification of the notation)%
\begin{equation*}
\begin{pmatrix}
(1,C_{0}.\exp (O(\varepsilon _{n}))) \\ 
(\a _{n},A_{0}.\exp (e_{r}\. +O(\varepsilon _{n})))
\end{pmatrix}%
\end{equation*}%
with $C_{0}\in G_{0}$. We also recall the corresponding splitting $G_{0}\times
G_{+}\hra G$, and suppose that there are no error terms, i.e. that%
\begin{equation*}
\mathcal{\tilde{R}}^{(n)}\Phi =%
\begin{pmatrix}
(1,C_{0}) \\ 
(\a _{n},A_{0}.\exp (e_{r}\. ))
\end{pmatrix}%
\end{equation*}

Since $\ba (0)$ is singular, depending on the homotopy of the cocycle in $G_{0}$,
we may not be able to conjugate $C_{0}$ to the $Id$ with a conjugant
commuting with $A_{0}$, but this can be assured by iterating at most $\x 
_{0}=\x _{G_{0}}$ times. This fact is a more general case of the
phenomenon observed in perturbations of $E_{1/2}(\. )$-like models in $%
SO(3)$. An action like the one we are studying may be associated to a
cocycle non-homotopic to the $Id$, and renormalization may converge to
constants for a part of the non-triviality of its homotopy class while the
rest is still close to a geodesic in the same homotopy class. At this step,
we can iterate $\x _{0}$ times and obtain a cocycle of singular dynamics,
but whose homotopy within the $0$ energy component allows conjugation of the
action to the actual geodesic. This action, $(\Lambda _{\x _{0}})_{\ast }%
\mathcal{\tilde{R}}^{(n)}\Phi $, can thus be conjugated to%
\begin{equation*}
\begin{pmatrix}
(1,Id) \\ 
(\x _{0}\a _{n},A_{0}.\exp (\x _{0}e_{r}\. ))
\end{pmatrix}%
\end{equation*}%
with a different $A_{0}$, but still commuting with $\ba (0)$. A
constant non-abelian perturbation of $A_{0}.\exp (\x _{0}\ba %
(0)\. )$ in a complex direction $j_{\r }$ such that $[e_{r},j_{\r 
}]\not=0$ will drive renormalization to a lower energy model.

\bigskip

We can, nonetheless, perturb $\mathcal{\tilde{R}}^{(n)}\Phi $ to a lower
energy cocycle without iterating. Let $H_{0}\in g_{0}$ be a preimage of $%
C_{0}$. Then, conjugation by $\exp (-H_{0}\. )$ gives the action%
\begin{equation*}
\begin{pmatrix}
(1,Id) \\ 
(\a _{n},\exp (e_{r}\. ).A_{0}(\. ))
\end{pmatrix}%
\end{equation*}%
where the zero-energy component $A_{0}(\. )=\exp (-H_{0}(\. +\a 
_{n})).A_{0}.\exp (H_{0}\. )$ may not be constant. If we chose a root $%
\r $ such that $h_{\r }\in g_{+}$, then $j_{\r }$ commutes with all
vectors in $g_{0}$, so that the action%
\begin{equation*}
\begin{pmatrix}
(1,Id) \\ 
(\a _{n},\exp (e_{r}\. ).A_{0}(\. ).\exp (zj_{\r })
\end{pmatrix}%
\end{equation*}%
will be of smaller energy. Iteration of this procedure, and renormalization
and normalization of the resulting actions will eventually give a cocycle of
minimal energy in the homotopy class of the initial cocycle and the
associated action%
\begin{equation*}
\begin{pmatrix}
(1,Id) \\ 
(\a _{n},\exp (e_{r^{\prime }}\. ).A_{0})
\end{pmatrix}%
\end{equation*}%
where $A_{0}$ is a different constant and $e_{r^{\prime }}$ will be $0$ iff
the cocycle were homotopic to constants. If we renormalize a $m$-lattice of
such an action, with $1\leq m\leq \x _{G}$ depending only on the homotopy
class of the cocycle, renormalization will converge to constants.
\selectlanguage{english}
\chapter{Perturbations of regular geodesics} \label{Normal form theorem}

Based on the a priori estimates obtained in the previous chapter and on
techniques developed in \cite{KrikAst} (chapter 4) and in \cite{Krik2001},
we begin the local study of regular obstructions by proving that under a Diophantine
condition on the rotation in the basis and a smallness condition (related to
the Diophantine condition) on the perturbation, the cocycle $(\a,E_{r,a}(\. ).e^{U(\. )})$
can be reduced to $(\a ,E_{r,a'}(\.).e^{P(\. )})$, where $a^{\prime }$ is a constant in
the torus, and $P(\. )$ is a Fourier polynomial spectrally supported in $I_{\tilde{r}}$. Already,
the estimates and the discussion of the previous chapter show that, unless $P(\. ) =0$, such a
$(\a ,E_{r,a'}(\. ).e^{P(\.)})$ is not of degree $r$, which gives a good description of the local
structure of the conjugacy class of $(\a ,E_{r,a}(\. ))$, where in
fact the diagonal constant $a$ is not very important.

The motivation of the result becomes easier after the study of the linearized equation, at the end of section
\ref{Local reduction lemmas regular geodesics}, but let us briefly
describe the procedure. In the previous chapter we saw that some low
frequencies in the spectrum of the perturbation of a regular geodesic
do not allow, if they are dominant, the perturbed cocycle to be of
energy $|e_{r}|$. It is reasonable, therefore, to think of them as
obstructions to reducibility
to $E_{r,a} (\. )$. The first reaction to this observation was that of R. Krikorian in \cite{Krik2001}, where he supposed that
a given cocycle in $\T \times SU(2)$, perturbation of a periodic geodesic of the group (automatically regular), is actually
of the same degree as the geodesic. Then, he showed that if the rotation is Diophantine, a convergent K.A.M. scheme can be
defined which constructs a conjugation reducing the cocycle to the periodic geodesic. At each step, the obstructions to
reducibility were shown to die out along with the perturbation thanks
to the assumption on the degree.

A careful study of the procedure shows that a K.A.M. scheme (and not a fixed-point argument) is needed because there is
a loss of derivatives in the estimates, which, however, is uniform throughout the scheme. Therefore, it is possible
that a more accurate result can be obtained, and this is done by dropping the assumption that the given cocycle is of
the same degree as the geodesic, but retaining only the assumption that the perturbation is small. Then, the Nash-Moser
inverse function theorem is used in order to show the following.
We can split any given perturbation, at least in the level of linear terms, into the obstructions plus a coboundary. The
obstructions cannot be eliminated, but we can consider them as part of the model around which linearize. Then, the
coboundary can still be eliminated, thanks to the absence of small divisors. The fact that this procedure can be carried out
with uniform estimates provided that the obstructions are small enough, allows the use of the inverse function theorem, which
guarantees that the exact problem can be solved. In other words, we can split any perturbation into two parts.
The first one, $P(\. )$, is the one that we cannot hope to eliminate.
But if we consider the model $(\a ,E_{r,a}(\. )e^{P(\. )})$
as the model of the dynamics, then the remaining part of the
perturbation can be eliminated by conjugation. Finally, the a priori
estimates show that the cocycle is of degree $r$ if, and only if, this
principal part $P(\. )=0$.

\section{Synopsis of the chapter}

In this section we apply Hamilton's inverse function theorem in order to obtain
a normal form theorem for cocycles in the neighborhood of a regular geodesic.
We remind that $\mathcal{P}_{r}$ was introduced in the previous chapter (def.
\ref{definition of obstructions for regular geodesics}). We denote by $\mathcal{E}_{0}$
the space of mappings in $f \in C^{\infty}(\T, g)$ such that
$\int \pi _{\ft}(f) =0 $. This tame Fréchet space inherits its topology
from $\mathcal{E} = C^{\infty}(\T, g)$, which in turn is the topology induced by the family of seminorms
\begin{equation*}
\| U \|_{s} = \max _{0 \leq \s \leq s} (\| \partial _{\s} U \| _{\infty } )
\end{equation*}
for $U \in \mathcal{E}$. We bring to the attention of the reader that these norms
correspond to the norms $\| U \| _{s}^{\max }$ in the notation of the first chapter.
Finally, $\ft $ denotes the unique maximal toral algebra of $g$ containing $e_{r}$.

With these notations, the normal form theorem reads
\begin{theorem}
\label{Normal Form Thm}Let $\a \in DC (\gamma , \tau) $, $r \in \Z ^{w}$ such that
$e_{r} \in \ft $ be regular, and $a \in \ft$. Then there
exists $\mathcal{W=W}(\a ,r\mathcal{)}$, a neighborhood of $0$ in $%
\mathcal{E}$, such that for any $U(\. )\in \mathcal{W}$, there exists a
unique $(\lambda ,B,P)\in \ft \times \mathcal{E}_{0}\times \mathcal{P}_{r}$ such that%
\begin{equation*}
E_{r,a}(\. ).e^{U(\. )}=e^{B(\. +\a )}.E_{r,a+\lambda }(\. ).e^{P(\. )}.e^{-B(\. )}
\end{equation*}
The triple $(\lambda , B,P)$ satisfies tame estimates with respect to $U$ and $\lambda 
$ depends continuously on it.
\end{theorem}

We remark that the important characteristic of $\mathcal{P}_{r} $ is 
not its support in the space of frequencies, but rather the fact that $\mathcal{P}%
_{r}$ and $Ad(E_{r,a}(\. )).\mathcal{P}_{r}$ have disjoint supports, for
all $a\in \ft $, as we have seen in the previous chapter.

We also have the following corollary.

\begin{corollary}
The orbits of $(\a ,E_{r,a}(\. ))$, with $a\in \ft $ and $%
\a \in DC(\gamma ,\tau )$, under the adjoint action of $\mathcal{E}_{0}$
form locally a tame Fr\'{e}chet manifold of codimension $2\sum\nolimits_{%
\rho \in \D_{+} }|\tilde{r}_{\rho }|+w$ in $SW^{\infty }(\T ,G)$. Or,
more importantly, reducibility to a normal form is a property with codimension
$2\sum\nolimits_{\rho \in \D_{+}}|\tilde{r}_{\rho }|$ in the neighborhood
of $(\a ,E_{r,a}(\. ))$ in $SW^{\infty }_{\a}(\T ,G)$.
\end{corollary}

A direct application in this particular case of the a priori estimates on
perturbations of regular geodesics shows that a such cocycle in normal form
is of energy smaller than that of $(\a ,E_{r,a}(\. ))$, unless $P(\. )=0$. We thus obtain

\begin{theorem} \label{NK local reducibility to regular geodesics}
Reducibility to a regular periodic geodesic over a Diophantine rotation $%
(\a ,E_{r,a}(\. ))$ is locally of codimension $2\sum\nolimits_{\rho
\in \D }|\tilde{r}_{\rho }|$. All cocycles of degree $r$ in this neighborhood
of $E_{r,a}(\. )$ are reducible to a normal form.
\end{theorem}

Combining this theorem with the renormalization scheme, we obtain the
following result.

\begin{theorem} \label{NK global reducibility to regular geodesics}
If $\a \in RDC$ and the cocycle $(\a ,A(\. ))$ is regular of degree $r$, then
$(\a ,A(\. ))$ is $C^{\infty }$-conjugate to $(\a ,E_{r,a}(\.)) $ for some $a\in \ft $.
\end{theorem}

\begin{proof}
If the cocycle $(\a, A(\.))$ is regular, then, by theorem \ref{Quantization of degree},
renormalization converges towards models of the type $(\a _{n} ,E_{r,a}(\. ))$. We remind that under the assumption
that the cocycle is regular, the constant $\chi_{G}$ becomes irrelevant, and therefore there exist
renormalization representatives $(\a_{n}, \tilde{A}_{n}(\.))$ arbitrary close in the $C^{\infty}$
topology to cocycles of the type $(\a _{n} ,E_{r,a}(\. ))$.

Since $\a \in RDC (\gamma , \tau)$, $\a_{n} \in DC (\gamma , \tau)$ infinitely many times.
Therefore, the $\a_{n}$ can be chosen uniformly Diophantine and the normal form theorem can be
applied to some $(\a_{n}, \tilde{A}_{n}(\.))$. The hypothesis that $(\a_{n}, \tilde{A}_{n}(\.))$
is of degree $r$ allows us to conclude.
\end{proof}

Finally, for cocycles in $\T \times SU(2)$ we have the following
result.

\begin{theorem}
If $\a \in RDC$ and $(\a ,A(\. ))\in SW^{\infty }(\T ,SU(2))$ is of non-zero degree,
then it is $C^{\infty }$-conjugate to $(\a ,E_{r}(\. +\theta ))$ for some $\theta \in \T$. Reducible
cocycles in $\T \times SU(2)$ are dense in the total space, if the rotation satisfies an $RDC$.
\end{theorem}

\begin{proof}
Every non-zero vector in $su(2)$ is regular, so that all positive energy cocycles
are regular too. Therefore, as in the proof of theorem \ref{NK global reducibility to regular geodesics},
we find that all cocycles of positive energy are conjugate to the corresponding periodic geodesic.

Since every such cocycle is accumulated by cocycles of lower energy, after a finite number of arbitrarily small
perturbations we can obtain a cocycle of $0$ energy. Then, renormalization of such a cocycle converges to constants
(we remind that $\chi _{SU(2)} = 1$), and the local density theorem \ref{RK local density} allows us to conclude.
\end{proof}

This density result was first obtained in \cite{Krik2001} under a somewhat
stricter (and easily recognizable as an artefact of the method) arithmetic
condition of full measure in $\T $. The calculations in the proofs
already indicate the existence of a local Fréchet manifold of cocycles
reducible to the normal form $(\a ,E_{r}(\. +\theta ))$ and show that
cocycles of smaller energy form an open dense set in the neighborhood of
normal forms. The arithmetic condition imposed was loosened to a simple $RDC$
in \cite{Fra2004}, but without the stronger result on the codimension and
the local description of the conjugacy classes.

\section{Local reduction lemmas} \label{Local reduction lemmas regular geodesics}

Let us firstly write and solve the linearized cohomological equation of
reduction of the perturbation to a second-order one. We suppose that
conjugation by a small $B(\. )\in C^{\infty }(\T ,g)$ reduces the
perturbation $U(\. )$ to a smaller one, noted by $V(\. )$%
\begin{equation*}
\exp (B(x+\a )).E_{r,a}(x).\exp (U(x)).\exp (-B(x))=E_{r,a}(x).\exp
(V(x))
\end{equation*}%
Linearization of this equation, under the assumption that $U(\. )$, $%
D(\. )$ and $V(\. )$ are small in $C^{0}$ and $V(\. )\ll U(\. )$, i.e. is of
second order with respect to $U(\. )$, gives
\begin{equation*}
Ad(E_{-r,-a}(x)).B(x+\a )-B(x)=-U(x)
\end{equation*}%
Introducing the coordinates $B(\. )=\sum_{\rho \in \tilde{\D}}B_{\rho
}^{\ft }(\. ).h_{\rho }+\sum_{\rho \in \D _{+}}B_{\rho
}(\. ).j_{\rho }$, as for $U(\. )$, we have the equations%
\begin{eqnarray}
B^{\ft}_{\rho }(\. +\a )-B_{\rho }^{\ft }(\. ) &=&-U_{\rho }^{%
\ft }(\. ),~\rho \in \tilde{\D}  \label{lin eq a} \\
e^{-2i\pi (\tilde{r}_{\rho }\. +a_{\rho })}.B_{\rho }(\. +\a
)-B_{\rho }(\. ) &=&-U_{\rho }(\. ),~\rho \in \D _{+}
\label{lin eq c}
\end{eqnarray}%
where $2\pi \tilde{r}_{\rho }=\langle e_{r},h_{\rho }\rangle \not=0$ and $%
r_{\rho }\in \N ^{\ast }$ for all $\rho \in \tilde{\D}$ (but in
general, $r_{\rho }\in \mathbb{Z}^{\ast }$, for $\rho \in \D _{+}$).

\bigskip

The set of equations involving the coordinates on the torus (eq. (\ref{lin eq a}%
)) are classical linear cohomological equations for which we refer to lemma \ref{Linear cohomological eq}.
We remind that the only topological obstruction to the solution of the equation concerns
the mean value of the rhs function, which must be $0$.

Solving the equation (\ref{lin eq c}) is less classical, so let us consider
its general form,%
\begin{equation*}
e^{2i\pi (mx+c)}f(x+\a )-f(x)=g(x)
\end{equation*}%
with $m\in \mathbb{Z}^{\ast }$ and $c\in \T $. The constant $c$ can
easily be seen to be irrelevant, since composition by a translation $x\mapsto x-c/m$ gives the equation the form%
\begin{equation*}
e^{2i\pi mx}f(x+\a )-f(x)=g(x)
\end{equation*}%
where we have kept the same notation for $f(\. )$ and $f(\. -c/m)$,
similarly for $g(\. )$. 

Application of the Fourier transform to this last
form of the equation gives
\begin{equation}
e^{2i\pi (k-m)\a }\hat{f}(k-m)-\hat{f}(k)=\hat{g}(k)  \label{dif eq}
\end{equation}
for all $k\in \Z $.

Let us suppose for simplicity that $m$ is positive, the remaining case being
treated similarly. The forward solution of the equation is then found by
forward iteration of the following formula:%
\begin{equation*}
\hat{f}(k)=e^{-2i\pi k\a }(\hat{g}(k+m)+\hat{f}(k+m))
\end{equation*}%
This gives%
\begin{equation}
\hat{f}_{+}(k)=\sum_{j=1}^{\infty }e^{-2i\pi j(k+\frac{j-1}{2}m)\a }\hat{%
g}(k+jm)  \label{forw.sol}
\end{equation}

Backward iteration of the formula 
\begin{equation*}
\hat{f}(k)=\hat{g}(k)+e^{2i\pi (k-m)\a }\hat{f}(k-m)
\end{equation*}%
gives the backward solution 
\begin{equation}
\hat{f}_{-}(k)=\sum_{j=0}^{\infty }e^{2i\pi j(k-\frac{(j+1)}{2}m)\a }%
\hat{g}((k-j)m)  \label{backw sol}
\end{equation}

\bigskip

The forward and the backward solution are constructed independently,
and $f_{+} (\. ) = \sum\nolimits _{k} \hat{f}_{+} (k) e^{2i \pi k \.} $
(resp. $f_{-} = (\. ) \sum\nolimits _{k} \hat{f}_{-} (k) e^{2i \pi k \.} $)
is in $C^{\infty} (\T ,g)$ only if the summation is for $k > l$ (rep. $k \leq l' $)
where $l \in \Z $ (resp. $l'$) is to be fixed. Clearly, determination of
all the Fourier coefficients of the solution imposes that $l' = l-1$.
If, now we define $\hat{f} : \Z \ra \C $ by
\begin{equation*}
{\hat{f}}(k)=%
\begin{cases}
\hat{f}_{+}(k), & k>l \\ 
\hat{f}_{-}(k), & k\leq l%
\end{cases}%
\end{equation*}
Then, $\hat{f} (\. ) $ is the Fourier series of a function $f : \T \ra \C $ which
solves the eq. \ref{dif eq} for $k\not\in \{l+1,...,l+m\}$, as shows the
form of the equation. For the remaining values of $k$, we find that%
\begin{equation*}
e^{2i\pi (k-m)\a }\hat{f}(k-m)-\hat{f}(k)=e^{2i\pi (k-m)\a }\hat{f}_{-}(k-m)-\hat{f}_{+}(k)
\end{equation*}
The a priori estimates obtained in the previous chapter suggest that we
chose $l=-m$, since, with this choice, the set of frequencies in $\{l+1,...,l+m\}$
is equal to the set $I_{\r }$ of def. \ref{definition of obstructions for regular geodesics}
if $m = r_{\r }$. In other words, the frequencies in $\{l+1,...,l+m\}$ are, for a given choice
of $l$, the obstruction to the solution of eq. \ref{dif eq} and we use the arbitrariness
of the choice of $l$ in order to place the obstructions in the frequencies controlled by
the a priori estimates.

For this reason, we will abuse the notation of $I_{\r }$ and denote the set
$\{-m+1,...,0\}$ by $I_{m }$.

For the given choice of $l$, depending on $m$, we define $\hat{f} : \Z \ra \C $ by
\begin{equation*}
{\hat{f}}(k)=%
\begin{cases}
\hat{f}_{+}(k), & k> m \\ 
\hat{f}_{-}(k), & k\leq m
\end{cases}%
\end{equation*}
and $f : \T \ra \C $ is the inverse Fourier transform of $\hat{f} (\. ) $.

Since, for any $p \in \N $, $|k|^{p} | \hat{g}(k)| \ra 0 $ as $|k| \ra \infty$, we obtain the estimate
\begin{equation*}
\left\Vert f\right\Vert _{s}\leq C_{s}\left\Vert g\right\Vert _{s+3}
\end{equation*}%
and for truncations up to an arbitrary order $N$, we find that%
\begin{eqnarray*}
\left\Vert T_{N}f\right\Vert _{s} &\leq &C_{s}N^{2}\left\Vert g\right\Vert
_{s} \\
\left\Vert R_{N}f\right\Vert _{s} &\leq &C_{s,s^{\prime }}N^{s-s^{\prime
}+3}\left\Vert g\right\Vert _{s^{\prime }}
\end{eqnarray*}%
with $s\geq s^{\prime }\geq 0$. Clearly, if for any given $m$ we choose $l=-m$, the
constants in the estimates depend only on $m$.

We have therefore proved

\begin{proposition}
\label{solution of difference equation}The function $f=\mathcal{F}^{-1}(\hat{%
f})$ as defined above solves the equation%
\begin{equation*}
f(x+\a )-e^{2i\pi mx}f(x)=g(x)-\Gamma _{m}g(x)
\end{equation*}%
where $\mathcal{F}\{\Gamma _{m}g\}(k)=\hat{g}(k)-(e^{2i\pi (k-m)\a }\hat{%
f}_{-}(k-m)-\hat{f}_{+}(k))$ for $k\in \{-m+1,..,0\}$ and $0$ otherwise, is
a trigonometric polynomial of order at most $m-1$. The following estimates
hold for all $s\geq 0$:%
\begin{eqnarray*}
\left\Vert f\right\Vert _{s} &\lesssim &\left\Vert g\right\Vert _{s+3} \\
\left\Vert \Gamma _{m}g\right\Vert _{s} &\lesssim &\left\Vert g\right\Vert
_{s}
\end{eqnarray*}%
where the constants depend only on $m$ and $s$.

The mapping $\Phi :g\mapsto (f,\Gamma _{m}g)$ is an invertible
continuous linear application from $H^{s+3}(\T ,\R )$ to
$H^{s}(\T ,\R )\times \{\mathcal{F}^{-1}(\sum_{I_{m}}a_{k}e^{2i \pi k \. }),a_{k}\in \C \}$.
\end{proposition}

Summing up what we have proved, we can state the following proposition.

\begin{proposition} \label{local reduction prop}
\label{sol lin eq} Let $E_{r,a}(\. )=\exp (\sum_{\rho \in \tilde{\D}%
}2i\pi h_{\rho }(r_{\rho }\. +a_{\rho }))$ be a regular periodic
geodesic, $U(\. )=\sum_{\r \in \tilde{\D}}U_{\r }^{\ft }(\. ).h_{\r }
+\sum_{\r \in \D _{+}}U_{\r }(\. ).j_{\r }$ be small enough in
$C^{\tau +3 }(\T ,g)$ and $\a \in DC(\gamma ,\tau) $.

Then, there exist $B(\. )$ and $\Gamma _{\tilde{r}}U(\. )$ (notice the
abuse in the notation) in $C^{\infty }(\T ,g)$, where%
\begin{eqnarray*}
B(\. ) &=&\sum_{\rho \in \tilde{\D}}B_{\rho }^{\ft }(\.
).h_{\rho }+\sum_{\rho \in \D _{+}}B_{\rho }(\. ).j_{\rho } \\
\Gamma _{\tilde{r}}U(\. ) &=&\sum_{\rho \in \D _{+}}\Gamma _{\rho
}U_{\rho }(\. ).j_{\rho }=\sum_{\rho \in \D _{+}}\sum_{k\in I_{\rho
}}p_{\rho }(k)e^{2i\pi k\. }.j_{\rho }
\end{eqnarray*}%
solving the equation%
\begin{equation*}
Ad(E_{r,a}(\. )^{\ast }).B(\. +\a )-B(\. )=-U(\. )+\Gamma _{%
\tilde{r}}U(\. )+\sum_{\r \in \tilde{\D}}\hat{U}_{\rho }^{\ft }(0).h_{\r }
\end{equation*}
with the following estimates:
\begin{eqnarray*}
\left\Vert B(\. )\right\Vert _{s} &\leq & \gamma C_{s} \left\Vert U(\.
)\right\Vert _{s+\tau +2} \\
\left\Vert \Gamma _{\tilde{r}}U(\. )\right\Vert _{s} &\leq
&C_{s} \left\Vert U(\. )\right\Vert _{s}
\end{eqnarray*}%
where the constants depend only on $r$, but not on $a$ or $\a $.

The $B_{\rho }^{\ft }(\. )$ are unique up to an additive constant
in $\R $ (thus unique if we impose that their constant Fourier
coefficient be $0$). The rest $\Gamma _{r}U(\. )$ is uniquely determined
by $U(\. )$ and the choice of its spectral support.
\end{proposition}

Let us also define the spaces in which the obstructions take values

\begin{definition}
We will denote by $\mathcal{P}_{\tilde{r}}$ the image of the mapping
$\Gamma _{\tilde{r}}$ in $C^{\infty }(\T ,g)$ and by $\mathcal{P}_{%
\tilde{r}}^{0}$ the space $\ft \oplus \mathcal{P}_{\tilde{r}}$.
\end{definition}

This proposition indicates that a K.A.M.-like scheme can be used, just as in 
\cite{Krik2001}. The scheme is defined by replacing the perturbation $%
U(\. )$ by a Fourier-truncation of a sufficiently high order $N$, and
solving the linearized equation for the low frequencies. This reduces the
perturbation to the sum of a term of order $C_{s}N^{a}\left\Vert
U\right\Vert _{0}\left\Vert U\right\Vert _{s}$ and the rest $\Gamma
_{h}U(\. )$. The hypothesis that the cocycle is of degree $r$ gives the
estimate%
\begin{equation*}
\| \Lambda^{0}_{\tilde{r}}U(\. )\| _{s}\lesssim
\| (Id-\Lambda^{0}_{\tilde{r}})U(\. )\| _{s}
\end{equation*}
This estimate implies that the rest is in fact of second order
with respect to the original perturbation, and the reduction can be
iterated. The congergence of the scheme shows that all small enough perturbations of $(\a,E_{r,a}(\. ))$
of degree $r$ are $C^{\infty }$ conjugate to a $(\a,E_{r,a^{\prime }}(\. ))$.
The reason why this is a purely local theorem, i.e. if the conjugation
exists it can be chosen small, is the fact that the arithmetic properties
of the diagonal constant $a$ in $E_{r,a}(\. )$ are irrelevant and the only
small denominator phenomena are those appearing in the reduction of
the toral coordinates. Since, however, they are related only to $\a $,
and not to $a$, they are uniform throughout the K.A.M. scheme and
do not pose any problems to the convergence of the scheme.

This fact enables us to choose a Nash-Moser inverse function theorem
approach (see \cite{Ham1982}). The proof of such a theorem consist essentially in
the construction of a convergent classical K.A.M. scheme which replaces the fixed point argument
in the standard inverse function theorem for Banach spaces. Heuristically, the theorem
states that if the equation of reduction to the normal form has can be uniquely
solved in all the tangent spaces in a neighborhood of $0$ and with uniform
estimates and uniform loss of derivatives, then the K.A.M. scheme converges.
In other words, if we can make one step of the scheme (i.e. solve the linearized
equation) in each tangent space, but with uniform estimates and loss of derivatives,
then the inversion of the mapping by the Nash-Moser theorem does the rest of the convergent scheme.

Nonetheless, K.A.M. theory is not made obsolete by the Nash-Moser theorem,
since the latter is, as all such theorems, of purely local nature. As a
consequence, an inverse function approach can only construct conjugations
of the same order as the perturbation of the normal form. On the other hand,
one can construct reducible cocycles for which the hypothesis (apart from
regularity assumptions) of the theorem (i.e. the resolution of the linearized
equation with good estimates) fails. Consider, eg, $\a \in \T \setminus \Q$ and
$k \in \Z ^{*}$. Then, the cocycles in $\T \times SU(2)$ of the form
\begin{equation} \label{example in SU(2)}
(\a , A e^{U(\.)})
\end{equation}
with $A = \{ e^{i \pi k \a },0 \} _{SU(2)}$ and $U(\.) = \{ 0, \e e^{2i \pi k \. } \} _{su(2)}$
($\e \in \R ^{*} $ is small) can be reduced by the $2$-periodic conjugation
$E_{-k/2} (\.) = \{ e^{-i \pi k \. } , 0 \} _{SU(2)}$ to the constant cocycle
\begin{equation*}
(\a , E_{-k/2} (\a ) A \exp (\{ 0, \e  \} _{su(2)}))
\end{equation*}
as shows a direct calculation. The conjugation, however is of order one, while the perturbation
is of order $\e $. We refer the reader who is unfamiliar with this kind of problems
to Chapter \ref{Local theory} for the proof that the linearized equation cannot
be solved for this cocycle due to the fact that the constant $A$ is \textit{resonant}.
As a consequence, there is no good candidate for a conjugation of the same order
as the perturbation, and the inverse function theorem approach fails, even though
the cocycle is reducible. On the other hand, conjugation by $E_{-k/2} (\.)$ can be
integrated in a K.A.M. scheme as a process of \textit{reduction of resonances}
(as in, say, \cite{El2001}), if we renounce in constructing a convergent
sequence of conjugations.

After this digression, let us return to the subject of this chapter and point out
that the phenomena that obstruct the existence of a conjugation of the order of
the perturbation are absent, as indicates proposition \ref{local reduction prop}.
This fact results in the exactness of theorem \ref{NK local reducibility to regular geodesics}.

In order to establish the contrast between theorem \ref{NK local reducibility to regular geodesics}
and the phenomena observed in the local theory, we present, somewhat prematurely,
the equivalent to our theorem in the local case, as obtained by R. Krikorian
in \cite{KrikAst}, after having introduced some notation.
Let $u\in G$, then call $\Gamma _{u}$ a space complementary to $Im(Id-Ad_{u}) $ in $g$ and
\begin{equation*}
l_{u} : g \ra \ker (Id-Ad_{u})
\end{equation*}
be the projection with respect to the Cartan-Killing form.
The Fréchet space of smooth mappings $\T ^{d}\ra g$ such that $l_{u}(B(0))=0$
will be denoted by $\mathcal{E}_{u}$. An element $u$ of $G$
is called Diophantine with respect to $\a \in \T $ if the
eigenvalues of its adjoint action have this property, i.e. if
\begin{equation*}
| u_{\r } - k \a |_{\Z} \geq \dfrac{\gamma ^{-1}}{ |k| ^{\tau} }
\end{equation*}
for all $k \in \Z ^{d} \setminus \{ 0 \} $. We can now state
\begin{theorem}[R. Krikorian] \label{RK normal form}
Let $u_{0}\in G$ be in $DC_{\a }(\gamma ,\tau )$ and let $\Gamma =\Gamma
_{u_{0}}$ and $\mathcal{E=E}_{u_{0}}$. Let also $\a $ be Diophantine.
Then, there exist an $\varepsilon >0$ and $s>0$ such that if $u\in C^{\infty
}(\T ^{d},G)$ and $\left\Vert u-u_{0}\right\Vert _{s}< \e $,
there exists a unique couple $(B,C)\in \mathcal{E}\times \Gamma $ such that%
\begin{equation*} \label{normal form local case}
u(\. )=e^{C}.e^{B(\. +\a )}.u_{0}e^{-B(\. )}
\end{equation*}
The smallness condition $\varepsilon $ is given by $c.\gamma ^{c^{\prime }}$%
, where the constants $c$ and $c^{\prime }$ do not depend on $\gamma $.
\end{theorem}

The contrast between the two theorems resides in the factor $e^{C}$ of
eq. \ref{normal form local case}, which measures the failure of the
inverse function theorem to obtain reducibility. As will be made clear
in chapter \ref{Local theory} (and as has been made clear in \cite{KrikAst})
the factor $e^{C}$ is precisely due to phenomena resembling to those observed in the discussion
following eq. \ref{example in SU(2)}.

The local normal form theorem can be used in the proof of a theorem for the reducibility in
positive measure in the parameter space for generic one-parameter families.
However, the stronger theorem of reducibility in full measure is true, but
cannot be obtained by the inverse function theorem, since it is based on the
reduction of resonances.

\section{Preliminary transformations and notation} \label{Preliminary transformations and notation}

Let us denote by $\exp ^{-1}$ the inverse of the exponential mapping,
defined in a neighborhood of the $Id$, and taking values in a neighborhood
of $0\in g$. The space orthogonal to $\ft $ for the Cartan-Killing form, $\ft ^{\bot }$,
is the complex space spanned by the vectors $j_{\rho }$, $\rho \in \D_{+} $.
The orthogonal projection on $\ft $ will be denoted by $\pi _{\ft }$.
Finally, we will denote by $\mathcal{V}=C^{\infty }(\T ,g)$ and call
$\mathcal{W}$ a neighborhood of $0\in \mathcal{V}$.

We can easily eliminate the obstruction to the solution of eq. (\ref{lin eq
a}) using the following fact. The mapping
\begin{equation*}
h\mapsto \int \pi _{\ft }(\exp ^{-1}(e^{h}.e^{U(\. )}))
\end{equation*}%
is a local diffeomorphism in a neighborhood of $0$ in $\ft $, called 
$\mathcal{N} $ (cf. also \textit{Corollary 8.1} of \cite{Krik2001}). We
can therefore introduce a fibration of $\mathcal{W}$ into
$\mathcal{N}\times \mathcal{W} _{0}$, where
$\mathcal{W}_{0}\mathcal{=}\{U\in \mathcal{W} , \pi _{\ft }(\hat{U}(0))=0\}$.

Therefore, if $\left\Vert U(\. )\right\Vert _{0}\in \mathcal{W}$ is small
enough, there exists a unique couple $(h,U_{0}(\. ))\in \mathcal{N}\times 
\mathcal{W}_{0}$ such that%
\begin{eqnarray*}
E_{r,a}(\. ).A.e^{U(\. )} &=&E_{r,a}(\. ).e^{h}.e^{U_{0}(\. )} \\
&=&E_{r,a+h}(\. )e^{U_{0}(\. )}
\end{eqnarray*}%
where $U_{0}(\. )\in \mathcal{W}_{0}$ and the mapping $U(\. )\mapsto U_{0}(\. )$ is continuous from $C^{s}$ into itself,
for all $s$. The smallness condition on the perturbation can be weakened to a smallness condition on
$\| U(\. )-\pi _{\ft }(\hat{U}(0))\| _{0}$. Finally, we can chose $\mathcal{W}_{0}$ so that
$Ad(\TT ).\mathcal{W} _{0}=\mathcal{W}_{0}$.

\bigskip

Since the rest $C\in \Gamma $ in the local normal form theorem is needed in
order to absorb the topological obstruction to the resolution of the
linearized cohomological equation, we are lead to replace it by a more
general mapping, since the topological obstruction for the solution of (%
\ref{lin eq c}) is no longer a constant.

Essentially, the normal form theorem should measure the failure to solve the
equation%
\begin{equation*}
Conj_{\exp (B(\. ))}.(\a ,E_{r,a}(\. ).e^{U(\. )})=(\a
,E_{r,a}(\. ))
\end{equation*}%
This failure will be represented by the additional terms $e^{\lambda }$ and $%
e^{D(\. )}$ in the following formula:%
\begin{equation*}
E_{r,a}(\. ).e^{U(\. )}=e^{-B(\. +\a )}.E_{r,a+\lambda }(\.
).e^{B(\. )}.e^{D(\. )}
\end{equation*}%
This equation, solved for the perturbation term, gives%
\begin{equation}
e^{U(\. )}=E_{r,a}(-\. ).e^{-B(\. +\a )}.E_{r,a+\lambda }(\.
).e^{B(\. )}.e^{D(\. )}  \label{nor form eq}
\end{equation}%
The formula that we seek to obtain is%
\begin{equation*}
E_{r,a}(\. ).e^{U(\. )}=e^{-B(\. +\a )}.E_{r,a^{\prime }}(\.
).e^{P(\. )}.e^{B(\. )}
\end{equation*}%
with $P(\. )$ a Fourier polynomial taking values in $\ft ^{\bot }$%
.

\bigskip

We remind that for $X\in g$ small enough and $\D X$ in the tangent space
at $X$, we call%
\begin{equation*}
\widetilde{\D X}=K(X).\D X=(D(e^{X}).\D X).e^{-X}
\end{equation*}

We shall write $\D B=\sum \D B(k).e^{2i\pi kx}$ for a variation of $%
B(\. )\in \mathcal{V}$, where $\D B(k)=\sum_{\rho \in \tilde{\D}%
}\D B_{\rho }^{\ft }(k).h_{\rho }+\sum_{\rho \in \D%
_{+}}\D B_{\rho }(k).j_{\rho }$ are the Fourier coefficients. Mappings
will not be denoted by $B(\. )$, but simply by $B$, in order to avoid
confusion with the notation $\D B(k)$.

We also recall the notation $R_{\a } : x \mapsto x+\a $ for the translation
of $\T $ into itself.

\section{Proof of the normal form theorem}

Let us therefore consider the smooth mapping $\tilde{\Psi}:((\ft %
,0)\times (g,0)^{3}~;\T ,\ft )\rightarrow ((g,0)~;\T ,%
\ft )$%
\begin{equation*}
(\lambda ,X,Y,Z;~x,a)\mapsto \exp ^{-1}(E_{r,a}(-x).e^{-X}.E_{r,a+\lambda
}(x).e^{Y}.e^{Z})
\end{equation*}%
where $x$ and $a$ are presently seen as parameters.

We remark that (\ref{nor form eq}) can be solved by inverting the mapping
\begin{equation*}
\Psi (\lambda ,B,D)=\tilde{\Psi}(\lambda ,B\circ R_{a },B,D ; a)
\end{equation*}
and, therefore, the next step is to prove the
differentiability properties required by Hamilton's theorem and choose the
functional spaces so that the mapping will be invertible.

The rest of this section is thus occupied by the proof of
\begin{proposition} \label{applicability of Hamilton}
Suppose that $\a \in DC(\gamma ,\tau )$, $r\in \Z ^{w}$ is such
that $e_{r}\in g$ is regular and let $a\in \ft $, the unique maximal
torus passing by $e_{r}$. Then there exists $\mathcal{\tilde{W}}$, a
neighborhood of $0$ in the local Fr\'{e}chet manifold
\begin{equation*}
\mathcal{V}_{r}=\{(\lambda ,B,D)\in \ft \times 
\mathcal{E}_{0}\times \mathcal{E},~D \in Ad(e^{-B}).\mathcal{P}_{\tilde{r}}\}
\end{equation*}%
such that the mapping
\begin{equation*}
\begin{array}{cccc}
\Psi : & \mathcal{\tilde{W}} & \rightarrow  & \mathcal{E} \\ 
& (\lambda ,B,D) & \mapsto  & \exp ^{-1}(E_{r,a}^{*}(\.
).e^{B \circ R_{\a} }.E_{r,a+\lambda }(\. ).e^{-B}.e^{-D)})
\end{array}%
\end{equation*}%
satisfies the conditions of Hamilton's theorem (theorem \ref{Hamilton inverse function})
for all $(\lambda ,B,D) \in \mathcal{\tilde{W}}$.
Moreover, $\mathcal{\tilde{W}}$ can be chosen so that $(\lambda ,B,D)\in \mathcal{\tilde{W}}$ if
and only if $(\lambda ,B \circ R_{\theta },D \circ R_{\theta })\in \mathcal{\tilde{W}}$,
for all $\theta \in \T $. The size of the neighborhood in $\mathcal{E}$ is bounded by%
\begin{equation*}
\left\Vert u\right\Vert _{s_{0} }\leq Cst.\max (\gamma ,1)^{-2}
\end{equation*}%
where $s_{0} \in \N $ and the constants do not depend on $\gamma $, but only on $r$.
\end{proposition}

\subsubsection*{Regularity of the mapping}

Let us begin by proving
\begin{proposition}
$\Psi $ is a $C^{\infty}$-tame mapping $\mathcal{N}\times \mathcal{W}^{2} \rightarrow \mathcal{E}$ defined in a neighborhood of $0\in \ft \times \mathcal{E}^{2}$.
\end{proposition}

\begin{proof}
Let us call $U=\tilde{\Psi}(\lambda ,X,Y,Z)$. A direct calculation
shows that%
\begin{equation*}
K(U).\D U=Ad(E_{r,a}(-\. )e^{-X})(-\widetilde{\D X}+\D \lambda
+Ad(E_{r,a+\lambda }(\. )).\widetilde{\D Y}+Ad(E_{r,a+\lambda }(\.
).e^{Y}).\widetilde{\D Z}
\end{equation*}

We can now define $\Psi :\mathcal{N}\times \mathcal{W}^{2}\mathcal{%
\rightarrow E}$ by%
\begin{equation*}
\Psi (\lambda ,B,D)=\tilde{\Psi}(\lambda ,B\circ R_{\a },B,D;A)
\end{equation*}%
If we pose now $U=\Psi (\lambda ,B,D)\in \E $ and $\D  U=(\D \lambda ,\D B,\D D)$
(we remind that $\widetilde{\D U}=K(U).\D U$) we see that
\begin{equation*}
Ad(e^{B\circ R_{\a }}E_{r,a}(\. )).\widetilde{\D U}=
-\widetilde{\D B} \circ R_{\a }+\D \lambda +Ad(E_{r,a+\lambda }(\. )).\widetilde{\D B}+
Ad(E_{r,a+\lambda }(\. ).e^{B}).\widetilde{\D D}
\end{equation*}
which proves the proposition.
\end{proof}

\subsubsection*{Construction of the inverse}
We now try to restrict the space on which $\Psi $ is defined, in order to
obtain inversibility for the restricted mapping. More precisely, we will define
$\mathcal{V} _{r}$, the local Fréchet manifold of proposition \ref{applicability of Hamilton}
so that, if we still denote by $\Psi $ the mapping $\mathcal{V} _{r} \ra \mathcal{V} $, $D \Psi $
is bijective at each point of $\mathcal{V} _{r}$, and the inverse satisfies tame estimates.

We thus solve the equation 
\begin{equation*}
\D U^{\prime }=-\widetilde{\D B}\circ R_{\a }+Ad(E_{r,a+\lambda
}(\. )).\widetilde{\D B}+\D \lambda +Ad(E_{r,a+\lambda }(\.
).e^{B}).\widetilde{\D D}
\end{equation*}%
for any $\D U^{\prime }\in \mathcal{E}$, thus establishing the
surjectivity of $D\Psi $ in the tangent space at $(\lambda ,B,D)\in \mathcal{N}\times \mathcal{W}%
^{2}$. We will solve the equation in the following equivalent form%
\begin{equation}
\D U^{\prime \prime }=Ad(E_{r,a+\lambda }^{*}(\. )).\widetilde{\D B}%
\circ R_{\a }-\widetilde{\D B}-\D \lambda -Ad(e^{B}).\widetilde{\D D}
\label{eq 1}
\end{equation}%
where
\begin{equation*}
\| \D U^{\prime \prime } \|_{s} \leq C_{s} (1 + \| B \|_{s}) \| \widetilde{\D U} \|_{s}
\end{equation*}
It can be seen directly that $\pi _{\ft }(\widetilde{\D B}(0))$ is
not determined by the equation, so, in order to gain uniqueness for $\D B(0)) \in \mathcal{W}_{0}$, we will use
$\pi _{\ft }(\widetilde{\D B}(0))$ as free parameters.

\bigskip

Direct application of Proposition \ref{sol lin eq} grants the existence of $%
(\D b,\D p)\in \mathcal{V}_{0}\times \mathcal{P}_{\tilde{r}}^{0}$ such
that%
\begin{equation*}
\D U^{\prime \prime }=Ad(E_{r,a+\lambda }^{*}(\. )).\D b\circ
R_{\a }-\D b+\D p
\end{equation*}%
and $(\D b,\D p)$ satisfy tame estimates with respect to $\D U$,
uniformly for $a$ and $\lambda $:
\begin{eqnarray*}
\left\Vert \D b\right\Vert _{s} &\leq &C_{s}^{\prime }\left\Vert \D
U^{\prime \prime }\right\Vert _{s+\tau +3}\leq C_{s+\tau +3}^{\prime }
(1 + \| B \|_{s+\tau +3}) \| \widetilde{\D U} \|_{s+\tau +3}\\
\left\Vert \D p\right\Vert _{s} &\leq &C_{s}^{\prime }\left\Vert \D
U^{\prime \prime } \right\Vert _{s}\leq C_{s}^{\prime }
(1 + \| B \|_{s}) \| \widetilde{\D U} \|_{s}
\end{eqnarray*}
where $C_{s}^{\prime }=\max (1, \gamma ) C_{s}^{\prime \prime }(r)$. We will omit the prime in
the notation of this constant.

We can therefore rewrite eq. \ref{eq 1} in the form%
\begin{equation}
Ad(E_{r,a+\lambda }^{*}(\. )).\D b\circ R_{\a }-\D b+\D
p=Ad(E_{r,a+\lambda }^{*}(\. )).\widetilde{\D B}\circ R_{\a }-
\widetilde{\D B}-\D \lambda -\widetilde{\D P}  \label{eq.3}
\end{equation}%
We have called $\widetilde{\D P}=Ad(e^{B}).\widetilde{\D D}$, which by convexity inequalities satisfies
\begin{equation*}
\| \widetilde{\D P} \|_{s} \leq C_{s} (1+\|B \|_{s}) \| \widetilde{\D D} \|_{s}
\end{equation*}

\bigskip

The situation for the inversion of $\D D\ra \widetilde{\D P}$ and the decomposition of
$\widetilde{\D P}$ into a coboundary plus an obstruction, however, is slightly
more complicated. Observation of the formula that we seek to prove indicates that
we should seek $D$ in $Ad(e^{-B}).\mathcal{P}_{\tilde{r}}$. For such a $D$,
we have $\widetilde{\D D}=Ad(e^{-B}).\widetilde{\D P}$, where $\widetilde{\D P}=K(P).\D P$
for $P, \D P \in \mathcal{P}_{\tilde{r}}$, which justifies the notation $\widetilde{\D P}$.
For a given such $P$ and $\D P$, Proposition \ref{sol lin eq} gives the existence of $\D D_{cob}$
and a Fourier polynomial $\Gamma _{\tilde{r} }.\widetilde{\D P}$ such that%
\begin{equation*}
\widetilde{\D P}=Ad(E_{r,a+\lambda }^{*}(\. )).\D D_{cob}\circ R_{\a }-
\D D_{cob}+\Gamma _{r}. \widetilde{\D P}
\end{equation*}
The family of mappings $\mathcal{\tilde{P}}_{r} \circlearrowleft$
\begin{equation*}
\D P \mapsto \tilde{\Gamma}_{\tilde{r}}(P). \D P=\Gamma _{\tilde{r}}\circ K(P). \D P
\end{equation*}
is a family of linear mappings of between vector spaces of the same
finite dimension, depending on the parameter $P \in \mathcal{P}_{r}$. For $P=0$,
$\tilde{\Gamma}_{\tilde{r}} (0)=Id $, and the estimates in Proposition \ref{sol lin eq} imply that the
family depends continuously on $P$, so that $ \D D_{cob}$ and $\tilde{\Gamma}_{r}\D P$
satisfy uniform estimates with respect to $\D P$, for all $P$ sufficiently small in $C^{0}$:
\begin{eqnarray*}
\| \D D_{cob} \| _{s} &\leq &C_{s+\tau +3} (1+ \| P \| _{s+\tau +3}) \| \D P  \| _{s+\tau +3} \\
\| \tilde{\Gamma}_{\tilde{r}}(P).\D P \| _{s} &\leq & C_{s} (1+ \| P \| _{s}) \| \D P  \| _{s}
\end{eqnarray*}
Similar estimates are also satisfied with respect to the original variable $\D D$
\begin{eqnarray*}
\| \D D_{cob} \| _{s} &\leq &C_{s+\tau +3} (1+ \| B \| _{s+\tau +3}+ \| P \| _{s+\tau +3})
\| \D D  \| _{s+\tau +3} \\
\| \tilde{\Gamma}_{\tilde{r}}(P).\D P \| _{s} &\leq & C_{s} (1+ \| B \| _{s}+ \| P \| _{s})
\| \D D  \| _{s}
\end{eqnarray*}
so long as $\| B \| _{s+\tau +3}, \| P \| _{s+\tau +3} \leq 1$.

We can therefore solve the equation
\begin{equation*}
\tilde{\Gamma} _{r}.\D P=\pi _{\ft ^{\bot }}\D p
\end{equation*}
with good estimates and write the equation for the remaining terms in the form
\begin{equation*}
Ad(E_{r,a+\lambda }^{*}(\. )).\D b^{\prime }\circ R_{\a }-\D
b^{\prime }+\pi _{\ft }\D p=Ad(E_{r,a+\lambda }^{*}(\. )).\D 
\tilde{B}\circ R_{\a }-\widetilde{\D B}-\D \lambda
\end{equation*}
where $\D b^{\prime }=\D b-\D D_{cob}$. Projection on $\mathfrak{%
t}$ gives $\D \lambda =-\pi _{\ft }\D p$, and finally we
find $\widetilde{\D B}=\D b^{\prime } + \D c $, where $\D c \in \ft $ is to be determined.

\bigskip

Let us now study the inversion of the change of variables $\D B \mapsto \widetilde{\D B}$.

\begin{lemma}
\label{inversion}Let $\widetilde{\D B}\in \mathcal{V}_{0}$ satisfy \ref{eq.3}
(it is then automatically unique, by virtue of Proposition \ref{sol lin eq}%
). Then there exists a unique $\D c \in \ft$ such that
$\D B=K^{-1}(B). (\widetilde{\D B} + \D c ) \in \mathcal{V}_{0}$. Moreover, the mapping
$\widetilde{\D B} \ra \D B$ is tame.
\end{lemma}

\begin{proof}
Let us call $\D B^{\prime }=K(B)^{-1}.\widetilde{\D B}$. Since $B$ is to
be considered small, $\int K^{-1}(B)$ is invertible as for $\| B \|_{0}$ small enough
\begin{equation*}
\| \int K^{-1}(B) -Id \| \leq C \| B \|_{0}
\end{equation*}
in operator norm. Therefore, we can find a unique 
$\D c \in \ft $ such that $\pi _{\ft } \int K^{-1}(B). \D c = \pi _{\ft} \D
B_{\ft }^{\prime }(0)$. Let, now,
\begin{equation*}
\D B=\D B^{\prime }-K^{-1}(B).b
\end{equation*}%
Then, $\int \pi _{\ft }(\D B)=0$ and thus $\D B\in \mathcal{E%
}_{0}$. Additionally, $K(B).\D B^{\prime }$ and $K(B).\D B$ differ
only by a constant diagonal matrix, so that eq. \ref{eq.3} is satisfied with 
$K(B).\D B$ in the place of $\widetilde{\D B}$.
Finally,
\begin{eqnarray*}
\norm{\D B}_{s} &\lesssim & \norm{\D B^{\prime}}_{s } + (1+ \norm{ B}_{s} ) |\D c| \\
&\lesssim & (1+ \norm{ B}_{s} ) \left( \norm{\widetilde{\D B}}_{s} +
(1+ \norm{ B}_{0} )\norm{\widetilde{\D B}}_{s} \right)
\end{eqnarray*}
Admitting a uniform bound, say $\norm{ B}_{0} < 1$, we obtain the desired estimate
\begin{equation*}
\norm{\D B}_{s} \leq C_{s} (1+ \norm{ B}_{s} ) \norm{\widetilde{\D B}}_{s}
\end{equation*}
\end{proof}

We remark that, thanks to the local nature of the setting, the changes of
variables are continuous mappings from $C^{s}(\T ,\R )$
into itself, with norms bounded by constants depending only on $\mathcal{N}$
and $\mathcal{W}$, for each $s$. Therefore, the loss of derivatives is only
due to the inversion of the operator involved in the linearized equation
studied in the previous section.
\selectlanguage{english}
\chapter{Revisiting the local theory} \label{Local theory}
The next natural step in our study, since we follow the path "global-to-local",
should be to treat the case where the obstruction to reducibility
is intermediate, i.e. the perturbations of singular geodesics. As we will see in the next chapter,
the study of singular geodesics is reduced to a parallel application of the local reduction lemmas
of the local theory as we will develop it in this chapter with the K.A.M. scheme used in \cite{Krik2001}.

For this reason, we will present the local theory, for which the main reference is \cite{KrikAst}.
We will adapt and simplify the proof of theorem \ref{RK local density} proved in the reference. Additionally, we improve the local almost quasi-reducibility theorem obtained as a corollary of the proof
of theorem \ref{RK local density} to a local almost reducibility theorem (see chapter 2 for the 
definitions of these notions).

The proof of the local density theorem by R. Krikorian works in the context of a unitary representation of
$G$. If one reads carefully the proofs, they can see that the adjoint representation of $G$ on $g$ works just
as well. It is not a faithful representation, but since its kernel $Z_G$ is discrete, it does not interfere
significantly in the theory. In order to be able to use the Fourier transform, the unitary group in which
$G$ was embedded, was in turn embedded in a space of matrices with complex coefficients.
In our context, the role of the ambient matrix vector space is taken up by the linear transformations
of the real vector space $g$, an object which is again defined intrinsically.

Summing up, we re-obtain in a slightly more general context the well known theorem of the density of
reducible cocycles in the neighborhood of constant cocycles. We also prove the local almost-reducibility theorem,
which strengthens the almost quasi-reducibility theorem of R. Krikorian. In addition, the proof of the
density theorem is significantly simpler, since it does not use the reducibility in positive measure
for $1$-parameter families. This last element becomes important in the study of perturbations of
singular geodesics that we will take up in the next chapter, and allows an elegant proof of a slightly
sharper and more natural global density theorem.

\section{Notation}

Throughout this chapter, we suppose that a $1$-periodic perturbation
of a constant cocycle is given
\begin{equation*}
(\a ,A. e^{U(\. )})
\end{equation*}
We also suppose that $\a \in DC(\gamma ,\tau )$, i.e.
\begin{equation*}
|k\a |_{\Z }=dist(k\a ,\Z )\geq \frac{\gamma ^{-1}}{|k|^{\tau }}, \, k \in \Z ^{*}
\end{equation*}

We remind that we denote by $w$ the rank of $G$, by $q$ the number of
positive roots of $G$, and by $f$ the dimension of the real vector space $g$
and by $\ft = \ft _{a}$ a maximal toral algebra passing by $a$, a preimage of $A$ in $g$ by $\exp$.
If $\tilde{\D}$ are the roots in a basis of Weyl with respect to $\ft$,
we denote by $(H_{\r  })_{\r  \in \tilde{\Delta}}$ the dual basis of $\ft $.
The choice of a Weyl's basis induces a decomposition of $g$ into a direct sum of subspaces
\begin{equation*}
g = ( \oplus _{\r \in \tilde{\D}} \R H_{\r } ) \oplus ( \oplus _{\r \in \D _{+}}\C j_{\r } )
\end{equation*}
where the $(H_{\r } )$ span $\ft $ and the $(j_{\r } ) \cup (i j_{\r } )$, where $i=\sqrt{-1}$,
form an orthonormal basis of the space othogonal to $\ft $ with respect to the Cartan-Killing form.
We remind the reader that if $s \in \ft $, then
\begin{equation*}
[s, j_{\r }] = 2i \p \r (s) j_{\r } = 2i \p s_{\r } j_{\r }
\end{equation*}
in the notations that we have adopted.
For mappings $U(\. ) : \T \ra g $ we use the notation
\begin{eqnarray*}
U(\. ) &=& \sum\nolimits_{\r \in \tilde{\D }}U_{\r }^{\ft }(\. )H_{\r }
+\sum\nolimits_{\r \in \D _{+}}U_{\r }(\. )j_{\r } \\
&=& U ^{\ft }(\. ) +\sum\nolimits_{\r \in \D _{+}}U_{\r }(\. )j_{\r }
\end{eqnarray*}

Since for every $\r^{\prime }\in \D _{+}$, there exist positive integers $m_{\r ^{\prime},\r }$ such that
\begin{equation*}
\r  ^{\prime }=\sum_{\tilde{\Delta}}m_{\r  ^{\prime },\r  }\r 
\end{equation*}
there exist rational numbers $p_{\r  ^{\prime },\r  }=l_{\r  ^{\prime
},\r  }/D$ with $|l_{\r  ^{\prime },\r  }|\leq b$, such that
\begin{equation*}
H_{\r  ^{\prime }}=\sum_{\tilde{\D }}p_{\r ^{\prime },\r }H_{\r }
\end{equation*}
where $D,b\in \N ^{\ast }$ are considered fixed.

We also call $c=c_{G}$, the cardinal of $Z_{G}$.

Finally, a real number $\beta $ will be called \textit{Diophantine with respect to} $\a $
if it satisfies
\begin{equation*}
|\beta - k\a |_{\Z } \geq \frac{\gamma ^{-1}}{|k|^{\tau }}
\end{equation*}
The set of such numbers will be denoted by $DC_{\a } (\gamma , \tau)$.

\section{Synopsis of the chapter}

In this chapter, we revisit, simplify and sharpen the local theory, as known since
\cite{KrikAst}, namely theorem \ref{RK local density}, and obtain
\begin{theorem} \label{NK local density}
Let $\a  \in DC(\gamma ,\tau )$. Then, there exist $\epsilon >0$ and $%
s_{0}\in \N ^{\ast }$ such that, if $U(\.) \in C^{\infty} (\T ,g) $ satisfies $\left\Vert U(\.  )\right\Vert
_{s_{0}}<1$, $\left\Vert U(\.  )\right\Vert _{0}<\epsilon $ and $A\in G$,
then the cocycle $(\a  ,Ae^{U(\. )})$ is accumulated in the $C^{\infty
}$ topology by reducible cocycles.
\end{theorem}

This theorem was in fact already known in a slightly weaker form, as a comparison with theorem \ref{RK local density} can show.
Its proof was based on a local quasi-reducibility theorem (see definition \ref{def almost reducibility} and the discussion
below), and obtained from it in a slightly more complicated way than we derive it from the following local almost reducibility
theorem, which in turn was not known.
\begin{theorem} \label{NK almost reducibility}
Let $\a  \in DC(\gamma ,\tau )$. Then, there exist $\epsilon >0$ and $s_{0}\in \N ^{\ast }$ such that, if
$\| U(\.  )\| _{s_{0}}<1$, $\| U(\.  )\| _{0}<\epsilon $ and $A\in G$, then the cocycle $(\a  ,Ae^{U(\.  )})$
is almost reducible.
\end{theorem}

The proof of this theorem is grants control of the growth of the conjugants, which is polynomial,
versus an exponential decay of the perturbations. This fact grants the preceding density theorem. It is with the same
argument that R. Krikorian obtained his local density theorem from his almost quasi-reducibility, but with a complication.
The problem of longer periods (the difference between almost reducibility and quasi-reducibility) made it necessary to use
the classical K.A.M. theory and reducibility in positive measure in the parameter space. We managed to overcome this
difficulty and rid the argument of this complication by reinterpreting H. Eliasson's generalized K.A.M. scheme. We followed
the adapted version of it for compact Lie groups as in \cite{KrikAst}, but interpreted the reduction of resonances (see section
\ref{reduction section}), as the reduction of the good linear model in the presence of resonances. This fact allowed us
a better control of the procedure and the recovery of the loss of periodicity (inherent in the theory in general compact groups)
at each step of the K.A.M scheme, which gives almost reducibility, instead of quasi-reducibility.

These last two theorems admit global analogues. Since we use the convergence of
renormalization, we strengthen the arithmetic condition to a recurrent
Diophantine one and obtain the following statements.

\begin{theorem} \label{NK global almost reducibility in 0 energy}
Let $\a  \in RDC$ and $(\a  ,A(\.  ))\in SW_{\a  }^{\infty }(\T ,G)$ be of degree $0$, and let
$m \in \N ^{*}$ be the minimal natural number such that $(\a  ,A(\.  ))^{m }$ is homotopic to the $Id$.
Then, the cocycle $(\a  ,A(\.  ))^{m }$ is almost reducible.
\end{theorem}

The global analogue of the local density theorem is as follows

\begin{theorem} \label{NK global density in 0 energy}
Let $\a  \in RDC$. Then, reducible cocycles are dense in the class of $0$%
-degree cocycles in $SW_{\a  }^{\infty }(\T ,G)$ that are
homotopic to constants.
\end{theorem}

Since we have already seen that positive energy cocycles in $SW^{\infty }(%
\T ,SU(2))$ are accumulated by cocycles of smaller energy, by a
simple induction argument we can obtain the following theorem.

\begin{theorem} \label{NK SU(2) global density}
Let $\a  \in RDC$. Then, reducible cocycles are dense in $SW_{\a }^{\infty }(\T ,SU(2))$.
The same holds for cocycles in $SW_{\a }^{\infty }(\T ,SO(3))$ which are homotopic to constants.
If $(\a , A(\. )) \in SW_{\a }^{\infty }(\T ,SO(3))$ is not homotopic to constants, there exist
arbitrarily small perturbations of $(\a , A(\. ))^{2} \in SW_{2 \a }^{\infty }(\T ,SO(3))$
(thus one-periodic perturbations), such that $(\a , A'(\. ))^{2} \in SW_{2 \a }^{\infty }(2 \T ,SO(3))$
is accumulated by reducible cocycles (the perturbed cocycle is considered as a $2$-periodic cocycle).
\end{theorem}

\begin{proof}
The only part which is not clear is the one on cocycles in $\T \times SO(3)$ which are not homotopic to constants.

Let $(\a , A(\. ))$ be such a cocycle of energy $0$. Then it has renormalization representatives $\mod 2$
which are arbitrarily close to constants, and therefore it is accumulated (in $C^{\infty }(2\T ,SO(3))$) by
reducible cocycles.

If $(\a , A(\. ))$ has positive energy, then its second iterate can be conjugated to a normal form
$(2\a , R_{4r\pi (\. +\theta})$ by a $1$-periodic conjugation. Therefore, we can perturb $(\a , A(\. ))^{2}$
to a cocycle of $0$ energy. However, renormalization of the second iterate has to be done in
$SW^{\infty }(2 \T ,SO(3))$ if we want to draw conclusions for $(\a , A(\. ))$. Renormalization of the
action generated by $(2,Id)$ and $(2\a , A'(\. ))$, where $A'(\. ) = A_{2}(\. )e^{U(\. )}$ (after a change of scale
by a factor $2$) drives us to the case treated first in the proof.
\end{proof}

This complication will have to be dealt with in the proof of the main theorem of this mémoire in the case
of multiply connected groups.

The theorems in the local case are of course valid for cocycles with an arbitrary number
of frequencies, but their global analogues are not known to hold, since their proof requires
renormalization of the dynamics.

\section{Strategy of the proof of theorem \ref{NK almost reducibility}}

The proof of the  local almost reducibility theorem is reduced at a first time to writing and
solving with good estimates the reduction of a perturbation of a constant cocycle to a perturbation
of order two, and then iterating the procedure by means of a K.A.M. scheme.

The convergence of the scheme needed for the second part has been taken from \cite{FK2009}, so for the time being
we focus on the first part, the \textit{linear problem}.

Let us assume that a $Y(\. ):\T \ra g$, small enough, can reduce the perturbation
$U(\. )$ to $V(\. )$, with $V(\. )$ of order two with respect to $U(\. )$.
The equation satisfied by such a conjugant is
\begin{equation*}
e^{Y(\. +\a )}.A.e^{U(\. )}.e^{-Y(\. )}=A^{\prime }.e^{V(\. )}
\end{equation*}
or
\begin{equation*}
e^{Ad(A^{\ast }). Y(\. +\a )}.e^{U(\. )}.e^{-Y(\. )}=A^{\ast }A^{\prime }.e^{V(\. )}
\end{equation*}
Linearization of this equation under the smallness assumptions implies that
such a $Y(\. )$ must satisfy the equation%
\begin{equation} \label{linearized local eq}
Ad(A^{\ast }).Y(\. +\a )-Y(\. )=-U(\. )
\end{equation}
We introduce the coordinates
\begin{equation*}
U(\. )=\sum\nolimits_{\r \in \tilde{\D }}U_{\r }^{\ft }(\. )H_{\r  }
+\sum\nolimits_{\r \in \D _{+}}U_{\r }(\. )j_{\r }
\end{equation*}
where $U_{\r }^{\ft }(\. )$ are real and $U_{\r }(\. )$ complex $1$-periodic functions.
We will write $Y(\.  )$ in an analogous way. The linearized equation then decomposes
into two different types of equations as follows.

The simpler case is that of the abelian part, which reads
\begin{equation*}
Y_{\r }^{\ft }(\. +\a )-Y_{\r }^{\ft } (\.)=-U_{\r }^{\ft }(\. )
\end{equation*}
for which we refer the reader to lemma \ref{Linear cohomological eq}.

The second type of equation is the one that characterizes the reduction of perturbations of
constants. It reads
\begin{equation} \label{untruncated local eq}
e^{-2i\pi \r (a)}Y_{\r }(\. +\a )-Y_{\r }(\. )=-U_{\r}(\. )
\end{equation}%
for $\r \in \D _{+}$. Application of the Fourier transform gives
\begin{equation} \label{local eq in Fourier}
(e^{2i\pi ( k\a  - \r  (a) )}-1)\hat{Y}_{\r  }(k)=-\hat{U}_{\r }(k),~ k \in \Z
\end{equation}
Therefore, the Fourier coefficient $\hat{Y}_{\r }(k)$, $k\in \Z $, is not determined
by the equation if
\begin{equation*}
k\a  -\r  (a)\in \Z
\end{equation*}%
or equivalently if%
\begin{equation*}
a\in P(G)+\a  P(G)
\end{equation*}
For the definition of $P(G)$, see section \ref{root-space decomp}. If $a \notin P(G)+\a  P(G)$, but is close to it,
reduction of the corresponding Fourier coefficients deteriorates the estimates.

In order to avoid the loss of derivatives, a problem already
encountered in Lemma \ref{Linear cohomological eq}, we solve the eq. \ref{local eq in Fourier}
for only a finite number of coefficients, which we will determine depending on the arithmetic properties of $a$.

\bigskip

For this equation, we can distinguish three different possibilities.

The first one, which gives rise to a classical K.A.M. scheme as in \cite{KrikAst},
occurs when $\r (a)$ is far from $k \a $,
\begin{equation} \label{diophantine roots}
| \r (a) -k \a |_{\Z} \geq K ^{-1}
\end{equation}
for all $0 < |k| \leq N $ and some $K>0$, comparable with $N$. Unlike with classical K.A.M., we suppose explicitly that
the inequality \ref{diophantine roots} is violated for $k=0$. Let us denote by $I_{0}$
the roots $\r \in \D _{+}$ for which $a_{\r }$ satisfies this property. Then, we can
solve eq. \ref{untruncated local eq} for $0 <|k| \leq N$, i.e. we solve eq. \ref{linearized local eq}
with $U_{\r}(\. )$ replaced by the homogeneous truncaction $\dot{T} _{N} U_{\r}(\. )$,
and obtain a solution to the equation
\begin{equation*} 
e^{-2i\pi \r (a)}Y_{\r }(\. +\a )-Y_{\r }(\. )=-T_{N} U_{\r}(\. )+\hat{U}_{\r}(0 )
\end{equation*}
with $Y_{\r } (\. ) $ of the order of $U_{\r } (\. )$.

The second one, the most gentle in fact, occurs when $\r (a)$ is far from $k \a $,
\begin{equation*} 
| \r (a) -k \a |_{\Z} \geq K ^{-1}
\end{equation*}
for all $0 \leq |k| \leq N $, in which case we will say that it is Diophantine
with respect to $\a $. Let us denote by $I_{nr }$ the roots $\r \in \D _{+}$ for
which $a_{\r }$ satisfies this property. Then, we can solve eq. \ref{local eq in Fourier} for
$0 \leq |k| \leq N $, i.e. we solve eq. \ref{untruncated local eq} with $U_{\r}(\. )$ replaced by
the truncaction $T _{N} U_{\r}(\. )$, and obtain a solution $Y_{\r } (\. ) $ of the order of $U_{\r } (\. )$.

The third and last case occurs when the inequality \ref{diophantine roots} is violated
by some $k_{\r }$ with $0 < |k_{\r }| \leq N$. We will show that with no loss of generality we can
restrict our study to the case where $|k_{\r }| \leq N'$, with $N' < N$, but with $N-N'$ of the
order of $N$. Moreover, with a good choice of the parameters $K$ and $N$, we can ensure that
if such a $k_{\r } $ exists, then it is unique, and, additionally, that the inequality
\ref{diophantine roots} is verified for $0 < |k- k_{\r }| \leq N $. Finally, we will ensure that,
if $I_{r }$ is the set of such roots, then $I_{0} \cup I_{d} \cup I_{r} = \D _{+}$.
For the roots in $I_{r }$ we can only solve the equation
\begin{equation*}
e^{-2i\pi \r (a)}Y_{\r }(\. +\a )-Y_{\r }(\. )= -  \dot{T}_{N}^{(k_{\r })}
\end{equation*}
where
\begin{equation} \label{decentered truncation}
\dot{T}_{N}^{(k_{\r })} = \sum _{0 < |k- k_{\r }| \leq N } \hat{U} _{\r}(k) e^{2i \p k  \. }
\end{equation}
when again $Y_{\r }(\. )$ is of the order of $U_{\r}(\. )$.

Therefore, we have solved the equation \ref{linearized local eq}, but with
$U(\. )$ on the rhs replaced by
\begin{equation*}
\dot{T}_{N }U^{\ft }(\.)+ \sum_{\r \in I_{d}}T_{N }U_{\r }(\. )j_{\r }+
\sum_{\r \in I_{r}}\dot{T}_{N }^{(k_{\r })}U_{\r }(\. )j_{\r }+
\sum_{\r \in I_{0}}\dot{T}_{N }U_{\r }(\. )j_{\r }
\end{equation*}
If we call
\begin{equation*}
ObU(\. ) = \hat{U}^{\ft }(0)+\sum_{\r \in I_{r}}U_{\r}(k_{\r })e^{2i\pi k_{\r }\. }j_{\r }+
\sum_{\r \in I_{0}}\hat{U}_{\r }(0)j_{\r }
\end{equation*}
the equation that we have solved reaches the form
\begin{equation*}
Ad(A^{\ast }).Y(\. +\a ) +U(\. ) -Y(\. )= Ob U (\. ) + \text{Rest}
\end{equation*}
where the term "Rest" represents the rests of the truncation operators $\dot{T}_{N}^{(k_{\r }) }$,
$\dot{T}_{N}$ and $T_{N}$, accordingly with the subspace of $g$. The corresponding result for
the non-linear problem is the conjugaction of $(\a , A e^{U (\. )})$ to
$(\a , A e^{Ob U (\. )}e^{U' (\. )})$, with $U' (\. )$ of second order with respect to $U (\. )$.

We remark that, since we have solved a linearized equation, where second order terms are
neglected, linearization of the equation
\begin{equation*}
Conj_{\exp (Y(\. ))} (\a , A e^{Ob U (\. )}e^{\tilde{U} (\. )}) = (\a , A e^{Ob U (\. )}e^{U' (\. )} )
\end{equation*}
with $\tilde{U}  (\. ) = U (\. ) - Ob U (\. ) +O (|U| ^{2})$ and $U' (\. )$ of second order,
gives exactly the same equation, solution and estimates. We can therefore interpret the procedure
as linearization in the neighborhood of a the non-constant cocycle which incorporates the
obstructions to the solution of the linearized equation.

Iteration of such a conjugation lemma is thus not possible, since the term $Ob U (\. )$ is of
the order of $U (\. )$. Moreover, since $Ob U (\. )$ has non-zero Fourier coefficients
exactly where the linear theory fails, a different approach has to be taken.

This approach resembles to, without being identical with, the reduction of resonances of
\cite{El1992} and \cite{KrikAst}, and in this context it consists in constructing a conjugation
of a controlled norm that reduces the exact model around which we linearize, $(\a , A e^{Ob U (\. )})$,
to a constant. A good candidate for such a conjugation is a periodic geodesic $ e^{-H \. } : \T \ra \TT _{A}$, with
$\TT _{A}$ a maximal torus passing by $A$. We have, then,
\begin{equation*}
Conj_{\exp (-H\. )} (\a , A e^{Ob U (\. )}) = (\a , e^{-H \a }A \exp (Ad (e^{-H \.} ).Ob U (\. )) )
\end{equation*}
so that the only possibly non-constant term is $Ad (e^{-H \.} ).Ob U (\. )$. Given the form
of $Ob U (\. ) : \T \ra g$, we calculate
\begin{eqnarray*}
Ad (e^{-H \.} ).Ob U (\. ) &=& \hat{U}^{\ft }(0)+
Ad (e^{-H \.} ) ( \sum_{\r \in I_{r}}U_{\r}(k_{\r })e^{2i\pi k_{\r }\. }j_{\r }+
\sum_{\r \in I_{0}}\hat{U}_{\r }(0)j_{\r } ) \\
&=& \hat{U}^{\ft }(0)+
( \sum_{\r \in I_{r}}U_{\r}(k_{\r })e^{2i\pi (k_{\r } -\r (H) )\. }j_{\r }+
\sum_{\r \in I_{0}}\hat{U}_{\r }(0) e^{-2i\pi \r (H) \. } j_{\r } 
\end{eqnarray*}
Consequently, $Ob U (\. )$ can be made constant after such a conjugation if we can solve
the equations $\r (H) = k_{\r }$ for $\r \in I_{r}$, and $\r (H) = 0$ for $\r \in I_{0}$.
This is done in section \ref{reduction section}, where it is shown also that such an $H$ 
can be constructed so as to satisfy $ |H| \lesssim N' $, where we remind that $|k_{\r}| \leq N'$.

With such an $H$, we have
\begin{equation*}
Conj_{\exp (-H\. )} (\a , A e^{Ob U (\. )}e^{\tilde{U} (\. )}) = (\a , \tilde{A} \exp (Ad (e^{-H \.} ).U' (\. )) )
\end{equation*}
where $U' (\. )$ is quadratic with respect to $U (\. )$, and therefore we can iterate the procedure.

The second part of the reduction, the construction of the conjugation
$\exp (-H\. )$, shows that the obstructions to the solution of the linear equation are reducible
cocycles. Therefore, the linearization of the problem in the beginning was linearization in the
neighborhood of reducible, instead of constant, cocycles.

There is in fact a slight complication due to the fact that $e^{-H \.}$ is not in general $1$-periodic,
but this is resolved by post-conjugating with another periodic geodesic, as our approach allows us
to treat this problem of longer periods at each step of the algorithm, instead of the accumulated
loss of periodicity after having applied the algorithm a sufficient number of times.

\section{Resonances and reduction of resonant modes} \label{reduction section}

The preceding section has shown the importance of the relation of $\r (a) \mod \Z $, for $\r \in \D _{+}$,
with $k \a \mod \Z$. More precisely, if $\r (a) = k_{\r } \a \mod \Z $ for some $k_{\r } \in \Z $, then
the corresponding Fourier coefficient $\hat{U}_{\r } (k_{\r })$ of $U_{\r } (\. )$
cannot be eliminated in the linear equation. In the perturbative
setting, if this equality is approximate and $\r (a) - k_{\r } \mod \Z $ is small with respect to $N^{-1} $,
then reduction of the corresponding Fourier coefficient deteriorates the estimates. This motivates the following definition.
\begin{definition}
A constant $\beta \in \R$ such that%
\begin{equation*}
k\a  -\beta \in \Z
\end{equation*}%
for some $k\in \Z ^{\ast }$ is called ($\a $-)\textit{resonant}. The set of such constants
is denoted by $Res(\a )$ and its complement, the set of \textit{non-resonant} roots, $NR(\a )$.

Let $I \subset \Z ^{*} $. If there exists $k \in I$ such that
\begin{equation*}
\left\vert \beta -k\a \right\vert _{\Z }=dist(\beta -k\a ,\Z )\leq K^{-1}
\end{equation*}
we will say that $\beta $ is in $RS(\a ,K)$ in $I$. In the case where
$I = \{ k \in \Z , 0 < |k| \leq N \} $, the sets of such constants will be
denoted by $RS(\a ,N,K)$ and $DS(\a ,N,K)$ respectively. We will then say that
$\beta $ is $K$-resonant to the order $N$.

A constant $a\in g$ is in $RS(\a ,K)$ in $I$ (resp. in $RS(\a ,N,K)$) if there
exists a root in $\r \in \D _{+}$ such that $\r  (a)$ is in $RS(\a ,K)$ in $I$
(resp. in $RS(\a ,N,K)$). If no such root exists, $a\in DS(\a ,K)$ in $I$
(resp. in $DS(\a ,N,K)$).
\end{definition}

We stress that we consider $0$ as non-resonant, i.e. we do not authorise $k=0$
in the definition of $Res$ and that of $RS $.

With this notation, for any given $a \in g $ and $\ft =\ft _{a} $, a maximal toral
algebra passing by $a$, we can introduce a partition of the roots in $\D _{+}$ in
\begin{eqnarray} \label{partition of roots}
I_{r} &=&\{\r \in \D _{+},~\r (a)=k_{\r }\a +l_{\r }\in \a \Z ^{\ast } + \Z \} \nonumber \\
I_{0} &=&\{\r \in \D _{+},~\r (a)=l_{\r }\in \Z \} \\
I_{d} &=&\{\r \in \D _{+},~\r (a)\not\in \a \Z + \Z \}=\D _{+}\backslash (I_{r}\cup I_{0}) \nonumber
\end{eqnarray}
Let us also associate to each $\r \in I_{r}$ the (unique) $k_{\r }$ satisfying
$\r (a)=k_{\r }\a +l_{\r }$. This notation is to be fixed throughout the chapter.
We can now state

\begin{lemma} \label{reduction of resonant modes}
Let $a \in \ft \subset g$, call $(a_{\r })_{\r \in \D _{+}}=(\r (a))_{\r \in \D _{+}}\in \C ^{q}$
and define the corresponding partition of the roots in $\D_{+}$, as in eq. \ref{partition of roots}.
Then, there exists $H\in \ft$ such that, for all $\r \in I_{r}$, we have $\r (H) = k_{\r }$, and
for all $\r \in I_{0}$, $\r (H) = 0$. The vector $H$ is of the form
$\sum \dfrac{k^{\prime}_{\r }}{D} H_{\r  }$, where the summation is over a subset of
$( I_{r} \cup I_{d} ) \cap \tilde{\D }$, and $k_{\r  }^{\prime }\in \Z $.
\end{lemma}

We remind that $q=\#\Delta _{+}$ denotes the number of positive roots, which
are not linearly independent. We also recall that any root $\r  ^{\prime }$ in
$\D _{+}\backslash \tilde{\D }$ can be written in the form $\sum_{\r  \in 
\tilde{\Delta}}m_{\r  ^{\prime },\r  }\r  $ where the $m_{\r  ^{\prime
},\r  }$ are positive integers. The entries of inverses of invertible
submatrices of $(m_{\r  ^{\prime },\r  })$ are rational numbers of the form
$l_{\r  ,\r  ^{\prime }}/D$, where $|l_{\r  ,\r  ^{\prime }}|\leq e$ is
an integer and $D\in \N ^{\ast }$ and $l_{\r  ,\r  ^{\prime }}$ and 
$D$ are not necessarily coprime.

\begin{proof}
Let us now call
\begin{eqnarray*}
\tilde{I}_{r} &=&I_{r}\cap \tilde{\D } \\
\tilde{I}_{0}&=&I_{0}\cap \tilde{\D } \\
\tilde{I}_{d}&=&I_{d}\cap \tilde{\D }
\end{eqnarray*}%
It is then clear that%
\begin{equation*}
I_{r} \subset vect(\tilde{I}_{r},\tilde{I}_{d})\cap \Delta _{+}
\end{equation*}
so that we can chose $I_{r}^{l}$ a linearly independent subset of $I_{r}$,
and complete it with $I^{l}_{0}$, a maximal set of linearly independent roots in
$ I_{0} \setminus \mathrm{vect} (\tilde{I}_{0} )$ in order to obtain a maximal linearly independent
subset of roots in $\tilde{I}_{r}\cup \tilde{I}_{d}$. Then, we can solve the linear system%
\begin{eqnarray*}
\sum_{\r  \in \tilde{I}_{r}\cup \tilde{I}_{d}}m_{\r  ^{\prime },\r 
}u_{\r  } &=&k_{\r  ^{\prime }},~\r  ^{\prime }\in I_{r}^{l} \\
\sum_{\r  \in \tilde{I}_{r}\cup \tilde{I}_{d}}m_{\r  ^{\prime },\r 
}u_{\r  } &=& 0,~\text{if}~\r  ^{\prime} \in I^{l}_{0}
\end{eqnarray*}%
and the $u_{\r  }$ are rational numbers of the form $k_{\r  }^{\prime }/D$, and $k_{\r  }^{\prime }=Dk_{\r  }$
if $\r  \in \tilde{I}_{r}$. The vector
\begin{equation*}
\sum_{\r  \in \tilde{I}_{r}\cup \tilde{I}_{d}} \dfrac{k_{\r  }^{\prime }}{D} H_{\r }
\end{equation*}
by construction satisfies the conclusions of the lemma, since the resonance of a linear combination of resonant roots
is the linear combination of the resonances: if
$\r ' = \sum\nolimits _{\r  \in \tilde{I}_{r}} m_{\r  ^{\prime },\r } \r $
is resonant, $k_{\r '} = \sum\nolimits _{\r  \in \tilde{I}_{r}}m_{\r  ^{\prime }, \r },k_{\r }$.
\end{proof}

\begin{remark}
If we suppose that $|k_{\r  }|\leq N$, for all $\r  \in I_{r}$,
then there exists a constant $b$ depending only on $G$ such that
$|k^{\prime }_{\r }|\leq b D N$, for all
$\r  \in \tilde{I}_{r}\cup \tilde{I}_{d}$. This follows from the fact
that the resonances $k_{\r }$ for the roots in $I_{r}$ and the integers $k^{\prime }_{\r }$
defined for roots in $\tilde{I}_{r} \cup \tilde{I}_{d}$ are related by submatrices of the
Cartan matrix of the group.

The notation for the integers $k^{\prime }_{\r }$
as solutions to the linear system of the proof of lemma \ref{reduction of resonant modes}
is also to be considered fixed throughout the chapter.
\end{remark}

Before stating a direct corollary, we remind that for a function
$f : \T \ra \mathbb{K} $, where $\mathbb{K} = \R, \C $, we call
$\s (f) = \{ k \in \Z, \hat{f} (k) \not= 0 \} $.

\begin{corollary} \label{reducibility of obstructions}
\label{reduction to cst} Let $A\in G$, $U(\.  )\in C^{\infty }(\T ,g)$ and $H$ constructed
as in lemma \ref{reduction of resonant modes}. Then, if $U(\.  )$ is spectrally supported in the
obstructions and the resonances,
\begin{eqnarray*}
\sigma (U_{\r}^{\ft}(\. )) & \subset &\{ 0 \} \\
\sigma (U_{\r}(\. )) & \subset &\{ 0 \} \text{ if } \r \in I_{0} \\
\sigma (U_{\r}(\. )) & \subset &\{k_{\r} \} \text{ if } \r \in I_{r} \\
\sigma (U_{\r}(\. )) &= &\emptyset \text{ otherwise}
\end{eqnarray*}
the cocycle
$(\a  ,Ae^{U(\.  )})$ is reducible $\mod 1$.
\end{corollary}

\begin{proof}
If $U(\.  )$ satisfies the hypothesis of the corollary, then it has the form
\begin{equation}
U(\.  ) = \hat{U}^{\ft }(0)+  \sum_{\r \in I_{r}}U_{\r}(k_{\r })e^{2i\pi k_{\r }\. }j_{\r }+
\sum_{\r \in I_{0}}\hat{U}_{\r }(0)j_{\r } 
\end{equation}
If, now, the vector $H$ satisfies the conclusions of lemma \ref{reduction of resonant modes},
we find directly that
\begin{eqnarray}
e^{-H(\.  +\a  )}Ae^{U(\.  )}e^{H\.  } &=& e^{-H\a  }A\exp ( Ad(e^{-H \. }).U(\. )) \\
&=& e^{-H\a  }A\exp ( Ad(e^{-H \. }).( U^{\ft}(0) + \sum U_{\r }(\. )j_{\r })
\end{eqnarray}
and we calculate
\begin{eqnarray*}
Ad(e^{-H \. }).( U^{\ft}(0) + \sum U_{\r }(\. )j_{\r } )&=&  U^{\ft}(0) +
Ad(e^{-H \. }). ( \sum_{\r \in I_{r}} e^{2i\pi k_{\r  }\.} \hat{U}_{\r  }(k_{\r })j_{\r  }
+ \sum_{\r \in I_{0}} \hat{U}_{\r  }(0 )j_{\r  }) \\
&=&  U^{\ft}(0) + \sum_{\r \in I_{r}} e^{2i\pi ( k_{\r  } - \r (H)) \.} \hat{U}_{\r  }(k_{\r })j_{\r  }
+ \sum_{\r \in I_{0}} \hat{U}_{\r  }(0 )j_{\r  }
\end{eqnarray*}
Since, by construction of the vector $H$, $k_{\r  } = \r (H) $, we find that
\begin{equation*}
Ad(e^{-H \. }).U(\. ) = s \in g
\end{equation*}
This shows that
\begin{equation*}
Conj_{\exp (-H \.)} (\a ,Ae^{U(\.  )}) = (\a , \tilde{A})
\end{equation*} 
i.e. that $(\a ,Ae^{U(\.  )})$ is reducible $\mod c D$ via the torus morphism $e^{-H\.  }$.

Since $e^{-H(\. +1)}e^{H\. }$ is constant, we can imitate the last part of the proof of proposition
\textit{2.2.4} in \cite{KrikAst} and regain periodicity.

Firstly, we show that there exists a maximal torus passing both by $e^{- H }$ and $\tilde{A}$, where
\begin{equation*}
\tilde{A}=e^{-H\a  }A\exp (s)
\end{equation*}
We initially show that $s \in g$ commutes with $e^{H}$:
\begin{eqnarray*}
Ad( e^{-H}). s &=&  U^{\ft}(0) + Ad( e^{-H}).( \sum_{\r \in I_{r}}  \hat{U}_{\r  }(k_{\r })j_{\r  } +
\sum_{\r \in I_{0}} \hat{U}_{\r  }(0 )j_{\r  }) \\
&=&  U^{\ft}(0) + ( \sum_{\r \in I_{r}}  \hat{U}_{\r  }(k_{\r }) e^{- 2i \p \r (H)} j_{\r  } +
\sum_{\r \in I_{0}} \hat{U}_{\r  }(0 )  e^{- 2i \p \r (H)} j_{\r  }) \\
 &=&  s
\end{eqnarray*}
since for $\r \in I_{r} \cup I_{0} $, $\r (H) \in \Z $. Therefore, $e^{-H }$ and $e^{-H\a  }A\exp (\lambda s)$
commute for $\lambda \in [0,1]$. Since $e^{-H\a  }A$ and $e^{-H\.  }$ are on the same maximal torus
$\TT _{A} \subset Z^{0}_{G}(e^{-H}) $ (the neutral component of $Z_{G}(e^{-H})$),
this implies that $\tilde{A} \in Z^{0}_{G}(e^{-H})$, and consequently there exists
$\TT _{\tilde{A}}$, a maximal torus passing both by $e^{-H  }$ and $\tilde{A}$.

Therefore, there exists a group morphism $C(\.  ): \R \ra \TT _{\tilde{A}}$ such that $C(1)=e^{-H}$, so that
\begin{eqnarray*}
Ae^{U(\.  )} &=&e^{-H(\.  +\a  )}\tilde{A}e^{H\.  } \\
&=&e^{-H(\.  +\a  )}C(\.  +\a  ).C(-\a  )\tilde{A}C^{\ast }(\.  )e^{H\.  } \\
&=&e^{-H(\.  +\a  )}C(\.  +\a  ).A'.C^{\ast }(\.  )e^{H\.  }
\end{eqnarray*}%
Therefore, $Ae^{U(\.  )}$ is reducible by $C^{\ast }(\.  )e^{-H\.  }$
which is $1$-periodic since%
\begin{eqnarray*}
C^{\ast }(\.  +1)e^{-H(\.  +1)} &=&C^{\ast }(\.  )C^{\ast
}(1)e^{-H}e^{-H\.  } \\
&=&C^{\ast }(\.  )e^{-H\.  }
\end{eqnarray*}
\end{proof}

This corollary concludes what we need to prove in the linear algebraic
context before passing on to the regime of perturbation theory. The
connection of these results with K.A.M. theory is made with the following
lemma. It is a generalization (already found in \cite{KrikAst}) of a rather simple observation, which
constitutes the basis of L. H. Eliasson's results for quasiperiodic skew-systems in $\T ^{d}\times SO(3)$
(see \cite{El1988}) and those based upon them. It asserts that for a suitably chosen $K$, big enough with respect to $N$,
there can be at most one resonant mode.

\begin{lemma}
\label{Dioph csts}Let $\a \in DC(\gamma ,\tau )$, $\,N^{\prime },N,\in \N ^{\ast }$,
and $K>0$ such that $K\geq 2^{\tau +1}\gamma N^{\tau }$. Then,
\begin{enumerate}
\item If $\beta \in RS(\a ,N,K)$, there exists a unique $k_{0}$, for
which%
\begin{equation*}
\left\vert \beta -k \a \right\vert _{\Z }<K^{-1},~0<|k|\leq N
\end{equation*}

Moreover,
\begin{equation*}
|\beta -k\a  |_{\Z} \geq  K^{-1},~0<|k- k_{0} |\leq N^{\prime }
\end{equation*}
for any $N^{\prime }$ such that $K\geq 2\gamma  N^{\prime ~\tau }$ and
\begin{equation*}
|\beta -k \a |_{\Z } \geq  K^{-1},~0<|k-k_{0}|\leq N^{\prime }
\end{equation*}
for any $N^{\prime }$ such that $K\geq 2\gamma N^{\prime ~\tau }$.
In particular, $|\beta |_{\Z }\geq K^{-1}$.

\item If $\beta \in RS(\a ,N,K)$, then $\beta$ is non-resonant (in the linear-algebraic sense, i.e
$\beta - k \a \not\in \Z$) to the order $N_{nr}<(\gamma^{-1} K)^{1/\tau }$ for $k \in \Z ^{\ast }$
with $0 < |k-k_0| \leq N_{nr}$.
\item If $\beta \in RS(\a ,N,K)$, then $| \beta |_{Z} \geq K^{-1} $
\end{enumerate}
\end{lemma}

\begin{proof}
\begin{enumerate}
\item Let $k_{1},k_{2}\in \Z $ with $0<|k_{i}|\leq N$ and $k_{1}\not=k_{2}$, such that
$\left\vert \beta -k_{i} \a \right\vert _{ \Z }\leq K^{-1}$, for $i=1,2$, i.e.
\begin{equation*}
\left\vert \beta -k_{i} \a -l_{i} \right\vert \leq K^{-1}
\end{equation*}%
for some $l_{i}\in \Z $. We have therefore, with $k=k_{1}-k_{2}$, so
that $0<|k|\leq 2N$, and $l=l_{1}-l_{2}$%
\begin{equation*}
\left\vert k\a -l\right\vert \leq 2K^{-1}
\end{equation*}%
Since $\a \in DC(\gamma ,\sigma )$, $k\not=0$ implies%
\begin{equation*}
\left\vert k\a \right\vert _{\Z }\geq \frac{\gamma ^{-1}}{%
|k|^{\tau }}\geq \frac{\gamma ^{-1}}{(2N)^{\tau }}
\end{equation*}
This is possible only if
\begin{equation*}
\frac{\gamma ^{-1}}{(2N)^{\tau }}<2K^{-1}
\end{equation*}%
which contradicts the choice of $K$ and $N$.
\item For the second part, we need only remark that
\begin{eqnarray*}
| \beta -k\a  | _{\Z } &=& | \beta - k_{0} \a +k-k_{0} \a | _{\Z } \\
&\geq & | k-k_{0} \a | _{ \Z }-| \beta -k_{0} \a | _{ \Z } \\
&\geq & \frac{\gamma ^{-1}}{|k-k_{0} |^{\tau }}-K^{-1} \\
&\geq &K^{-1}
\end{eqnarray*}%
provided that%
\begin{equation*}
K\geq 2\gamma |k-k_{0} |^{\tau }
\end{equation*}
\item Finally, since
\begin{equation*}
| \beta -k_{0} \a -l_{0} | \leq K^{-1}
\end{equation*}%
we have%
\begin{eqnarray*}
|\beta |_{\Z } &\geq &| k_{0} \a | _{\Z }-K^{-1} \\
&\geq &\frac{\gamma ^{-1}}{|k_{0}|^{\tau }}-K^{-1} \\
&\geq &2^{\tau +1}K^{-1}-K^{-1} \\
&\geq &K^{-1}
\end{eqnarray*}
\end{enumerate}
\end{proof}

\bigskip

If we fix $K$ and $N^{\prime }$ satisfying $K\geq 2\gamma N^{\prime ~\tau }$ and we assume that resonances
occur only to the order $N<N^{\prime }$, we have the following corollary.

\begin{corollary}
\begin{enumerate}
\item If $\beta \in RS(\a ,N,K)$, then $|\beta |_{\Z }\geq (2(%
\frac{N^{\prime }}{N})^{^{\tau }}-1)K^{-1}$.

\item If $|\beta |_{\Z }<K^{-1}$, then $\beta \in DS(\a ,N^{\prime },K)$.
\end{enumerate}
\end{corollary}

This lemma and its corollary are necessary in order to apply the
linear-algebraic reduction of resonances to a K.A.M. scheme, since in the
latter case constants are only approximately Diophantine, $0$ or resonant. A
resonant constant can have only one resonant mode, but, if the constants are
not chosen suitably, they may be more than one when we allow closeness to
resonances up to $K^{-1}$. On the other hand, we need to be able to solve
exactly the linear system that reduces the resonant modes to constants.

\section{The local conjugation lemma} \label{The local conjugation lemma, local theory}

Before stating and proving the proposition which will be used in the K.A.M. scheme, we will fix some notation.
We let $b$ be the maximum of the operator norms of (inverses of) square submatrices of $(m_{\r  ^{\prime },\r  })$%
, $D$ as before and $\tilde{b}=2b$. For a given $N \in \N ^{\ast }$ we let $N_{0}=N$, $N_{i}=\tilde{b}N_{i-1}$,
$i=1,...,q+1$. To every order of truncation we associate the sets of roots%
\begin{eqnarray*}
I_{r}^{(i)} &=&Res(K,N_{i}) \\
I_{0}^{(i)} &=&\{\r  \in \D _{+},~|a_{\r  }|<K^{-1}\} \\
DS^{(i)} &=&DS(K,N_{i}) \\
I_{d}^{(i)} &=&DS^{(i)}\backslash I_{0}^{(i)}
\end{eqnarray*}
We also define the $k_{\r } \in \Z $ equal to $0$ for all roots in $I_{0}^{(i)}$, and
$k_{\r  }$ if $\r  \in I_{r}^{(i)}$, and we will sometimes omit the superscript $(i)$ for simplicity
in notation. We remark that if we fix a constant $A$ in $G$ and a maximal torus passing by it, the
root-space decomposition of $g$ results in the canonical definition of the Fourier modes of
mappings in $C^{\infty }(\T ,g)$ in
\begin{equation*}
\fZ = (\bigsqcup _{\r \in \tilde{\D }} \Z ^{\ft }_{\r }) \bigsqcup (\bigsqcup _{\r \in \D _{+} } \Z _{\r })
\end{equation*}
where the first factor accounts for the coordinates in the maximal torus (and therefore for real functions)
and the second for $\C $-valued functions in the directions $j_{\r }$. The spectrum $\s $ of a mapping
$U(\. ) \in C^{\infty }(\T ,g)$ is viewed in a canonical way as a subset of $\fZ $ as
\begin{equation*}
\s (U) = (\bigsqcup _{\r \in \tilde{\D }} \s (U^{\ft }_{\r })) \bigsqcup (\bigsqcup _{\r \in \D _{+} } \s (U_{\r }))
\end{equation*}
In this notation, we can define $\bk = \bk _{a} \subset  \fZ $ such that $(\bk )_{\r} =\{ k_{\r } \}$ if
$\r \in I_{0} \cup I_{r }$ and $(\bk )_{\r} = \emptyset $ otherwise, so that the hypothesis of lemma
\ref{reducibility of obstructions} amounts to $\s (ObU(\. )) \subset \bk _{A} $. We will need the notation
$|\bk | = \max _{\r \in \D ^{+}} |(\bk )_{\r} |$ where by convention $|\emptyset  | = 0$.

Clearly, the frequencies of mappings in $C^{\infty }(D\T ,g)$ belong to the finer set
\begin{equation*}
\fZ _{D} = (\bigsqcup _{\r \in \tilde{\D }} \dfrac{1}{D}\Z ^{\ft }_{\r }) \bigsqcup
(\bigsqcup _{\r \in \D _{+} } \dfrac{1}{D}\Z _{\r })
\end{equation*}
and $\fZ \hookrightarrow \fZ _{D} $ in a canonical way. Even though $\fZ $ is not a lattice,
the structure of additive group within each components remains of interest, since the adjoint action of
a $cD$-periodic geodesic translates the frequencies within each component of $\fZ $ without mixing them:
\begin{equation*}
\s (Ad(e^{H\.}).U(\. )) = (\bigsqcup _{\r \in \tilde{\D }} \s (U(\. ) ^{\ft }_{\r })_{\r}) \bigsqcup
(\bigsqcup _{\r \in \D _{+} } ( \s (U_{\r }(\. )) - \dfrac{k_{\r }^{\prime}}{D})_{\r }
\end{equation*}
if $H$ is of the form $\sum \dfrac{k^{\prime}_{\r }}{D} H_{\r  }$.

To such a partition of the roots, we associate the truncation operators
\begin{eqnarray} \label{truncation operators local theory}
T_{N_{i }}^{(\bk )}U(\. ) &=&\dot{T}_{N_{i }}U^{\ft }(\.)+
\sum_{\r \in I_{d}}T_{N_{i }}U_{\r }(\. )j_{\r }+
\sum_{\r \in I_{r}}\dot{T}_{N_{i }}^{(k_{\r })}U_{\r }(\. )j_{\r }+
\sum_{\r \in I_{0}}\dot{T}_{N_{i }}U_{\r }(\. )j_{\r } \\
ObU(\. ) &=&\hat{U}^{\ft }(0)+\sum_{\r \in I_{r}}U_{\r}(k_{\r  })e^{2i\pi k_{\r }\. }j_{\r }+
\sum_{\r \in I_{0}}\hat{U}_{\r }(0)j_{\r } \\
R_{N_{i }}^{(\bk )}U(\.  ) &=&U(\.  )-T_{N_{i }}^{(k)}U(\.  )-ObU(\. )
\end{eqnarray}
and we refer the reader to eq. \ref{decentered truncation} for the definition of the less familiar quantities.
Finally, we define the pierced box
\begin{equation*}
\mathcal{\dot{B}}(\bk ,N_{i }) \subset
(\bigsqcup _{\r \in \tilde{\D }} \Z _{\r }) \bigsqcup (\bigsqcup _{\r \in \D _{+} } \Z _{\r })
\end{equation*}
to be the spectral support of mappings in the image of $T_{N_{i }}^{(\bk )}$, and the box
$\mathcal{B}(\bk ,N_{i })$ as the spectral support of mappings in the image of $T_{N_{i }}^{(\bk )}+Ob$.

\begin{proposition}
Let $\a  \in DC(\gamma ,\tau )$, $N \in \N ^{\ast }$, and $K>0$ such that
$K\geq 2^{\tau +1}\gamma \tilde{b}^{q+1} N^{\tau }$. Then for any $A\in G$, there exists
$0\leq i\leq q$ and $\bk  = \bk _{A}\in \fZ $ with $|\bk _{A}|\leq N_{i}$, such that

\begin{enumerate}
\item $Ad(A)\in DS(\a  ,K)$ in $\mathcal{\dot{B}}(\bk _{A},N_{i+1} )$ and
is non-resonant in $\mathcal{\dot{B}}(\bk _{A},2N_{i+1})\subset \fZ $.

\item If $\sigma (U(\. )) \subset \bk $, then there exists a $cD$ -periodic torus morphism
$B(\. )=\exp (-H_{A}\.  ):cD \T \ra \TT _{A}$ with $\s (Ad(B(\. )))$ included in
$ \mathcal{B} ( \emptyset ,bDN_{i})\times \{ \emptyset \}\subset \fZ _{D}$, reducing $(\a ,Ae^{U(\.  )})$
to a constant $(\a  ,A^{\prime })$.
\end{enumerate}
\end{proposition}

We remind that in fact $(\a  ,Ae^{U(\.  )})$ is reducible modulo $1$. In order to establish this stronger
fact, one need only follow the proof of corollary \ref{reducibility of obstructions} in this perturbed
setting. The vectors $H$ and $s$ remain the same ("resonant modes are exact"), but $A$ in the perturbed setting
is at most $K^{-1}$-away from the $A$ considered in the exact setting ("resonances are approximate"). This
indefiniteness of $A$ only tilts the torus in which we regain periodicity (the mapping $C(\. )$ of the proof),
but the argument proving the existence of the torus remains valid since $K^{-1}$ is small, and $Ad(A).H=H$ and
$Ad(e^{H}).s=s$ in both cases.

\begin{proof}
Let $A\in G$, $\TT _{A}$ a maximal torus passing by $A$, and $%
(e^{2i\pi a_{\r  }})_{\r }$ the eigenvalues of the adjoint action of $A$. For $%
0\leq i\leq q$ we observe that, since $I_{r}^{(i)}\subset I_{r}^{(i+1)}$, $%
DS^{(i+1)}\subset DS^{(i)}$ and $DS^{(i)}\cup I_{r}^{(i)}=\D _{+}$,
there exists such an $i$ for which $I_{r}^{(i)}=I_{r}^{(i+1)}$, so that%
\begin{equation*}
\Delta _{+}=DS^{(i+1)}\cup I_{r}^{(i)}
\end{equation*}
is a partition of the root system. Immitating the proof of the previous
lemma, we define $\tilde{I}_{r}=I_{r}^{(i)}\cap \tilde{\Delta}$, $\tilde{I}_{d}=I_{d}^{(i)}\cap \tilde{\Delta}$
and $\tilde{I}_{0}=I_{0}^{(i)}\cap \tilde{\Delta}$.

It follows directly from the definitions that if $\r  \in span(\tilde{I}%
_{0})$, then $|a_{\r  }|\leq bK^{-1}$. On the other hand, if $\r  \in
I_{r}^{(i)}$, then $|a_{\r  }|\geq (2\tilde{b}-1)K^{-1}>bK^{-1}$, so that $%
I_{r}^{(i)}\not\subset \mathrm{span} (\tilde{I}_{0})$. As a consequence, it remains true in the perturbative setting
that the resonance of a linear combination of resonant roots is the linear combination of the resonances: if
$\r ' = \sum\nolimits _{\r  \in \tilde{I}_{r}} m_{\r  ^{\prime },\r } \r $ and $\r ' \in I_{r}^{(i)}$, then
$k_{\r '} = \sum\nolimits _{\r  \in \tilde{I}_{r}}m_{\r  ^{\prime }, \r } k_{\r }$.

We can therefore adapt the proof of the lemma \ref{reduction of resonant modes} to this setting
and obtain a vector $H$ of the form $\sum_{\r  \in \tilde{\D }}\frac{k_{\r  }^{\prime }}{D}H_{\r  }$
satisfying%
\begin{eqnarray*}
\r  (H) &=&k_{\r  },~\text{if}~\r  \in I_{r}^{(i)} \\
\r  (H) &=&0,~\text{if}~\r  \in I_{0} \\
|k_{\r  }^{\prime }| &\leq &bDN_{i},~k_{\r  }^{\prime }\in \mathbb{Z}
\end{eqnarray*}%
Therefore, if $\sigma (U(\.  ))$ is contained in the resonant modes of the
directions in $I_{r}^{(i)}$ and the constant terms in $I_{0}$, the cocycle
$(\a  ,Ae^{U(\.  )})$ is reducible by $\exp (-H\.  )$.
\end{proof}

Since the set of frequencies described in the second point of the previous
proposition is exactly the set of obstructions to resolving the linear
cohomological equation with good estimates, we can now prove the following
proposition, after having introduced the notation
\begin{eqnarray*}
N_{i}^{(u)} &=&N_{i}+N_{i+1} \\
N_{i}^{(l)} &=&N_{i}-N_{i-1}
\end{eqnarray*}
for $1\leq i \leq q$. Both $N_{i}^{(u)}$ and $N_{i}^{(l)}$ are of the order of $N$.

\begin{proposition}
\label{conj non-dioph}Let $\a  \in DC(\gamma ,\tau )$ and $K\geq 2^{\tau
+1}\gamma \tilde{b}^{q+1}N^{\tau }$. Let, also, $(\a  ,Ae^{U(\. )})\in SW^{\infty }(\T ,G)$ with%
\begin{equation*}
c_{1,0}K(N_{q}^{(u)})^{2}\e  _{0}<1
\end{equation*}%
where $\e  _{s}=\left\Vert U\right\Vert _{s}$. Then, there exists
$0\leq i\leq q$ and a conjugation $Y(\.  )\in C^{\infty }(\T ,g)$ such that
\begin{equation*}
e^{Y(\.  +\a  )}.A.e^{U(\.  )}.e^{-Y(\.  )}=Ae^{ObU(\. )}e^{U^{\prime }(\. )}
\end{equation*}%
The cocycle $(\a ,Ae^{ObU(\. )})$ is reducible by a $cD$-periodic torus morphism $B(\.  ):cD\T \ra \TT _{A}$ with
$\sigma (Ad(B(\. )))$ included in $\mathcal{B} (\emptyset ,bDN_{i})\times \{0\}\in \fZ _{D}$.
The mapping $B(\. )$ is $exp (-H\. )$ as constructed in lemma \ref{reduction of resonant modes} .

Finally, the following estimates hold:
\begin{eqnarray*}
\left\Vert Y(\.  )\right\Vert _{s} &\leq &c_{1,s}K(N_{i}^{(u)})^{s+1/2}\e  _{0} \\
\left\Vert Ob(U(\.  ))\right\Vert _{s} &\leq
&c_{1,s}N_{i}^{s+1/2}\e  _{0}
\end{eqnarray*}%
and $U^{\prime }(\.  )\in C^{\infty }(\T ,g)$ is such that%
\begin{equation*}
\e  _{s}^{\prime }\leq
c_{2,s}K^{2}(N_{i}^{(u)})^{2}((N_{i}^{(u)})^{s}\e  _{0}+\e 
_{s})\e  _{0}+C_{s,s^{\prime }}(N_{i+1}^{(l)})^{s-s^{\prime }+2}\e  _{s^{\prime }}
\end{equation*}
\end{proposition}

\begin{proof}
Let us write the cohomological equation in the neighborhood of $(\a  ,A)$:%
\begin{equation*}
e^{Y(\.  +\a  )}Ae^{U(\.  )}e^{-Y(\.  )}=A^{\prime }(\. 
)e^{U^{\prime }(\.  )}
\end{equation*}%
where $A^{\prime }(\.  )$ is not too far from $A$ and $U^{\prime }(\. )$
is supposed to be much smaller than $U(\. )$. We remark that since we
want to obtain a purely local conjugation proposition we allow $A^{\prime }(\. )$ to be non-constant
and to incorporate the obstructions to the solution of the linear equation. The following form of
the equation is therefore more adapted to the problem%
\begin{equation*}
e^{Ad(A^{\ast })Y(\.  +\a  )}e^{U(\.  )}e^{-Y(\.  )}=A^{\ast
}A^{\prime }(\.  )e^{U^{\prime }(\.  )}
\end{equation*}%
Since $A^{\ast }A^{\prime }(\.  )=exp(V(\.  ))$ is of the same order as $%
U(\.  )$, we can linearize the preceeding equation to%
\begin{equation*}
Ad(A^{\ast })Y(\.  +\a  )+U(\.  )-Y(\.  )=V(\.  )
\end{equation*}

For the constant $A\in G$, we fix the the index $i$ of the previous proposition and use the partition
$\Delta _{+}=I_{r}\cup I_{d}\cup I_{0}$ and $\bk = \bk_{A}\in \fZ $ associated to $i$.
We have the following estimates
\begin{eqnarray*}
\| T_{N_{i+1}}^{(\bk _{A})}U(\. )\| _{s} &\leq &C_{s}(N_{i}^{u})^{s+1/2}\e _{0} \\
\| R_{N_{i+1}}^{(\bk_{A})}U(\.  )\| _{s} &\leq &C_{s,s^{\prime }}(N_{i+1}^{l})^{s-s^{\prime }+2}\e _{s^{\prime }} \\
\| ObU(\.  )\| _{s} &\leq &C_{s}(N_{i})^{s+1/2}\e _{0}
\end{eqnarray*}%
concerning the corresponding truncation operators.

The definition of the partition $I_{r}\cup I_{d}\cup I_{0}$ implies that the
equation
\begin{equation*}
Ad(A^{\ast })Y(\.  +\a  )+U(\.  )-Y(\.  )=ObU(\. )+R_{N_{i+1}}^{(\bk _{A})}U(\.  )
\end{equation*}
can be solved, with the estimate%
\begin{equation*}
\left\Vert Y(\.  )\right\Vert _{s}\leq C_{s}K (N_{i}^{u})^{s+1/2}\e _{0}
\end{equation*}

Coming back to the non-linear problem (see also section \ref{estimate of quadratics} of the appendix),
we find that
\begin{equation*}
e^{Y(\.  +\a  )}.A.e^{U(\.  )}.e^{-Y(\.  )}=Ae^{ObU(\. 
)}e^{U^{\prime }(\.  )}
\end{equation*}%
where $U^{\prime }(\.  )$ is of second order with respect to $Y(\. )$ and $U(\.  )$
(and $ObU(\.)$), so that
\begin{equation*}
\e  _{s}^{\prime }\leq
c_{2,s}K^{2}(N_{i}^{(u)})^{2}((N_{i}^{(u)})^{s}\e  _{0}+\e 
_{s})\e  _{0}+C_{s,s^{\prime }}(N_{i+1}^{(l)})^{s-s^{\prime }+2}\e  _{s^{\prime }}
\end{equation*}

Finally, by the construction of the mapping $ObU(\.  )$ and the vector
$k_{A}$, it follows that the cocycle $(\a  ,Ae^{ObU(\. )})$ is
reducible by the torus morphism constructed in the previous proposition, and
therefore reducible modulo $cD$.
\end{proof}

\begin{corollary} \label{corollary for iteration}
Let $\a  \in DC(\gamma ,\tau )$ and $K\geq 2^{\tau
+1}\gamma \tilde{b}^{q+1}N^{\tau }$. Let, also, $(\a  ,Ae^{U(\. )})\in SW^{\infty }(\T ,G)$ with%
\begin{equation*}
c_{1,0}K(N_{q}^{(u)})^{2}\e  _{0}<1
\end{equation*}%
where $\e  _{s}=\left\Vert U\right\Vert _{s}$. Then, there exists
$0\leq i\leq q$ and a conjugation $G(\.  )\in C^{\infty }(\T ,G)$
such that%
\begin{equation*}
G(\.  +\a  ).A.e^{U(\.  )}.G^{*}(\.  )=A^{\prime} e^{U' (\.  )}
\end{equation*}
The mapping $U^{\prime }(\.  )$ is $1$-periodic, and $G(\. )$ and $U(\. )$ satisfy the following estimates
\begin{eqnarray*}
\left\Vert G(\.  )\right\Vert _{s} &\leq &  c_{1,s} N_{i}^{s} + c_{1,s}K(N_{i}^{(u)})^{s+1/2}\e  _{0} \\
\e  _{s}^{\prime } &\leq &c_{2,s}K^{2}(N_{i}^{(u)})^{2}((N_{i}^{(u)})^{s}\e  _{0}+\e _{s})\e  _{0}+
C_{s,s^{\prime }}(N_{i+1}^{(l)})^{s-s^{\prime }+2}\e _{s^{\prime }}
\end{eqnarray*}
\end{corollary}

\begin{proof}
We only need to apply successively the conjugations constructed in the proof of corollary
\ref{reduction to cst}. By the corollary we have
\begin{equation*}
Conj_{\exp (Y(\. ))} (\a  ,Ae^{U(\. )}) = (\a  ,Ae^{ObU(\. )}e^{\tilde{U}(\. )})
\end{equation*}
where $\tilde{U}(\. )$ is quadratic with respect to $U(\. )$ and $Y(\. )$. We observe that
\begin{equation*}
\sigma (Ad(e^{-H \. })) \subset \mathcal{B}(\emptyset ,bDN_{i})
\end{equation*}
so that
\begin{equation*}
 \s ( Ad(e^{-H \. }) R^{\bk }_{N} U(\.  ) \subset  \mathcal{B}(\emptyset ,bDN^{l})^{c}
\end{equation*}
which implies
\begin{equation*}
\| ( Ad(e^{-H \. }) R^{\bk }_{N} U(\.  ) \| _{s} 
\leq C_{s,s^{\prime }}(N_{i+1}^{l})^{s-s^{\prime }+2}\e _{s^{\prime }}
\end{equation*}
By construction, the conjugation by $e^{-H \. }$ satisfies
\begin{equation*}
Conj_{\exp (-H\. )} (\a  ,Ae^{ObU(\. )}e^{\tilde{U}(\. )}) = (\a  ,Ae^{s}e^{U'(\. )})
\end{equation*}
where $s \in g$ and $U'(\. ) =Ad(e^{-H\. })\tilde{U}(\. )$.

Finally, we can choose the torus morphism $C(\. )$ constructed in the corollary for the reduction
of the period of the conjugation to be of minimal $C^{1}$ norm, so that the adjoint action of $C(\. )$
affects the norms by constants depending only on $G$. Since the mapping
$G(\. ) = C(\. ) e^{-H \. } e^{Y( \. )}$ is $1$-periodic, it satisfies the conclusions of the corollary.
\end{proof}

The proof of this corollary concludes the preparation for the K.A.M. scheme.

\section{Iteration and conclusion} \label{The iteration local theory}

In this section, we will follow \cite{FK2009} since the scheme needed in
order to conclude is only a special case of \textit{Theorem 10}, proved in
the appendix C of the paper. For the sake of completeness, we present the context and the statement of the theorem.

Let $a>0$, $0<\sigma_{0}\leq 1$, $M,m, \mu , \overline{\mu} >0 $ such that
\begin{equation*}
\dfrac{M}{m} < 1+ \sigma _{0}, \; \dfrac{m}{\mu } <1, \; 0 \leq 2 \overline{\mu} < \mu
\end{equation*}
We point out that $\overline{\mu}$ is taken positive in the paper, but inspection of the proof shows
that $\overline{\mu} =0 $ is admissible.

For such choices of parameters, there exists $g>0$
\begin{equation*}
\dfrac{1}{\mu - \overline{\mu}} < g < \min (\dfrac{1+\sigma _{0}}{M}, \dfrac{1}{m}, \dfrac{1}{\overline{\mu}})
\end{equation*}
and $0<\sigma < \sigma _{0}$ such that
\begin{equation*}
1+ \sigma < \mu - \overline{\mu} g
\end{equation*}
that is
\begin{equation*}
\dfrac{1+ \sigma}{\mu - \overline{\mu}} < g < \min (\dfrac{1+\sigma _{0}}{M}, \dfrac{1}{m}, \dfrac{1}{\overline{\mu}})
\end{equation*}

For $s, \overline{s}$ , let $C_{s, \overline{s}} : [0,\infty) \ra [1,\infty)$ 
be a family of continuous functions on $[0,2]$ such that $C_{s, \overline{s}} (t) = 1$ if $t > 2$,
increasing with respect to $s, \overline{s} \geq 0$ and let
$\overline{C} _{s} = \max _{0 \leq t \leq 2} C_{s, s} (t)$.

Under these assumptions we get the following theorem
\begin{theorem} \label{FK scheme}
There exists $s_{0}>0$ such that if $\e _{p,s}$ is a double sequence satisfying
\begin{equation*}
\e _{p+1,s} \leq C_{s, \overline{s}} (1+ \lambda _{p}^{a} \e _{p,0} )
( \lambda _{p}^{a+Ms} \e _{p,0}^{1+\sigma _{0}} +\lambda _{p}^{a+ms} \e _{p,0} \e _{p,s} +
\lambda _{p}^{a-(\overline{s}-s) \mu } ( \e _{p,\overline{s}} 
+ \lambda _{p}^{\overline{s} \, \overline{\mu }} \e _{p,0} ))
\end{equation*}
for any $s,  \overline{s}, p \in \N $, where
$\lambda _{p}=L^{(1+\sigma )^{p}}$, $L = \overline{C}_{s_{0}}$ and if
\begin{equation*}
\e _{0,0} \leq (\overline{C} _{s_{0}})^{-\frac{s_{0}}{g}}, \; \e _{0,s_{0}} \leq 1
\end{equation*}
then, for any $s \in \N $ $\e _{p,s} = O(\lambda _{p}^{-\infty} )$.
Moreover, $s_{0}$ does not depend on the sequence $(C_{s})_{s}$
and can be taken of the form $(a + 1)\xi (\sigma _{0}, M, m, \mu , \overline{\mu})$.
\end{theorem}

\begin{remark}
A small comment on the theorem and its proof is due. In its statement, the theorem asserts that, provided that the
smallness conditions are satisfied and the inductive estimates hold for all positive integers $p$, then the K.A.M.
scheme converges exponentially fast. In practice, however, the inductive estimates hold under certain smallness
conditions which have to be verified at each step. In our case, the smallness conditions, which are related to the applicability of the Hausdorff-Campbell formula,
are of the type $C L_{p}^{a} \e _{p, 0} <1$ in the notation of the theorem above.

This problem is resolved by the very first lemma in the proof of the theorem \ref{FK scheme} which asserts that if
$\gamma _{0}$ is big enough (the largeness depending on the parameters of the theorem), $s_{0}= g \gamma _{0}$
and $b= \kappa \gamma _{0}$ where $\kappa >0$ is such that
\begin{equation*}
\dfrac{1+ \sigma + \kappa}{\mu - \overline{\mu}} < g
\end{equation*}
and if the smallness assumptions of the theorem for $p=0$ are satisfied, then we have inductively
\begin{eqnarray*}
\e _{p,0} < \lambda _{p}^{-\gamma _{0}} \\
\e _{p,s_{0}} < \lambda _{p}^{b}
\end{eqnarray*}
which justifies the use of the inductive estimates. We have omitted the constants in the expressions for the sake
of simplicity.
\end{remark}

\subsection{Choice of constants}

In order to apply theorem \ref{FK scheme}, we begin by fixing some notation. We call
\begin{equation*}
\e  _{n,s}=\left\Vert U_{n}(\.  )\right\Vert _{s}
\end{equation*}%
We also define the sequence of orders of truncation $N_{n}$ and the
diophantine constants $K_{n}$ and we need to have%
\begin{eqnarray*}
K_{n} &\geq &2^{\tau +1}\gamma \tilde{b}^{q+1}N_{n}^{\tau } \\
N_{n,i} &=&\tilde{b}^{i}N_{n} \\
N_{n,i}^{l} &\geq &N_{n} \\
N_{n,i}^{u} &\leq &2\tilde{b}^{q+1}N_{n}
\end{eqnarray*}%
If we also take into account the customary definition of $%
N_{n}=N_{n-1}^{1+\sigma }=N_{1}^{(1+\sigma )^{n-1}}$, where $N=N_{1}$ and $\s$ are to
be fixed, we may also chose $K_{n}$ to grow as a power of $N_{n}$, so that we have%
\begin{eqnarray*}
N_{n} &=&N^{(1+\sigma )^{n-1}} \\
K_{n} &=& N_{n}^{\t + \nu}  \geq 2^{\tau +1} \gamma \tilde{b}^{q+1}N_{n}^{\tau }
\end{eqnarray*}
where for $\nu >0$ fixed, the last inequality holds if $N$ is big enough.

\subsection{Estimates} \label{Estimates KAM local theory}

In view of corollary \ref{corollary for iteration}, we have the following situation. A
cocycle $(\a  ,Ae^{U(\.  )})=(\a  ,A_{1}e^{U_{1}(\.  )})\in SW^{\infty }(\T ,G)$ is given,
and we want to construct inductively a sequence of $1$-periodic conjugations
\begin{eqnarray*}
L_{n}(\.  ) &=&G_{n}(\.  )...G_{1}(\.  ) \\
&=&C_{n}(\.  )e^{-H_{n}\.  }e^{Y_{n}(\.  )}...C_{1}(\. 
)e^{-H_{1}\.  }e^{Y_{1}(\.  )}
\end{eqnarray*}
such that%
\begin{equation*}
(\a  ,A_{1}e^{U_{1}(\.  )})=Conj_{L_{n}(\.  )}(\a 
,A_{n}e^{U_{n}(\.  )})
\end{equation*}%
At each step, $Y_{n}(\.  )$ is to be given by the local conjugation lemma \ref{conj non-dioph}, the $cD$-periodic torus morphism $e^{-H_{n}\.  }$ reduces the obstructions of $U_{n}(\.  )$ seen as a perturbation of the
constant cocycle $(\a  ,A_{n})$, and $C_{n}(\.  )$ commutes with%
\begin{equation*}
A_{n+1}=e^{-H_{n}\a  }A_{n}\exp (Ad (e^{-H\. }).Ob_{n}U_{n}(\.  ))
\end{equation*}%
and $C_{n}(1)=e^{-H_{n}}$, so that $G_{n}(\.  )$ and consequently $Ad(G_{n}(\.  ))$ are $1$-periodic.

Let us now express the inductive estimates at the $n$-th step as a function of $N_{n}$.
The corollary \ref{corollary for iteration} is applicable at the step $n$ if
\begin{equation*}
c_{1,0}K_{n}(N_{n,q}^{u})^{2}\e  _{n,0}<1
\end{equation*}
or, with a different constant $c_{1,0}$,%
\begin{equation} \label{smallness condition}
c_{1,0}N_{n}^{ \t +\nu +2}\e  _{n,0}\leq 1
\end{equation}
We find that
\begin{eqnarray*}
\e  _{n+1,s} &\leq &c_{2,s}N_{n}^{2\t +2\nu +2}(N_{n}^{s}\e 
_{n,0}+\e  _{n,s})\e  _{n,0}+C_{s,s^{\prime }}N_{n}^{s-s^{\prime
}+2}\e  _{n,s^{\prime }} \\
\left\Vert G_{n}(\.  )\right\Vert _{s} &\leq &c_{1,s}N_{n}^{\t + \nu
+s+1/2}\e  _{n,0}+c_{3,s}N_{n}^{s+1/2}
\end{eqnarray*}
or, in the form of the estimates as they appear in theorem \ref{FK scheme}
\begin{equation} \label{inductive estimates}
\e _{n+1,s} \leq C_{s, \overline{s}} ( N_{n}^{2\t +2\nu +2+s} \e _{n,0}^{2} +N_{n}^{2\t +2\nu + 2} \e _{n,0} \e _{n,s}
+ N_{n}^{2-(\overline{s}-s) } \e _{n,\overline{s}} )
\end{equation}

Therefore, the estimates satisfy the hypotheses of theorem \ref{FK scheme} with
\begin{eqnarray*}
a&=& 2\t +2\nu +2 \\
M&=&1 \\
\sigma _{0}&=&1 \\
m&=&0 \\
\mu &=&1 \\
\overline{\mu } &=& 0
\end{eqnarray*}
and with these choices, we can determine $\sigma $ as in theorem \ref{FK scheme}.

\subsection{Iteration and convergence}

A choice of a big enough $N \geq \overline{C} _{s_{0}}$ implies that the smallness condition
\ref{smallness condition} is satisfied provided that $\e _{0,0} \leq N ^{-\frac{s_{0}}{g}}$.
Therefore, the scheme can be initiated, and the remark following theorem \ref{FK scheme} justifies
the application of corollary \ref{corollary for iteration} at each step, and consequently the
inductive estimates in eq. \ref{inductive estimates} are valid.

Therefore, \ref{FK scheme} is applicable and gives that, for all $s\in \N$,
\begin{equation*}
\e  _{n,s}=O(N_{n}^{-\infty })
\end{equation*}
i.e. that for any $s \in \N$ and for any $\s \in N$, there exists a positive constant $C = C(s, \s )$
such that
\begin{equation*}
\e  _{n,s}\leq C N_{n}^{-\s }
\end{equation*}
For shortness, we will write $\| U_{n} (\.) \| = O_{C^{\infty}} (N_{n}^{- \infty}) $.

\subsection{Proof of theorem \ref{NK local density}}

A direct consequence of the application of the generalized K.A.M. scheme of theorem \ref{FK scheme} is the proof
of theorem \ref{NK almost reducibility}, since the convergence of every $\e _{n,s}$ with $s$ fixed and
$n \ra \infty$ is precisely the definition of almost reducibility.

The fact that $\e _{n,s}$ converges exponentially fast to $0$ becomes important in the proof of the
local density theorem:
\begin{proof} [Proof of theorem \ref{NK local density}]
In the previous section, we obtained that, if $U_{1}(\.  )$ satisfies the smallness conditions of
theorem \ref{FK scheme}, we can construct a sequence of conjugations
$L_{n}(\.  )\in C^{\infty }(c\T ,G)$ such that%
\begin{eqnarray*}
\left\Vert L_{n}(\.  )\right\Vert _{s} &\lesssim &nN_{n}^{s} \\
(\a  ,A_{1}e^{U_{1}(\.  )}) &=&Conj_{L_{n}(\.  )}(\a 
,A_{n}e^{U_{n}(\.  )}) \\
\left\Vert U_{n}(\.  )\right\Vert _{s} &=&O(N_{n}^{-\infty })
\end{eqnarray*}

We have constructed the sequence of conjugations $L_{n}(\.  )$ satisfying%
\begin{eqnarray*}
A_{1}e^{U_{1}(\.  )} &=&L_{n}(\.  +\a  )A_{n}e^{U_{n}(\. 
)}L_{n}^{\ast }(\.  ) \\
&=&L_{n}(\.  +\a  )A_{n}L_{n}^{\ast }(\.  )e^{\tilde{U}_{n}(\.  )}
\end{eqnarray*}%
where $\tilde{U}_{n}(\.  )=Ad(L_{n}(\.  ))U_{n}(\.  )$. Since%
\begin{equation*}
\left\Vert \tilde{U}_{n}(\.  )\right\Vert _{s}\lesssim \e 
_{n,s}+nN_{n}^{s}\e  _{n,0}
\end{equation*}%
we obtain directly that $\tilde{U}_{n}(\.  )\ra 0$ in the $C^{\infty }$ topology. Therefore,
$L_{n}(\. +\a  )A_{n}L_{n}^{\ast }(\.  )\ra A_{1}e^{U_{1}(\.  )}$, and by construction the cocycles
\begin{equation*}
(\a  ,L_{n}(\.  +\a  )A_{n}L_{n}^{\ast }(\.  ))
\end{equation*}%
are reducible.
\end{proof}

The proof of the global analogues of the theorems are obtained in the same way as theorem
\ref{NK global reducibility to regular geodesics} from theorem \ref{NK local reducibility to regular geodesics},
and we only point out that the local theorems are true for cocycles in $SW(\T ^{d},G)$ for any $d \in \N ^{*}$,
the only difference being that the loss of derivatives in the local conjugation lemmas is $s+d/2$
in place of $s+1/2$ and the resulting non-significant changes in the choice of constants in the K.A.M. scheme.
\selectlanguage{english}
\chapter{Perturbations of singular geodesics} \label{Perturbations of singular geodesics}

We now return to the study of the non-local problem. More specifically, we
take up the local study the remaining case of singular geodesics, which is
the final step for the proof of the theorem on the global density of reducible cocycles and its consequences.

In this setting, we suppose that $e_{r}$ is a singular vector in $g$, so
that neither of the factors in the splitting $G_{0}\times G_{+}$ is trivial.
The models toward which renormalization converges, modulo an iteration at
most $\x _{G _{0}}$ times, are of the form $E_{r}(\. )A$, where $E_{r}(\. )$ and $A$ commute.

We fix $\ft $, a maximal toral algebra in $g$ and $\TT $ the corresponding torus passing by $A$.
We can therefore, write such a cocycle in the form
\begin{equation*}
(\a ,E_{r}(\. ).A.e^{U(\. )})
\end{equation*}
where $A\in \TT $ and $U(\. )$ is small in the $C^{\infty }$ topology.

The particularity of the local study of singular geodesics is the
coexistence of local and global phenomena, which demands a posteriori
estimates on the rest, as in the study of regular geodesics, together with a
reduction scheme as in the local case. The independence of the two parts of
the dynamics in the first order allows us, however, to apply separately the
reduction within $G_{0}$ and within the directions that do not commute with $G_{+}$.

\section{Notation}

Throughout this chapter, we suppose given a $1$-periodic perturbation
of a $1$-periodic geodesic of the form%
\begin{equation*}
(\x \tilde{\a} ,E_{\x \tilde{r}}(\. )A\exp (U(\. )))
\end{equation*}%
where $E_{r}(\. )$ and $A$ commute, and $1 \leq \x \leq \x _{0}$. We call
\begin{equation*}
(\a ,E_{r}(\. )A\exp (\tilde{U}(\. )))=(\x \tilde{\a} ,E_{\x \tilde{r}}(\. )A\exp (U(\. )))
\end{equation*}

We also suppose that $\a \in DC(\gamma ,\tau )$, i.e.%
\begin{equation*}
|k\a |_{\Z }=dist(k\a ,\Z )\geq \frac{\gamma ^{-1}}{%
|k|^{\tau }}
\end{equation*}
This is implied by the hypothesis that $\tilde{\a } \in DC(\tilde{\gamma} ,\tau )$, only with a
different constant $\gamma$. We also suppose that $e_{r}$ is singular, so that the corresponding splitting
$G_{0}\times G_{+}$ has no trivial factors.

We remind that we denote by $w$ ($w_{0}$) the rank of $G$ (resp. $G_{0}$),
by $q$ (resp. $q_{0}$) the number of positive roots of $G$ (resp. $G_{0}$),
and by $f$ (resp. $f_{0}$) the dimension of the real vector space $g$ (resp. $g_{0}$). If $\tilde{\D}$
are the roots in a basis of Weyl, we denote by $\tilde{\D}_{0}$ the roots in the corresponding basis of the
root system of $\ft _{0}=\ft \cap g_{0}$, by $I^{(0)}$ the positive roots of $G_{0}$ and
$I^{(+)} = \D_{+} \setminus I^{(0)} $. In other terms,
\begin{eqnarray*}
I^{(+)} &=&\{\r \in \Delta _{+},\r (e_{r})\not=0\} \\
I^{(0)} &=&\{\r \in \Delta _{+},\r (e_{r})=0\}
\end{eqnarray*}
We also remind that for $\r ^{\prime }\in I^{(0)}$, there exist
positive integers $(m_{\r^{\prime },\r })_{\r } , \r ^{\prime }\in I^{(0)}$ such that
\begin{equation*}
\r ^{\prime }=\sum_{\tilde{\D}_{0}}m_{\r ^{\prime },\r }\r 
\end{equation*}
there exist rational numbers $p_{\r ^{\prime },\r }=q_{\r ^{\prime
},\r }/D_{0}$ with $|q_{\r ^{\prime },\r }|\leq b_{0}$, such that%
\begin{equation*}
H_{\r ^{\prime }}=\sum_{\tilde{\D}_{0}}p_{\r ^{\prime },\r }H_{\r }
\end{equation*}%
where $D_{0},b_{0} \in \N ^{\ast }$ are considered fixed. The matrices in these relations
are submatrices of the matrices satisfying the same relations in $g$, and consequently $D_{0}\leq D$ and $b_{0}\leq b$.

We also define the orthogonal projections $\p_{0} : g \ra g_{0} $ and $\p_{+}=id - \p_{0} $.\footnote{Note that
$\p _{+}$ is not a projection on $g _{+}$ but on $(g_{0})^{\perp}$.}

As we have already seen in the concrete examples in chapter \ref{dynamics of torus-reducible cocycles},
the vectors $H_{\r }$ corresponding to roots in $\tilde{\D} \setminus \tilde{\D}_{0}$ do not form a basis
of $g_{+}$. This fact is nonetheless not significant, so we complete the basis $(H_{\r })_{\r \in \tilde{\D}_{0}}$
to $(H_{\r })_{\r \in \tilde{\D}}$, a basis of $g$, with a set of orthonormal vectors in $g_{+}$,
indexed by roots in $\Delta _{+} \setminus \tilde{\D}_{0}$, even though the completion of the basis is non-canonical.

On the other hand, wherever the properties of the vector $e_{r}$ are important, the reader should rather think
in terms of the dual of the Weyl's basis of $\ft $ extending that of $\ft _{0}$, even though this basis is not used explicitly.
The complex directions $j_{\r } \in g$, $\r \in \tilde{\D } $ are those associated to such a Weyl's basis of $\ft $.

We also use the notation $c_{0}=c_{G_{0}}$ and $\mathcal{Z} = Z^{0}_{G}(\exp (\R e_{r}))$.

\section{Synopsis of the chapter}

In a first time, as in the previous chapters, we prove the local almost reducibility to normal forms in the
case of singular geodesics:
\begin{theorem} \label{local almost reducibility to singular geodesics}
Let $\a \in DC(\gamma ,\tau )$ and $r\in \N ^{w}$ such that $E_{r}(\. )$ be a $1$-periodic singular geodesic and
$A\in Z^{0}_{G}(\exp (\R e_{r})) = \mathcal{Z} $. Then, there exist $\epsilon >0$ and $s_{0}\in \N ^{*}$ such that if
\begin{eqnarray*}
\left\Vert U(\. )\right\Vert _{0} = \e _{0} &<&\epsilon  \\
\left\Vert U(\. )\right\Vert _{s_{0}} = \e _{s_{0}} &<&1
\end{eqnarray*}
and the cocycle $(\a ,E_{r}(\. )Ae^{U(\. )})$ is of degree $r$, then it is almost reducible to
$(\a ,E_{r}(\. )A^{\prime })$ where also $A^{\prime }\in \mathcal{Z}$: There exists a
sequence of conjugations $B_{n}(\. ) \in C^{\infty}(\T , G)$
and a sequence of constants $A_{n}\in \mathcal{Z}$, such that
\begin{equation*}
\left( ( E_{r}( \. )A_{n} \right)  ^{*} \left( B_{n}(\. +\a ).E_{r}(\. )Ae^{U(\. )}.B_{n}^{*}(\. ) \right) \ra Id
\end{equation*}
in $C^{\infty}(\T , G)$.
\end{theorem}
We point out that the failure of exact reducibility to the normal form is due to small divisor phenomena in $G_{0}$.
All the complications in this part of the theory, in comparison with the theory of regular geodesics, are due to the presence
of the non-trivial factor $G_{0}$. Consequently, all the proofs and the statements of the theorems stand between the simplicity
and the sharpness of those in Chapter \ref{Normal form theorem} and the complications of the local theory of the previous
chapter.

Besides almost reducibility to periodic geodesics, we can also obtain the local density of cocycles reducible to
periodic geodesics in the class of cocycles of degree $r$:
\begin{theorem}
Let $\a \in DC(\gamma ,\tau )$, $r\in \N ^{w}$ such that $E_{r}(\. )$ be a $1$-periodic singular geodesic and
$A\in \mathcal{Z} $. Then, cocycles reducible to $( \a ,E_{r}(\. )A^{\prime })$ where
$A^{\prime }\in \mathcal{Z}$ are dense in the class of cocycles of the form $(\a ,E_{r}(\. )Ae^{U(\. )})$
which are of degree $r$ and $U(\. )$ satisfies the smallness conditions of the previous theorem. More formally,
The closure of
\begin{equation*}
\{ Conj_{B(\.) }( \a ,E_{r}(\. )A^{\prime }), B(\.) \in C^{\infty}(\T , G), A^{\prime} \in \mathcal{Z} \}
\end{equation*}
in $SW^{\infty}_{\a}$ contains all cocycles in
\begin{equation*}
\{ ( \a ,E_{r}(\. )A).e^{U(\.)} ), A \in \mathcal{Z},U(\.) \in C^{\infty}(\T , G), \e _{0} < \epsilon ,
\e _{s_{0}} <1 , \emph{deg} ( \a ,E_{r}(\. )A).e^{U(\.)} ) = r \}
\end{equation*}
\end{theorem}
The proof of this theorem is obtained by the comparison of the rate of convergence of the scheme with the explosion
of the norms of the reducing conjugations, exactly as in the local theory.

As previously, if we relax the arithmetic condition on the frequency to an $RDC$,
we obtain the global analogues of these theorems, namely

\begin{theorem}
Let $\a \in RDC$ and $(\a ,A(\. ))\in SW^{\infty }(\T ,G)$
be of degree $r\in \N ^{w}$ such that $E_{r}(\. )$ be a $1$-periodic singular geodesic. Moreover,
let $m \in \N ^{*}$ be the minimal natural number such that
$(\a  ,A(\.  ))^{m }$ is homotopic to the $Id$. Then, there exists $\x \in \N $,
with $1\leq \x \leq m$ and $\x | m$, such that $(\a ,A(\. ))^{\x }$ can be conjugated arbitrarily close in
$C^{\infty } (\T ,G)$ to cocycles of the form $(\x \a ,E_{\x r}(\. )A)$, where $A\in \mathcal{Z}$.
\end{theorem}
We also prove the following theorem.
\begin{theorem}
Let $\a$ and $(\a ,A(\. ))$ as in the previous theorem. Then, for the same $\x \in \N ^{*}$,
$(\a ,A(\. ))^{\x }$ is accumulated in $SW_{\x \a}^{\infty }( \T ,G)$ by cocycles
conjugate to $(\x \a ,E_{\x r}(\. )A)$ in $SW_{\x \a}^{\infty }(\T ,G)$, where $A\in \mathcal{Z}$.
\end{theorem}

Finally, we state the following corollary

\begin{corollary} \label{accumulation by smaller energy}
Let $\a \in RDC$ and $(\a ,A(\. ))\in SW^{\infty }(\T ,G)$ be of positive energy, and $m$ as above. Then,
there exists $\x \in \N $, with $1\leq \x \leq m$ such that $(\a ,A(\. ))^{\x }$
is accumulated in $C^{\infty }(\T ,G)$ by cocycles of smaller energy.
\end{corollary}

\begin{proof}
Since $(\a ,A(\. ))^{\x }$ is almost reducible to a cocycle of the form $(\a ,E_{r}(\. )A)$
where $Ad(A).e_{r} = e_{r}$, and for all $a\in g$ small enough and such that $[a,e_{r}]\not=0$,
$(\x \a ,E_{\x r}(\. )Ae^{a})$ is of smaller energy than $(\a ,A(\. ))^{\x }$, the result follows.
\end{proof}

We can now obtain the main theorem of this thesis, the theorem of
global density of reducible cocycles, assuming the proofs of the
theorems previously stated in this section.

\section{Proof of the global density theorem}
The proof of theorem \ref{Global density theorem} is in fact a finite induction argument and uses the yet unproven
theorems of this chapter, as well as the theorems proved in chapters \ref{Normal form theorem}
and \ref{Local theory}.

\begin{proof}
Let $(\a ,A(\. ))\in SW^{\infty }(\T ,G)$. If the cocycle is of energy $0$, the conclusion of the theorem is true
by the theorems proved in the study of the local case, see the synopsis of chapter \ref{Local theory}, and more
specifically theorem \ref{NK global almost reducibility in 0 energy}.

If the cocycle is of degree $r\not=0$, then, by corollary \ref{accumulation by smaller energy}, there exists
$1\leq \x \leq \min (m , \x _{0})$ and $\x | m$ such that $(\a ,A(\. ))^{\x }$ is accumulated by cocycles of
smaller energy. Let us consider $(\x \a ,A^{\prime }(\. ))$, a perturbation of $(\a ,A(\. ))^{\x }$, and assume
that $\emph{en}(\x \a ,A^{\prime }(\. ))>0$, since if $\emph{en}(\a ,A^{\prime }(\. ))=0$ the proof is concluded.
If $\emph{en}(\x \a ,A^{\prime }(\. ))>0$, we need to renormalize $(\x \a ,A^{\prime }(\. ))$. Since we want to
relate the result of the renormalization scheme with the dynamics of the cocycle $(\a ,A(\. ))$, we cannot apply
the continued fractions algorithm to $\x \a \in \T $, even though
$(\x \a ,A^{\prime }(\. )) \in SW^{\infty }_{\x \a}(\T ,G)$.
We can, nonetheless, consider $\x \a \in \x \T $, which places
$(\x \a ,A^{\prime }(\. )) \in SW^{\infty }_{\x \a}(\x \T ,G)$. Rescaling by $\x $ we obtain the cocycle
$( \a ,A^{\prime }( \x ^{-1} \. )) \in SW^{\infty }_{\a}( \T ,G)$, which is of smaller energy than
$(\a ,A(\. ))$, and whose order in the homotopy group is smaller.

If to $(\a ,A(\. ))$ we associate the splitting of the dynamics $G_{0}\times G_{+}\hra G$, and to
$(\a ,A^{\prime }(\x ^{-1}\. ))$ the splitting $G_{0}^{\prime }\times G_{+}^{\prime }\hra G$, we find
that $G_{0} \hra G_{0}^{\prime }$ in a canonical way. This is due to the fact that the splitting is
preserved by renormalization, since the degree is preserved. Therefore, if $r' $ is the degree of
$(\a ,A^{\prime }(\x ^{-1}\. )) $, $G_{0}^{\prime}$ (or rather $g_{0}^{\prime }$) is characterized by
the equation $[e_{r'} ,s ]=0$ for $s \in g$.
This equation is satisfied for all $s \in g$ such that $[e_{r } ,s ]=0$, since the perturbations that decrease
the energy of the cocycle $(\a ,A(\. ))$ do not commute with $e_{r }$. More informally, such perturbations
shrink $G_{+}$ and and consequently make $G_{0}$ grow.

If we denote by $\x ' $ the natural number corresponding to $(\a ,A^{\prime }(\x ^{-1}\. )) $,
$\x $ divides $\x ^{\prime }$ and thus
\begin{equation*}
(\a ,A(\. ))^{\x ^{\prime }} = (\x  \a ,A_{\x }(\. ))^{\x ^{\prime } / \x }
\end{equation*}
is accumulated in $SW_{\x ' \a }^{\infty }(\x \T ,G)$ by cocycles of energy smaller than
$\emph{en}(\a ,A^{\prime }(\x ^{-1 } \. ))$, just by virtue of the corresponding theorem of
accumulation by cocycles of smaller energy applied to $(\a ,A^{\prime }(\x ^{-1}\. )) $.

Since the perturbations can be chosen arbitrarily small, we can repeat this step a finite number of times, until
we obtain a cocycle for which $\emph{en} = 0$, or, equivalently, $G_{+}^{\prime }= \{ Id\}$ and $G_{0}^{\prime } = G$.
\end{proof}

\section{Strategy of the proof of theorem \ref{local almost reducibility to singular geodesics}}

As announced in the previous chapter, the proof local almost reducibility to singular geodesics
consists essentially in a parallel reduction of the positive and $0$-energy parts of the perturbation.
The linear equation corresponding to the first part is the one encountered in chapter \ref{Normal form theorem},
while the one corresponding to the latter is that of the local theory.

This coexistence of close-to-constants and far-from-constants dynamics does not allow the use of an inverse
function theorem as in
the case of regular geodesics, since the reduction of resonant modes in $G_{0}$ has to be treated with non-converging
conjugations. As we will see, this forces a non-uniform choice for the spectral support of the obstructions,
since the reduction of resonant modes does not commute with all $j_{\r } , \r \in I ^{(+)}$.
This is in fact a sufficient reason for the inapplicability of the inverse function theorem, as we know
that reduction of resonant modes cannot be avoided for generic cocycles.
We also need to adapt the a priori estimates obtained for perturbations of regular geodesics (see chapter 
\ref{A priori estimates on perturbations of regular geodesics}) in order to control the obstructions in the
positive-energy regime, as perturbations within $G_{0}$ will increase the energy of the path.

In order to cope with these problems, we use the K.A.M. scheme of the previous chapter combined with the one described
in \cite{Krik2001}, after having proved that the local reduction lemmas for local and non-local dynamics
can actually be combined.

Let us, therefore, consider $(\a ,A(\. )) = (\a ,E_{r}(\. )A\exp (U(\. )))$, a $1$-periodic perturbation of a
singular geodesic, and assume that a $Y(\. ):\T \ra g$, small enough, can reduce the perturbation $U(\. )$ to
$V(\. )$, with $V(\. )$ of second order with respect to $U(\. )$. Let us call $(\a ,A'(\. ))$ the cocycle thus
obtained. The equation satisfied by such a conjugant is
\begin{equation*}
e^{Y(\. +\a )}.E_{r}(\. )A.e^{U(\. )}.e^{-Y(\. )}=E_{r}(\. )A^{\prime}.e^{V(\. )}
\end{equation*}
Linearization of this equation under the smallness assumptions implies that
such a $Y(\. )$ must satisfy the equation%
\begin{equation} \label{linearized local eq sing geod}
Ad(E_{r}^{\ast }(\. )A^{\ast }).Y(\. +\a )-Y(\. )=-U(\. )
\end{equation}
The linearized equation then decomposes
into two different types of equations (and a third for the coordinates in the torus) as follows.

Projection of the equation on $g_{0}$ and introduction of some obvious notation gives
\begin{equation*}
Ad(A^{\ast })Y^{(0)}(\.  +\a  )+U^{(0)}(\.  )-Y^{(0)}(\.  )=Ob_{0}U(\. )+R_{N_{i+1}}^{0,(\bk _{A})}U(\.  )
\end{equation*}
where $Y(\. )$ satisfies good estimates and the cocycle $(\a , A.e^{Ob_{0} U(\. )})$ is reducible by
$B(\. ) :c_{0} D_{0} \T \ra G_{0} \hra G$.

Projection of the equation on the positive energy component gives
\begin{eqnarray*}
Ad(E_{r}(\. )^{\ast }.A^{*}).Y^{(+)}(\. +\a )+U^{(+)}(\. )-Y^{(+)}(\. ) &=&
\Gamma _{\tilde{r}}U^{(+)}(\. )+\sum_{\D _{++}}\hat{U}_{\r }^{\ft }(0).H_{\r }+
\sum _{\D _{++}} R_{N_{i+1}} U_{\r}(\.  ) j_{\r} \\
&=& Ob_{+}U(\. )+\sum _{\D _{++}} R_{N_{i+1}} U_{\r}(\.  ) j_{\r}
\end{eqnarray*}

If, now, we call $ObU(\. )=Ob_{0}U(\. )+Ob_{+}U(\. )$, we find that
\begin{equation*}
Conj_{\exp (Y(\. ))}(\a ,E_{r}(\. )Ae^{U(\. )})=(\a ,E_{r}(\. )A'e^{ObU(\. )}e^{V(\. )})
\end{equation*}
where $ObU(\. )$ is of the order of $U(\. )$, and $V(\. )$ is of second order. Since, however,
$\p _{0} ObU(\. )= Ob _{0}U(\. )$ may not be constant, the length of the paths $A'(\. )$ and $A_{2}^{\prime } (\. )$
will be greater than $|e_{r}|$ and $2|e_{r}|$ respectively, and therefore the a priori estimates cannot
give any additional information.

This fact shows that, before applying the a priori estimates, we have to reduce the $0$-energy part of
the obstruction. This reduction may translate the positive energy part of the obstruction, since
$\r (H)$ may be non-zero also for roots in $I ^{(+)}$. This fact forces a modification in the definition
of the operator $\Gamma _{\tilde{r}}$, so that, the frequencies of $ Ad (e^{-H\.} ). \Gamma _{\tilde{r}}$
belong to the standard intervals $[0,-r_{\r}-1] \cap \Z $ or $[-r_{\r}+1,0] \cap \Z $, depending on the
sign of $r_{\r}$ (see def. \ref{definition of obstructions for regular geodesics}).
Moreover, since the conjugation (called $C(\. )$ in lemma \ref{reducibility of obstructions}) takes
values in a different torus, it may mix the frequencies of the positive energy obstructions, so we need to
apply the estimates for a $D$-periodic perturbation of the geodesic, obtain their smallness, and then
regain periodicity.

The conjugation lemma that we obtain this way gives rise to a slightly non-standard K.A.M. scheme for the
following reason. We have no direct estimates for the $C^{0}$ norm of the perturbation after conjugation,
and we cannot have any, since the only way to control it is through the $C^{1}$ norm. Therefore, the
Kolmogorov-Hadamard interpolation inequalities have to be applied between the $C^{1}$ and $C^{s_{0}}$ norms,
instead of $C^{0}$ and $C^{s_{0}}$, and we can then use the fact that the $C^{0}$ norm of the first derivative
controls the $C^{0}$ norm of the function in order to gain smallness of all norms and continue the procedure.

\section{A priori estimates}

In this section we adapt the apriori estimates obtained for perturbations of
regular geodesics in the case of perturbations of singular geodesics, and we
prove a similar lemma. The differences are in fact due to the different
definition of the truncation operator, taking into account the fact that $%
e_{r}$ is singular, and to the fact that perturbations of longer periods need to
be considered. Therefore, in what follows we let $U(\. ):P \T \ra g$, where $P\in \N ^{\ast }$.

\subsubsection{Estimates on the energy of the path}

If $U(\. )$ is small and $P$-periodic, then the path $A(\.
)=E_{Pr}(\. ).A.e^{\tilde{U}(\. )}$, where $\tilde{U}(\. )=U(P\.
)$ is $1$-periodic. We remind that, since%
\begin{equation*}
U(\. )=\sum\nolimits_{\r \in \tilde{\Delta}}\sum\nolimits_{k\in 
\Z }\hat{U}_{\r }^{\ft }(k)e^{2i\pi P^{-1}k\. }H_{\r
}+\sum\nolimits_{\r \in \Delta _{+} }\sum\nolimits_{k\in \Z }\hat{U}%
_{\r }(k)e^{2i\pi P^{-1}k\. }j_{\r }
\end{equation*}%
where%
\begin{equation*}
\hat{U}_{\r }^{(\ft )}(k)=\int_{0}^{1}U_{\r }^{(\ft %
)}(P\. )e^{-2i\pi k\. }
\end{equation*}%
we find immediately that%
\begin{equation*}
\tilde{U}(\. )=\sum\nolimits_{\r \in \tilde{\Delta}}\sum\nolimits_{k%
\in \Z }\hat{U}_{\r }^{\ft }(k)e^{2i\pi k\. }h_{\r
}+\sum\nolimits_{\r \in \Delta _{+} }\sum\nolimits_{k\in \Z }\hat{U}%
_{\r }(k)e^{2i\pi k\. }j_{\r }
\end{equation*}

Estimation of its energy gives%
\begin{equation*}
\left\Vert a(\. )\right\Vert _{L^{2}}^{2}=P^{2}\emph{en}^{2}+2\int
\langle e_{Pr},\tilde{u}(\. )\rangle +\left\Vert \tilde{u}(\.
)\right\Vert _{L^{2}}^{2}
\end{equation*}%
This last equation implies that, in the first order, the projection of $%
u(\. )$ on $g_{0}$ contributes only (positively) in the term $\left\Vert
u(\. )\right\Vert _{L^{2}}^{2}$, and not in $\int \langle e_{r},u(\.
)\rangle $, which justifies the notation
\begin{eqnarray*}
I^{(+)} &=&\{\r \in \Delta _{+},\r (e_{r})\not=0\} \\
I^{(0)} &=&\{\r \in \Delta _{+},\r (e_{r})=0\}
\end{eqnarray*}

We also modify the definition of the truncation operators:

\begin{definition}
Let $\tilde{r}_{\r }\in \Z ^{*}$ for $\r \in I^{(+)}$,
and $P\in \N ^{\ast }$. We will call $I_{P,\r }=\{ 1,2,...-\tilde{r} _{\r } P +1 \}$
if $\tilde{r}_{\r }<0$, and $I_{P,\r }=\{-1,-2,..,-\tilde{r}_{\r }P-1 \}$ if $\tilde{r}_{\r }>0$.
We also let $\Lambda _{P\tilde{r}}$ and $\Lambda _{P,\r }$ be the truncation operators
\begin{eqnarray*}
\Lambda _{P\tilde{r}}U(\. ) &=&\sum\nolimits_{\r \in I^{(+)}}\Lambda
_{P,\r }U_{\r }(\. )j_{\r } \\
&=&\sum\nolimits_{\r \in I^{(+)}}\sum\nolimits_{k\in I_{P,\r }}\hat{U}%
_{\r }(k)e^{2i\pi k\. }j_{\r }
\end{eqnarray*}

We will also use the non-homogeneous truncation operator\footnote{The superscript $0$ does not refer to $g_{0}$.}
\begin{eqnarray*}
\Lambda _{P\tilde{r}}^{0}U(\. ) &=&\sum\nolimits_{\r \in
I^{(+)}}\Lambda _{P,\r }^{0}U_{\r }(\. )j_{\r } \\
&=&\Lambda _{P\tilde{r}}U(\. )+\sum\nolimits_{\r \in I^{(+)}}\hat{U}%
_{\r }(0)j_{\r }
\end{eqnarray*}
\end{definition}

A calculation similar to the one made in the regular case gives%
\begin{equation*}
\left\Vert a(\. )\right\Vert _{L^{2}}^{2}=P^{2}\emph{en}^{2}+\sum_{\r
\in I^{(+)}}4\pi ^{2}P\tilde{r}_{\r }\sum\nolimits_{k\in \Z }k|%
\hat{U}_{\r }(k)|^{2})+\sum_{\r \in \Delta _{+}}\sum_{k\in \Z %
}4\pi ^{2}k^{2}|\hat{U}_{\r }(k)|^{2}+\| \partial \pi _{\mathfrak{t%
}}\tilde{U}(\. )\| _{L^{2}}^{2}+O(\| \tilde{U}\| _{H^{1}}^{3})
\end{equation*}%
and the resulting bound%
\begin{equation*}
\left\Vert a(\. )\right\Vert _{L^{2}}^{2}\leq P^{2}\emph{en}%
^{2}-\sum_{\r \in I ^{(+)}}4\pi ^{2}\| \Lambda _{Pr}\tilde{U}%
(\. )\| _{L^{2}}^{2}+\| \partial (Id-\Lambda _{P\tilde{r}%
})\tilde{U}(\. )\| _{L^{2}}^{2}+O(\| \tilde{U}%
\| _{H^{1}}^{3})
\end{equation*}%
so that the following lemma is true

\begin{lemma}
Let $(\a ,E_{r}(\. )A\exp (U(\. )))$ be of degree $r$ and $U(\.
)\in C^{\infty }(P\T ,g)$ be such that $U(P\. )$ is small enough
in $H^{1}$. Then, there exists a positive constant such that
\begin{equation*}
\left\Vert \Lambda _{P\tilde{r}}U(P\. )\right\Vert _{L^{2}}\lesssim
\left\Vert \partial (Id-\Lambda _{P\tilde{r}})U(P\. )\right\Vert _{L^{2}}
\end{equation*}
\end{lemma}
The dependence of the constant on $P$ is not important for us, since in fact $P=D$, the loss of
periodicity in the reduction of resonant modes.
\subsubsection{Estimation of the energy of the second iterate}

We use the bound on the $L^{2}$ norm of the derivative, obtained in the
regular case%
\begin{equation*}
\left\Vert a_{2}(\. )\right\Vert _{L^{2}}^{2}\leq 4P^{2}\emph{en}%
^{2}+C_{1}\| \partial \tilde{U}(\. )\|
_{L^{2}}^{2}-C_{2}\| [e_{r},\tilde{U}^{\symbol{94}}(0)]\|
(\left\Vert [e_{r},\tilde{U}^{\symbol{94}}(0)]\right\Vert -\left\Vert
\partial \tilde{U}(\. )\right\Vert _{L^{2}})
\end{equation*}%
Since $\left\Vert [e_{r},\tilde{U}^{\symbol{94}}(0)]\right\Vert
^{2}=\sum_{\r \in I ^{(+)}}4\pi ^{2}\tilde{r}_{\r }^{2}|\hat{U}%
_{\r }(0)|^{2}$, the following lemma holds.

\begin{lemma}
Let $(\a ,E_{r}(\. )Ae^{\tilde{U}(\. )})$ be a $P$-periodic
perturbation (small enough in $H^{1}$) of the cocycle $(\a ,E_{r}(\.
)A)$. Let also, without loss of generality, $\hat{U}(0)$ be orthogonal to $%
e_{r}$ and $\sum_{\r \in I ^{(+)}}|\hat{U}_{\r }(0)|^{2}\not=0$.
Then, there exists a constant $C_{0}$ depending only on $r$ such that if%
\begin{equation*}
\| \partial \tilde{U}\| _{L^{2}}^{2}\leq C_{0}\sum_{\r
\in I ^{(+)}}|\hat{U}_{\r }(0)|^{2}
\end{equation*}%
then $(\a ,E_{r}(\. )Ae^{U(\. )})$ is of energy strictly smaller
than that of $(\a ,E_{r}(\. )A)$.
\end{lemma}

\subsubsection{Conclusion}

In the same way, we conclude that the following a priori estimate holds for
perturbations of singular geodesics, which is to be compared with lemma \ref{A priori estimates}.

\begin{lemma} \label{A priori estimates sing geod}
There exists a positive constant $C$, depending on $r$ but not on $A$ and $U(\. )$, such that if
$(\a ,E_{r}(\. )Ae^{U(\. )})$ is of degree $r$, where $U(\. )$
is $P$-periodic and small enough in $H^{1}$, then it satisfies%
\begin{equation*}
\left\Vert \Lambda _{P\tilde{r}}^{0}U(P\. )\right\Vert _{L^{2}}\leq
C\left\Vert (Id-\Lambda _{P\tilde{r}})\partial U(P\. )\right\Vert _{L^{2}}
\end{equation*}
\end{lemma}

We remark that the only difference with the estimates obtained previously is
in the fact that the perturbation within $g_{0}$ contributes only in
increasing the energy of the path. This is expressed in the estimates by the
fact that functions in $g_{0}$ are in the kernel of the truncation operator $%
\Lambda _{P\tilde{r}}$. Since these estimates give no information on the
pertrubation if the order of the perturbation in $g_{0}$ is bigger than the
one not commuting with $g_{+}$, we need to reduce both parts.

\section{Local conjugation lemma}

Let us assume that a $Y(\. ):\T \ra g$, small enough, can reduce the perturbation $U(\. )$ to $V(\. )$
with $V(\. )\ll U(\. )$. The equation satisfied by such a conjugant is%
\begin{equation*}
e^{Y(\. +\a )}.E_{r}(\. ).A.e^{U(\. )}.e^{-Y(\. )}=E_{r}(\. ).A^{\prime }.e^{V(\. )}
\end{equation*}%
or
\begin{equation*}
(Ad(E_{r}^{\ast }(\. )A^{\ast })e^{Y(\. +\a )}).e^{U(\.
)}.e^{-Y(\. )}=A^{\ast }A^{\prime }.e^{V(\. )}
\end{equation*}
Linearization of this equation under smallness assumtions implies that such
a $Y(\. )$ must satisfy the equation%
\begin{equation*}
Ad(E_{r}^{\ast }(\. )A^{\ast }).Y(\. +\a )-Y(\. )=-U(\. )
\end{equation*}

Let us introduce the coordinates
\begin{equation*}
U(\. )=\sum\nolimits_{\r \in \tilde{\Delta}}U_{\r }^{\ft }(\. )H_{\r } +
\sum\nolimits_{\r \in \D _{+}}U_{\r }(\. )j_{\r }
\end{equation*}%
where as usual $U_{\r }^{\ft }(\. )$ are real and $U_{\r }(\. )$ complex periodic functions. We will write $Y(\. )$ in an
analogous way. The linearized equation then decomposes in three different
types of equations as follows.

The simpler case is that of the diagonal part, which reads%
\begin{equation*}
Y_{\r }^{\ft }(\. +\a )-Y_{\r }^{\ft }(\.
)=-U_{\r }^{\ft }(\. )
\end{equation*}%
for which we refer the reader to lemma \ref{Linear cohomological eq}. In a
similar way, if $\r \in I^{(+)}$, i.e. if $\r (e_{r})\not=0$, then the
equation in the direction $j_{\r }$ is of the type of the lemma \ref%
{solution of difference equation}. Finally, for roots in $I^{(0)}$, we refer
to proposition \ref{conj non-dioph}.

Putting all the corresponding estimates together gives the following lemma concerning the
solution of the linearized equation.
\begin{lemma} \label{local reduction singular geodesics}
Let $\a \in DC(\gamma ,\tau )$ and $(\a ,E_{r}(\. ).A.e^{U(\. )})\in SW^{\infty }(\T ,G)$
and $K\geq 2^{\tau +1}\gamma \tilde{b} _{0}^{q_{0}+1}N^{\tau }$. We remind that we call $\e _{s}=\| U\| _{s}$.
Then, there exists $0\leq i\leq q_{0}$ and mapping $Y(\. )\in C^{\infty }(\T ,g)$ and
$U'(\. )\in C^{\infty }(\T ,g)$ such that
\begin{equation*}
Ad (E^{*}_{r}(\. ).A^{*}) Y(\. +\a ) +U(\. ) -Y(\. )=ObU(\. ) +\tilde{R}_{n} U(\. )
\end{equation*}
and $Conj_{\exp (Y(\. ))}(\a ,E_{r}(\. ).A.e^{U(\. )}) = (\a ,E_{r}(\. ).A.e^{Ob U(\. )}.e^{U'(\. )})$.
The conjugant $Y(\. )$ satisfies the estimates
\begin{equation*}
\left\Vert Y(\. )\right\Vert _{s}\leq
c_{1,s}K (N_{i}^{u})^{s+1/2}\e _{0}
\end{equation*}%
The mapping $ObU(\. ) =  Ob_{(\a ,E_{r}(\. ).A)}U(\. )$ satisfies the property that there exists a torus
morphism $B(\. ):c_{0} D_{0}\T \rightarrow G_{0}\hra G$ such
that%
\begin{eqnarray}
\pi _{g_{0}}(Ad(B(\. ))ObU(\. )) &=&U_{0}\in g_{0} \oplus g_{+} \\  \label{config spectral support ob}
\Lambda _{D_{0} \tilde{r}}^{0}(Ad(B( \. ))ObU( \. )) &=& Ad(B( \. ))ObU( \. )-U_{0}
\end{eqnarray}
We have the following estimates%
\begin{equation*}
\begin{tabular}{ccl}
$\left\Vert ObU(\. )\right\Vert _{s}$ & $\leq $ & $C_{s}N_{i}^{s+1/2}%
\e _{0}$ \\ 
$\| \tilde{R}_{N}U(\. )\| _{s}$ & $\leq $ & $C_{s, s'}
(N_{i+1}^{l})^{s-s^{\prime }+2}\e _{s^{\prime }}$ \\ 
$\e _{s}^{\prime }$ & $\leq $ & $%
c_{2,s}K^{2}(N_{i}^{u})^{2}((N_{i}^{u})^{s}\e _{0}+\varepsilon
_{s})\e _{0} +C_{s,s^{\prime }}(N_{i+1}^{l})^{s-s^{\prime }+2}\varepsilon
_{s^{\prime }}$
\end{tabular}
\end{equation*}
\end{lemma}
We remind that $N_{i+1} = \tilde{b}_{0}N_{i}$, $N^{u}_{i}=N_{i}+N_{i+1}$, and $N_{i}^{l} = N_{i} - N_{i-1}$.
\begin{proof}
The coordinates $Y_{\r }^{\ft } (\. ) $, $\r \in \tilde{\Delta}$, as
has been shown, satisfy the equation%
\begin{equation*}
Y_{\r }^{\ft }(\. +\a )+U_{\r }^{\ft }(\.
)-Y_{\r }^{\ft }(\. )=\hat{U}_{\r }^{\ft %
}(0)+R_{N_{i}}U_{\r }^{\ft }(\. )
\end{equation*}%
for which the estimates follow directly.

We now focus on the roots in $I^{(0)}$. Following the solution of the linear
equation in the close-to-constants case, we chose a suitable $i$ between $0$
and $q_{0}$ and partition the roots in $I^{(0)}$ into $I_{res}^{(0)}\cup
I_{nr}^{(0)}$, the resonant and the non-resonant ones. We also define $Ob_{0}$, the
projection operator on $C^{\infty }(\T ,g_{0})$ which projects to the resonant modes
$\bk  = (k_{\r })_{\r \in I^{(0)}}$ and the constant ones for roots close to $0$. Finally, we use the truncation
operator $R_{N_{i+1}}^{(\bk ),0}$ on $C^{\infty }(\T ,g_{0})$. With these notations, we have
\begin{equation*}
\sum_{\r \in I^{(0)}}[Ad(A^{\ast })Y_{\r }(\. +\a )+U_{\r
}(\. )-Y_{\r }(\. )]j_{\r }=Ob_{0}U(\. )+R_{N_{i+1}}^{(\bk ),0}U(\. )
\end{equation*}%
As in local theory (see proposition \ref{conj non-dioph}), there exists a torus morphism
$B(\. ):c_{0} D_{0} \T \ra G_{0}$ satisfying
\begin{equation*}
| LB(\. ) | \leq b_{0} D_{0} N_{i}
\end{equation*}%
and such that%
\begin{equation*}
Ad(B(\. ))Ob_{0}U(\. )=U_{0}
\end{equation*}%
a constant in $g_{0}$. The announced estimates on $\e ' _{s}$, $Ob_{0}U(\. ) = \p _{0 } ObU(\. )$
and the conjugation are obtained in the same way as in the proposition treating the purely local case.

Finally, we consider the roots in $I^{(+)}$ and define the integers $k_{\r }^{\prime }$ for such roots by the relation%
\begin{equation*}
[LB(\. ) ,j_{\r }]=2i\pi \frac{k_{\r }^{\prime }}{D_{0}}j_{\r }
\end{equation*}
Clearly, if $k_{\r }^{\prime }$ is non-zero, the reduction of resonant
modes in $G_{0}$ translates the Fourier coefficients of the rest of the
linear cohomological equation for roots in $I^{(+)}$. The a priori
estimates of lemma \ref{A priori estimates sing geod}, however, give us information only on a part
of the spectrum of the perturbation determined by $r$. In order to take this effect into
account, we need to translate the rest of the equation by $-\frac{k_{\r }^{\prime }}{D_{0} }$ in each direction,
which amounts in fact to resolving the auxiliary equation
\begin{multline*}
\sum \nolimits _{\r \in I^{(+)}}[Ad(A^{\ast }E_{r}^{\ast }(\. ))\tilde{Y}_{\r
}(\. +\a )+U_{\r }(\. )-\tilde{Y}_{\r }(\. )]j_{\r }= \\
Ad(B(\. ))\sum \nolimits _{\r \in I^{(+)}}U_{\r }(\. )j_{\r
}+\Gamma _{D_{0} r}[Ad(B(\. ))\sum \nolimits _{\r \in I^{(+)}}U_{\r }(\. )j_{\r }]
\end{multline*}
for which the given function function $\tilde{U} (\.)= Ad(B(\. ))\sum \nolimits _{\r \in I^{(+)}}U_{\r }(\. )j_{\r }$
satisfies the estimate $\tilde{\e } _{s} \leq C_{s} N_{i}^{s+1/2} \e _{s} $, since its Fourier coefficients are translated.
By solving this equation, we obtain a rest corresponding to $\tilde{U}$ in standard configuration. If we conjugate
the equation back by $Ad(B^{\ast }(\. ))$, we obtain $Y_{\r }(\. )$
\begin{equation*}
\sum \nolimits _{\r \in I^{(+)}} Y_{\r }(\. ) j_{\r } = Ad(B^{\ast } (\. )) \sum \nolimits _{\r \in I^{(+)}}
\tilde{Y} _{\r }(\. ) j_{\r }
\end{equation*}
which solves the cohomological equation for $U(\. )$ but with a translated rest. The rest
$Ad(B^{\ast } (\. )) . \Gamma _{D_{0} r}[\tilde{U} (\.)]$
has the appropriate spectral support for eq. \ref{config spectral support ob} to be satisfied.
The periodicity is regained in the same way as in the
local case, since the mapping $\tilde{Y}_{\r }(\. )$ is obtained by simple operations on the Fourier coefficients of
\begin{equation*}
Ad(B (\. ))\sum \nolimits _{\r \in I^{(+)}}U_{\r }(\. )j_{\r }
\end{equation*}%
which commute with algebraic conjugation by $Ad(B(\. ))$.

This translation deteriorates the estimates for the rest obtained in proposition \ref{solution of difference equation}
for $C^{s}$ norms by a factor $(b_{0}N_{i})^{s+1/2}$, while the rest of the estimates remain as in the proposition.

This last part is proved in more detail in lemma \ref{translated difference eq} in the appendix, for the sake
of completeness.
\end{proof}

We remark that, by using smallness assumptions on the $C^{1}$ norm of the
perturbation, we can show that the significant part of the obstruction is
its projection on $g_{0}$, i.e. that $ObU(\. )-Ob_{0}U(\. )$ is of the
order of $U^{\prime }(\. )$. Therefore, the non-constant part of the rest
in the complement of $g_{0}$ can be incorporated to $U^{\prime }(\. )$,
thus giving the cocycle $(\a ,E_{r}(\. ).A.e^{Ob_{0}U(\.
)}e^{U^{\prime }(\. )})$ where%
\begin{equation*}
(\a ,E_{r}(\. ).A.e^{Ob_{0}U(\. )})=Conj_{B(\. )}(\a
,E_{r}(\. ).B(\a )A.\exp (\frac{1}{D_{0}}\int_{0}^{D_{0}}Ad(B(\.
))Ob_{0}U(\. )))
\end{equation*}%
so that the following corollary is true.
\begin{corollary} \label{corollary iteration sing geod}
Let $\a \in DC(\gamma ,\tau )$, $A \in \mathcal{Z} = Z_{G}(\exp (\R e_{r} )$,
$(\a ,E_{r}(\. ).A.e^{U(\. )})\in SW^{\infty }(\T ,G)$ be of degree $r $, and
$K\geq 2^{\tau +1}\gamma \tilde{b}_{0}^{q_{0}+1}N^{\tau }$. Let, also,
\begin{equation*}
c_{1,0}K(N_{q_{0}}^{u})^{3}\e _{1}  < 1
\end{equation*}
where $\e _{s}=\left\Vert U\right\Vert _{s}$. Then, there exists $%
0\leq i\leq q_{0}$ and a conjugation $G(\. )=C(\. )B(\. )e^{Y(\.
)}\in C^{\infty }(\T ,g)$ such that
\begin{equation*}
G(\. +\a )E_{r}(\. )Ae^{U(\. )}G^{\ast }(\. )=E_{r}(\.
)A^{\prime }e^{U^{\prime }(\. )}
\end{equation*}
where also $A' \in \mathcal{Z}$. The conjugant $G(\. )$ satisfies the estimates
\begin{equation*}
\left\Vert G(\. )\right\Vert _{s}\leq
c_{1,s}(K(N_{i}^{u})^{s+1/2}\e _{0}+N_{i}^{s+1/2})
\end{equation*}%
Finally, if we call
\begin{eqnarray*}
\tilde{\e }_{s}= \| \partial U (\. ) \|_{s} \\
\tilde{\e }^{\prime }_{s}= \| \partial U^{\prime } (\. ) \|_{s} \\
\e _{0}^{\prime}= \| U' (\. ) \|_{0}
\end{eqnarray*}
we have the following estimates:
\begin{eqnarray*}
\e ^{\prime} _{0} & \leq & C_{2} \tilde{\e }^{\prime } _{0} \\
\tilde{\e }^{\prime } _{s} & \leq &
c_{2,s}K^{2}(N_{i}^{u})^{2}((N_{i}^{u})^{s} \tilde{\e } _{0}+
\tilde{\e }_{s}) \tilde{\e } _{0}+C_{s,s^{\prime }}(N_{i+1}^{l})^{s-s^{\prime
}+2} \tilde{\e } _{s^{\prime }}
\end{eqnarray*}
\end{corollary}
We note that the smallness condition is now imposed on $\e_{1}$, so that the a priori estimates can be
applied. If we admit an a priori bound on $\e _{0}$, the $\e_{1}$ in the smallness condition can be
substituted by $\tilde{\e }_{0}$. We also note that the quantities with tildes are the norms of the
\textit{derivative} of the perturbation.
\begin{proof}
We need only to show that the positive energy part of the obstruction $\p _{+ } Ob U(\. )$ is of
second order with respect to $U(\. )$. By construction of the mappings,
\begin{eqnarray*}
B(\. +\a )e^{Y(\. +\a )}E_{r}(\. )Ae^{U(\. )}e^{-Y(\.
)}B^{\ast }(\. )&=& E_{r}(\. )Ae^{Ad(B(\. ))ObU(\. )}e^{\tilde{U}^{\prime }(\. )} \\
&=& E_{r}(\. )A'e^{\Lambda _{D_{0} \tilde{r}}^{0}Ad(B(\. )))ObU(\. )}e^{\tilde{U}^{\prime }(\. )}
\end{eqnarray*}
where $A'=A\exp (Ad(B(\. )).Ob_{0}U(\. ) ) \in \mathcal{Z}$ is a constant, and
$\delta _{s}^{\prime } = \| \tilde{U}^{\prime }(\. ) \| _{s}$ satisfies
\begin{equation*}
\delta _{s}^{\prime }\leq
c_{2,s}K^{2}(N_{i}^{u})^{2}((N_{i}^{u})^{s}\e _{0}+\e
_{s})\e _{0}+C_{s,s^{\prime }}(N_{i+1}^{l})^{s-s^{\prime
}+2}\e _{s^{\prime }}
\end{equation*}
Under the assumption that $\delta _{1}^{\prime } \leq C$, where $C \ll 1$ is a fixed positive number, condition
which is implied by the smallness assumption of the statement, $c_{1,0}K(N_{q_{0}}^{u})^{3}\e _{1}  < 1$,
the a priori estimates (lemma \ref{A priori estimates sing geod}) are applicable and give
\begin{equation*}
\| \Lambda _{D_{0} r}^{0} Ad(B(\. ))Ob_{+} U(\. ) \|_{0} \leq C_{1} \| \partial \tilde{U} ^{\prime} (\.) \| _{0}
\end{equation*}
Therefore, if we call
\begin{equation*}
U^{\prime } (\. )= \exp ^{-1} ( e^{Ad(B(\. ))ObU(\. )}e^{\tilde{U}^{\prime }(\. )} )
\end{equation*}
and
\begin{equation*}
\tilde{\e }_{s}= \| \partial U (\. ) \|_{s} \text{  and }
\tilde{\e }^{\prime }_{s}= \| \partial U^{\prime } (\. ) \|_{s}
\end{equation*}
we have
\begin{eqnarray*}
\e ^{\prime} _{0} & \leq & C_{2} \tilde{\e }^{\prime } _{0} \\
\tilde{\e }^{\prime } _{s} & \leq &
c_{2,s}K^{2}(N_{i}^{u})^{2}((N_{i}^{u})^{s} \tilde{\e } _{0}+
\tilde{\e }_{s}) \tilde{\e } _{0}+C_{s,s^{\prime }}(N_{i+1}^{l})^{s-s^{\prime
}+2} \tilde{\e } _{s^{\prime }}
\end{eqnarray*}
where we have used the a priori bound for $\e _{0}$.

Finally, conjugation by $C(\. )$ deteriorates the estimations by multiplicative constants and regains periodicity.
\end{proof}

\section{Iteration and conclusion}
\subsection{Choice of constants}
We now fix some notation similar to the one in the previous chapter. We call
\begin{eqnarray*}
\tilde{\e }_{n,s}= \| \partial U_{n} (\. ) \|_{s} \\
\e _{n,0} = \| U_{n} (\. ) \|_{0}
\end{eqnarray*}
We also define the sequence of orders of truncation $N_{n}$ and the diophantine constants $K_{n}$ as before
\begin{eqnarray*}
K_{n} &=& N_{n}^{\t + \nu}  \\
N_{n} &=&N^{(1+\sigma )^{n-1}} \\
N_{n,i}^{l} &\geq &N_{n} \\
N_{n,i}^{u} &\leq &2\tilde{b}^{q+1}N_{n}
\end{eqnarray*}
where for $\nu >0$ fixed, $K_{n}\geq 2^{\tau +1}\gamma \tilde{b}_{0}^{q_{0}+1}N_{n}^{\tau } $ provided that
$N$ is big enough.

\subsection{Estimates} \label{Estimates KAM local theory}

We now wish to apply the corollary \ref{corollary iteration sing geod} to the cocycle
$(\a  ,E_{r}(\. ).A.e^{U(\.  )})=(\a  ,E_{r}(\. ).A_{1}.e^{U_{1}(\.  )})\in SW^{\infty }(\T ,G)$
where $A \in Z_{G}(\exp (\R e_{r} ))$ in order to construct inductively a sequence of $1$-periodic conjugations
\begin{eqnarray*}
L_{n}(\.  ) &=&G_{n}(\.  )...G_{1}(\.  ) \\
&=&C_{n}(\.  )e^{-H_{n}\.  }e^{Y_{n}(\.  )}...C_{1}(\. 
)e^{-H_{1}\.  }e^{Y_{1}(\.  )}
\end{eqnarray*}
such that
\begin{equation*}
(\a  ,E_{r}(\. ).A_{1}.e^{U_{1}(\.  )})=Conj_{L_{n}(\.  )}(\a E_{r}(\. ).A_{n}.e^{U_{n}(\.  )})
\end{equation*}
with $A_{n} \in Z_{G}(\exp (\R e_{r} ))$ and $U_{n}(\.  ) \ra 0$ in $C^{\infty}$ faster than the rate of
growth of $L_{n}$. At each step, $Y_{n}(\.  )$ is to be given by the local conjugation lemma
\ref{local reduction singular geodesics}, $B(\. )$ reduces the obstructions in $g_{0 }$, and $C(\. )$
regains periodicity of the conjugants.

Let us therefore express the inductive estimates and smallness conditions at the $n$-th step as a
function of $N_{n}$. Firstly, we admit a bound $\| U_{1} \|_{0} \leq C_{0} \ll 1$, which also has to
be proven uniform for all $n$. If this is the case, the lemma of local reduction is applicable at the step $n$ if
\begin{equation*}
c_{1,1}K_{n}(N_{n,q}^{u})^{3}\tilde{\e }  _{n,0}<1
\end{equation*}
or, with a different constant $c_{1,1}$,
\begin{equation} \label{smallness condition sing geod}
c_{1,1}N_{n}^{ \t +\nu +3}\tilde{\e }  _{n,0} < 1
\end{equation}
We find that
\begin{eqnarray*}
\tilde{\e } _{n+1,s} & \leq &
c_{2,s}K_{n}^{2}(N_{n,q_{0}}^{u})^{2}((N_{n,q_{0}}^{u})^{s} \tilde{\e } _{n,0}+
\tilde{\e } _{n,s}) \tilde{\e } _{n,0}+C_{s,s^{\prime }}(N_{n,q_{0}+1}^{l})^{s-s^{\prime}+2} \tilde{\e } _{n,s^{\prime }} \\
\e _{n+1,0} & \leq & C_{2} \tilde{\e } _{n+1,0}
\end{eqnarray*}
or, since $N_{n,q_{0}}^{u}$ and $N_{n,q_{0}}^{l}$ are of the order of $N_{n}$,
\begin{eqnarray*}
\tilde{\e } _{n+1,s} & \leq &
c_{2,s}K_{n}^{2}(N_{n})^{2}((N_{n})^{s} \tilde{\e } _{n,0}+
\tilde{\e } _{n,s}) \tilde{\e } _{n,0}+C_{s,s^{\prime }}(N_{n})^{s-s^{\prime}+2} \tilde{\e } _{n,s^{\prime }} \\
\e _{n+1,0} & \leq & C_{2} \tilde{\e } _{n+1,0}
\end{eqnarray*}
or, in the form of the estimates as they appear in theorem \ref{FK scheme}
\begin{equation} \label{inductive estimates}
\tilde{\e } _{n+1,s} \leq C_{s, \overline{s}} ( N_{n}^{2\t +2\nu +2+s} \tilde{\e } _{n,0}^{2}
+N_{n}^{2\t +2\nu + 2} \tilde{\e } _{n,0} \tilde{\e } _{n,s} + N_{n}^{2-(\overline{s}-s) }
\tilde{\e } _{n,\overline{s}} )
\end{equation}

Therefore, the estimates satisfy the hypotheses of theorem \ref{FK scheme} with
\begin{eqnarray*}
a&=& 2\t +2\nu +3 \\
M&=&1 \\
\sigma _{0}&=&1 \\
m&=&0 \\
\mu &=&1 \\
\overline{\mu } &=& 0
\end{eqnarray*}
and with these choices, we can determine $\sigma $ as in theorem \ref{FK scheme}.

\subsection{Iteration and convergence}

A choice of a big enough $N \geq \overline{C} _{s_{0}}$ implies that the smallness condition
\ref{smallness condition sing geod} is satisfied provided that $\tilde{\e } _{0,0} \leq N ^{-\frac{s_{0}}{g}}$.
Therefore, the scheme can be initiated, and the remark following theorem \ref{FK scheme} justifies
the application of corollary \ref{corollary iteration sing geod} at each step, and consequently the
inductive estimates in eq. \ref{inductive estimates} are valid. It is crucial that at each step the
estimate $\e _{n,0}  \leq  C_{2} \tilde{\e } _{n,0}$ implies that the $C^{0}$ norm, which we do
not control directly dies out along with the $C^{0}$ norm of the first derivative, so that the use
of the a priori estimates is justified.

Therefore, theorem \ref{FK scheme} is applicable and gives that, for all $s\in \N$,
\begin{equation*}
\tilde{\e }  _{n,s}=O(N_{n}^{-\infty })
\end{equation*}
which implies that $\e   _{n,s}=O(N_{n}^{-\infty })$ as well. The density theorem and the accumulation by
cocycles of smaller energy are concluded as before, since conjugations grow only polynomially with $N_{n}$.

\section{Proof of theorem \ref{local almost reducibility to singular geodesics}}

Finally, a word on the proof of the local almost reducibility and
local density of cocycles reducible to the normal forms.

These theorems are obtained in the same way as in their counterparts
for perturbations of constant cocycles. The K.A.M. scheme whose
convergence we proved in the preceding section constructs a sequence
of conjugations that conjugate a given cocycle
$(\a ,E_{r}(\. )Ae^{U(\. )})$ ever closer to one of the form
$(\a ,E_{r}(\. )A')$, on the assumption of a smallness condition on
$U(\. )$ and provided that
\begin{equation*}
deg (\a ,E_{r}(\. )Ae^{U(\. )}) = r
\end{equation*}
Therefore, local almost reducibility of all such cocycles
to normal forms is an immediate corollary of the convergence of the
scheme. Finally, the given cocycle is conjugated exponentially close
to the normal form, whereas the norms of conjugations grow only
polynomially. This proves the local density theorem.


\selectlanguage{english}
\chapter{Appendix} \label{appendix 1}
\section{Missing proofs of some lemmas}
We begin by the lemma concerning the abelian subsets of simply connected Lie groups.
\begin{proof}[Proof of Lemma \ref{parties abel. simpl. conn.}]
Let $A,B \in G $ be commuting elements and suppose that $\TT \ni A $ is a maximal torus, such that
$ Ad(B) . \TT = \TT $. Such a torus exists, by the proof of lemma \ref{parties abeliennes}. If $Ad(B)|_{\TT }= Id$,
then $\TT $ contains both $A$ and $B$, so the proof is finished. If not, let us consider
$\Phi (\. ) : [ 0,1 ]  \ra \TT $ be a geodesic of minimal length connecting the $Id$ with $A$ and such that
$\tilde{\Phi}  = Ad(B).\Phi : [ 0,1 ]  \ra \TT $ intersects $\Phi $ only at its extremities. This amounts to
considering $a \in \ft $ ($\ft $ is the Lie algebra of $\TT $), a preimage of $A$, such that $Ad(B).a \neq a $.
This preimage satisfies, however, $Ad(B).a \in \ft $. This implies that $a$ decomposes in $ a_{1} + a_{2} $,
both in $\ft $, such that $Ad(B).a_{1} = a_{1}$ (and $a_{1}$ is possibly $0$), and $Ad(B).a_{2} \neq a_{2}$,
but $a_{2} \neq 0$. Since $A$ and $B$ commute, 
\begin{equation*}
e = a - Ad(B).a = (a_{1} + a_{2}) - Ad(B).(a_{1} + a_{2}) = a_{2} - Ad(B). a_{2} \neq 0
\end{equation*}
is a preimage of the $Id$. Since $e+Ad(B).e =  a_{2} - Ad(B^{2}). a_{2} $, and since $Ad(B)|_{\TT } $ is of finite order $m $,
$\sum \nolimits _{0}^{m } Ad(B^{i}).e =0 $. We call $\x $ the minimal positive integer for which this sum is $0$. We also
remark that $\langle e , a_{2} \rangle = - \langle e , Ad(B).a_{2} \rangle$, so that if we call $a' $ the projection
of $a_{2}$ parallel to $e$, we have $\exp (2a' ) = Id$.

If we assume, now, that $G $ is simply connected, the paths $\Phi $ and $\tilde{\Phi }$ are homotopic with fixed
endpoints. In particular, there exists such a homotopy $\Psi _{t}(\. )$ such that for $t\in [0,1]$ fixed,
$\Psi _{t}(\. ) $ is a minimal geodesic connecting the $Id$ with $A$, $\Psi _{0}(\. ) = \Phi(\. )$, and
$\Psi _{1}(\. ) = \tilde{\Phi}(\. )$. It remains to show that we can find a homotopy $\Psi $ such that for
some fixed parameter value $t_{0}$, $Ad(B).L\Psi _{t_{0}}(\. ) = L\Psi _{t_{0}}(\. )$.

The fact that the homotopy $\Psi _{t}(\. )$ consists of geodesics implies that there exists a path
$B(\. ) : [0,1] \ra G $ such that $\Psi _{t}(\. ) = Ad(B(t )).\Psi _{0}(\. ) $. This path can be continuously
deformed to a path still denoted by $B(\. )$, such that $Ad(B(\. )).(a - a') = a-a' $. The corresponding homotopy
is still denoted by $\Psi$. This choice is possible, since
$Ad(B(0 )).a = a $. Therefore, for some parameter value $t_{0}$, $Ad(B.B(\. )).a' = a' $, and there exists
a maximal torus passing both by $A$ and $B$, since $\Psi _{t_{0}}(1 ) = A$ and
$Ad(B).L\Psi _{t_{0}}(\. ) = L\Psi_{t_{0}}(\. )$.

If the group is not simply connected, the  argument remains valid if we
replace $B$ by $B^{\x}$, with $\x\leq m $.

\end{proof}

We now give the proof of the lemma on Lie brackets of Birkhoff sums used in the proof of lemma \ref{higher order}:

\begin{lemma}
\label{summation}Let $u\in L^{2}(\T ,\R ^{d})$, $\mathcal{U}$
a unitary operator in $L^{2}(\T ,\R ^{d})$, and $[\.  ,\. 
]$ a bilinear antisymmetric application on $\R ^{d} \times \R ^{d}$, invariant by $%
\mathcal{U}$: $[\mathcal{U}\.  ,\mathcal{U}\.  ]=\mathcal{U}.[\. 
,\.  ]$. Then,%
\begin{equation*}
\frac{1}{n^{2}}\sum\limits_{k=0}^{n}\sum\limits_{l=0}^{k}[\mathcal{U}^{k}u,%
\mathcal{U}^{l}u]\ra  0\text{ in }L^{1}(\T ,\R ^{d})
\end{equation*}
\end{lemma}

\begin{proof}
Let%
\begin{equation*}
\phi =\lim \frac{1}{n}\sum\limits_{k=0}^{n}\mathcal{U}^{k}u
\end{equation*}%
which is the projection of $u$ on the subspace of $U$-invariant functions,
and call $\psi =u-\phi $. We have immediately%
\begin{equation*}
\lim \frac{1}{n}\sum\limits_{k=0}^{n}\mathcal{U}^{k}\psi =0
\end{equation*}%
in $L^{2}$. By the antisymmetry of the bracket,%
\begin{equation}
\sum\limits_{k=0}^{n}\sum\limits_{l=0}^{k}[\mathcal{U}^{k}u,\mathcal{U}%
^{l}u]=\sum\limits_{k=0}^{n}k[\mathcal{U}^{k}\psi ,\phi
]+\sum\limits_{k=0}^{n}(n-k)[\phi ,\mathcal{U}^{k}\psi
]+\sum\limits_{k=0}^{n}\sum\limits_{l=0}^{k}[\mathcal{U}^{k}\psi ,\mathcal{%
U}^{l}\psi ]  \label{sums}
\end{equation}

If we call $s_{n}=\| \sum\limits_{k=0}^{n}\mathcal{U}^{k}\psi \| _{L^{2}}$, we can rearrange the terms
in $\sum \limits_{k=0}^{n}(n-k)\mathcal{U}^{k}\psi $ and obtain%
\begin{equation*}
\| \sum\limits_{k=0}^{n}(n-k)\mathcal{U}^{k}\psi \| _{L^{2}}\leq \sum\limits_{k=0}^{n}s_{k}
\end{equation*}%
If, for $\e  >0$ we let $N>1$ such that for $n\geq N$, $s_{n}\leq
n\e  $, then we have%
\begin{eqnarray*}
\sum\limits_{k=0}^{n}s_{k}
&=&\sum\limits_{k=0}^{N}s_{k}+\sum\limits_{k=N}^{n}s_{k} \\
&\leq &\frac{1}{2}N(N+1)\left\Vert \psi \right\Vert _{L^{2}}+\frac{%
\e  }{2}n(n+1)
\end{eqnarray*}%
Choosing $n$ big enough ($n\gg N$), we can make $\frac{1}{n^{2}}%
\sum\limits_{k=0}^{n}s_{k}$ smaller than $\e  $.

Since $\sum\limits_{k=0}^{n}k\mathcal{U}^{k}\psi
+\sum\limits_{k=0}^{n}(n-k)\mathcal{U}^{k}\psi =n\sum\limits_{k=0}^{n}%
\mathcal{U}^{k}\psi $, we find that $\frac{1}{n^{2}}\sum\limits_{k=0}^{n}k%
\mathcal{U}^{k}\psi $ also converges to $0$. Consequently, the first two
sums in eq. (\ref{sums}) converge to $0$ in $L^{1}$.

Now, we estimate, with a similar choice of $\e  $ and $N$, $\frac{1}{%
(nN)^{2}}\sum\limits_{k=0}^{nN}\sum\limits_{l=0}^{k}[\mathcal{U}^{k}\psi ,%
\mathcal{U}^{l}\psi ]$ in $L^{1}$. The summation can be decomposed into
summation over $\frac{1}{2}n(n-1)$ squares of the type%
\begin{equation*}
\sum\limits_{k=jN}^{(j+1)N-1}\sum\limits_{l=iN}^{(i+1)N-1}[\mathcal{U}%
^{k}\psi ,\mathcal{U}^{l}\psi ]=[\sum\limits_{k}\mathcal{U}^{k}\psi
,\sum\limits_{l}U^{l}\psi ],~1\leq i,j\leq n-1
\end{equation*}%
which, by invariance of the norms by $U$, we can bound by%
\begin{equation*}
\| \sum\limits_{k=jN}^{(j+1)N-1}\sum\limits_{l=iN}^{(i+1)N-1}[%
\mathcal{U}^{k}\psi ,\mathcal{U}^{l}\psi ]\| _{L^{1}}\lesssim N^{2}\e 
\end{equation*}%
and $n$ triangles of the type%
\begin{equation*}
\mathcal{U}^{jN}\sum\limits_{k=0}^{N}\sum\limits_{l=0}^{k}[\mathcal{U}%
^{k}\psi ,\mathcal{U}^{l}\psi ],~1\leq j\leq n-1
\end{equation*}%
which can be bounded by%
\begin{equation*}
\| \sum\limits_{k=0}^{N}\sum\limits_{l=0}^{k}[\mathcal{U}^{k}\psi ,%
\mathcal{U}^{l}\psi ]\| _{L^{1}}\lesssim \frac{1}{2}%
N(N-1)\left\Vert \psi \right\Vert _{L^{2}}
\end{equation*}%
Thus,
\begin{equation*}
\frac{1}{(nN)^{2}}\| \sum\limits_{k=0}^{nN}\sum\limits_{l=0}^{k}[%
\mathcal{U}^{k}\psi ,\mathcal{U}^{l}\psi ]\| _{L^{1}}\lesssim \frac{%
\left\Vert \psi \right\Vert _{L^{2}}}{2n}+\frac{\e  }{2}
\end{equation*}%
and an appropriate choice of $n$ gives the desired estimate.
\end{proof}

The following, used in the proof of proposition \ref{convergence}, can be
also found in \cite{Fra2004} (lemma 2.4).

\begin{lemma}
\label{egorov}Let $(f_{n}(\.  ))$ be a sequence of non-negative, uniformly
bounded measurable functions, such that $f_{n}(\.  )\ra  0$ almost
surely with respect to the lebesgue measure $\mu $. Then, if $(c_{n})$ is a
sequence of positive numbers converging to $0$, we have%
\begin{equation*}
\frac{1}{c_{n}}\int_{x}^{x+c_{n}}f_{n}(\.  )\ra  0,~a.s.(x)
\end{equation*}
\end{lemma}

\begin{proof}
By the Egorov theorem, for any $k\in \N ^{\ast }$, there exists a
closed set $B_{k}$ of Lebesgue measure at least $1-\frac{1}{k}$ such that%
\begin{equation*}
\lim \sup_{x\in B_{k}}f_{n}(x)=0
\end{equation*}
We denote by $A_{k}$ the Lebesgue density points of $B_{k}$, i.e. the points 
$x\in B_{k}$ for which%
\begin{equation*}
\lim _{\e \ra 0 ^{+}} \frac{meas([x-\e  ,x+\e  ])}{2\e  }=1
\end{equation*}%
This property is true for $a.e.$ point of $B_{k}$. Then, we find that for $%
x\in A_{k}$,%
\begin{eqnarray*}
\frac{1}{c_{n}}\int_{x}^{x+c_{n}}f_{n}(\.  ) &=&\frac{1}{c_{n}}%
\int_{[x,x+c_{n}]\cap A_{k}}f_{n}(\.  )+\frac{1}{c_{n}}\int_{[x,x+c_{n}]%
\backslash A_{k}}f_{n}(\.  ) \\
&\leq &\lim \sup_{x\in A_{k}}f_{n}(x)+M\frac{\mu ([x,x+c_{n}]\backslash
A_{k})}{c_{n}}
\end{eqnarray*}%
where $M$ is a uniform bound for $f_{n}(\.  )$. Since the right hand side
converges to $0$, the lemma is proved.
\end{proof}

\section{Lemma on conjugation close to singular geodesics}

Finally, we state and prove the lemma needed in lemma \ref{local reduction singular geodesics} in order
to quantify the influence of the translation of $\Gamma _{m}$ in the space of frequencies on the solution
of \ref{dif eq}. More precisely,
\begin{lemma} \label{translated difference eq}
Let $m \in \Z ^{*}$ and $g(\. ) \in C^{\infty} (\T, \C )$. Then, there exist positive constants $C_{s}$
such that, for any $N,N' \in \N$ with $0 \leq N' < N$, there exists $f (\. ) $, a solution to eq.
\begin{equation*}
T_{N} f(\. +\a )-e^{2i\pi mx\. } T_{N} f(\. )=g(\. )-\Gamma _{N', m}g(\. ) + r_{N} (\. )
\end{equation*}
with $\s (\Gamma _{N', m}g(\. ) ) \subset \{-m+1-N',..,-N'  \} $ satisfying the estimates
\begin{eqnarray*}
\| T_{N}f\| _{s} &\leq & C_{s}( N^{(u)})^{s+3} \| g\| _{s+3} \\
\| r_{N} \| _{s} &\leq & C_{s}( N^{(l)})^{s'-s+3} \| g\| _{s} \\
\| \Gamma _{N', m} g\| _{s} &\leq & C_{s} (N^{\prime })^{s+1/2 } \| g \|_{0}
\end{eqnarray*}
where $N^{(u)} = N+ N'$ and $N^{(l)} = N- N'$. The constants do not depend on $N$ and $N'$.
\end{lemma}

\begin{proof}
Let us revisit the formulas (\ref{forw.sol}) and (\ref{backw sol}) for the forward and backward solution,
respectively. The solution of the auxiliary equation
\begin{equation*}
\phi (\. +\a )-e^{2i\pi m\.} \phi (\.)= e^{-2i N' \p \. } g(\.)-\Gamma _{ m}(e^{-2i N' \p \. } g(\.))
\end{equation*}
gives the formulas
\begin{eqnarray*}
\hat{\phi }_{+}(k)=\sum_{j=1}^{\infty }e^{-2i\pi j(k+N'+\frac{j-1}{2}m)\a }\hat{g}(k+jm + N') \\
\hat{\phi }_{-}(k)=\sum_{j=0}^{\infty }e^{2i\pi j(k+N'-\frac{j+1}{2}m)\a } \hat{g}((k-j)m + N' )
\end{eqnarray*}
respectively for the Fourier coefficients of the forward and backward solution. If we call
$f= e^{2i N' \p \. } \phi $, then it satisfies
\begin{equation*}
f(\. +\a )-e^{2i\pi m\. } f(\. )=g(\. )-\Gamma _{N', m}g(\. )
\end{equation*}
with
\begin{equation*}
\Gamma _{N', m}g(\. ) = e^{2i N' \p \. } \Gamma _{ m}(e^{-2i N' \p \. } g(\.))
\end{equation*}
The forward and backward sums for $f(\. ) $ become
\begin{eqnarray*}
\hat{f} _{+}(k)=\sum_{j=1}^{\infty }e^{-2i\pi j(k+N'+\frac{j-1}{2}m)\a }\hat{g}(k+jm ) \\
\hat{f} _{-}(k)=\sum_{j=0}^{\infty }e^{2i\pi j(k + N'-\frac{(j+1)}{2}m)\a } \hat{g}((k-j)m )
\end{eqnarray*}
This implies the estimates of the lemma exactly as in the proof of \ref{solution of difference equation},
since $N'$ appears only in the factors $e^{-2i\pi j(k+N'+\frac{j-1}{2}m)\a }$ and
$e^{2i\pi j(k + N'-\frac{(j+1)}{2}m)\a }$ which do not affect the estimates, and in the translation
of the truncation operators. Moreover, inspection of the proof shows that if $N'$ is in $\frac{1}{D} \Z$,
but $g(\. )$ is $1$-periodic, the solution $f(\. )$ is in fact $1$-periodic.
\end{proof}

In the proof of lemma \ref{local reduction singular geodesics} we actually use a vector-valued version
of the lemma above, with the different $m$ uniformly bounded by the size of resonances in $G_{0}$.

\section{Estimate on the Hausdorff-Campbell formula without linear terms} \label{estimate of quadratics}

The following estimate is useful in K.A.M. theory, since the latter passes by the solution
of liearized equations coming from non-linear problems. When we use these approximate solutions
to attack the initial problem, we have to estimate the series of the Hausdorff-Campbell formula
(see \cite{Helg}) without constant and linear terms. This is done in the appendix of \cite{KrikAst},
or in \cite{FK2009} (eq. 6.1 and 6.2) as follows.

If $Q$ is quadratic in $(f, g)$, i.e. if $Q$ is $C^{2}$, $Q(0, 0) = 0$ and $DQ(0, 0) = 0$, we have
\begin{equation*}
\| Q(f, g) \| _{s} \leq C_{s} (1+ \| f \| _{0} + \| g \| _{0})^{s+1} (\| f \| _{0} + \| g \| _{0})
(\| f \| _{s} + \| g \| _{s})
\end{equation*}
which simplifes to
\begin{equation*}
\| Q(f, g) \| _{s} \leq C_{s} (\| f \| _{0} + \| g \| _{0})
(\| f \| _{s} + \| g \| _{s})
\end{equation*}
if we admit a priori bounds for $\| f \| _{0}$ and $\| g \| _{0}$.


\bibliographystyle{smfalpha} 
\bibliography{bibliothese} 



\cleardoublepage 
\renewcommand{\indexname}{Index des notations} 
\printindex 

\end{document}